\documentclass[11pt]{amsart}
\usepackage[utf8]{inputenc}
\usepackage[T1]{fontenc}
\usepackage{lmodern}
\usepackage{microtype}
\usepackage[leqno]{amsmath}
\usepackage{amssymb}
\usepackage{mathtools}
\usepackage{enumerate}
\usepackage{graphicx}
\usepackage[usenames,dvipsnames]{xcolor}
\usepackage[
colorlinks=true,linkcolor=NavyBlue,urlcolor=RoyalBlue,citecolor=PineGreen,bookmarks=true,bookmarksdepth=3,bookmarksopen=true,bookmarksopenlevel=2,unicode=true,linktocpage]{hyperref}

\numberwithin{equation}{section}
\numberwithin{figure}{section}

\newtheorem{theorem}{Theorem}[section]
\newtheorem{remark}[theorem]{Remark}
\newtheorem{lemma}[theorem]{Lemma}
\newtheorem{proposition}[theorem]{Proposition}
\newtheorem{corollary}[theorem]{Corollary}

\newtheorem{definition}[theorem]{Definition}
\newtheorem{example}[theorem]{Example}

\let\C\relax

\newcommand{\C}{\mathbf{C}}
\newcommand{\D}{\mathbf{D}}
\newcommand{\E}{\mathbf{E}}

\newcommand{\h}{\mathbf{H}}
\newcommand{\N}{\mathbf{N}}
\newcommand{\Z}{\mathbf{Z}}
\newcommand{\p}{\mathbf{P}}
\newcommand{\Q}{\mathbf{Q}}
\newcommand{\R}{\mathbf{R}}

\newcommand{\Fh}{\mathfrak {h}}
\newcommand{\Fg}{\mathfrak {g}}

\newcommand{\CA}{\mathcal {A}}

\newcommand{\CC}{\mathcal {C}}
\newcommand{\CD}{\mathcal {D}}
\newcommand{\CE}{\mathcal {E}}
\newcommand{\CF}{\mathcal {F}}
\newcommand{\CI}{\mathcal {I}}

\newcommand{\CK}{\mathcal {K}}
\newcommand{\CL}{\mathcal {L}}

\newcommand{\CT}{\mathcal {T}}

\newcommand{\CW}{\mathcal {W}}
\newcommand{\CX}{\mathcal {X}}
\newcommand{\CZ}{\mathcal {Z}}

\newcommand{\CH}{\mathcal {H}}

\newcommand{\SLE}{{\rm SLE}}
\newcommand{\CLE}{{\rm CLE}}

\newcommand{\dist}{\mathrm{dist}}

\newcommand{\diam}{\mathrm{diam}}

\newcommand{\im}{\mathrm{Im}}
\newcommand{\re}{\mathrm{Re}}

\newcommand{\one}{{\bf 1}}

\newcommand{\wt}{\widetilde}
\newcommand{\wh}{\widehat}
\newcommand{\ol}{\overline}
\newcommand{\ul}{\underline}

\newcommand{\giv}{\,|\,}

\newcommand{\BCLE}{\mathrm{BCLE}}

\newcommand{\cwBCLE}{\BCLE^{\boldsymbol {\circlearrowright}}}

\newcommand{\quant}[3][]{{{\mathfrak q}}_{#3}\if\relax\detokenize{#1}\relax\else^{#1}\fi(#2)}
\newcommand{\median}[2][]{{\mathfrak m}_{#2}\if\relax\detokenize{#1}\relax\else^{#1}\fi}
\newcommand{\mediant}[2][]{\wt{\mathfrak m}_{#2}\if\relax\detokenize{#1}\relax\else^{#1}\fi}

\newcommand{\lebneb}[1]{{\mathfrak N}_{#1}}

\newcommand{\Fd}{\mathfrak d}
\newcommand{\met}[3]{\Fd(#1,#2;#3)}

\newcommand{\metres}[4]{\Fd^{#1}(#2,#3;#4)}

\newcommand{\metapprox}[4]{\Fd_{#1}(#2,#3;#4)}
\newcommand{\mettapprox}[4]{\wt{\Fd}_{#1}(#2,#3;#4)}
\newcommand{\metapproxres}[5]{\Fd_{#1}^{#2}(#3,#4;#5)}

\DeclarePairedDelimiter\abs{\lvert}{\rvert}
\DeclarePairedDelimiter\norm{\lVert}{\rVert}

\newcommand*{\defeq}{\mathrel{\mathop:}=}
\newcommand*{\eqdef}{=\mathrel{\mathop:}}

\newcommand*{\mmiddle}[1]{\mathrel{}\middle#1\mathrel{}}

\newcommand*{\wtwt}[1]{\widetilde{\raisebox{0pt}[0.9\height]{$\widetilde{#1}$}}}
\newcommand*{\whwh}[1]{\widehat{\raisebox{0pt}[0.9\height]{$\widehat{#1}$}}}
\newcommand*{\whwt}[1]{\widehat{\raisebox{0pt}[0.9\height]{$\widetilde{#1}$}}}

\newcommand*{\Cont}{\operatorname{Cont}}
\newcommand*{\Fill}{\operatorname{fill}}

\newcommand*{\sle}[1]{$\SLE_{#1}$}
\newcommand*{\slek}{\sle{\kappa}}
\newcommand*{\slekp}{\sle{\kappa'}}
\newcommand*{\slekr}[1]{$\SLE_{\kappa}(#1)$}
\newcommand*{\cle}[1]{$\CLE_{#1}$}
\newcommand*{\clek}{\cle{\kappa}}
\newcommand*{\clekp}{\cle{\kappa'}}

\newcommand{\dGH}{d_{\mathrm {GH}}}
\newcommand{\dH}{d_{\mathrm {H}}}
\newcommand{\dGHf}{d_{\mathrm {GHf}}}
\newcommand{\dP}{d_{\mathrm {P}}}

\newcommand*{\metregions}[1][]{\mathfrak{C}\if\relax\detokenize{#1}\relax\else_{#1}\fi} 
\newcommand*{\cserial}{c_{\mathrm{s}}} 
\newcommand*{\cparallel}{c_{\mathrm{p}}} 
\newcommand*{\ac}[1]{\mathbf{a}_{#1}} 
\newcommand*{\act}[1]{\wt{\mathbf{a}}_{#1}}
\newcommand{\metapproxac}[4]{\Fd_{#1,\mathbf{a}_{#1}}(#2,#3;#4)}
\newcommand{\metapproxacres}[5]{\Fd_{#1,\mathbf{a}_{#1}}^{#2}(#3,#4;#5)}

\newcommand{\mettapproxacres}[5]{\wt{\Fd}_{#1,\wt{\mathbf{a}}_{#1}}^{#2}(#3,#4;#5)}

\newcommand*{\distE}{\operatorname{dist_E}}
\newcommand*{\diamE}{\operatorname{diam_E}}
\newcommand*{\dpath}[1][]{d_{\mathrm{path}}\if\relax\detokenize{#1}\relax\else^{#1}\fi} 

\newcommand*{\len}[2]{L_{#1}(#2)}
\newcommand*{\lenmetres}[2]{\len{\Fd^{#1}}{#2}}
\newcommand*{\paths}[4]{P(#1,#2;#3;#4)}
\newcommand*{\spaths}[4]{P_s(#1,#2;#3;#4)}

\newcommand*{\dsle}{{d_{\SLE}}}
\newcommand*{\double}{{\mathrm{dbl}}}
\newcommand*{\ddouble}{{d_{\double}}} 
\newcommand*{\angledouble}{{\theta_{\double}}} 

\newcommand{\bestexp}{{\alpha_{\mathrm{apr}}}}
\newcommand{\bubbleexp}{{\alpha_{\mathrm{bubble}}}}

\newcommand{\GFF}{{\mathrm{GFF}}}
\newcommand{\bad}{{\mathrm{bad}}}
\newcommand{\bubble}{{\mathrm{bubble}}}

\newcommand*{\pmed}{p_{\mathrm{int}}}
\newcommand*{\intpts}[1]{\CX^{\mathrm{int}}_{#1}} 
\newcommand*{\intptsapprox}[2]{\CX^{\mathrm{int}}_{#1,#2}}
\newcommand*{\intptsapproxbubble}[2]{\CX^{\mathrm{int}}_{#1,#2}}

\begin{document}

\title[Tightness of metrics for non-simple CLEs]{Tightness of approximations to metrics on non-simple conformal loop ensemble gaskets}

\author{Valeria Ambrosio, Jason Miller, and Yizheng Yuan}

\begin{abstract}
We study a class of approximation schemes aimed at constructing conformally covariant metrics defined in the gasket of a conformal loop ensemble (CLE$_\kappa$) for $\kappa \in (4,8)$. This is the range of parameter values so that the loops of a CLE$_\kappa$ intersect themselves, each other, and the domain boundary. Its gasket is the closure of the union of the set of points not surrounded by a loop. The class of approximation schemes includes approximations to the geodesic metric and to the resistance metric. We show that the laws of these approximations are tight, and that every subsequential limit is a non-trivial metric on the CLE$_\kappa$ gasket satisfying a natural list of properties. Subsequent work of the second two authors will show that the limits exist and are conformally covariant both in the setting of the geodesic and resistance metrics. We conjecture that the geodesic (resp.\ resistance) metric describes the scaling limit of the chemical distance (resp.\ resistance) metric associated with discrete models that converge in the limit to CLE$_\kappa$ for $\kappa \in (4,8)$ (e.g., critical percolation for $\kappa=6$).
\end{abstract}

\date{\today}
\maketitle

\setcounter{tocdepth}{1}

\tableofcontents

\parindent 0 pt
\setlength{\parskip}{0.20cm plus1mm minus1mm}

\section{Introduction}

\subsection{Overview}

The \emph{Schramm-Loewner evolution} ($\SLE_\kappa$; $\kappa \geq 0$) is the canonical model of a conformally invariant non-crossing random planar curve which lives in a simply connected domain.  It was introduced by Schramm in \cite{s2000sle} as the candidate for the scaling limit of the loop-erased random walk and since its introduction it has been conjectured to describe the scaling limit of the interfaces that arise in many different lattice models in two dimensions.  Such convergence results have now been proved in a number of cases \cite{s2001percolation,ss2009dgff,s2010ising,lsw2004lerw}.  The \emph{conformal loop ensembles} ($\CLE_\kappa$; $\kappa \in [8/3,8]$) were introduced in \cite{s2009cle,sw2012cle} and are the loop variant of $\SLE_\kappa$.  A $\CLE_\kappa$ consists of a countable collection of non-crossing loops in a simply connected domain each of which looks locally like an $\SLE_\kappa$.  The $\CLE_\kappa$'s are conjectured to describe the joint scaling limit of all of the interfaces in many discrete lattice models in two dimensions and such convergence results have been established in a number of cases \cite{bh2019ising,cn2006cle,ks2019fkising,lsw2004lerw}.  We remark that a number of convergence results towards $\SLE_\kappa$ and $\CLE_\kappa$ have been proved in the setting of random lattices \cite{s2016hc,kmsw2019bipolar,gm2021percolation,gm2021saw,gkmw2018active,lsw2017schnyder} using the framework developed in \cite{s2016zipper,dms2021mating}.

The behavior of an $\SLE_\kappa$ becomes increasingly fractal as the value of $\kappa > 0$ increases.  In particular, $\SLE_\kappa$ curves are simple for $\kappa \in (0,4]$, self-intersecting but not space-filling for $\kappa \in (4,8)$, and are space-filling for $\kappa \geq 8$ \cite{rs2005basic}.  Moreover, the Hausdorff dimension of the range of an $\SLE_\kappa$ was computed in \cite{rs2005basic,b2008dimension} and is given by $\min(1+\kappa/8,2)$.

A $\CLE_\kappa$ is a countable collection of loops in a simply connected domain which do not cross themselves or each other and which locally look like an $\SLE_\kappa$.  They are defined for $\kappa \in [8/3,8]$ where at the extreme $\kappa=8/3$ it consists of the empty collection of loops and at the other extreme $\kappa=8$ it corresponds to a single space-filling loop.  When $\kappa \in (8/3,4]$, the loops are simple and do not intersect the domain boundary or each other while for $\kappa \in (4,8)$ the loops intersect themselves, each other, and the domain boundary.   If $\Gamma$ is a $\CLE_\kappa$ then the \emph{carpet} (resp.\ \emph{gasket}) of $\Gamma$ is the set of points which are not surrounded by a loop of $\Gamma$ for $\kappa \in (8/3,4]$ (resp.\ $\kappa \in (4,8)$).  The reason for the difference in terminology is that the case $\kappa \in (8/3,4]$ can be thought of as corresponding to a random analog of the Sierpi\'nski carpet (since the loops do not intersect each other) while the case $\kappa \in (4,8)$ can be thought of as corresponding to a random analog of the Sierpi\'nski gasket (since the loops do intersect each other). For $\kappa \in (4,8)$, the \emph{thin gasket} of $\Gamma$ is defined to be the set of points which are not disconnected from the boundary of the domain by any loop in $\Gamma$. The dimension of the $\CLE_\kappa$ carpet/gasket was computed in \cite{ssw2009radii,msw2014dimension,nw2011soup}, and the conjectural dimension of the thin gasket was recently computed in \cite{nqsz-backbone}.  The canonical conformally covariant volume measure on the $\CLE_\kappa$ carpet/gasket was constructed in \cite{ms2022clemeasure} and gives the conjectured scaling limit of the properly renormalized measure which assigns unit mass to each vertex of a discrete model which converges to $\CLE_\kappa$ in the scaling limit.

The goal of this paper is to initiate the program constructing the \emph{chemical distance metric} and the \emph{effective resistance metric} for $\CLE_\kappa$ in the regime that $\kappa \in (4,8)$.  The chemical distance (or intrinsic) metric describes the length of the shortest path between each pair of points in the gasket. The effective resistance metric describes a diffusion process on the gasket. Recall that the $\CLE_\kappa$ gasket is a fractal set, and therefore we need an approximation procedure to construct these metrics. Our main result will be the tightness of a class of approximation procedures. We will introduce a list of abstract conditions for approximations of $\CLE_\kappa$ metrics which encompass certain natural approximations to the chemical distance and to the resistance metric. For these (approximate) $\CLE_\kappa$ metrics we prove some estimates that will also be useful in subsequent work. We further show that from the subsequential limits we can construct metrics on the $\CLE_\kappa$ gasket and they also satisfy our list of axioms. In particular, we obtain useful estimates for the $\CLE_\kappa$ metrics constructed from the subsequential limits.  (We remark that the tightness result for the chemical distance metric in the regime $\kappa \in (8/3,4)$ was proved in \cite{m2021tightness}, but the argument in \cite{m2021tightness} is very different from the one presented in this article as in the former setting the loops do not intersect each other or the domain boundary while they do in the present setting.)

It will be shown in subsequent work of the second two authors that there is a unique geodesic (resp.\ resistance) metric on the $\CLE_\kappa$ gasket that is conformally covariant and Markovian, and further that their approximations converge to these metrics. We conjecture that the geodesic $\CLE_\kappa$ metric is equal to the scaling limit of the chemical distance metric for those models which converge to the corresponding $\CLE_\kappa$ in the limit. Similarly, we conjecture that the $\CLE_\kappa$ resistance metric is equal to the scaling limit of the effective resistance on the clusters of these models, or equivalently that the simple random walk on these models converge to the diffusion process on the corresponding $\CLE_\kappa$ gasket. In the case of two-dimensional critical percolation on the triangular lattice, forthcoming work \cite{dmmy2025percolation} will show that the geodesic (resp.\ resistance) metric on \cle{6} describes the scaling limit of the chemical distance (resp.\ effective resistance) metric on its clusters.

\begin{figure}[ht]
\centering
\includegraphics[width=0.4\textwidth]{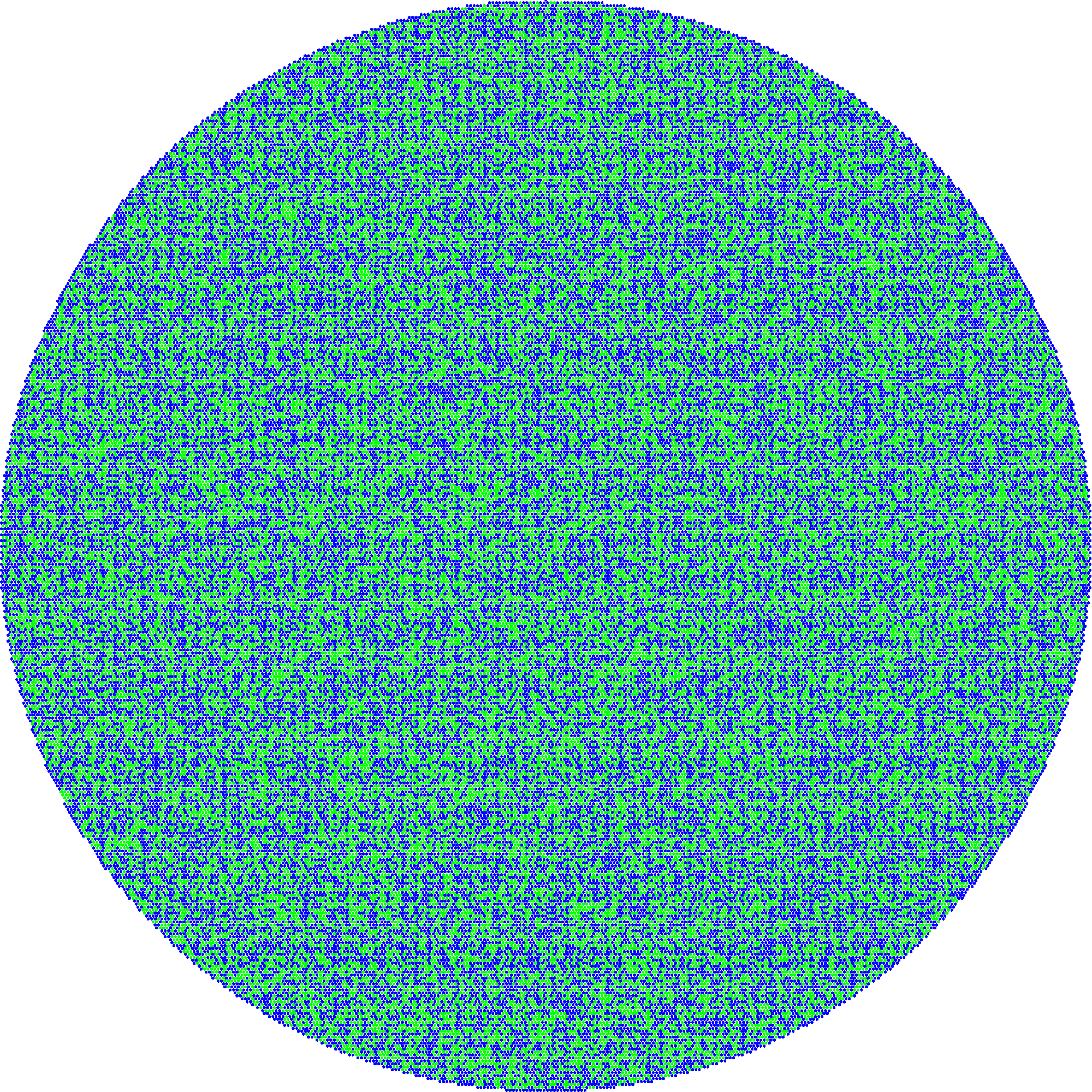} \hspace{0.05\textwidth} \includegraphics[width=0.4\textwidth]{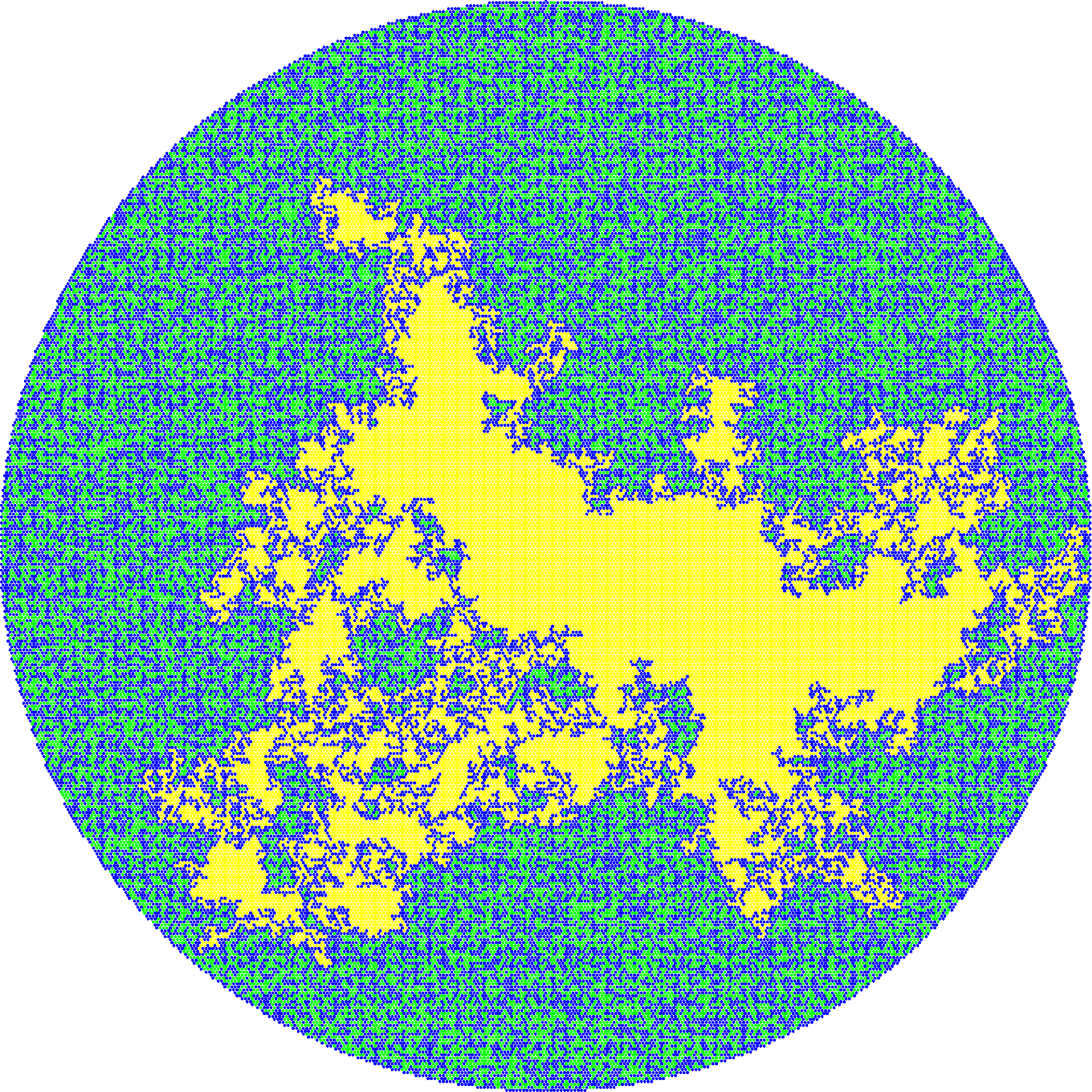}

\includegraphics[width=0.6\textwidth]{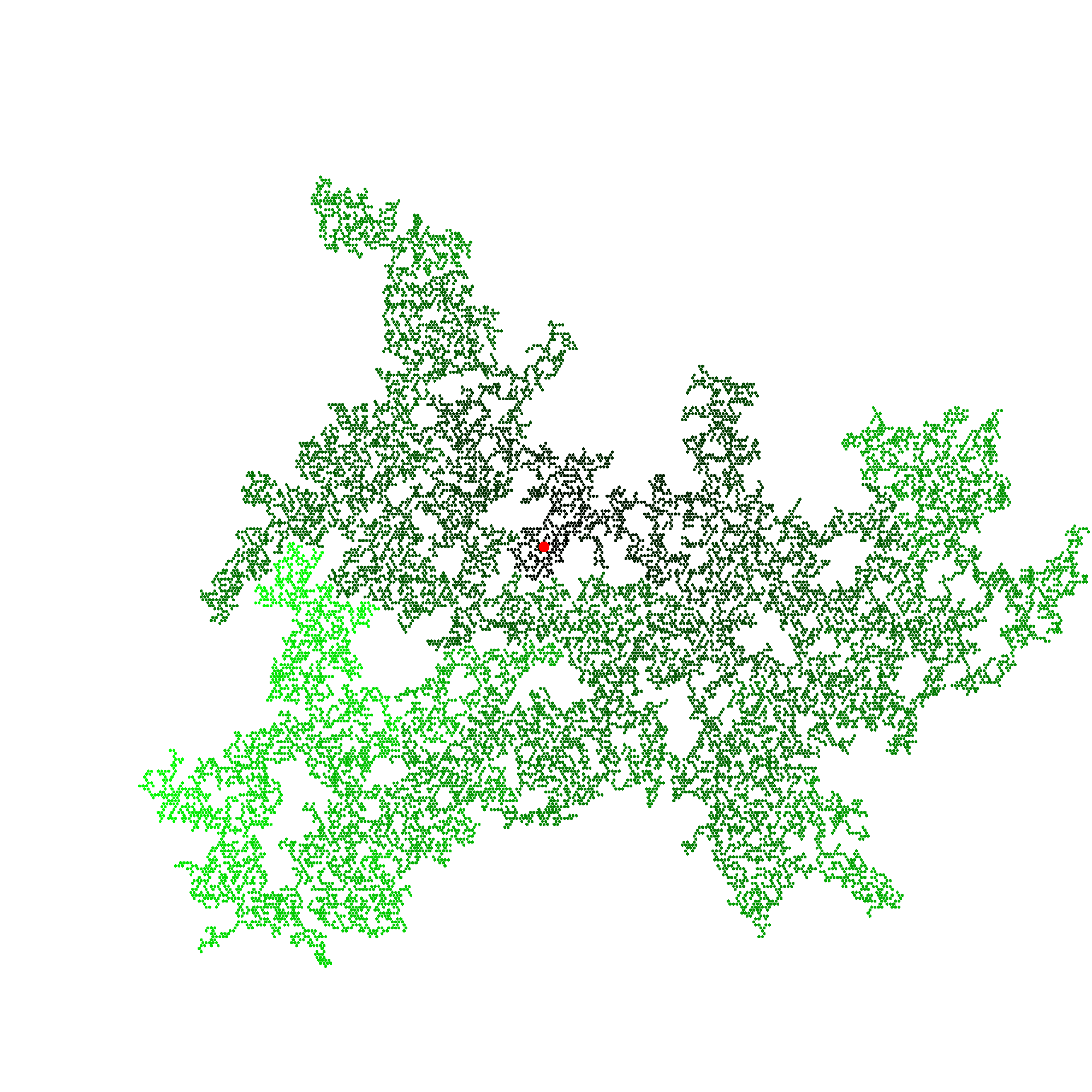}
\caption{\textbf{Top left:} Critical percolation configuration generated on the hexagonal lattice in the unit disk. \textbf{Top right:} Regions surrounded by the outermost cluster that surrounds the origin. \textbf{Bottom:} The corresponding cluster of open sites, colored according to their graph distance to a point (shown in red) near the center.}
\end{figure}

In the rest of the paper we will denote by $\kappa'\in(4,8)$ the parameter in the intersecting regime to be consistent with \cite{ms2016ig1,ms2017ig4} and we will write $\kappa = 16/\kappa' \in (2,4)$. Suppose that $\Gamma$ is a $\CLE_{\kappa'}$ and $\Upsilon_\Gamma$ is its gasket.  Since some points in $\Upsilon_\Gamma$ are visited by loops of $\Gamma$ twice, we will need to distinguish the so-called prime ends in $\Upsilon_\Gamma$.  We will give the precise definition just below, but the idea is that a point visited by a loop of $\Gamma$ twice may correspond to two different points in $\Upsilon_\Gamma$.  We can build some metrics on $\Upsilon_\Gamma$ that are approximate versions of the intrinsic metric and the resistance metric, respectively. We will phrase our main result for a class of so-called \emph{approximate \clekp{} metrics} which includes the approximations mentioned above. Our main result is that for such approximations $(\metapprox{\epsilon}{\cdot}{\cdot}{\Gamma})_{\epsilon > 0}$ there are normalization constants $\median{\epsilon}$ such that the family of $\median{\epsilon}^{-1} \metapprox{\epsilon}{\cdot}{\cdot}{\Gamma}$ is tight. Further, every subsequential limit yields again a \clekp{} metric in the class of metrics that we consider. We will let $\median{\epsilon}$ be the median of some appropriately chosen quantity in terms of $\metapprox{\epsilon}{\cdot}{\cdot}{\Gamma}$ which will be defined in Section~\ref{se:intersections_setup}. Before we state our main results, we need to formally define the class of approximate \clekp{} metrics and specify the topology with respect to which we will prove tightness.

\subsection{Approximate CLE metrics}
\label{se:assumptions}

\newcommand*{\epsexp}{a_0} 

We define a class of ``approximate CLE metrics'' for which we will show a tightness result. We begin by describing the setup we consider.

\begin{definition}
\label{def:admissible_path}
Let $\Gamma$ be a collection of loops in $\C$. We say that a path $\gamma$ is \emph{admissible for $\Gamma$} if it does not cross any loop of $\Gamma$.\footnote{This means that for any subsegment $\ell \subseteq \CL$ of a loop $\CL \in \Gamma$, if we let $\varphi$ be a conformal transformation from the unbounded connected component of $\C\setminus\ell$ to $\C\setminus\ol{\D}$, then $\varphi(\gamma \setminus \Fill(\ell))$ extends to a continuous path in $\C\setminus\D$ (where $\Fill(\ell)$ denotes the complement of the unbounded component of $\C \setminus \ell$).}. For each $U \subseteq \C$, we let $\paths{x}{y}{U}{\Gamma}$ denote the set of admissible paths for $\Gamma$ from $x$ to $y$ within $U$.
\end{definition}

Suppose that $\CL$ is a random non-self-crossing loop in $\C$ (we will always be considering the case where $\CL$ is a particular \clekp{} loop), and let $D$ be the regions surrounded by $\CL$ (i.e.\ the union of connected components of $\C \setminus \CL$ with winding number $1$). Given $\CL$, let $\Gamma_D$ be a conditionally independent \clekp{} in each connected component of $D$, and let $\Gamma = \{\CL\} \cup \Gamma_D$.  Let $\Upsilon_\Gamma$ be the gasket of $\Gamma_D$ (viewed as a metric space embedded in the plane, see Section~\ref{subsec:topology}). For $U \subseteq \C$ and $x,y \in U \cap \Upsilon_\Gamma$, we let
\begin{equation}\label{eq:dpath}
 \dpath[U](x,y) = \inf\{ \diamE(\gamma) : \gamma \in \paths{x}{y}{U}{\Gamma} \},
\end{equation}
where $\diamE$ denotes the diameter with respect to the Euclidean metric.  We write $\dpath(x,y) = \dpath[\ol{D}](x,y)$.

For each open, simply connected $U \subseteq \C$, let $\Gamma_{U^*} \subseteq \Gamma_D$ be the collection of loops that are entirely contained in $U$, and let $U^* \subseteq U \cap D$ be the set of points that are not on or inside any loop of $\Gamma_D\setminus\Gamma_{U^*}$. We view $U^*$ as the metric space equipped with the metric
\[ d_{U^*}(x,y) = \inf\{ \diamE(\gamma) : \gamma \in \paths{x}{y}{U}{\Gamma\setminus\Gamma_{U^*}} \} . \]
Recall \cite[Lemma~3.2]{gmq2021sphere} that the conditional law of $\Gamma_{U^*}$ given $\Gamma \setminus \Gamma_{U^*}$ is that of a conditionally independent collection of \clekp{}'s in each connected component of $U^*$.

We will assume that there is an ``internal metric'' defined in each region consisting of components that are bounded between a finite number of \clekp{} loops. More precisely, suppose $\CL_1,\ldots,\CL_n \in \Gamma_D$, and $V \subseteq D$ is a union of connected components of
$D \setminus (\CL_1 \cup \cdots \cup \CL_n)$ that are not inside the loops $\CL_1,\ldots,\CL_n$. Let
\[ d_{\ol{V}}(x,y) = \inf\{ \diamE(\gamma) : \gamma \in \paths{x}{y}{\ol{V}}{\{\CL,\CL_1,\ldots,\CL_n\}} \} . \]
We let $\metregions$ be the collection of regions $V \subseteq D$ as described above (for some choice of $\CL_1,\ldots,\CL_n$) such that $\ol{V}$ is simply connected with respect to $d_{\ol{V}}$ (where $\ol{V}$ denotes the completion with respect to $d_{\ol{V}}$). For $U \subseteq \C$ open, we let $\metregions[U] = \{ V \in \metregions : \ol{V} \subseteq U \}$.

\begin{remark}
 In order to prove tightness of the metrics, we only need to consider simply connected $\ol{V}$ in the definition above. But we can also construct internal metrics for general $\ol{V}$ if we are given approximations for general $\ol{V}$ that satisfy the same axioms.
\end{remark}

We now define what we mean by an (approximate) \clekp{} metric. Our prominent examples will be (approximations to) geodesic and resistance metrics on the \clekp{} gasket (see Section~\ref{se:met_examples} for more details). The series law and the parallel law are signature properties of the resistance metric. We will need a slightly different variant which we call the generalized parallel law. For geodesic metrics, both are trivially satisfied. Further, we will introduce approximate versions of these properties, due to the reason that the approximations of these metrics will satisfy them only under extra restrictions (due to the fact that there is some interference in an $\epsilon$-neighborhood where $\epsilon$ is the spacial scale of approximation).

Suppose that $\epsilon \ge 0$ (zero is allowed). Suppose that we have a family of random functions
\[ \metapproxres{\epsilon}{V}{\cdot}{\cdot}{\Gamma}\colon (\ol{V} \cap \Upsilon_\Gamma) \times (\ol{V} \cap \Upsilon_\Gamma) \to [0,\infty]  \]
for each $V \in \metregions$, and they satisfy
\[
 \metapproxres{\epsilon}{V}{x}{y}{\Gamma} = \metapproxres{\epsilon}{V}{y}{x}{\Gamma}
 \quad\text{and}\quad
 \metapproxres{\epsilon}{V}{x}{y}{\Gamma} \le \metapproxres{\epsilon}{V}{x}{z}{\Gamma}+\metapproxres{\epsilon}{V}{z}{y}{\Gamma} .
\]
In the case $\epsilon = 0$, we assume additionally that $\metres{V}{x}{x}{\Gamma} = 0$ for every $x \in \ol{V} \cap \Upsilon_\Gamma$, and that the functions $\metres{V}{\cdot}{\cdot}{\Gamma}$ are continuous with respect to $\dpath[\ol{V}]$. (When $\epsilon > 0$, we do not assume this. In either case we do not assume them to be positive off the diagonal.)
These functions are not required to be determined by $\Gamma$, and we view $(\Gamma,\metapprox{\epsilon}{\cdot}{\cdot}{\Gamma})$ as a random variable with values in the product space.

\begin{figure}[ht]
\centering
\includegraphics[width=0.3\textwidth]{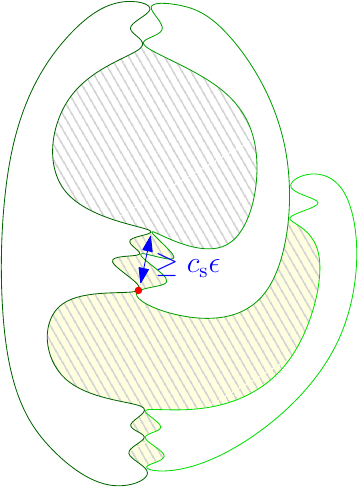}
\hspace{0.05\textwidth}
\includegraphics[width=0.45\textwidth]{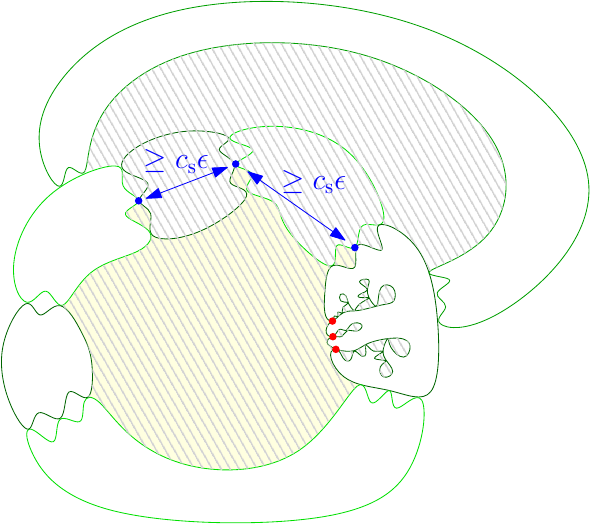}
\caption{\textbf{Left:} Illustration of the compatibility axiom. Shown is an example of regions $V \subseteq V'$ ($V$ is shaded in yellow, and $V'$ is tiled in grey) where $V' \setminus V$ is contained in a ``dead end'' which is separated by the point $u$ shown in red. The compatibility axiom states that in the region below the red dot we have $\metapproxres{\epsilon}{V}{\cdot}{\cdot}{\Gamma} = \metapproxres{\epsilon}{V'}{\cdot}{\cdot}{\Gamma}$. \textbf{Right:} Illustration of the extra condition in the monotonicity axiom. Shown is an example of regions $V \subseteq V'$ for which the monotonicity property holds ($V$ is shaded in yellow, and $V'$ is tiled in grey). The regions $V' \setminus V$ are either contained in ``dead ends'' (separated by the red dots) or they are separated by $\{z_1,z_2,...\}$ (blue dots) satisfying condition~\eqref{it:mon_large_loops}.}
\label{fi:axioms_comp_mon}
\end{figure}

\begin{figure}[ht]
\centering
\includegraphics[width=0.3\textwidth]{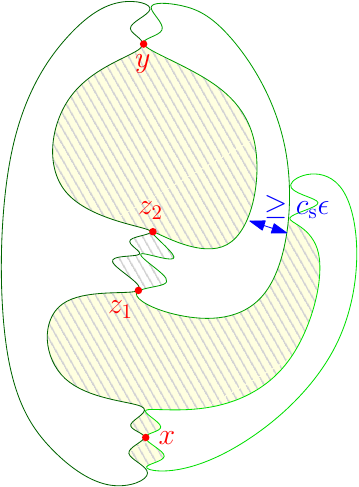}
\hspace{0.05\textwidth}
\includegraphics[width=0.45\textwidth]{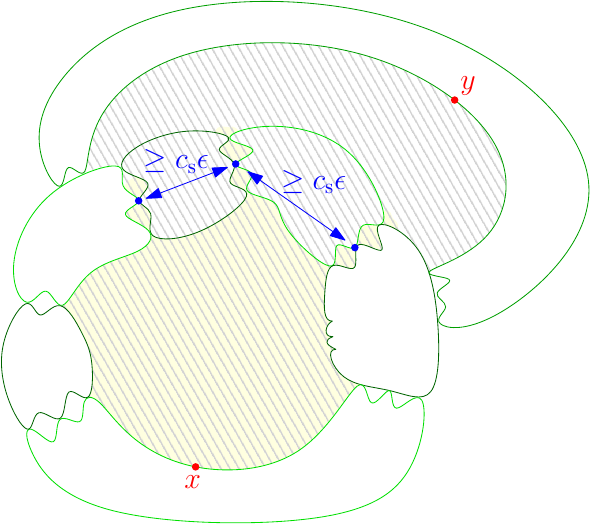}
\caption{\textbf{Left:} Illustration of the series law. It holds under the extra condition that the two yellow regions have Euclidean distance at least $\cserial\epsilon$ from each other. The region $V$ is tiled in grey. \textbf{Right:} Illustration of the generalized parallel law with $N=3$. The region $V$ is tiled in grey, the points $z_1,...,z_N$ are shown in blue. The region $V_x$ needs to contain a $\cserial\epsilon$-neighborhood of $K_x$ in $V$ (shaded in yellow).}
\label{fi:axioms_series_parallel}
\end{figure}

Let $\cserial \ge 0$ and $\cparallel(N) > 0$, $N \in \N$, be constants. The constant $\cparallel(N)$ serves to quantify the generalized parallel law (we have $\cparallel(N) = N$ in the case of resistance metrics, and $\cparallel(N) = 1$ in the case of geodesic metrics). We think of $\cserial\epsilon$ as the Euclidean range in which the approximations of the metric in $V$ may be affected by a $\cserial\epsilon$-neighborhood of $V$. Suppose that the following properties hold almost surely.

\textbf{Separability:} 
For every $V \in \metregions$,
\begin{equation}
\label{eq:separability_axiom}
 \metapproxres{\epsilon}{V}{x}{y}{\Gamma} = \lim_{V' \searrow V} \metapproxres{\epsilon}{V'}{x}{y}{\Gamma} ,\quad x,y \in \ol{V} \cap \Upsilon_\Gamma .
\end{equation}
The limit is in the sense that convergence holds for any decreasing sequence of $V'_n \in \metregions$ with $V'_n \supseteq V$ and $\sup_{u \in \ol{V'_n} \cap \Upsilon_\Gamma} \dpath[\ol{V'_n}](u,V) \to 0$.

This requirement guarantees that the entire collection of internal metrics is determined by a countable subset of them. In particular, the family lives in a Polish space.

\textbf{Markovian property:} Let $U \subseteq \C$ be open, simply connected. The conditional law of the collection $(\metapproxres{\epsilon}{V}{\cdot}{\cdot}{\Gamma})_{V \in \metregions[U]}$ given $\Gamma\setminus\Gamma_{U^*}$ and $(\metapproxres{\epsilon}{V'}{\cdot}{\cdot}{\Gamma})_{V' \in \metregions[\C\setminus\ol{U}]}$ is almost surely measurable with respect to~$U^*$.

\textbf{Translation invariance:} Let $U \subseteq \C$ be open, simply connected. There exists a probability kernel $\mu^{U^*}$ such that for each $z \in \C$, the conditional law of $(\metapproxres{\epsilon}{V}{\cdot}{\cdot}{\Gamma})_{V \in \metregions[U+z]}$ given $(U+z)^*$ is $(T_z)_* \mu^{U^*}(\cdot -z)$ where $T_z \Fd(x,y) = \Fd(x-z,y-z)$ denotes translation by $z$.

\textbf{Compatibility:} Let $V,V' \in \metregions$, $V \subseteq V'$, and $x,y \in \ol{V} \cap \Upsilon_\Gamma$ such that for every $u \in V' \setminus V$ there is a point $z \in \ol{V}$ with $\dpath[\ol{V'}](z, V' \setminus V) \ge \cserial\epsilon$ that separates $u$ from $x,y$ in $\ol{V'} \cap \Upsilon_\Gamma$. Then $\metapproxres{\epsilon}{V}{x}{y}{\Gamma} = \metapproxres{\epsilon}{V'}{x}{y}{\Gamma}$.

See the left side of Figure~\ref{fi:axioms_comp_mon} for an illustration. We remark that for (the approximations of) geodesic metrics, the extra condition $\dpath[\ol{V'}](z, V' \setminus V) \ge \cserial\epsilon$ is not necessary. For the approximations of the resistance metric, however, we will explain in \cite{my2025resuniqueness} that this property holds only under this extra condition.

\textbf{Monotonicity:} 
Let $V,V' \in \metregions$, $V \subseteq V'$. Let $x,y \in \ol{V} \cap \Upsilon_\Gamma$ and suppose one of the following:
\begin{enumerate}[(i)]
 \item\label{it:mon_dead_ends} For every $u \in V' \setminus V$ there is a point $z \in \ol{V}$ that separates $u$ from $x,y$ in $\ol{V'} \cap \Upsilon_\Gamma$.
 \item\label{it:mon_large_loops} There exist $z_1,z_2,\ldots \in \ol{V}$ with $\abs{z_i-z_{i'}} \ge \cserial\epsilon$ for $i\neq i'$ such that no \emph{simple} admissible path in $\ol{V}$ from $x$ to $y$ intersects $\{z_1,z_2,\ldots\}$, and the set $\{z_1,z_2,\ldots\}$ separates $x,y$ from $V' \setminus V$ in $\ol{V'} \cap \Upsilon_\Gamma$.
\end{enumerate}
Then there are points $x',y' \in \ol{V} \cap \Upsilon_\Gamma$ with $\dpath[\ol{V}](x',x) \le \epsilon$, $\dpath[\ol{V}](y',y) \le \epsilon$ such that
\begin{equation}\label{eq:approx_monotonicity_axiom}
 \metapproxres{\epsilon}{V'}{x'}{y'}{\Gamma} \le \metapproxres{\epsilon}{V}{x}{y}{\Gamma} .
\end{equation}

\begin{remark}
 For (the approximations of) geodesic metrics, we have $\metapproxres{\epsilon}{V'}{\cdot}{\cdot}{\Gamma} \le \metapproxres{\epsilon}{V}{\cdot}{\cdot}{\Gamma}$ for \emph{every} $V,V'$ with $V \subseteq V'$. For the approximations of the resistance metric, however, we will explain in \cite{my2025resuniqueness} that this property holds only under extra restrictions which are described in~\eqref{it:mon_dead_ends},\eqref{it:mon_large_loops}, and that we may need to move the points $x,y$ to some nearby points $x',y'$. See the right side of Figure~\ref{fi:axioms_comp_mon} for an illustration of the conditions~\eqref{it:mon_dead_ends},\eqref{it:mon_large_loops}.
\end{remark}

\textbf{Series law:} Let $V \in \metregions$, and $x,y,z_1,z_2 \in \ol{V} \cap \Upsilon_\Gamma$ be such that $z_1$ separates $x$ from $y,z_2$ in $\ol{V} \cap \Upsilon_\Gamma$, and $z_2$ separates $y$ from $x,z_1$ in $\ol{V} \cap \Upsilon_\Gamma$, and such that $\distE(K_x,K_y) \ge \cserial\epsilon$ where $K_x$ (resp.\ $K_y$) are the connected components of $\ol{V} \cap \Upsilon_\Gamma \setminus \{z_1,z_2\}$ containing $x$ (resp.\ $y$). (Here, $\distE$ denotes the distance with respect to the Euclidean metric.) Then 
\[ \metapproxres{\epsilon}{V}{x}{y}{\Gamma} \ge \metapproxres{\epsilon}{V}{x}{z_1}{\Gamma}+\metapproxres{\epsilon}{V}{z_2}{y}{\Gamma}.\]
(See the left side of Figure~\ref{fi:axioms_series_parallel} for an illustration.)

Further, if $z$ separates $x$ from $y$ in $\ol{V} \cap \Upsilon_\Gamma$, then
\[ \metapproxres{\epsilon}{V}{x}{y}{\Gamma} \ge \metapproxres{\epsilon}{V}{x}{z}{\Gamma} . \]

\textbf{Generalized parallel law:} Let $V \in \metregions$, and let $x,y,z_1,\ldots,z_N \in \ol{V} \cap \Upsilon_\Gamma$ such that $\abs{z_i-z_{i'}} \ge \cserial\epsilon$ for $i \neq i'$, and $x,y$ are separated in $\ol{V} \cap \Upsilon_\Gamma \setminus\{z_1,\ldots,z_N\}$. Let $K_x$ be the connected component of $\ol{V} \cap \Upsilon_\Gamma \setminus\{z_1,\ldots,z_N\}$ containing $x$, and let $V_x \subseteq V$ such that $\ol{V_x} \supseteq \{u \in V : \dpath[\ol{V}](u,K_x) \le \cserial\epsilon\}$. Then
\[ \cparallel(N)\metapproxres{\epsilon}{V}{x}{y}{\Gamma} \ge \min_{i}\metapproxres{\epsilon}{V_x}{x}{z_i}{\Gamma} . \]

See the right side of Figure~\ref{fi:axioms_series_parallel} for an illustration. For (the approximations of) geodesic metrics we have $\cparallel(N) = 1$ for every $N$, and it suffices that $\ol{V_x} \supseteq K_x$. For the approximations of the resistance metric, however, we need to require $\ol{V_x}$ to contain a $\cserial\epsilon$-neighborhood of $K_x$ in $V$ (this will be explained in \cite{my2025resuniqueness}). For instance, this holds with $V_x = V$.

\begin{remark}
 In fact, for the main results of this paper, we only need to require the generalized parallel law above to hold with $V_x = V$. The variant with general $V_x$ is only used in the proof of Theorem~\ref{th:nondegeneracy}.
\end{remark}

\begin{definition}\label{def:cle_metric}
 Suppose that $\CL$ is a random non-self-crossing loop in $\C$, and let $D$ be the regions surrounded by $\CL$. Given $\CL$, let $\Gamma_D$ be a conditionally independent \clekp{} in each connected component of $D$, and let $\Gamma = \{\CL\} \cup \Gamma_D$. We call a family of random metrics $(\metapproxres{\epsilon}{V}{\cdot}{\cdot}{\Gamma})_{V \in \metregions}$ coupled with $\Gamma$ that satisfies the assumptions above an \emph{approximate \clekp{} metric}. In the case $\epsilon = 0$ we call $(\metres{V}{\cdot}{\cdot}{\Gamma})_{V \in \metregions}$ a \emph{\clekp{} metric}. We call $\metapproxres{\epsilon}{V}{\cdot}{\cdot}{\Gamma}$ the \emph{internal metric in $V$}.
\end{definition}

We use the shorthand notation $\metapprox{\epsilon}{\cdot}{\cdot}{\Gamma}$ to refer to the family of internal metrics. For $U$ open, simply connected, we refer to $(\metapproxres{\epsilon}{V}{\cdot}{\cdot}{\Gamma})_{V \in \metregions[U]}$ as the \emph{internal metrics within $U^*$}.

Our goal will be to construct \clekp{} metrics via approximation schemes. Suppose $\cserial \ge 0$, $\cparallel(N) > 0$ are fixed constants. Suppose $\metapprox{\epsilon}{\cdot}{\cdot}{\Gamma}$ is an approximate \clekp{} metric for each $\epsilon \in (0,1]$, with respect to the same constants $\cserial \ge 0$, $\cparallel(N) > 0$. The goal of this paper is to show tightness of the family $(\median{\epsilon}^{-1}\metapprox{\epsilon}{\cdot}{\cdot}{\Gamma})$ for suitable renormalization constants $\median{\epsilon} > 0$. Note that our assumptions on the approximate \clekp{} metrics are only valid when we are spatially further away than the approximation scale $\epsilon > 0$. On smaller scales they will not give us any useful bounds. Therefore we need to assume additionally that the distances across smaller scales are small compared to the renormalization factor $\median{\epsilon}$. Roughly speaking, we define $\median{\epsilon}$ as the median (or any quantile) of the $\metapprox{\epsilon}{\cdot}{\cdot}{\Gamma}$-distance ``across the region bounded between two intersecting macroscopic \clekp{} loops'' which we describe now. The definition that we will give below is not exactly the same as the one that we use in the proofs of the paper (see Section~\ref{se:intersections_setup}).  It is a simplified version that serves to give the reader a feel for the setup, but it will turn out that the two definitions are comparable for good approximation schemes (defined below).  In order to differentiate between the definition of $\median{\epsilon}$ given in Section~\ref{se:intersections_setup} and the definition we will describe just below, we will denote the latter by $\wh{\mathfrak{m}}_\epsilon$.

\begin{figure}[ht]
\centering
\includegraphics[width=0.45\textwidth]{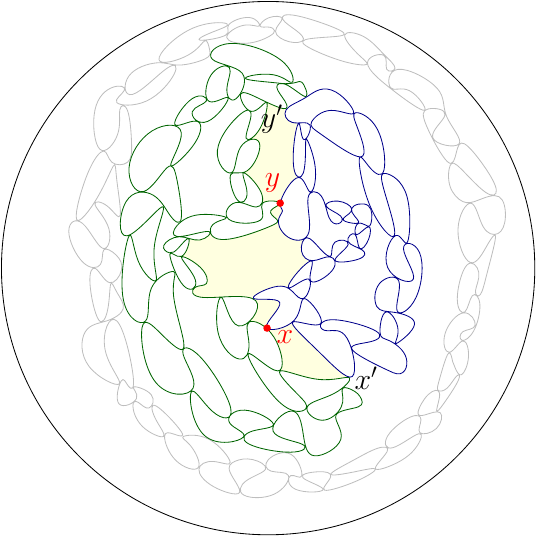}
\caption{\label{fi:region_intersecting_loops} We renormalize the approximate \clekp{} metrics by the median of the distance across the region between two typical intersecting \clekp{} loops as illustrated in the figure. The loops $\CL_1,\CL_2$ are shown in green and blue, respectively, and the region $U_{x',y'}$ is shown in yellow.}
\end{figure}

In the following, we let $\Gamma_\D$ be a nested \clekp{} in $\D$, and let $\CL$ be the outermost loop of $\Gamma_\D$ that surrounds~$0$. Then $\Gamma_D$ consists of the loops of $\Gamma_\D$ in the next nesting level inside $\CL$. This choice of $\CL$ is made just for concreteness; as we comment in Remark~\ref{rm:domain_choice} below, the concrete choices are not important as long as we consider interior clusters of a nested \clekp{}.

For $0<r_1<r_2$, we write $A(z,r_1,r_2) = B(z,r_2)\setminus\ol{B(z,r_1)}$.  Let $E$ be the event that there are two loops $\CL_1,\CL_2 \in \Gamma_D$ and four intersection points $x,y \in \CL_1 \cap \CL_2 \cap B(0,1/4)$, $x',y' \in \CL_1 \cap \CL_2 \cap A(0,1/2,3/4)$ such that if $U_{x',y'}$ is the union of connected components between the outer boundaries of $\CL_1,\CL_2$ from $x'$ to $y'$, then $x,y \in \ol{U}_{x',y'}$. On the event $E$, we let $X_\epsilon = \metapproxres{\epsilon}{U_{x',y'}}{x}{y}{\Gamma}$ where $x$ (resp.\ $y$) is the first (resp.\ last) point on $\ol{U}_{x',y'} \cap \CL_1 \cap \CL_2 \cap B(0,1/4)$ (see Figure~\ref{fi:region_intersecting_loops}). We let $\wh{\mathfrak{m}}_\epsilon$ be the median of $X_\epsilon$ conditionally on $E$.

\begin{definition}\label{def:good_approximation}
 Suppose that we have the setup above where $\CL$ is the outermost loop of $\Gamma_\D$ that surrounds~$0$. Suppose that $\metapprox{\epsilon}{\cdot}{\cdot}{\Gamma}$ is an approximate \clekp{} metric for each $\epsilon \in {(0,1]}$ with respect to the same constants $\cserial \ge 0$, $\cparallel(N) > 0$. We say that $(\metapprox{\epsilon}{\cdot}{\cdot}{\Gamma})_{\epsilon \in {(0,1]}}$ is a \emph{good approximation scheme} if there exists a constant $\ac{\epsilon} > 0$ for each $\epsilon \in {(0,1]}$ and a constant $\epsexp > 0$ such that
\begin{equation}\label{eq:eps_bound_ass}
 \lim_{\epsilon \searrow 0} \frac{\ac{\epsilon}}{\epsilon^{\epsexp}\median{\epsilon}} = 0
\end{equation}
and for each $r>0$ we have
\begin{equation}\label{eq:approx_error_asymp}
 \lim_{\epsilon \searrow 0} \p\left[ \sup_{V \in \metregions} \sup_{\substack{\dpath[\ol{V}](x,y) < \epsilon ,\\ \dpath(x, \Upsilon_\Gamma \setminus \ol{V}) \ge r}} \metapproxres{\epsilon}{V}{x}{y}{\Gamma} \le \ac{\epsilon} \right] = 1 .
\end{equation}
\end{definition}

In our approximations to the geodesic metric (see Section~\ref{se:met_examples}), we can bound $\metapproxres{\epsilon}{V}{x}{y}{\Gamma}$ by a small deterministic constant $\ac{\epsilon}$ whenever $\dpath[\ol{V}](x,y) < \epsilon$. For the approximations to the resistance metric considered in \cite{my2025resuniqueness}, it turns out that there is no such deterministic bound. Therefore we need to phrase the condition~\eqref{eq:approx_error_asymp} as a probabilistic bound. The condition~\eqref{eq:eps_bound_ass} says that the error $\ac{\epsilon}$ is asymptotically negligible compared to the normalizing factor $\median{\epsilon}$.
We keep the assumption~\eqref{eq:eps_bound_ass}, \eqref{eq:approx_error_asymp} separate from the definition of an approximate \clekp{} metric in Definition~\ref{def:cle_metric} because it is not compatible with scaling the metric in Lemma~\ref{le:scaled_metric}. The assumption~\eqref{eq:eps_bound_ass}, \eqref{eq:approx_error_asymp} is technical but rather mild, and is satisfied for all ``reasonable'' approximation schemes. Note also that the assumption~\eqref{eq:eps_bound_ass} includes that $\median{\epsilon} > 0$, in particular the approximating metrics are not identical to zero.

The following scaling property is immediate from the definition.

\begin{lemma}
\label{le:scaled_metric}
 Suppose $\metapprox{\epsilon}{\cdot}{\cdot}{\Gamma}$ is an approximate \clekp{} metric coupled with $\Gamma$. Let $\lambda > 0$, and define $\mettapprox{\lambda^{-1}\epsilon}{\cdot}{\cdot}{\lambda^{-1}\Gamma} = \metapprox{\epsilon}{\lambda\cdot}{\lambda\cdot}{\Gamma}$. This defines an approximate \clekp{} metric.
\end{lemma}

We note that the Markovian property implies that the conditional law of the internal metrics within $U^*$ depends only on the region $U^*$ and not on the choice of $U$ from which it arises.

\begin{lemma}\label{le:cond_laws_compatible}
 Let $U \subseteq U'$ be open, simply connected. On the event $(U')^* \subseteq U$, the conditional law of $(\metapproxres{\epsilon}{V}{\cdot}{\cdot}{\Gamma})_{V \in \metregions[U']}$ given $(U')^*$ is the same as the conditional law of $(\metapproxres{\epsilon}{V}{\cdot}{\cdot}{\Gamma})_{V \in \metregions[U]}$ given $U^*$.
\end{lemma}

\begin{proof}
 Let $E$ be the event that $(U')^* \subseteq U$, and in particular $(U')^* = U^*$. Therefore on $E$ the conditional law of $(\metapproxres{\epsilon}{V}{\cdot}{\cdot}{\Gamma})_{V \in \metregions[U]}$ given $(U')^*$ and $U^*$ is the same as the conditional law of $(\metapproxres{\epsilon}{V}{\cdot}{\cdot}{\Gamma})_{V \in \metregions[U']}$ given $(U')^*$. On the other hand, by the Markovian property, the conditional law of $(\metapproxres{\epsilon}{V}{\cdot}{\cdot}{\Gamma})_{V \in \metregions[U]}$ given $(U')^*$ and $U^*$ is measurable with respect to $U^*$, and therefore the same as its conditional law given just $U^*$.
\end{proof}

We introduce another technical notation which will be used in the proofs of this paper. The reason is that the assumptions on the approximate \clekp{} metrics fail to produce estimates for scales smaller than $\epsilon$. For smaller scales we would like to use the bound~\eqref{eq:eps_bound_ass} for the approximation errors. However, the assumption~\eqref{eq:approx_error_asymp} only holds asymptotically, but in this paper we will need estimates that hold uniformly for all approximate \clekp{} metrics. Therefore we will introduce the following notation in the case $\epsilon > 0$. For $\ac{\epsilon} > 0$ fixed, we let
\begin{equation}
\label{eq:shortcutted_metric}
 \metapproxacres{\epsilon}{V}{x}{y}{\Gamma} = \inf_{(u_i)} \sum_i \begin{cases} \metapproxres{\epsilon}{V}{u_i}{u_{i+1}}{\Gamma} & \text{if}\quad \dpath[\ol{V}](u_i,u_{i+1}) \ge \epsilon,\\ \ac{\epsilon} & \text{if}\quad \dpath[\ol{V}](u_i,u_{i+1}) < \epsilon , \end{cases}
\end{equation}
where the infimum is taken over finite sequences $(u_1,\ldots,u_L)$ with $u_1 = x$, $u_L = y$. As $\epsilon \searrow 0$, the condition~\eqref{eq:approx_error_asymp} implies that $\metapproxres{\epsilon}{V}{\cdot}{\cdot}{\Gamma} = \metapproxacres{\epsilon}{V}{\cdot}{\cdot}{\Gamma}$ in components away from $\Upsilon_\Gamma \setminus \ol{V}$ with probability converging to $1$.

For the metric $\mettapprox{\lambda^{-1}\epsilon}{\cdot}{\cdot}{\Gamma} = \metapprox{\epsilon}{\lambda\cdot}{\lambda\cdot}{\lambda\Gamma}$ in Lemma~\ref{le:scaled_metric} we set $\act{\lambda^{-1}\epsilon} = \ac{\epsilon}$ and $\mettapproxacres{\lambda^{-1}\epsilon}{V}{\cdot}{\cdot}{\Gamma} = \metapproxacres{\epsilon}{\lambda V}{\lambda\cdot}{\lambda\cdot}{\lambda\Gamma}$.

Note that the family $(\metapproxacres{\epsilon}{V}{\cdot}{\cdot}{\Gamma})_{V \in \metregions}$ also satisfies the compatibility and the generalized parallel law. To see the compatibility, observe that the series law implies that if a sequence $(u_i)$ in~\eqref{eq:shortcutted_metric} contains a point $u_{i_0}$ such that for some $u \in (V' \setminus V) \cap \Upsilon_\Gamma$, the corresponding point $z$ does not separate $u_{i_0}$ from $u$, then $(u_i)$ is not optimal in~\eqref{eq:shortcutted_metric}. To see the generalized parallel law, observe that for each sequence $(u_i)$ in~\eqref{eq:shortcutted_metric} there is a pair $u_i,u_{i+1}$ that are separated by $\{z_1,\ldots,z_N\}$.

\subsubsection{Examples}
\label{se:met_examples}

\noindent\textbf{Approximation of geodesic metric (chemical distance).}  To construct a geodesic metric on the \clekp{} gasket, we use \emph{geodesic approximation schemes} which are defined in Section~\ref{se:geodesic_metric}. Suppose that for each admissible path $\gamma$ we can define an approximation $\lebneb{\epsilon}(\gamma)$ of its fractal length. We let $\metapproxres{\epsilon}{V}{x}{y}{\Gamma}$ be the infimum of $\lebneb{\epsilon}(\gamma)$ over all admissible paths within $\ol{V}$ from $x$ to $y$. Examples for such approximations include the following (see Example~\ref{ex:geodesic_approx}).

\emph{Example 1:} Let $\lebneb{\epsilon}(\gamma)$ be the Lebesgue measure of the $\epsilon$-neighborhood of $\gamma$. (Note that we can set $\ac{\epsilon} = 4\pi\epsilon^2$.)

\emph{Example 2:} Let $\lebneb{\epsilon}(\gamma)$ be the largest number $N$ such that there exist $t_1 < t_2 < \cdots < t_N$ with $\abs{\gamma(t_j)-\gamma(t_{j-1})} \ge \epsilon$ for each $j$. (Note that we can set $\ac{\epsilon} = 1$.)

We will show in \cite{my2025geouniqueness} that all geodesic approximation schemes converge to the same limit (up to a multiplicative constant). This is the \emph{intrinsic metric on the \clekp{} gasket}, also called \emph{chemical distance}. In \cite{dmmy2025percolation} we will show that the chemical distance metric on critical percolation converges to this metric.

\noindent\textbf{Effective resistance.}  We will explain in \cite{my2025resuniqueness} that for certain graph approximations of the \clekp{} gasket, the associated effective resistance metrics satisfy our definitions of approximate \clekp{} metrics. As we will explain there, the extra restrictions in the compatibility and monotonicity axioms are needed for these approximations. The scaling limit of these approximations will be the unique (up to a multiplicative constant) \emph{\clekp{} resistance metric}. This will give rise to the \emph{diffusion process on the \clekp{} gasket}. We will further show in \cite{dmmy2025percolation}, that the random walk on a critical percolation cluster converges to this diffusion process.

\noindent\textbf{Subsequential limits.}  We show in Section~\ref{se:construction_metric} that subsequential limits of approximate \clekp{} metrics give rise to \clekp{} metrics with $\epsilon = 0$.

\subsection{Topology}
\label{subsec:topology}

We will now describe the topology with respect to which we will prove tightness.  In the case of the simple $\CLE$s \cite{m2021tightness}, tightness was proved with respect to a topology which was defined in terms of the Euclidean metric.  In the case of the non-simple $\CLE$s, it is natural to consider the ``prime ends'' in the gasket rather than points since each double point of a loop corresponds to two ``prime ends'' which have a positive distance from each other. Therefore the associated $\CLE$ metric cannot be continuous with respect to the ordinary Euclidean metric. This leads us to view the gasket as an abstract metric space equipped with an embedding into the plane. We equip the set of CLE gaskets with the Gromov-Hausdorff-function topology which is defined in Appendix~\ref{app:ghf}. (This topology has also been considered in \cite{bck2017tightness} in the context of trees embedded in the plane.) Moreover, we also need to keep track of the separation points in the gasket which play an essential role in the definition of the CLE metrics in Section~\ref{se:assumptions}. Since the usual Gromov-Hausdorff topology does not see which points are separating, we add in this extra information by encoding them as additional subsets of the space.

Suppose that~$\CL$ is a random non-self-crossing loop in $\C$, and let $D$ be the regions surrounded by~$\CL$. Given~$\CL$, let $\Gamma_D$ be a conditionally independent \clekp{} in each connected component of $D$, and let $\Gamma = \{\CL\} \cup \Gamma_D$.  Let~$\Upsilon_\Gamma$ be the gasket of $\Gamma_D$. For each open, simply connected $U \subseteq \C$, we consider the gasket $\Upsilon_{\Gamma_{U^*}}$ of $\Gamma_{U^*}$ which consists of conditionally independent \clekp{} in each connected component of $U^*$. We view~$\Upsilon_{\Gamma_{U^*}}$ as the tuple $(X,\dpath[U],\Pi)$ where $(X,\dpath[U])$ is the metric space completion of $\Upsilon_{\Gamma_{U^*}} \setminus \bigcup \Gamma$ with respect to the metric $\dpath[U](x,y) = \inf\{ \diam(\gamma) : \gamma \in \paths{x}{y}{U}{\Gamma} \}$ in~\eqref{eq:dpath}, and $\Pi\colon X \to \ol{U^*}$ is the embedding of the gasket into the plane (where double points of loops in $\Gamma$ have two distinct preimages in $X$). We encode the separation points in $X$ by the sets
\[ A_n = \{ (x,y,z_1,\ldots,z_n) \in X^{2+n} : \text{The set $\{z_1,\ldots,z_n\}$ separates $x$ from $y$} \} . \]
Finally, we regard each $\metapproxres{\epsilon}{U}{\cdot}{\cdot}{\Gamma}$ as a function defined on $X^2$. We can naturally extend $\Pi$ to $X^2$ by $\Pi(x,y) = (\Pi(x),\Pi(y))$. We view $(X^2,\dpath[U],\Pi,\Fd)$ as the space $(X^2,\dpath[U])$\footnote{Here $\dpath[U]((x,y),(x',y')) = \dpath[U](x,x')+\dpath[U](y,y')$ denotes the product metric.} equipped with the function $(\Pi,\Fd)$. We equip the space of such tuples with the Gromov-Hausdorff-function topology introduced in Appendix~\ref{app:ghf}.\footnote{In the GHf topology we can view the compact subsets $A_n$ as constant $1$ functions defined on $A_n$.}

In this paper, we will identify $(\Upsilon_{\Gamma_{U^*}}, \Fd)$ with the infinite tuple $((X^2,\dpath[U],\Pi,\Fd),(X^{2+n},A_n)_{n\in\N})$ and equip the space of such tuples with the product topology. Note that by Lemma~\ref{le:ghf_subspace}, when taking limits of such tuples in the product topology, the limiting spaces remain compatible.

\begin{lemma}\label{le:ghf_subspace_proj}
 Suppose $(X_n,d_n,\Pi_n) \to (X,d,\Pi)$ in the GHf topology where $\Pi_n\colon X_n \to \C$, $\Pi\colon X \to \C$ are continuous. Let $V \subseteq \C$ be open, and let $X'_n = \ol{\Pi_n^{-1}(V)}$, $X' = \ol{\Pi^{-1}(V)}$. Suppose that $(X'_n,d_n,\Pi_n) \to (\wt{X}',\wt{d},\wt{\Pi})$ in the GHf topology. Then there is an isometric embedding $\psi\colon X' \to \wt{X}'$ such that $\Pi\big|_{X'} = \wt{\Pi}\circ\psi$.
\end{lemma}

\begin{proof}
 By Lemma~\ref{le:ghf_subspace}, $(\wt{X}',\wt{d},\wt{\Pi})$ is isometric to a subspace of $X'$ which we will identify with $\wt{X}'$. We need to show that $X' \subseteq \wt{X}'$. It suffices to show that $\Pi^{-1}(V) \subseteq \wt{X}'$. By Lemma~\ref{le:ghf_common_embedding}, we can assume that $(X_n,d_n)$ and $(X,d)$ are embedded in a compact space $(W,d_W)$ such that $d_\infty(\Pi,\Pi_n) \to 0$. Then, for each $x \in \Pi^{-1}(V)$ there is a sequence $x_n \to x$ with $\Pi_n(x_n) \to \Pi(x)$. Since $V$ is open, we have $x_n \in \Pi_n^{-1}(V)$ for large $n$. This shows the claim.
\end{proof}

\begin{lemma}\label{le:gasket_compact}
 Suppose that we have the setup above with $\CL$ being a loop of a \clekp{} in $\D$. Almost surely, for each open, simply connected $U \subseteq \D$ and $V \Subset U$, the space $(\Upsilon_{\Gamma_{U^*}} \cap \ol{V}, \dpath[U])$ equipped with the topology specified above is compact.
\end{lemma}

\begin{proof}
 The space $\Upsilon_{\Gamma_{U^*}} \cap \ol{V}$ is complete as a closed subspace of a complete metric space. We need to show that the space is totally bounded.
 
 We denote with the subscript `E' distances with respect to the Euclidean metric and with the subscript `path' distances with respect to the $\dpath$ metric.
 
 Suppose that $0<r<\distE(V,\partial U)$. Let $\{x_i\}_{i=1,\dots,m}$ be such that $\Upsilon_\Gamma \cap \ol{V} \subseteq \bigcup_{i=1}^m B_{\mathrm E}(x_i, r^{2})$. For each $y\in B_{\mathrm E}(x_i, r^{2})$ let
 \[ V_{x_i}(y)=\left\{z\in B_{\mathrm E}(x_i, r^{2}) \ :\  \dpath[B_{\mathrm E}(x_i,r)](y,z)<\infty\right\} . \]
 Note that $V_{x_i}(y)\subseteq B_{\mathrm{path}}^{U}(y, 2r)$ for each  $i=1,\ldots,m$. Hence, it suffices to argue that there exists a constant $K=K(r)$ such that a.s.\ for every $i=1,\ldots,m$ there are $y_1,\ldots, y_K\in \Upsilon_\Gamma \cap B_{\mathrm E}(x_i, r^{2})$ so that  $\Upsilon_\Gamma \cap B_{\mathrm E}(x_i, r^{2})=\bigcup_{j=1}^K V_{x_i}(y_j)$.
 
 Suppose that $K$ is the minimal number such that this holds. Then the CLE contains at least $2K$ crossings of $A_{\mathrm E}(x_i, r^{2}, r)$ for some $i$. By \cite[Lemma~4.4]{amy-cle-resampling} we can choose $K$ large enough so that for each $x \in r^{2}\Z^2 \cap V$ the probability of $2K$ crossings of $A_{\mathrm E}(x, 2r^{2}, r)$ is at most $O(r^5)$. Taking a union bound we conclude that on an event with probability $1-O(r)$ this does not occur for any $x_i$, $i=1,\ldots,m$. Hence, a.s.\ for $r$ small enough we can cover $\Upsilon_{\Gamma_{U^*}} \cap \ol{V}$ with $O(Kr^{-4})$ balls of radius $r$ w.r.t.\ $\dpath[U]$.
\end{proof}

\subsection{Main results}
\label{subsec:main_statement}

In our main theorem statements we consider the following setup. Let $\kappa' \in (4,8)$. Let $\Gamma_\D$ be a nested $\CLE_{\kappa'}$ in $\D$, let $\CL$ be the outermost loop of $\Gamma_\D$ that surrounds~$0$, and let $D$ be the regions surrounded by $\CL$. Let $\Gamma_D$ be the loops of $\Gamma_\D$ contained in $\ol{D}$, and let $\Gamma = \{\CL\} \cup \Gamma_D$. Let $\metapprox{\epsilon}{\cdot}{\cdot}{\Gamma} = (\metapproxres{\epsilon}{V}{\cdot}{\cdot}{\Gamma})_{V \in \metregions}$ be a family of approximate \clekp{} metrics coupled with $\Gamma$, with some fixed constants $\cserial,\cparallel(N)$, and assume that $(\metapprox{\epsilon}{\cdot}{\cdot}{\Gamma})_{\epsilon \in (0,1]}$ is a good approximation scheme in the sense of Definition~\ref{def:good_approximation}.

\begin{remark}\label{rm:domain_choice}
 The particular choice of the ambient domain $\D$ is not important. It is just useful in the proofs to work with a bounded domain. In the proofs we only need to consider small regions $V \in \metregions$ in the interior of $\D$, and therefore we can equally consider an arbitrary simply connected domain in place of $\D$, or also a whole-plane \clekp{}.
 
 Similarly, the choice of the cluster is not important. We will see in the proof that our results hold for all interior clusters of $\Gamma_\D$.
 
 On the other hand, it \emph{is} important to consider interior clusters.  The reason for this is that one expects that geodesics in the $\CLE_{\kappa'}$ gasket are fractal and have dimension strictly larger than $1$.  If we consider a region $D$ with smooth boundary, then it is natural to expect that the metric will degenerate along $\partial D$ in the (subsequential) limit as $\epsilon \to 0$.
\end{remark}

\begin{theorem}\label{th:tightness_metrics}
Let $\cserial,\cparallel(N)$ be given. Consider the setup described in the preceding paragraph. Suppose that $(\metapprox{\epsilon}{\cdot}{\cdot}{\Gamma})_{\epsilon \in (0,1]}$ is a good approximation scheme as in the Definition~\ref{def:good_approximation}. Then there exists a family of normalizing constants $\median{\epsilon} > 0$ such that for each subsequence $(\epsilon_n)$ with $\epsilon_n \to 0$, if $V \in \metregions$ is chosen according to some probability distribution given $\Gamma$, the laws of $\median{\epsilon_n}^{-1} \metapproxres{\epsilon_n}{V}{\cdot}{\cdot}{\Gamma}$ restricted to subsets with positive $\dpath$-distance to $\Upsilon_\Gamma \setminus \ol{V}$ are tight with respect to the topology described in Section~\ref{subsec:topology}.
\end{theorem}

In the statement of Theorem~\ref{th:tightness_metrics}, if we view the random function $\metapproxres{\epsilon_n}{V}{\cdot}{\cdot}{\Gamma}$ in the topology of uniform convergence on subsets with positive $\dpath$-distance to $\Upsilon_\Gamma \setminus \ol{V}$, and view the collection $(\median{\epsilon_n}^{-1} \metapproxres{\epsilon_n}{V}{\cdot}{\cdot}{\Gamma})_{V \in \metregions}$ in the product topology, then the theorem says that the laws of this collection of internal metrics are tight. We chose to phrase Theorem~\ref{th:tightness_metrics} for a particular $V$ in order to make a simple and clean statement using the definition of the topology we gave in Section~\ref{subsec:topology}. Note that we can select $V$ in an arbitrary way depending only on $\Gamma$; natural choices are e.g.\ selecting for given $m \in \N$ and $U \subseteq \D$ the $m$ largest loops of $\Gamma$ and the maximal regions $V$ that are bounded between these loops and are contained in $U$.

For each sequence $(\epsilon_n)$ with $\epsilon_n \to 0$ we can use Theorem~\ref{th:tightness_metrics} to extract a subsequence such that the laws of $(\median{\epsilon_n}^{-1} \metapproxres{\epsilon_n}{U}{\cdot}{\cdot}{\Gamma})$ converge along the subsequence for a countable dense collection of $U$. From such a converging subsequence, we can then construct a \clekp{} metric in the sense of Definition~\ref{def:cle_metric} (with $\epsilon=0$). Conceptually, if we let $(D^U)$ denote the subsequential limit, we define
\[ \metres{V}{\cdot}{\cdot}{\Gamma} = \lim_{U \searrow V} D^U \quad\text{for each } V \in \metregions . \]
The precise construction will be given in Section~\ref{se:construction_metric}.

\begin{theorem}
Let $(\epsilon_n)$ be a sequence with $\epsilon_n \to 0$, and let $\met{\cdot}{\cdot}{\Gamma}$ be the metric constructed from $(\median{\epsilon_n}^{-1} \metapprox{\epsilon_n}{\cdot}{\cdot}{\Gamma})$ in Section~\ref{se:construction_metric}. Then $\met{\cdot}{\cdot}{\Gamma}$ is a \clekp{} metric in the sense of Definition~\ref{def:cle_metric} (with $\epsilon=0$).
\end{theorem}

Of course, the theorem would be empty if the renormalized metrics converge to zero. Our renormalization constants which are defined in Section~\ref{se:intersections_setup} are chosen so that in all the cases we consider, the limits are not identically zero. To avoid adding more technical complications, we have decided not to give a general criterion that ensures this. Instead, we prove that any \clekp{} metric is a.s.\ non-degenerate \emph{provided it is not identically zero}. (See Remark~\ref{rm:compatibility_renormalisation} for a further discussion.)

\begin{theorem}\label{th:nondegeneracy}
Suppose that $\met{\cdot}{\cdot}{\Gamma}$ is a \clekp{} metric in the sense of Definition~\ref{def:cle_metric} (with $\epsilon = 0$). Then either a.s.\ $\metres{V}{\cdot}{\cdot}{\Gamma} = 0$ for each $V \in \metregions$ or a.s.\ $\metres{V}{\cdot}{\cdot}{\Gamma}$ is a true metric for each $V \in \metregions$ (i.e.\ $\metres{V}{x}{y}{\Gamma} > 0$ whenever $x \neq y$).
\end{theorem}

We also prove various estimates for the \clekp{} metrics.

\begin{theorem}\label{th:cle_metric_hoelder}
Let $\met{\cdot}{\cdot}{\Gamma}$ be a \clekp{} metric. Almost surely, the metric $\metres{V}{\cdot}{\cdot}{\Gamma}$ is H\"older continuous with respect to $\dpath[\ol{V}]$ for every $V \in \metregions$. If we restrict the metric to the thin gasket $\CT_\Gamma$, then it is H\"older continuous with respect to the Euclidean metric.
\end{theorem}

We also have superpolynomial tails for the H\"older constants of $\metapproxres{\epsilon}{V}{\cdot}{\cdot}{\Gamma}$ and $\metres{V}{\cdot}{\cdot}{\Gamma}$, see Section~\ref{subsec:thm1_proof} for precise statements.

We further show polynomial bounds on the scaling behavior of \clekp{} metrics. The exponents $d_\SLE$, $\ddouble$ are defined at the beginning of Section~\ref{se:preliminaries}. To state this result, we condition on the event that $\CL \subseteq \D$, so that for $\lambda \in (0,1]$ the conditional law of $\lambda\CL$ given $\CL \subseteq \D$ is absolutely continuous with respect to the law of $\CL$.\footnote{This can be seen as follows. Consider the \clekp{} exploration in $\lambda\D$. On the event that $\lambda\CL \subseteq \lambda\D$, the exploration will surround a region $U \Subset \lambda\D$ before discovering $\lambda\CL$. The region $U$ can be discovered by a pair of flow lines $\eta_u,\eta_{u,v}$ where $\eta_{u,v}$ is reflected off $\eta_u$ in the opposite direction (see Section~\ref{se:fl_interaction} for details). The \clekp{} configuration inside $U$ are then determined by the GFF values in $U$. By the absolute continuity of the GFF, the law of this configuration is absolutely continuous with respect to the law under a GFF on $\D$. This gives that the law of $\lambda\CL$ is absolutely continuous with respect to the law of \emph{some} loop of $\Gamma_\D$ surrounding $0$. By the resampling operations from \cite{amy-cle-resampling} recalled in Section~\ref{se:cle_resampling}, we then conclude that the law is absolutely continuous with respect to the law of the \emph{outermost} loop surrounding $0$.} In particular, if we have a \clekp{} metric $\met{\cdot}{\cdot}{\Gamma}$, then the \clekp{} metric $\metres{\lambda V}{\lambda\cdot}{\lambda\cdot}{\lambda\Gamma}$ is defined.

\begin{theorem}[tightness across scales]\label{th:metric_scaling}
 Suppose that $\Gamma$ has the law described at the beginning of Section~\ref{subsec:main_statement} \emph{conditioned} on the event that $\CL \subseteq \D$. Let $\met{\cdot}{\cdot}{\Gamma}$ be a non-trivial \clekp{} metric. There is a constant $\median[\lambda]{} > 0$ for each $\lambda \in (0,1]$ such that the family $\{ ((\median[\lambda]{})^{-1}\metres{\lambda V}{\lambda\cdot}{\lambda\cdot}{\lambda\Gamma})_{V \in \metregions} \}_{\lambda \in (0,1]}$ is tight and its subsequential limits are non-trivial. Moreover, the constants satisfy
 \[ r^{d_\SLE+o(1)}\median[\lambda]{} \le \median[r\lambda]{} \le r^{\ddouble+o(1)}\median[\lambda]{} \]
 for $r,\lambda \in {(0,1]}$.
\end{theorem}

Finally, we show for geodesic approximation schemes (which are formally defined in Section~\ref{se:geodesic_metric}) that their subsequential limits are non-degenerate and are geodesic metrics.

\begin{theorem}\label{th:geodesic_approximation}
If $(\metapprox{\epsilon}{\cdot}{\cdot}{\Gamma})_{\epsilon \in (0,1]}$ is a good geodesic approximation scheme, then we can choose the renormalizing constants $\median{\epsilon}$ so that any subsequential limit $\met{\cdot}{\cdot}{\Gamma}$ of $(\median{\epsilon}^{-1}\metapprox{\epsilon}{\cdot}{\cdot}{\Gamma})_{\epsilon \in (0,1]}$ is a collection of non-degenerate geodesic metrics. Further, for each $V \in \metregions$ we have
\[ \metres{V}{x}{y}{\Gamma} = \inf_{\gamma \in \paths{x}{y}{\ol{V}}{\Gamma}} \lenmetres{D}{\gamma} ,\quad x,y \in \ol{V} \cap \Upsilon_\Gamma \]
where $\lenmetres{D}{\gamma}$ denotes the length of the path $\gamma$ under the metric $\metres{D}{\cdot}{\cdot}{\Gamma}$.
\end{theorem}

\subsection{Relationship with other work}
\label{subsec:otherwork}

We will now review some of the literature that is connected to this work.

\subsubsection{Critical percolation in two dimensions}

There are a number of works that have studied the chemical distance metric and the random walk on the two-dimensional critical percolation model which we conjecture to converge to the metrics that we construct in this paper for $\kappa = 6$.

The problem of using $\SLE$ techniques to study the chemical distance metric for critical percolation, especially in the context of determining the length of the shortest crossing across an $n \times n$ box (on the event that a crossing exists), is mentioned in Schramm's ICM contribution \cite[Problem~3.3]{s2007icm}.  It is natural to expect that length of the shortest crossing grows like $n^\alpha$ for some exponent $\alpha$.  The value of~$\alpha$ is not explicitly known, but was numerically estimated in \cite{hs1988percnumerical} to be approximately~$1.13$.  It was proved by Aizenman-Burchard \cite{ab1999holder} that $\alpha > 1$ and  Damron-Hanson-Sosoe proved in \cite{dhs2017chemical,dhs2021chemicalineq} that $\alpha$ is strictly smaller than the corresponding exponent for the length of the lowest crossing, verifying a conjecture of Kesten-Zhang \cite{kz1993conj}.  In particular, in the case of critical percolation on the triangular lattice the works \cite{dhs2017chemical,dhs2021chemicalineq} show that $\alpha < 4/3$ (this exact value comes from $\SLE$ techniques).  Let us also mention that radial versions of these statements are proved in \cite{sr2022radialchemical}.  There are also several numerical works which estimate the dimension of the limiting metric space and explore the extent to which the corresponding geodesics might be related to $\SLE$ curves \cite{pose2014shortest}.

For the simple random walk on the incipient infinite percolation cluster, Kesten \cite{kes-rw-percolation} proved its subdiffusive behavior with respect to Euclidean distance. The work \cite{gl2022chemicalsub} obtains stronger bounds for the diffusivity in terms of the chemical distance metric. There have been conjectures regarding the spectral dimension of the simple random walk, however the conjecture $d_s = 4/3$ made by Alexander and Orbach is believed to be false in dimensions smaller than $6$ \cite[Sect.~7.4]{hug-rwre}.

The present paper is the first step to constructing the continuum scaling limits of the critical percolation metrics.

\subsubsection{Critical percolation in high dimensions}

Critical percolation in high dimensions behaves very differently to the two-dimensional case. In the high-dimensional case, the critical percolation clusters are expected to be tree-like. Random walks on high-dimensional percolation models have been studied in \cite{bjks-rw-oriented-percolation,kn-alexander-orbach,hhh-rw-iic}. In this case, the Alexander-Orbach conjecture $d_s = 4/3$ does hold. The scaling limit of the simple random walk on a related model was proved in \cite{acf-rw-lattice-trees}. The chemical distance in high-dimensional percolation is studied in \cite{hs-percolation-high-dim}. In \cite{brbn-hypercube-percolation} the authors announced that they are planning to prove scaling limit results for the metric space induced by the high-dimensional percolation clusters. (They prove an analogous result for the hypercube percolation model.)

\subsubsection{Supercritical percolation}

There are also many works which have studied the analogous problems for supercritical percolation. In contrast to critical percolation, the supercritical percolation clusters resemble the Euclidean space (rather than being fractal). In particular, the chemical distance metric is comparable to the Euclidean metric, and the works \cite{gm1990supercritical,gm1990parabolic,ap1996chemical,gm2007chemical} obtain increasing strong bounds quantifying this. Similarly, the random walk on supercritical percolation clusters converge to ordinary Brownian motion, see \cite{bar-rw-supercritical,ss-rw-percolation,bb-rw-percolation,mp-rw-percolation} for variants of this result.

\subsubsection{Chemical distance metric for simple $\CLE$s}

As mentioned earlier, the analog of Theorem~\ref{th:tightness_metrics} for $\CLE_\kappa$ for $\kappa \in (8/3,4)$ was proved in \cite{m2021tightness}.  We emphasize that there are several key differences between the regime $\kappa \in (8/3,4)$ studied in \cite{m2021tightness} and the regime $\kappa \in (4,8)$ which is the focus of the present work which is why the argument given here is completely different from the one given in \cite{m2021tightness}.  Firstly, in the case $\kappa \in (8/3,4)$, the loops do not intersect themselves, each other, or the domain boundary.  This makes it possible to ``route'' paths around ``bad'' regions.  This is not possible in the case that $\kappa \in (4,8)$.  Secondly, in the case $\kappa \in (8/3,4)$, the topology of the associated metric is the same as the Euclidean topology on the carpet.  In contrast, in the case $\kappa \in (4,8)$ this is not true due to the existence of the non-trivial ``prime ends'' mentioned earlier.

\subsubsection{Uniform spanning tree metric}

The construction of the chemical distance metric for $\CLE_8$ and how it arises as the scaling limit of the chemical distance metric for the uniform spanning tree has been established in \cite{bck2017tightness,hs2018euclideanmating}.  In this situation, constructing the continuous metric is substantially simplified as the distance between two points is equal to the ``natural parameterization'' of the tree branch connecting them and the convergence can be extracted from \cite{lv2021natural}.

\subsubsection{Random planar maps and Liouville quantum gravity}

There have also been many works which are focused on constructing random planar metrics in the context of random planar maps and Liouville quantum gravity (LQG).  The latter refers to the random two-dimensional Riemannian manifold with formal metric tensor
\begin{equation}
\label{eqn:lqg_def}
e^{\gamma h(z)} (dx^2 + dy^2)
\end{equation}
where $\gamma \in (0,2)$ is a parameter, $dx^2  + dy^2$ denotes the Euclidean metric on a planar domain $D$, and $h$ is some form of the Gaussian free field (GFF).  Since the GFF is not a function but rather takes values in the space of distributions, \eqref{eqn:lqg_def} requires interpretation.  The associated metric was first constructed for $\gamma=\sqrt{8/3}$ in \cite{ms2020lqgtbm1,ms2021lqgtbm2,ms2021lqgtbm3} and then for all $\gamma \in (0,2)$ in \cite{dddf2020lqgmetric, gm2021lqgmetric}.  In contrast to the chemical distance metrics associated with $\CLE$, the LQG metric in this regime has the Euclidean topology.  The LQG metric in the so-called \emph{supercritical regime} was recently constructed in \cite{dg2023supertightness, dg2023superuniqueness} and this does not have the Euclidean topology.

There are several key differences between the LQG metric and the $\CLE$ metric.  First, the LQG metric is defined directly out the GFF.  This means that one can make use of many tools from the theory of Gaussian processes (concentration inequalities and independence) which do not have direct analogs in the context of $\CLE$.  Moreover, there are some special properties of the GFF (the Markov property and conformal invariance) whose $\CLE$ analogs are much more difficult to work with.  Second, in the context of LQG one considers surfaces modulo conformal transformation and so the Euclidean shape of a given domain is not important.  This allows one to make use to an extra array of independence properties which are not available in the context of $\CLE$.

Let us also mention the works which study random metrics in the context of random planar maps. Landmark works include the convergence of uniformly chosen random planar maps to the Brownian map given in \cite{lg2013bm,m2013brownian} and also the convergence to the so-called stable gasket or carpet in the context of random planar maps with large faces \cite{lgm2019largefaces,cmr2023largefaces}.  The techniques used in these papers are also quite different from the present work.  This is primarily due to the fact that in both of these settings there is an explicit description of the limiting metric space.  However, the ``quantum'' version of the $\CLE$ metric (where lengths are biased by the exponential of the GFF) should correspond to \cite{lgm2019largefaces,cmr2023largefaces} in the ``dense'' regime.

\subsection{Proof strategy and outline}  The remainder of this article is structured as follows.  We will collect some preliminaries in Section~\ref{se:preliminaries}.  In Section~\ref{se:intersection_exponent} we will obtain superpolynomial tails for the probability that the CLE metrics in the region between two intersecting loops exceed its median value. A key step in the proof involves estimating the distance between the endpoints of an individual connected component will be carried out in Section~\ref{se:bubble_exponent}. Finally, we will prove the tightness and regularity statements in the main theorems in Section~\ref{se:tightness_proof}, and the limiting CLE metric will be constructed in Section~\ref{se:construction_metric}.  Appendix~\ref{app:ghf} contains some preliminaries on the Gromov-Hausdorff-function topology, and Appendix~\ref{app:conditional_laws} contains some measure theoretic inputs regarding the tightness and convergence of conditional laws.  The purpose of Appendix~\ref{app:arms} is to collect results on various intersection exponents for SLE, and we will use them in Appendix~\ref{se:regularity} to obtain some strong estimates for the regularity of $\SLE_\kappa$ curves, in particular serving to quantify the degree to which bottlenecks can accumulate. In Appendix~\ref{se:cle_outer_boundary_exponent} we estimate the probability of multiple crossings of the exterior boundaries of \clekp{} loops.

\begin{figure}[ht]
\centering
\includegraphics[width=0.3\textwidth]{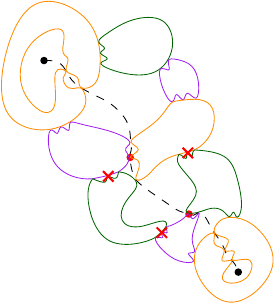}
\hspace{0.1\textwidth}
\includegraphics[width=0.3\textwidth]{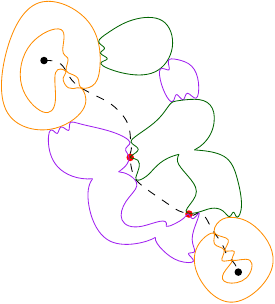}
\caption{\textbf{Left:} A simplified illustration of a few intersecting \clekp{} loops (only their outer boundaries are shown). \textbf{Right:} By resampling the \clekp{} near the points marked with an `x', we can compare the region between the two red points to a region that is bounded between two intersecting loops.}
\label{fi:outline_chaining}
\end{figure}

Let us now explain the strategy that we will use in more detail. By a variant of the Arzel\`a-Ascoli theorem, we have tightness of a sequence of random functions if we can prove H\"older continuity with uniform bounds on their H\"older constants. We perform a chaining argument to prove the latter. Due to the structure of \clekp{} in the non-simple regime $\kappa' \in (4,8)$, paths in the gasket will pass through intersection points of the \clekp{} loops (see the left side of Figure~\ref{fi:outline_chaining}). Therefore we need to gain uniform control over the distances between intersection points of loops.

In the first step we consider the region between the intersection of two \clekp{} loops. Such a region can be described by a pair of intersecting \slek{} flow lines (where $\kappa = 16/\kappa'$) with a particular angle difference $\angledouble$ \cite{mw2017intersections}. The precise setup will be given in Section~\ref{se:intersections_setup}. We define the normalizing constant $\median{\epsilon}$ be the median of $\metapprox{\epsilon}{x_0}{y_0}{\Gamma}$ where $x_0,y_0$ are two appropriately chosen points on the intersection of the two loops.

Our main aim is to show that
\begin{equation}
\label{eqn:main_concentration}
\p[ \metapprox{\epsilon}{\delta x_0}{\delta y_0}{\delta\Gamma} \geq \median{\epsilon} ] \to 0 \quad\text{as}\quad \delta \to 0
\end{equation}
faster than any power of $\delta$ (uniformly in $\epsilon$).

\begin{figure}[ht]
\centering
\includegraphics[width=0.2\textwidth]{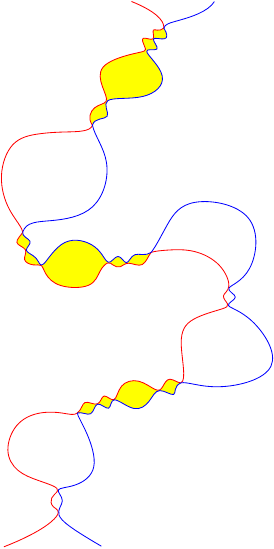}
\hspace{0.2\textwidth}
\includegraphics[width=0.25\textwidth]{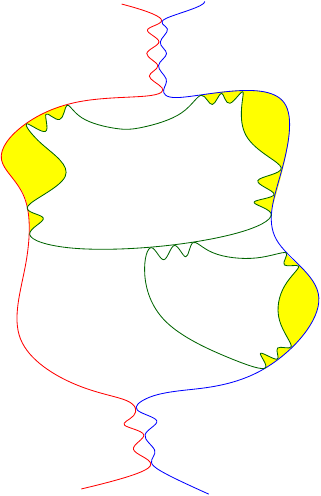}
\caption{\textbf{Left:} The self-similar structure of the regions bounded between two intersecting \clekp{} loops. \textbf{Right:} Inside each bubble, we find the same self-similar structure on both the left and the right side (only the outer boundaries of the \clekp{} loops are shown).}
\label{fi:outline_selfsimilarity}
\end{figure}

The first step to showing~\eqref{eqn:main_concentration} is first to prove an \emph{a priori} bound with a polynomial tail.  In particular, we will show that
\begin{equation}
\label{eqn:main_concentration_initial}
\p[ \metapprox{\epsilon}{\delta x_0}{\delta y_0}{\delta\Gamma} \geq \median{\epsilon} ] = O(\delta^{\ddouble+o(1)})
\end{equation}
where $\ddouble =  2-{(12-\kappa')(4+\kappa')}/({8\kappa'})$ is the double point dimension for $\SLE_{\kappa'}$ \cite{mw2017intersections}. We will prove~\eqref{eqn:main_concentration_initial} in Section~\ref{se:a_priori}. The idea is that the region $U$ bounded between the two intersecting flow lines has a self-similar structure, namely it contains approximately $\delta^{-\ddouble+o(1)}$ copies of $\delta U$ as sub-regions (see the left side of Figure~\ref{fi:outline_selfsimilarity}). If~\eqref{eqn:main_concentration_initial} were not true, then we argue that there is sufficient independence between the disjoint subregions so that it likely contains one with $\metapprox{\epsilon}{x_\delta}{y_\delta}{\Gamma} \geq \median{\epsilon}$. Since every path crossing $U$ needs to pass through these subregions, the series law implies that $\metapprox{\epsilon}{x_0}{y_0}{\Gamma} \geq \median{\epsilon}$ with high probability, contradicting the definition of $\median{\epsilon}$.

More generally, using a similar argument, we also show that
\begin{equation*}
\p[ \metapprox{\epsilon}{\delta x_0}{\delta y_0}{\delta\Gamma} \geq \delta^\zeta \median{\epsilon} ] = O(\delta^{\ddouble-\zeta+o(1)})
\end{equation*}
for $\zeta \ge 0$. This will allow us to bound distances by smaller constants than $\median{\epsilon}$ and sum them up via the triangular inequality.

To obtain~\eqref{eqn:main_concentration}, we let $\bestexp$ be the supremum of exponents $\alpha$ so that the left-hand side of~\eqref{eqn:main_concentration} is $O(\delta^\alpha)$. By~\eqref{eqn:main_concentration_initial} we have $\bestexp \geq \ddouble$. Our aim is to improve this a priori bound and show that $\bestexp = \infty$.  To prove this, we will use a ``bootstrapping'' argument to say that if $\bestexp < \infty$ then there exists $c > 0$ so that $\bestexp \geq \bestexp+c$, which is a contradiction. The argument consists of two parts. First, due to the structure of $U$ (recall Figure~\ref{fi:outline_selfsimilarity}), we can successively explore the region until we find an intersection point $x'$ such that $\metapprox{\epsilon}{x_0}{x'}{\Gamma} \geq \median{\epsilon}$. When this happens, we repeat the exploration from $x'$ onward and argue that we have sufficient independence so that the number of repetitions is bounded by a geometric random variable. It then remains to control the distances across the ``bubbles'' that terminate at the points $x'$ where the exploration has stopped. This is the most technical part of the proof and is carried out in Section~\ref{se:bubble_exponent}.

\begin{figure}[ht]
\centering
\includegraphics[width=0.3\textwidth]{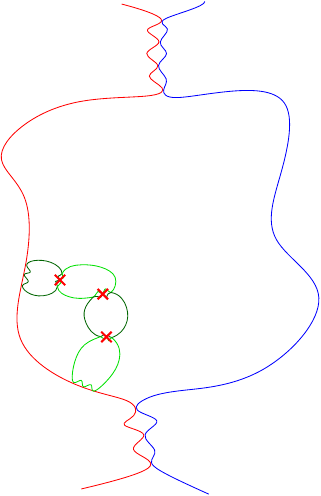}
\hspace{0.1\textwidth}
\includegraphics[width=0.3\textwidth]{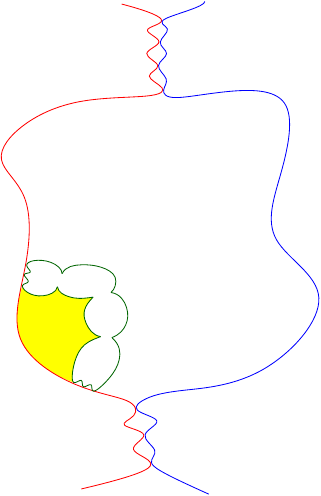}
\caption{By resampling the \clekp{} near the points marked with an `x', we can compare the regions bounded by a finite collection of loops to the type of regions bounded by a single loop.}
\label{fi:outline_bubble}
\end{figure}

For each such bubble, we consider the \clekp{} configuration inside the bubble. Using the self-similar structure of the \clekp{}, we argue that near the left (resp.\ right) boundary of the bubble, the \clekp{} configurations look like smaller copies of the region $U$. There are two cases. When a macroscopic loop intersects the left (resp.\ right) boundary of the bubble, the regions bounded between their intersection points have law comparable to that of $U$ (see the right side of Figure~\ref{fi:outline_selfsimilarity} for an illustration). To deal with the parts that lie between the intersections with multiple loops, we use the resampling argument developed in \cite{amy-cle-resampling} in order to compare their law to the region of the former type. The results in \cite{amy-cle-resampling} show that we can change the linking pattern of \clekp{} loops at their intersection points without significantly affecting the law of the \clekp{} configuration away from these points (see Figure~\ref{fi:outline_bubble} for an illustration). We conclude that, by our assumption, the $\metapprox{\epsilon}{\cdot}{\cdot}{\Gamma}$-distances between points on the left (resp.\ right) boundary of the bubble exceed $\median{\epsilon}$ with probability $O(\delta^{\bestexp})$.

We further argue that the \clekp{} near the left and the right side are sufficiently independent so that the probability that the distance exceeds $\median{\epsilon}$ on both sides is in fact $O(\delta^{2\bestexp})$. This would imply that left-hand side of~\eqref{eqn:main_concentration} is in fact bounded by $O(\delta^{2\bestexp})$. In reality, this independence holds only in regions that are sufficiently far away from the initial and final points of the bubble, and there will be various types of pathological behavior which we will need to rule out (e.g.\ pinch points where the two sides of a bubble come much closer to each other than the overall diameter of the bubble). Therefore we will only obtain a bound of the form $\bestexp \geq \bestexp+c$.

In order to argue that we have sufficient independence so that all of this works, we will need to establish some strong regularity statements which serve to quantify the occurrence of bottlenecks on $\SLE_\kappa$ curves.  These results may be of independent interest and are presented in Section~\ref{se:regularity}.

Once we have established~\eqref{eqn:main_concentration}, we then turn to complete the proof of tightness in Section~\ref{se:tightness_proof}.  The strategy is based on a chaining argument (in the spirit of the Kolmogorov-Chentsov theorem). For each $\delta > 0$ we consider a mesh of points that lie on the intersection of loops such that neighboring points have approximately Euclidean distance $\delta$. We will show for $\zeta \in (0,\ddouble)$ that the $\metapprox{\epsilon}{\cdot}{\cdot}{\Gamma}$-distances between neighboring points is at most $\delta^\zeta \median{\epsilon}$ with overwhelming probability. This is deduced from~\eqref{eqn:main_concentration} using various absolute continuity arguments comparing the regions between intersecting loops to the ones considered in~\eqref{eqn:main_concentration}. We will again rely on the resampling argument developed in \cite{amy-cle-resampling} to compare the regions bounded between multiple loops to the type of regions that are bounded by two loops as in~\eqref{eqn:main_concentration} (see Figure~\ref{fi:outline_chaining} for an illustration).

To conclude the tightness proof, suppose that $x,y \in \Upsilon_\Gamma$ are connected through a path $\gamma$ in the gasket, then we find a sequence of points on $\gamma$ that lie on the mesh points considered above, and hence we can bound the distance between $x,y$ by summing up the distances between pairs of points. This ultimately leads to a bound of the type $\metapprox{\epsilon}{x}{y}{\Gamma} \lesssim \diamE(\gamma)^\zeta$ which gives H\"older-continuity with respect to the metric $\dpath$. 

Finally, in Section~\ref{se:construction_metric} we construct a \clekp{} metric from each subsequential limit, and prove that it satisfies the axioms of a \clekp{} metric. The main noteworthy property is the Markovian property. Since the conditional independence is not preserved under weak limits in general, we need to leverage the tightness result to obtain convergence of the conditional laws. For this, we will use a strong continuity result from \cite{amy-cle-resampling} in the total variation sense for the conditional laws in the domain Markov property of CLE. The auxiliary lemmas regarding the conditional laws of weak limits are collected in Appendix~\ref{app:conditional_laws}.

\subsection*{Notation} All estimates in this paper will be uniform in $\epsilon \ge 0$ and within the class of approximate \clekp{} metrics with the same constants $\cserial,\cparallel(N)$ as in Definition~\ref{def:cle_metric}. Implicit constants such as in $O(\delta^\alpha)$, $o^\infty(\delta)$, etc.\ do not depend on $\epsilon,\delta$ (provided that $0 \le \epsilon < \delta \le 1$), and the choice of the approximate \clekp{} metric, but they may depend on $\cserial,\cparallel(N)$. We denote by $O(\delta^\alpha)$ a function that is bounded by a constant times $\delta^\alpha$. We denote by $o^\infty(\delta)$ a function that is $O(\delta^b)$ for any fixed $b>0$. We write $a \lesssim b$ if $a \le cb$ for a constant $c>0$, and we write $a \asymp b$ if $a \lesssim b$ and $b \lesssim a$. We always assume that $0 \le \epsilon < \delta \le 1$.

For $0<r_1<r_2$, we write $A(z,r_1,r_2) = B(z,r_2)\setminus\ol{B(z,r_1)}$. We sometimes use $\distE$ (resp.\ $\diamE$) to denote Euclidean distance (resp.\ diameter) in order to distinguish from the metric $\dpath$ defined in~\eqref{eq:dpath}. We write $U \Subset V$ to denote that $U$ is compactly contained in $V$, i.e.\ $\ol{U}$ is compact and $\ol{U} \subseteq V$.

\subsection*{Acknowledgements}  V.A., J.M., and Y.Y.\ were supported by ERC starting grant SPRS (804116). J.M.\ and Y.Y.\ also received support from ERC consolidator grant ARPF (Horizon Europe UKRI G120614), and Y.Y.\ in addition received support by the Royal Society.  We thank Maarten Markering for many helpful comments on an earlier draft of the paper.

\section{Preliminaries}
\label{se:preliminaries}

The purpose of this section is to collect a number of preliminaries which will be important for the rest of this work.  We will start in Section~\ref{subsec:sle} by reviewing some of the basic properties of $\SLE$.  We will then describe some of the basic properties of $\CLE$ in Section~\ref{subsec:cle}.  In Section~\ref{subsec:ig} we will review the coupling between $\SLE$ and the Gaussian free field (GFF) and collect some preliminary lemmas for $\SLE$ and the GFF.

Throughout this paper, we set
\[
\kappa \in (2,4), \quad
\kappa' = \frac{16}{\kappa} \in (4,8), \quad
d_\SLE=1+\frac{\kappa}{8}, \quad \lambda = \frac{\pi}{\sqrt{\kappa}}, \quad
\chi=\frac{2}{\sqrt{\kappa}}-\frac{\sqrt{\kappa}}{2} .
\]
Further, let
\begin{equation}\label{eq:angle_double}
\angledouble = (\kappa-2)\lambda/\chi = \pi (\kappa-2)/ (2-\kappa/2)
\end{equation}
be the double point angle and
\begin{equation}\label{eq:dim_double}
 \ddouble = 2 - \frac{(12-\kappa')(4+\kappa')}{8\kappa'}
\end{equation}
be the double point dimension of $\SLE_{\kappa'}$ (cf.\ \cite{mw2017intersections}).

\subsection{Conformal maps}

We collect a few useful lemmas concerning conformal maps.

\begin{lemma}
\label{le:bounded_crosscut}
There exists $c>1$ such that the following is true. Let $D$ be a simply connected domain with $0,\infty \notin D$, and $f\colon \h \to D$ a conformal map. Then there exists a crosscut $\Xi$ of $\h$ from $\R_+$ to $\R_-$ such that $\Xi \subseteq A(0,1,2) \cap \h$ and $f(\Xi) \subseteq A(0,c^{-1}|f(i3/2)|,c|f(i3/2)|)$.
\end{lemma}

\begin{proof}
With positive probability (say $1/100$), a Brownian motion starting from $i3/2$ will hit the interval $[1,2]$ before leaving the annulus $A(0,1,2)$. On the other hand, since $0,\infty \notin D$, by Beurling's estimate, the probability of a Brownian motion in $D$ starting from $f(i3/2)$ hitting $\partial B(0, c|f(i3/2)|)$ before hitting $\partial D$ is $O(c^{-1/2})$. Likewise, its probability of hitting radius $\partial B(0, c^{-1}|f(i3/2)|)$ before hitting $\partial D$ is also $O(c^{-1/2})$. For $c > 1$ large enough, these probabilities will be smaller than $1/100$, and therefore there exists a Brownian trajectory that hits the interval $[1,2]$ and whose image is contained in the desired annulus. The same is true for the interval $[-2,-1]$. Hence, picking a subcurve of the Brownian trajectories, there is also a crosscut of $\h$ with the desired property.
\end{proof}

\begin{lemma}
\label{lem:hyperbolic_geodesic_facts}
 Let $D \subseteq \C$ be a simply connected domain, and $x,y \in \partial D$ distinct. Let $\partial_1 D$ (resp.\ $\partial_2 D$) be the clockwise (resp.\ counterclockwise) arc of $\partial D$ between $x$ and $y$. Let $z$ be a point on the hyperbolic geodesic in $D$ from $x$ to $y$. Then
 \[ \dist(z,\partial_1 D) \le \frac{64}{\pi^2}\dist(z,\partial_2 D) . \]
\end{lemma}

\begin{proof}
 We can assume that $\dist(z,\partial_1 D) > \dist(z,\partial_2 D)$. Suppose $\varphi\colon D \to \D$ is a conformal transformation sending $x,y$ to $-i,i$. Since $z$ is on the hyperbolic geodesic from $x$ to $y$, we have $\varphi(z) \in [-i,i]$. Let $B$ be a Brownian motion starting at $z$. By the symmetry and conformal invariance of Brownian motion, the probability that $B$ first hits $\partial_1 D$ (resp.\ $\partial_2 D$) is $1/2$. On the other hand, by Beurling's estimate, the probability that $B$ hits $\partial_1 D$ before $\partial_2 D$ is at most $(4/\pi)(\dist(z,\partial_2 D)/\dist(z,\partial_1 D))^{1/2}$ (see, e.g., \cite[Section~3.8]{L05conf_proc} for the explicit constant in the Beurling estimate). The claim follows by combining these two facts.
\end{proof}

\subsection{Schramm-Loewner evolution}
\label{subsec:sle}

We will now review some basics for the $\SLE_\kappa$ processes and their variants.  The starting point for the definition of \emph{chordal $\SLE_\kappa$} is the \emph{chordal Loewner equation} which for $U \colon \R_+ \to \R$ continuous is given by
\begin{equation}
\label{eqn:loewner_ode}
\partial g_t(z) = \frac{2}{g_t(z) - U_t},\quad g_0(z) = z.
\end{equation}
For each $z \in \h$ we have that $(g_t(z))$ is defined up until $\tau_z = \inf\{t \geq 0 : \im(g_t(z)) = 0\}$.  The particular choice $U = \sqrt{\kappa} B$ for $B$ a standard Brownian motion and $\kappa \geq 0$ gives the definition of \emph{chordal $\SLE_\kappa$}.   It was proved by Rohde-Schramm \cite{rs2005basic} ($\kappa \neq 8$) and Lawler-Schramm-Werner \cite{lsw2004lerw} ($\kappa = 8$, see also \cite{am2022sle8}) that $\SLE_\kappa$ is generated by a continuous curve.  This means that there exists a continuous curve $\eta \colon \R_+ \to \ol{\h}$ such that for each $t \geq 0$ the domain of $g_t$ is equal to the unbounded component of $\h \setminus \eta([0,t])$.  As we mentioned earlier, the $\SLE_\kappa$ curves are simple for $\kappa \in (0,4]$, self-intersecting but not space-filling for $\kappa \in (4,8)$, and space-filling for $\kappa \geq 8$ \cite{rs2005basic}.

The \emph{$\SLE_\kappa(\rho)$ processes} are an important variant of $\SLE_\kappa$ where one keeps track of extra marked points. These processes were first introduced in \cite[Section~8.3]{lsw2003restriction}; see also \cite{sw2005coordinate} for a number of useful results about $\SLE_\kappa(\rho)$. It is given by solving~\eqref{eqn:loewner_ode} with the solution to
\begin{equation}
\label{eqn:sle_kappa_rho}
dW_t = \sqrt{\kappa} dB_t + \sum_i \re\left( \frac{\rho_i}{W_t - V_t^i}\right) dt, \quad dV_t^i = \frac{2}{V_t^i - W_t} dt,\quad V_0^i = z^i
\end{equation}
used in place of $U$ in~\eqref{eqn:loewner_ode}.  It is shown in \cite{ms2016ig1} that~\eqref{eqn:sle_kappa_rho} has a unique solution up until the \emph{continuation threshold}, which is the first time $t$ that the sum of the weights of the force points which have collided with $W_t$ is at most $-2$ and also that the resulting process is generated by a continuous curve.

We remark that it is also possible to make sense of the $\SLE_\kappa(\rho)$ processes in the regime that $\rho \in (-2-\kappa/2,-2)$.  For $\rho \in [\kappa/2-4,-2)$, the continuity of these processes was proved in \cite{ms2019lightcone} and for $\rho \in (-2-\kappa/2,\kappa/2-4)$ the continuity was proved in \cite{msw2017cleperc}.

\emph{Radial $\SLE_\kappa$} is given by solving the \emph{radial Loewner equation}
\begin{equation}
\label{eqn:radial_loewner}
\partial_t g_t(z) = g_t(z) \frac{W_t + g_t(z)}{W_t - g_t(z)},\quad g_0(z) = z
\end{equation}
where $W_t = e^{i \sqrt{\kappa} B_t}$ and $B$ is a standard Brownian motion.   For each $z \in \D$ we have that $(g_t(z))$ is defined up until $\tau_z = \inf\{t \geq 0 : |g_t(z)| = 1\}$.  Let $K_t = \{z : \tau_z \leq t\}$.  Then $g_t$ defines a conformal transformation $g_t \colon \D \setminus K_t \to \D$ with $g_t(0) = 0$ and $g_t'(0) > 0$.  As in the case of chordal $\SLE_\kappa$, for each $\kappa \geq 0$ there exists a continuous curve $\eta \colon \R_+ \to \ol{\D}$ so that for each $t \geq 0$ the connected component of $\D \setminus \eta([0,t])$ containing $0$ is equal to $\D \setminus K_t$.

Let
\begin{equation}
\Psi(w,z) = -z \frac{z+w}{z-w} \quad\text{and}\quad \wt{\Psi}(z,w) = \frac{\Psi(z,w) + \Psi(1/\ol{z},w)}{2}.
\end{equation}
Then the radial $\SLE_\kappa(\rho)$ processes are defined by taking $W$ to be the solution to
\begin{equation}
\label{eqn:radial_sle_kappa_rho}
d W_t = \left(-\frac{\kappa}{2} + \frac{\rho}{2} \wt{\Psi}(O_t,W_t) \right) dt + i \sqrt{\kappa} W_t dB_t,\quad dO_t = \Psi(W_t,O_t) dt.
\end{equation}
It was shown in \cite{ms2017ig4} that for each $\rho > -2$ the radial $\SLE_\kappa(\rho)$ processes correspond to a continuous curve.

\emph{Whole-plane $\SLE_\kappa(\rho)$} is defined by solving~\eqref{eqn:radial_loewner} for all $t \in \R$ (rather than just for $t \in \R_+$) with the time-stationary solution to~\eqref{eqn:radial_sle_kappa_rho} (which is defined for all $t \in \R$).   In this case, for each $z \in \C$ we have that $(g_t(z))$ is defined up until $\tau_z = \inf\{t \in \R : |g_t(z)| = 1\}$.  Let $K_t = \{z : \tau_z \leq t\}$ as before.  Then $g_t$ defines a conformal transformation $\C \setminus K_t \to \C \setminus \ol{\D}$ which fixes and has positive derivative at $\infty$.  It was proved in \cite{ms2017ig4} that for each $\rho > -2$ the whole-plane $\SLE_\kappa(\rho)$ processes are generated by a continuous curve.

\emph{Two-sided whole-plane $\SLE_\kappa$} refers to a curve in $\C$ from $\infty$ to $\infty$ through $0$.  It is defined by first sampling a whole-plane $\SLE_\kappa(2)$ curve $\eta$ from $\infty$ to $0$ and then concatenating with $\eta$ in the component of $\C \setminus \eta$ with $0$ on its boundary a chordal $\SLE_\kappa$ curve from $0$ to $\infty$.  The importance of these curves is that they describe the local behavior of an $\SLE_\kappa$ curve which has been conditioned to pass through a given interior point \cite{fie-two-sided-radial,zhan-decomposition}. An important property of the two-sided whole-plane $\SLE_\kappa$ is that when it is parameterized by its $d_\SLE$-dimensional Minkowski content, it is self-similar with index $1/d_\SLE$ and has stationary increments \cite{zhan-sle-loop}.

We collect below a few lemmas that will be used in the paper, distinguishing between the variants of SLE just mentioned.

\subsubsection{Chordal SLE}

\begin{lemma}
\label{le:variation_one_scale}
For any $p > d_\SLE$ there exists $c>0$ such that the following is true. Let $\varphi$ be a conformal transformation from $(D,x,y)$ to $(\h,0,\infty)$, and let $\delta = \dist(\varphi^{-1}(i),\partial D)$. Let $\eta$ be a chordal \slek{} in $(D,x,y)$. For $r,u>0$, let $E$ be the event that for every $s,t$ with $\varphi(\eta[s,t]) \subseteq [-u^{-1},u^{-1}]\times[u,u^{-1}]$ there are $r^{-p}u^{-c}$ times $s = t_1 < t_2 < \cdots < t_{r^{-p}u^{-c}} = t$ such that $\diam(\eta[t_{\ell},t_{\ell+1}]) \le r\delta$ for each $\ell$. Then $\p[E^c] = o^\infty(ru)$ as $ru \searrow 0$.
\end{lemma}
\begin{proof}
By Koebe's distortion theorem, there exists $c>0$ such that $|(\varphi^{-1})'(z)| \lesssim u^{-c}\delta$ on $[-u^{-1},u^{-1}]\times[u,u^{-1}]$. Therefore it suffices to find $s = t_1 < t_2 < \cdots < t_{r^{-p}u^{-c}} = t$ such that $\diam(\varphi(\eta[t_{\ell},t_{\ell+1}])) \le ru^c$ for each $\ell$.

Consider $\eta_{\h} = u\varphi(\eta)$ so that $\eta_{\h}[s,t] \subseteq [-1,1]\times[0,1]$. Suppose that there are $s,t$ with $\eta_{\h}[s,t] \subseteq [-1,1]\times[0,1]$ and $s = t_1 < t_2 < \cdots < t_{(ru^{c+1})^{-p}} = t$ such that $\diam(\eta_{\h}[t_{\ell},t_{\ell+1}]) \ge ru^{c+1}$ for each~$\ell$. Then, for $p' \in (d_\SLE,p)$, the $p'$-variation of $\eta_{\h}\big|_{[s,t]}$ will be at least $(ru^{c+1})^{-(p-p')/p'}$. By \cite[Proposition~6.3]{hy-sle-regularity}, this event has probability $o^{\infty}(ru^{c+1})$.
\end{proof}

\subsubsection{Two-sided whole-plane SLE}

\begin{lemma}
There exists $c>0$ such that if $\eta$ denotes a two-sided whole-plane $\SLE_\kappa$ in $\C$ from $\infty$ to $\infty$ through $0$, parameterized according to the natural parameterization with $\eta(0) = 0$, then
\begin{align}
\p[\diam(\eta[0,t]) \ge rt^{1/d_\SLE}] &\le c^{-1}\exp(-cr^{d_\SLE/(d_\SLE-1)}) , \label{eq:wpsle_ub}\\
\p[\diam(\eta[0,t]) \le r^{-1}t^{1/d_\SLE}] &= o^\infty(r^{-1}) \label{eq:wpsle_lb}
\end{align}
as $r\to\infty$.
\end{lemma}
\begin{proof}
The upper bound~\eqref{eq:wpsle_ub} is \cite[Proposition~3.10]{hy-sle-regularity}.  The lower bound~\eqref{eq:wpsle_lb} follows from \cite{rz2017higher} and \cite{zhan-decomposition}. Indeed, by scaling, it suffices to show $\E[(\Cont_\eta(B(0,1)))^n] < \infty$ for any $n \in \N$ where $\Cont_\eta(dz)$ denotes the Minkowski content measure of $\eta$. By the conformal covariance of the natural parameterization, the expectation is at most a constant times $\E_{\h,z}[(\Cont_{\eta_{\h,z}}(B(z,1)))^n]$ for two-sided radial SLE $\eta_{\h,z}$ in $(\h,0,\infty)$ through $z \in [0,1]\times[2,3]$ (where the constant does not depend on the choice of $z$). Consider the rooted measure $G(z)dz \otimes \p_{\h,z}$ which by \cite[Theorem~4.1]{zhan-decomposition} agrees with $\p_\h(d\eta_\h) \otimes \Cont_{\eta_\h}(dz)$ where $G$ is the $\SLE_\kappa$ Green's function in $(\h,0,\infty)$.  Hence,
\begin{align*}
\E[(\Cont_\eta(B(0,1)))^n] 
&\lesssim \int_{[0,1]\times[2,3]} \E_{\h,z}[(\Cont_{\eta_{\h,z}}(B(z,1)))^n] \,G(z)\,dz \quad\text{(since $\inf_{z \in [0,1] \times [2,3]} G(z) > 0$)} \\
&= \E_\h \int_{[0,1]\times[2,3]} (\Cont_{\eta_{\h}}(B(z,1)))^{n} \,\Cont_{\eta_\h}(dz) \\
&\le \E_\h[ (\Cont_{\eta_{\h}}([-1,2]\times[1,4]))^{n+1} ] \\
&< \infty,
\end{align*}
where in the final inequality we used \cite[Theorem~1.2]{rz2017higher}.  This is what we wanted to show.
\end{proof}

\subsection{$\CLE$ and $\BCLE$}
\label{subsec:cle}

We will now review the definition of the \emph{conformal loop ensembles} ($\CLE_{\kappa'}$) \cite{s2009cle} in the $\BCLE_{\kappa'}$ framework \cite{msw2017cleperc}.  We will focus on the case that ${\kappa'} \in (4,8)$ since this is the regime which is relevant for this work.  The starting point for the construction of the $\CLE_{\kappa'}$ is the so-called \emph{exploration tree}.  Suppose that $D \subseteq \C$ is a simply connected domain and $x \in \partial D$.   Fix a countable dense set $(y_n)$ in $\partial D$ and for each~$n$ we let~$\eta_n$ be an $\SLE_{\kappa'}({\kappa'}-6)$ process in $D$ from~$x$ to~$y_n$.   We assume that the force point for each $\eta_n$ is located at $x_+$ (i.e., infinitesimally to the right of $x$ on $\partial D$).  We assume that the $\eta_n$ are coupled together so as to agree until they separate their target points \cite{sw2005coordinate} and then evolve independently afterwards.  We can generate from the $\eta_n$ a collection of boundary touching loops as follows.  Fix a value of $m$ and let $[\sigma,\tau]$ be an interval of time in which $\eta_m$ makes an excursion from the counterclockwise arc of $\partial D$ from $x$ to $y_m$.  Now let $(y_{n_k})$ be a subsequence of the $y_n$ with $y_{n_k} \to \eta_m(\sigma)_+$ as $k \to \infty$ and let $\eta_{m,k}$ be the concatenation of $\eta_m|_{[\sigma,\tau]}$ with the part of $\eta_{n_k}$ from when it hits $\eta_m(\tau)$ up until when it first disconnects $y_{n_k}$ from $\eta_m(\sigma)$.  Then as $k \to \infty$, we have that $\eta_{m,k}$ converges to a loop from $\eta_m(\sigma)$ to itself.  The so-called $\cwBCLE_{\kappa'}(0)$ process consists of all such loops as both $m$ and $[\sigma,\tau]$ vary.

We say that $z \in D$ is surrounded by a \emph{true loop} of the $\cwBCLE_{\kappa'}(0)$ if it is surrounded by one of the aforementioned loops.  Otherwise and if $z$ is not on one of the loops, we say that $z$ is surrounded by a \emph{false loop} of the $\cwBCLE_{\kappa'}(0)$.

It is proved in \cite{ms2016ig3} that the law of the resulting loop ensemble does not depend on the choice of~$x$ (i.e., is root invariant) and in \cite{ms2017ig4} that it is \emph{locally finite}, meaning that for each $\epsilon > 0$ the number of loops which have diameter at least $\epsilon$ is finite.

The set of true loops of a $\cwBCLE_{{\kappa'}}(0)$ process in $D$ has the same law as the set of loops of a $\CLE_{\kappa'}$ process $\Gamma$ in $D$ that intersect the boundary.  We obtain the entire $\CLE_{\kappa'}$ by adding an independent $\cwBCLE_{{\kappa'}}(0)$ into each of the regions which is not surrounded by a loop (i.e, inside each of the false loops) and then iterating.  We can also define a \emph{nested} $\CLE_{\kappa'}$ by starting off with a non-nested $\CLE_{\kappa'}$ and then adding into each of the loops an independent $\CLE_{\kappa'}$.

Suppose that $\Gamma$ is a nested $\CLE_{\kappa'}$ in $D$ and $x \in \partial D$.  We can define a space-filling path $\Lambda$ from $\Gamma$ in the following way.  We let $\Lambda_0$ be the path that starts from $x$ and goes around $\partial D$ counterclockwise.  We define $\Lambda_1$ as follows. Choose an arbitrary point $y \in \partial D$ distinct from $x$, and let $\partial_{x,y} D$ be the counterclockwise arc of $\partial D$ from $x$ to $y$. Let $\Gamma_{x,y}$ be the loops of $\Gamma$ that intersect $\partial_{x,y} D$. For each $\CL \in \Gamma_{x,y}$ let $a_\CL$ (resp.\ $b_\CL$) be the first (resp.\ last) intersection point of $\partial_{x,y} D$ with $\CL$. Let $\CL^O$ be the clockwise segment of $\CL$ from $a_\CL$ to $b_\CL$, and $\CL^I$ its clockwise segment from $b_\CL$ to $a_\CL$. We let $\wt{\Gamma}_{x,y} \subseteq \Gamma_{x,y}$ be the collection of loops $\wt{\CL}$ that are not separated from $y$ by $\CL^I$ for any $\CL \in \Gamma_{x,y}$. We let $\Lambda_1$ be the path that begins like $\Lambda_0$, then traces the segments $\wt{\CL}^O$ for $\wt{\CL} \in \wt{\CL}_{x,y}$ in the order of $\partial_{x,y} D$, and finally traces the segments $\wt{\CL}^I$ for $\wt{\CL} \in \wt{\CL}_{x,y}$ in the reverse order. Given that we have defined $\Lambda_n$ for some $n \geq 0$, we define $\Lambda_{n+1}$ from $\Lambda_n$ as follows. Suppose $U$ is a connected component of $D \setminus \Lambda_n$. Let $x_U \in \partial U$ be the last point visited by $\Lambda_n$. For each $U$, after $\Lambda_n$ completes $\partial U$ and arrives at $x_U$, we extend it via the same procedure used to obtain $\Lambda_1$ from $\Lambda_0$ if $\partial U$ is surrounded by $\Lambda_n$ counterclockwise (the reflected procedure if $\partial U$ is surrounded by $\Lambda_n$ clockwise). In the limit as $n \to \infty$, we have that $\Lambda_n$ converges to a continuous path $\Lambda$ which explores all of the loops of $\Gamma$.  This path is the so-called space-filling $\SLE_{\kappa'}$; we will describe another construction of it just below from \cite{ms2017ig4} which is the perspective that was used to prove its continuity in \cite{ms2017ig4}.

\subsubsection{The thin gasket}

Suppose that $\Gamma$ is a \clekp{} in $D$. The \emph{thin gasket} $\CT_\Gamma$ of $\Gamma$ is the set of points in $\ol{D}$ that are connected to $\partial D$ by at least two disjoint admissible paths for $\Gamma$. In other words, $\CT_\Gamma$ is the set of points that are not in the interior components of the outer boundaries of the loops of $\Gamma$. The relevance of the thin gasket for our CLE metrics is that the CLE metrics we consider do not see ``dead ends'' that are disconnected by single CLE loops (every path that enters such a region needs to come back through the same point). In contrast to the full \clekp{} loops, their outer boundaries do not make double points. (This would correspond to a $6$-arm event in the language of lattice models.) The following lemma makes this quantitative.

\begin{lemma}
\label{le:thin_gasket_hoelder}
 For any $b>0$ there exists $\zeta > 0$ such that
 \[
  \p\left[ \sup_{u,v\in\CT_\Gamma} \frac{\dpath(u,v)}{\abs{u-v}^\zeta} > M \right] = O(M^{-b}) 
  \quad\text{as } M \to \infty .
 \]
\end{lemma}

The main input to the proof of Lemma~\ref{le:thin_gasket_hoelder} is the fact that the outer boundary of a $\CLE_{\kappa'}$ loop does not have double points with large probability, which is the content of Lemma~\ref{lem:thin_cle_lwb} in Appendix~\ref{app:arms}.

\begin{proof}[Proof of Lemma~\ref{le:thin_gasket_hoelder}]
 Recall $\alpha_{4,\kappa}$ from \eqref{eqn:double_exponent_simple}. Since $\alpha_{4,\kappa} > 2$ for $\kappa < 4$, we can choose $\zeta > 0$ so that $(1-\zeta)\alpha_{4,\kappa} - 2 > 0$.
 
 For $\delta \in (0,1)$ let $E_{\delta}^4$ be the event that $\Gamma$ contains a loop whose exterior boundary makes $4$ crossings across some annulus $A(z,2\delta,\delta^\zeta) \subseteq \D$. By Lemma~\ref{lem:thin_cle_lwb}, we have $\p[E_{\delta}^4] = O(\delta^{(1-\zeta)\alpha_{4,\kappa}-2})$.
 
 Suppose we are on the event $\bigcap_{2^{-k} < \delta} (E_{2^{-k}}^4)^c$. We claim that $\dpath(u,v) \le \delta^{-\zeta}\abs{u-v}^\zeta$ for any $u,v\in\CT_\Gamma$. This will imply the result since we can pick $\zeta > 0$ small.
 
 To see the claim, we consider two cases. If $\abs{u-v} \le \delta$, pick $k$ such that $\abs{u-v} \in [2^{-k-1},2^{-k}]$. Then $u,v \in B(z,2^{-k+1})$ for some $z$. On the event $(E_{2^{-k}}^4)^c$, we have that $\dpath(u,v) \le 2^{-k\zeta} \le (2\abs{u-v})^\zeta$. On the other hand, if $\abs{u-v} > \delta$, then $\dpath(u,v) \le 1 \le \delta^{-\zeta}\abs{u-v}^\zeta$.
\end{proof}

\subsubsection{Partial explorations and resamplings of $\CLE_{\kappa'}$}
\label{se:cle_resampling}

\newcommand{\domainpair}[1]{{\mathfrak {P}}_{#1}}
\newcommand{\outside}{{\mathrm{out}}}
\newcommand{\inside}{{\mathrm{in}}}
\newcommand{\resampled}{{\mathrm{res}}}

In \cite{amy-cle-resampling} we considered partial resamplings of a \clekp{} where we explore part of the \clekp{} and resample the unexplored parts according to their conditional law, yielding a coupling between two \clekp{} instances that agree on the explored parts. We describe the setup and some of the results which will be relevant for this work.

Let $D \subseteq \C$ be a simply connected domain and $\varphi \colon D \to \D$ be a conformal transformation.  Let $\domainpair{D}$ consist of all pairs $(U,V)$ of simply connected sub-domains $U \subseteq V \subseteq D$ with $\dist(\varphi(D\setminus V), \varphi(U)) > 0$.  (Note that $\domainpair{D}$ does not depend on the choice of $\varphi$.) 

Suppose $\Gamma$ is a nested \clekp{} in $D$, and $(U,V) \in \domainpair{D}$. The \emph{partial exploration} of $\Gamma$ in $D\setminus V$ until hitting $\ol{U}$, denoted by $\Gamma_\outside^{*,V,U}$, is the collection of maximal segments of loops and strands in $\Gamma$ that intersect $D \setminus V$ and are disjoint from $U$. Let $V^{*,U}$ be the connected component containing $U$ after removing from $D$ all loops and strands of $\Gamma_\outside^{*,V,U}$. This defines a simply connected domain with a finite number of marked points $\ul{x}^*$ on its boundary, corresponding to the unfinished strands in $\Gamma_\outside^{*,V,U}$, which are pairwise linked by an \emph{exterior planar link pattern} $\beta^*$ induced by the strands.

The following is one of the main results in \cite{amy-cle-resampling}. The \emph{multichordal \clekp{}} in $(D;\ul{x};\beta)$ where $D$ is a simply connected domain, $\ul{x}$ is a finite collection of distinct prime ends, and $\beta$ is an exterior planar link pattern between the marked points, is a probability measure on the collections of loops and paths in $\ol{D}$ where each path connects a pair of marked points. We refer to \cite{amy-cle-resampling} for a definition of the law, and will only summarize its main properties here. (See in particular \cite[Theorems~1.6 and~1.11]{amy-cle-resampling}.)

\begin{theorem}
\label{thm:cle_partially_explored}
Consider the setup described above. The conditional law of the remainder of $\Gamma$ given $\Gamma_\outside^{*,V,U}$ is a multichordal \clekp{} in $(V^{*,U};\ul{x}^*;\beta^*)$ with marked points $\ul{x}^*$ and exterior link pattern $\beta^*$ given as above. The family of multichordal \clekp{} laws is conformally invariant in the sense that if $\Gamma$ is a multichordal \clekp{} in $(D;\ul{x};\beta)$ and $\varphi\colon D \to \wt{D}$ is a conformal transformation, then $\varphi(\Gamma)$ is a multichordal \clekp{} in $(\wt{D};\varphi(\ul{x});\beta)$. For each \emph{interior link pattern} $\alpha$, the probability that the chords of a multichordal \clekp{} in $(\D;\ul{x};\beta)$ induce the link pattern $\alpha$ is positive and depends continuously on $\ul{x}$.
\end{theorem}

Further, we have continuity in total variation when we restrict to regions away from the boundary. We will need the following consequences of \cite[Proposition~6.3]{amy-cle-resampling}. Recall the definition of $U^*$ in Section~\ref{se:assumptions}. We refer to \cite[Section~6.1]{amy-cle-resampling} for the topology with respect to which continuity holds.

\begin{lemma}\label{le:partial_cle_tv}
 Let $(U,V) \in \domainpair{D}$ and $U_1 \Subset U$ open, simply connected. Let $\alpha^*$ denote the interior link pattern induced by the remainders of the unfinished loops in $\Gamma_\outside^{*,V,U}$ (which connect the marked points $\ul{x}^*$ in $V^{*,U}$). The conditional law of $(U_1^*,\alpha^*)$ given $\Gamma_\outside^{*,V,U}$ is a continuous function of $\Gamma_\outside^{*,V,U}$ to the space of probability measures equipped with the total variation distance.
\end{lemma}

\begin{lemma}\label{le:translation_tv}
 Let $U \Subset D$ be open, simply connected. For $z \in \C$ let $T_z(w) = w+z$. Then, as $z \to 0$, the pushforward of the law of $(U+z)^*$ under $T_{-z}$ converges in total variation to the law of $U^*$.
\end{lemma}

\begin{proof}
 The law of the pushforward of the law of $(U+z)^*$ under $T_{-z}$ is the same as the law of $U^*$ with respect to a \clekp{} in $D-z$. Therefore the claim follows from \cite[Proposition~6.3]{amy-cle-resampling}.
\end{proof}

Given a \clekp{} $\Gamma$, we can consider resamplings of $\Gamma$ within a set $W \subseteq D$. By this we mean that we repeatedly select regions $(U,V) \in \domainpair{D}$ with $V \subseteq W$ randomly in a way that is independent of the CLE configuration within $W$, and let $\wt{\Gamma}$ be the \clekp{} such that $\wt{\Gamma}_\outside^{*,V,U} = \Gamma_\outside^{*,V,U}$ and the remainder of $\wt{\Gamma}$ is sampled from its conditional law given $\wt{\Gamma}_\outside^{*,V,U}$ (independently of the remainder of $\Gamma$). This procedure can be repeated a random number of times, yielding a coupling of $\Gamma$ with another \clekp{}~$\Gamma^\resampled$.

\begin{lemma}[{\cite[Lemma~5.7]{amy-cle-resampling}}]
\label{lem:disconnect_interior}
Suppose that $\Gamma$ is a nested \clek{} in a simply connected domain $D$, and let $\Upsilon_\Gamma$ be its gasket. For each $a\in(0,1)$, $b>0$ there exists $p > 0$ such that the following holds. 

Let $z\in D$ and $j_0\in\Z$ be such that $B(z,2^{-j_0})\subseteq D$. There exists a resampling $\Gamma^\resampled_{z,j}$ of $\Gamma$ within $A(z,2^{-j},2^{-j+1})$ for each $j>j_0$ with the following property. Let $G^\resampled_{z,j}$ be the event that $\Gamma^\resampled_{z,j}$ has a loop $\wt\CL$ that
\begin{itemize}
 \item disconnects every point in $\Upsilon_{\Gamma^\resampled_{z,j}} \cap B(z,2^{-j})$ from $\partial B(z,2^{-j+1})$,
 \item the gasket and the collection of loops of $\Gamma$ and $\Gamma^\resampled_{z,j}$ remain the same in each connected component of $\C \setminus \wt\CL$ that intersects $B(z,2^{-j})$.
\end{itemize}
For $k\in\N$, let $\wt G_{z,j_0,k}$ be the event that the number of $j=j_0+1,\ldots,j_0+k$ so that
\[ \p[ G^\resampled_{z,j} \giv \Gamma ] \geq p\]
is at least $(1-a)k$.  Then
\[ \p[(\wt G_{z,j_0,k})^c] = O(e^{-b k}) \]
where the implicit constant does not depend on $z,j_0,k$.
\end{lemma}

\begin{lemma}[{\cite[Lemma~5.8]{amy-cle-resampling}}]
\label{lem:disconnect_boundary}
Suppose that $\Gamma$ is a nested \clekp{} in $\h$. For each $a\in(0,1)$, $b>0$ there exists $p > 0$ such that the following holds. 

There exists a resampling $\Gamma^\resampled_{j}$ of $\Gamma$ within $A(0,2^{-j},2^{-j+1}) \cap \h$ for each $j\in\N$ with the following property. Let $G^\resampled_{j}$ be the event that $\Gamma^\resampled_{j}$ has a loop $\wt\CL$ that
\begin{itemize}
 \item disconnects $[-2^{-j},2^{-j}]$ from $\partial B(0,2^{-j+1})$,
 \item the collection of loops of $\Gamma$ and $\Gamma^\resampled_{j}$ remain the same in each connected component of $\h \setminus \wt\CL$ adjacent to $[-2^{-j},2^{-j}]$.
\end{itemize}
For $k\in\N$, let $\wt G_{k}$ be the event that the number of $j=1,\ldots,k$ so that
\[ \p[ G^\resampled_{j} \giv \Gamma ] \geq p\]
is at least $(1-a)k$.  Then
\[ \p[(\wt G_{k})^c] = O(e^{-b k}) . \]
\end{lemma}

\begin{lemma}[{\cite[Lemma~5.9]{amy-cle-resampling}}]
\label{lem:break_loops}
Suppose that $\Gamma$ is a nested \clek{} in a simply connected domain $D$. For each $a\in(0,1)$, $b>0$ there exists $p > 0$ such that the following holds. 

Let $z\in D$ and $j_0\in\Z$ be such that $B(z,2^{-j_0})\subseteq D$. There exists a resampling $\Gamma^\resampled_{z,j}$ of $\Gamma$ within $A(z,2^{-j},2^{-j+1})$ for each $j>j_0$ with the following property. Let $G^\resampled_{z,j}$ be the event that
\begin{itemize}
 \item no loop in $\Gamma^\resampled_{z,j}$ crosses the annulus $A(z,2^{-j},2^{-j+1})$,
 \item denoting $\CC$ the loops in $\Gamma$ that cross $A(z,2^{-j},2^{-j+1})$, the collection of loops of $\Gamma$ and $\Gamma^\resampled_{z,j}$ remain the same in each connected component of $\C \setminus \bigcup\CC$ that is not surrounded by a loop in $\CC$ and intersects $B(z,2^{-j})$, and the gasket of $\Gamma$ in these components is contained in the gasket of $\Gamma^\resampled_{z,j}$.
\end{itemize}
For $k\in\N$, let $\wt G_{z,j_0,k}$ be the event that the number of $j=j_0+1,\ldots,j_0+k$ so that
\[ \p[ G^\resampled_{z,j} \giv \Gamma ] \geq p\]
is at least $(1-a)k$.  Then
\[ \p[(\wt G_{z,j_0,k})^c] = O(e^{-b k}) \]
where the implicit constant does not depend on $z,j_0,k$.
\end{lemma}

More generally, we proved that with extremely high probability as $k \to \infty$, for a large fraction of annuli $A(z,2^{-j},2^{-j+1})$ as above, any resampling target involving a bounded number of successful resamplings has a positive and bounded below conditional probability of being achieved in the annulus. In order to formally state the result, we recall the following definitions given in \cite[Section~5.1]{amy-cle-resampling}.

\begin{definition}[Target pivotals]
\label{def:resampling_target}   
Let $\Gamma$ be a \clekp{} in $D$.  Let $W\subseteq D$ open and $r > 0$. Suppose that we are given the following random variables.
\begin{enumerate}[(i)]
\item $M^*<\infty$.
\item  A set $\{(\ell_i^1,\ell^2_i)\}_{i=1,\ldots,M^*}$ of pairs of arcs in $\Gamma$ with $\ell_i^1,\ell^2_i \subseteq W$ and such that none of them overlap with each other.
\item A set of $M^*$ points $\ul{y}$ such that
\begin{itemize}
 \item $y_i\in  \ell_i^1\cap \ell^2_i$ and $B(y_i,4r)\subset W$ for $i=1,\ldots,M^*$,
 \item the endpoints of $\ell^1_i, \ell^2_i$ lie outside $B(y_i,4r)$,
 \item $B(y_i,4r)\cap B(y_j,4r)=\varnothing$ and $(\ell_i^1 \cup \ell_i^2) \cap B(y_j,4r) = \varnothing$ for all $i\neq j$.
\end{itemize}
\end{enumerate}
We denote these target pivotals by $\mathfrak F_{W,r} = \{(\ell^1_i,\ell^2_i,y_i)\}_{i=1,\ldots,M^*}$.

If $\Gamma^\resampled$ is a resampling of $\Gamma$, then we say that $\Gamma^\resampled$ is $\mathfrak F_{W,r}$-successful if the following hold.
\begin{enumerate}[(i)]
 \item The loops and strands of $\Gamma$ and $\Gamma^\resampled$ are identical when restricted to the set $D \setminus \bigcup_{i=1}^{M^*} B(y_i,4r)$.
 \item Let $\tau_i^j$ (resp.\ $\ol{\tau}_i^j$) be the first (resp.\ last) time that $\ell_i^j$ is in $\ol{B}(y_i,4 r)$, for $i=1,\ldots,M^*$, $j=1,2$. Then in $\Gamma^\resampled$ the connections between the four strands $\ell_i^1|_{[0,\tau_i^1]}$, $\ell_i^1|_{[\ol{\tau}_i^1,1]}$, $\ell_i^2|_{[0,\tau_i^2]}$, $\ell_i^2|_{[\ol{\tau}_i^2,1]}$ (assuming that they are parameterized on $[0,1]$) are flipped, i.e.\ the strand $\ell^{\resampled,1}_i$ (resp.\ $\ell^{\resampled,2}_i$) in $\Gamma^\resampled$ starting with $\ell_i^1|_{[0,\tau_i^1]}$ (resp.\ finishing with $\ell_i^1|_{[\ol{\tau}_i^1,1]}$) is linked to $\ell_i^2|_{[0,\tau_i^2]}$ or $\ell_i^2|_{[\ol{\tau}_i^2,1]}$. Moreover, it holds again that $\ell^{\resampled,1}_i \cap \ell^{\resampled,2}_i \cap B(y_i,r) \neq \varnothing$.
\end{enumerate}
\end{definition}

\begin{proposition}[{\cite[Proposition~5.3]{amy-cle-resampling}}]
\label{prop:resampling}
Suppose that $\Gamma$ is a nested \clekp{} in a simply connected domain $D$. Given any $M_0 \in \N$, $r > 0$, $a\in(0,1)$, and $b>0$, there exists $p \in (0,1)$ such that the following holds.

Let $j_0 \in\Z$ and $z\in D$ be such that $B(z,2^{-j_0})\subseteq D$. There exists a resampling $\Gamma^\resampled_{z,j}$ of $\Gamma$ within $W_{z,j} \defeq A(z,2^{-j},2^{-j+1})$ for each $j>j_0$ with the following property. Let $G^p_{z,j}$ be the event for $\Gamma$ that for any choice of target pivotals $\mathfrak F_{W_{z,j},r2^{-j}} = \{(\ell^1_i,\ell^2_i,y_i)\}_{i=1,\ldots,M^*}$ with $M^* \le M_0$, the conditional probability given $\Gamma$ that $\Gamma^\resampled_{z,j}$ is $\mathfrak F_{W_{z,j},r2^{-j}}$-successful is at least $p$. For $k\in\N$ let $\wt G_{z,j_0,k}$ be the event that the number of $j = j_0+1,\ldots,j_0+k$ such that $G^p_{z,j}$ occurs is at least $(1-a)k$. Then
\[
 \p[(\wt G_{z,j_0,k})^c] = O(e^{-b k}) 
\]
where the implicit constant does not depend on $z,j_0,k$.
\end{proposition}

\subsection{Imaginary geometry}
\label{subsec:ig}

We are now going to review some aspects from the coupling of $\SLE$ with the GFF developed in \cite{dub2009gff,s2016zipper,ms2016ig1, ms2016ig2, ms2016ig3, ms2017ig4} which will be used in this work.

Suppose that $D \subseteq \C$ is a domain.  We recall that the \emph{Gaussian free field} (GFF) $h$ on $D$ is the distribution valued Gaussian field with covariance function given by the Green's function $G$ for $\Delta$ on $D$.  See \cite{s2007gff} for an introduction to the GFF.

It is proved in \cite{s2016zipper,dub2009gff,ms2016ig1} that for each $\kappa \in (0,4)$ it is possible to make sense of the formal solutions to the equation
\[ \eta'(t) = e^{i h(\eta(t)) / \chi} \quad\text{where}\quad \chi = \frac{2}{\sqrt{\kappa}} - \frac{\sqrt{\kappa}}{2}\]
and it turns out that they are $\SLE_\kappa$ type-curves.  Such $\eta$ are referred to as \emph{flow lines} of $h$.  We will first explain how this works in the context of GFF flow lines starting from the domain boundary and then how it works when the flow lines start from interior points.

For each $\kappa \in (0,4)$ we let
\[ \kappa' = \frac{16}{\kappa},\quad \lambda = \frac{\pi}{\sqrt{\kappa}},\quad\text{and}\quad \lambda' = \frac{\pi}{\sqrt{\kappa'}}.\]

\subsubsection{Flow lines}
Suppose that $h$ is a GFF on $\h$ with boundary conditions given by $-\lambda$ (resp.\ $\lambda$) on $\R_-$ (resp.\ $\R_+$).  Then there exists a unique coupling of $h$ with an $\SLE_\kappa$ process $\eta$ from $0$ to $\infty$ which is characterized by the following property.  Suppose that $(f_t)$ is the centered Loewner flow for $\eta$ and $\tau$ is a stopping time for $\eta$ which is a.s.\ finite.  Then given $\eta|_{[0,\tau]}$,
\[ h \circ f_\tau^{-1} - \chi \arg (f_\tau^{-1})'\]
is a GFF on $\h$ with the same distribution of $h$.  Moreover, in this coupling we have that $\eta$ is a.s.\ determined by $h$.

More generally, suppose that we have $\ell, r \in \N$, $-\infty < x^{\ell,L} < \cdots < x^{1,L} \le 0 \le x^{1,R} < \cdots < x^{r,R} < +\infty$, and~$\rho^{i,q} \in \R$ for each $q \in \{L,R\}$ and $1 \leq i \leq \ell$ (resp.\ $1 \leq i \leq r$) if $q = L$ (resp.\ $q = R$).  We let $x^{\ell+1,L} = - \infty$, $x^{0,L} = 0^-$, $x^{0,R} = 0^+$, and $x^{r+1,R} = + \infty$.  Suppose that $h$ is a GFF on $\h$ with boundary conditions given by
\begin{equation}
\label{eqn:flow_line_bd}
-\lambda\left( 1+ \sum_{i=1}^j \rho^{i,L} \right) \quad\text{in} \quad (x^{j+1,L}, x^{j,L}] \quad \text{and} \quad  \lambda\left( 1+ \sum_{i=1}^j \rho^{i,R} \right) \quad\text{in}\quad (x^{j,R}, x^{j+1,R}] 
\end{equation}
for each $0 \leq j \leq \ell$ on the left side and $0 \leq j \leq r$ on the right side.  Then there exists a unique coupling between $h$ and an $\SLE_\kappa(\ul{\rho}^L; \ul{\rho}^R)$ process $\eta$ from $0$ to $\infty$ which is characterized by the following property.  Suppose that $(f_t)$ is the associated centered Loewner flow for $\eta$ and $\tau$ is a stopping time for $\eta$ which is a.s.\ finite.  Then given $\eta|_{[0,\tau]}$, the field $h \circ f_\tau^{-1} - \chi \arg (f_\tau^{-1})'$ is a GFF on $\h$ with boundary conditions given by
\begin{equation}
\label{eqn:flow_line_bd_map_back}
-\lambda\left( 1+ \sum_{i=1}^j \rho^{i,L} \right) \quad\text{in} \quad (f_\tau(x^{j+1,L}), f_\tau(x^{j,L})] \quad \text{and} \quad  \lambda\left( 1+ \sum_{i=1}^j \rho^{i,R} \right) \quad\text{in}\quad (f_\tau(x^{j,R}), f_\tau(x^{j+1,R})]
\end{equation}
for each $0 \leq j \leq \ell$ on the left side and $0 \leq j \leq r$ on the right side.  Moreover, in this coupling we have that $\eta$ is a.s.\ determined by $h$.

For each $x, \theta \in \R$ we can similarly define the flow line of $h$ with angle $\theta$ from $x$ to be the flow line of $h(\cdot + x) + \theta \chi$.  We also define the flow lines of a GFF $h$ on a simply connected domain $D \subseteq \C$ from $x$ to $y$ as follows.  Let $\varphi \colon \h \to D$ be a conformal transformation with $\varphi(0) = x$ and $\varphi(\infty) = y$.  Then the flow line of $h$ from $x$ to $y$ is equal to the image under $\varphi$ of the flow line of the GFF $h \circ \varphi - \chi \arg \varphi'$ on $\h$ from $0$ to $\infty$.

\begin{lemma}\label{le:sle_reversal}
Let $\kappa \in (2,4)$.  Suppose $h$ is a GFF on $\h$ with boundary values $\lambda$ on $\R_-$ and $\lambda-\pi\chi = \lambda(\kappa/2-1)$ on $\R_+$. Let $\eta$ be the angle $-\angledouble$ flow line of $h$ from $0$ to $\infty$. Let $\wt{\eta}$ be the angle~$0$ flow line of $h$ from~$\infty$ to~$0$. Then~$\eta$ and~$\wt{\eta}$ have the same law (modulo time reversal), namely \slekr{\kappa-4;-\kappa/2}.
\end{lemma}
\begin{proof}
This follows from the reversibility of \slekr{\rho_L;\rho_R} \cite{ms2016ig2}.
\end{proof}

\begin{lemma}\label{le:reflected_fl_law}
Let $\kappa \in (2,4)$.  Suppose $h$ is a GFF on $\h$ with boundary values $-\lambda(1+\rho_L)$  on $\R_-$ and $\lambda-\pi\chi = \lambda(\kappa/2-1)$ on $\R_+$ where $\rho_L > -2$. Let $\eta_1$ (resp.\ $\eta_2$) be the angle $0$ (resp.\ $-\angledouble$) flow line of $h$ from $0$ to $\infty$. Given $\eta_1$, let $\ol{\eta}_2$ be the angle $0$ flow line from $\infty$ to $0$ of the restriction of $h$ to the components of $\h\setminus\eta_1$ to the right of $\eta_1$ (i.e.\ $\ol{\eta}_2$ is reflected off $\eta_1$). Then the pair $(\eta_1,\eta_2)$ has the same law as $(\eta_1,\ol{\eta}_2)$ (modulo time reversal).
\end{lemma}

\begin{proof}
It follows from Lemma~\ref{le:sle_reversal} that the conditional law of both $\eta_2$ and $\ol{\eta}_2$ given $\eta_1$ is an \slekr{\kappa-4;-\kappa/2} in the components to the right of $\eta_1$.
\end{proof}

\subsubsection{Counterflow lines}
Suppose that $h$ is a GFF on $\h$ with boundary conditions given by $\lambda'$ (resp.\ $-\lambda'$) on $\R_-$ (resp.\ $\R_+$).  Then there exists a unique coupling of $h$ with an $\SLE_{\kappa'}$ process $\eta'$ from $0$ to $\infty$ which is characterized by the following property.  Suppose that $(f_t)$ is the centered Loewner flow for $\eta'$ and $\tau$ is a stopping time for $\eta'$ which is a.s.\ finite.  Then given $\eta'|_{[0,\tau]}$,
\[ h \circ f_\tau^{-1} - \chi \arg (f_\tau^{-1})'\]
is a GFF on $\h$ with the same distribution of $h$.  Moreover, in this coupling we have that $\eta'$ is a.s.\ determined by $h$.

More generally, suppose we have $\ell, r \in \N$, $-\infty < x^{\ell,L} < \cdots < x^{1,L} \le 0 \le x^{1,R} < \cdots < x^{r,R} < +\infty$, and $\rho^{i,q} \in \R$ for each $q \in \{L,R\}$ and $1 \leq i \leq \ell$ (resp.\ $1 \leq i \leq r$) if $q = L$ (resp.\ $q = R$).  We let $x^{\ell+1,L} = -\infty$, $x^{0,L} = 0^-$, $x^{0,R} = 0^+$, and $x^{r+1,R} = + \infty$.  Suppose that $h$ is a GFF on $\h$ with boundary conditions given by
\begin{equation}
\label{eqn:cfl_bd}
\lambda'\left( 1+ \sum_{i=1}^j \rho^{i,L} \right) \quad\text{in} \quad (x^{j+1,L}, x^{j,L}] \quad \text{and} \quad  -\lambda'\left( 1+ \sum_{i=1}^j \rho^{i,R} \right) \quad\text{in}\quad (x^{j,R}, x^{j+1,R}] 
\end{equation}
for each $0 \leq j \leq \ell$ on the left side and $0 \leq j \leq r$ on the right side.  Then there exists a unique coupling between $h$ and an $\SLE_{\kappa'}(\ul{\rho}^L; \ul{\rho}^R)$ process $\eta'$ from $0$ to $\infty$ which is characterized by the following property.  Suppose that $(f_t)$ is the associated centered Loewner flow for $\eta'$ and $\tau$ is a stopping time for $\eta'$ which is a.s.\ finite.  Then given $\eta'|_{[0,\tau]}$, the field $h \circ f_\tau^{-1} - \chi \arg (f_\tau^{-1})'$ is a GFF on $\h$ with boundary conditions given by
\begin{equation}
\label{eqn:cfl_bd_map_back}
\lambda'\left( 1+ \sum_{i=1}^j \rho^{i,L} \right) \quad\text{in} \quad (f_\tau(x^{j+1,L}), f_\tau(x^{j,L})] \quad \text{and} \quad  -\lambda'\left( 1+ \sum_{i=1}^j \rho^{i,R} \right) \quad\text{in}\quad (f_\tau(x^{j,R}), f_\tau(x^{j+1,R})]
\end{equation}
for each $0 \leq j \leq \ell$ on the left side and $0 \leq j \leq r$ on the right side.  Moreover, in this coupling we have that $\eta'$ is a.s.\ determined by $h$.

For each $x, \theta \in \R$ we can similarly define the counterflow line of $h$ with angle $\theta$ from $x$ to be the counterflow line of $h(\cdot + x) + \theta \chi$.  We also define the counterflow lines of a GFF $h$ on a simply connected domain $D \subseteq \C$ from $x$ to $y$ as follows.  Let $\varphi \colon \h \to D$ be a conformal transformation with $\varphi(0) = x$ and $\varphi(\infty) = y$.  Then the counterflow line of $h$ from $x$ to $y$ is equal to the image under $\varphi$ of the counterflow line of $h \circ \varphi - \chi \arg \varphi'$ from $0$ to $\infty$.

\subsubsection{Flow lines from interior points}

It is also possible to consider flow lines of a GFF starting from \emph{interior} rather than \emph{boundary} points.  Rather than being chordal $\SLE_\kappa$-type curves, they are whole-plane $\SLE_\kappa$-type curves.  Let us first explain how this works in the case of the whole-plane GFF.  Suppose that $h$ is a whole-plane GFF.  It is only possible to define the whole-plane GFF modulo additive constant. Since the flow lines of a GFF only depend on the field modulo $2 \pi \chi$, we view $h$ as a distribution modulo a global multiple of $2\pi \chi$.  Then it is shown in \cite{ms2017ig4} that there exists a unique coupling of $h$ with a whole-plane $\SLE_\kappa(2-\kappa)$ process $\eta$ from $0$ to $\infty$ so that for each stopping time $\tau$ for $\eta$ we have that the conditional law of $h$ given $\eta|_{[0,\tau]}$ is that of a GFF in $\C \setminus \eta([0,\tau])$ with the same boundary conditions along $\eta$ as in the case of a flow line starting from the boundary.  In this coupling, we have that $\eta$ is a.s.\ determined by $h$.  More generally, for each $z \in \C$ and $\theta$ the flow line of $h$ starting from $z$ with angle $\theta$ is defined to be the flow line of $h(\cdot + z) + \theta \chi$ starting from $0$ and this process has the law of an $\SLE_\kappa(2-\kappa)$ process from $z$ to $\infty$ and is a.s.\ determined by $h$.  The manner in which flow lines starting from interior points interact with each other is the same as for flow lines starting from boundary points.

By absolute continuity, we can also consider the flow lines of a GFF on a domain $D \subseteq \C$ and these paths locally look like whole-plane $\SLE_\kappa(2-\kappa)$ processes.

More generally, we can consider $h = \wt{h} - \alpha \arg(\cdot)$ where $\wt{h}$ is a whole-plane GFF, $\alpha > - \chi$, and with $h$ viewed modulo a global multiple of $2\pi(\chi + \alpha)$. Then the flow line of $h$ from $0$ is a whole-plane $\SLE_\kappa(\rho)$ curve where $\rho = 2-\kappa + 2\pi \alpha/\lambda$.  The same is true for the flow lines of $h$ from $0$ to $\infty$ with other angles.

\subsubsection{Coupling with $\CLE_{\kappa'}$}

Suppose that $h$ is a GFF on $\h$ with boundary conditions given by $\lambda'-\pi \chi$.  For each $x \in \R$ we let $\eta_x'$ be the counterflow line of $h$ from $\infty$ to $x$.  Then $\eta_x'$ is an $\SLE_{\kappa'}(\kappa'-6)$ in $\h$ from $\infty$ to $0$ where the force point is located infinitesimally to the right of $\infty$ (when standing at $\infty$ and looking towards $0$).  Moreover, the $\eta_x'$ as $x$ varies are coupled together in exactly the same way as in the definition of the $\CLE_{\kappa'}$ exploration tree.  For this reason, $\CLE_{\kappa'}$ is naturally coupled with $h$.  (The same construction also works if we instead have boundary conditions given by $-\lambda'+\pi \chi$ except in this case the force point of the counterflow line is immediately to the left of $\infty$ when standing at $\infty$ and looking towards $0$.)

\subsubsection{Interaction}
\label{se:fl_interaction}

The manner in which the flow and counterflow lines of the GFF interact with each other is described in \cite{ms2016ig1}.  Let us first describe how the flow lines interact with each other.  Suppose that $x_1,x_2 \in \R$ with $x_1 < x_2$ and $\theta_1, \theta_2 \in \R$ with $\theta_1 > \theta_2 - \pi$.  Let $h$ be a GFF on $\h$ with piecewise constant boundary conditions and let $\eta_i$ be the flow line of $h$ with angle $\theta_i$ from $x_i$ to $\infty$ for $i=1,2$.  Then $\eta_1$ stays to the left of $\eta_2$ if $\theta_1 > \theta_2$, $\eta_1$ merges with $\eta_2$ upon intersecting and does not subsequently separate if $\theta_1 = \theta_2$, and $\eta_1$ crosses $\eta_2$ from left to right upon intersecting and does not subsequently cross back if $\theta_1 \in (\theta_2-\pi,\theta_2)$.

\begin{figure}[ht]
\centering
\includegraphics[width=0.49\textwidth]{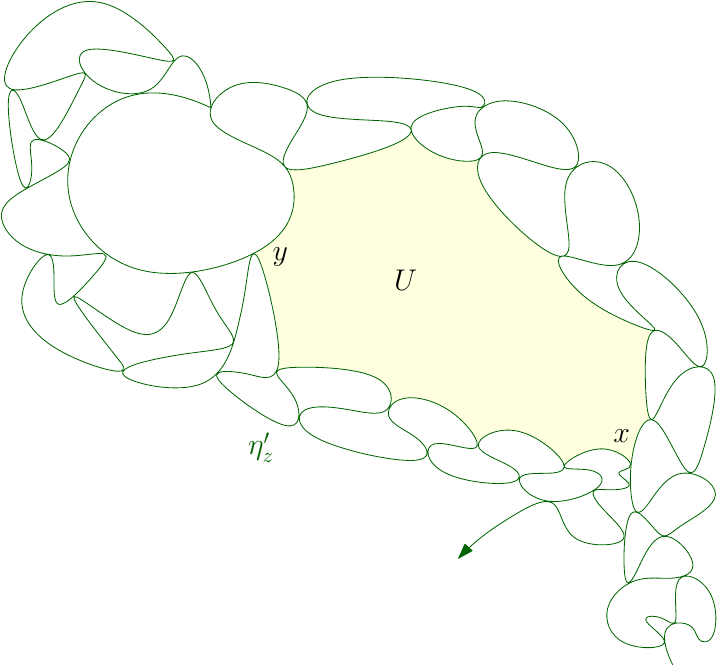}
\includegraphics[width=0.49\textwidth]{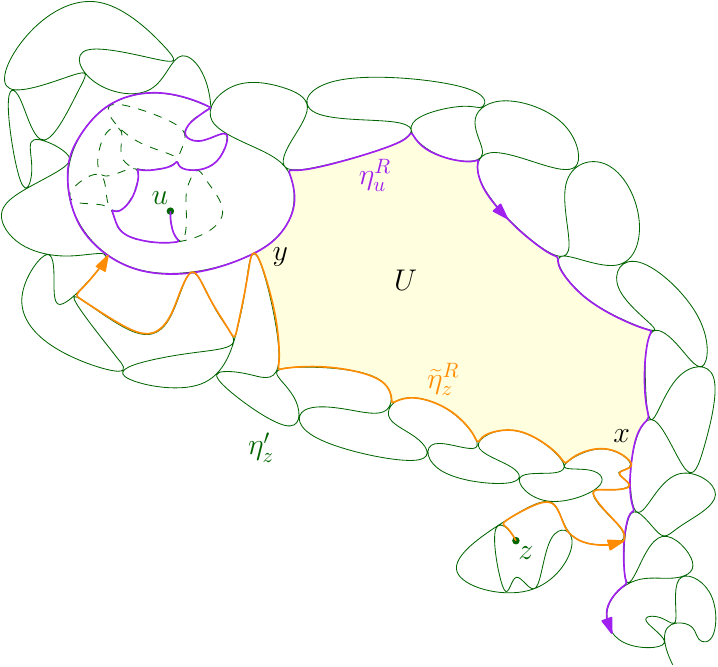}
\caption{By the duality, each component disconnected by a counterflow line can be detected by a pair of flow lines where one flow line is reflected off the other in the opposite direction.}
\label{fi:detection}
\end{figure}

The interaction between flow lines and counterflow lines is as follows. Suppose that $z \in \ol{\h}$ and $\eta'_z$ is a counterflow line of $h$ from $\infty$ targeting $z$. Then the left (resp.\ right) boundary of $\eta'_z$ seen from $z$ is a.s.\ equal to the flow line $\eta^L_z$ (resp.\ $\eta^R_z$) of $h$ from $z$ to $\infty$ with angle $\pi/2$ (resp.\ $-\pi/2$). (This is also called duality between \slek{} and \slekp{}.) This fact can be used to ``detect'' the components separated by a counterflow line using flow lines which we now explain. See Figure~\ref{fi:detection} for an illustration.

Suppose that $U$ is a component of $\h \setminus \eta'_z$, and say that it is surrounded by $\eta'_z$ counterclockwise. Let $x \in \partial U$ be the first (equivalently last) point visited by $\eta'_z$. Select any $y \in \partial U \setminus \{x\}$, and let $\partial_1 U \subseteq \partial U$ be the counterclockwise arc from $x$ to $y$. Note that $\eta'_z$ disconnects infinitely many components between visiting $y$ and coming back to $x$. Let $u$ be a point in one such component. The counterflow line $\eta'_u$ targeting $u$ then agrees with $\eta'_z$ at least until they visit $y$. By the duality, the flow line $\eta^R_u$ contains $\partial_1 U$. Given $\eta'_u$, due to the locality of counterflow lines, we can now view the remainder of $\eta'_z$ as a counterflow line of the restriction of $h$ to the component of $\h \setminus \eta'_u$ containing $z$ which we denote by $\wt{h}$ in the following. Applying the duality to $\wt{h}$, we see that the flow line $\wt{\eta}^R_z$ of $\wt{h}$ contains the remaining part of the boundary $\partial U \setminus \eta^R_u$. Note that it is important to let $\wt{\eta}^R_z$ be a flow line of $\wt{h}$, because the flow line $\eta^R_z$ of $h$ would merge with $\eta^R_u$ upon intersecting. In contrast, the flow line $\wt{\eta}^R_z$ reflects off $\eta^R_u$ in the opposite direction.

\subsubsection{Space-filling $\SLE_{\kappa'}$}
The flow lines of a GFF $h$ starting from interior points can be used to define a \emph{space-filling} version of $\SLE_{\kappa'}$ for each $\kappa' > 4$.  It is constructed in the following way.  Fix a countable dense set of points $(z_n)$ in $\C$.  For each $n$ we let $\eta_n$ be the flow line of $h$ starting from $z_n$ with angle $\pi/2$.  Then we define an ordering on the $(z_n)$ by saying that $z_i$ comes before $z_j$ if and only if $\eta_i$ merges with $\eta_j$ on its right side. It is shown in~\cite[Section 4]{ms2017ig4} that there is a unique space-filling curve that visits these points in order  and it is continuous when parameterized by Lebesgue measure. In the coupling the path is determined by the field $h$.
 
From the flow line interaction rules it follows that when starting a flow line with angle in $[-\pi/2,\pi/2]$ from a point $z$, the space-filling $\SLE_{\kappa'}$ path will visit the range of the flow line in reverse chronological order. For a counterflow line targeting a point $z$, the space-filling $\SLE_{\kappa'}$ will visit the points of the counterflow line in the same order and, whenever the counterflow line cuts off a component, the space-filling $\SLE_{\kappa'}$ fills up this component before continuing along the trajectory of the counterflow line.

The same construction also works if $h$ is a GFF on a domain $D \subseteq \C$ and its boundary conditions are chosen appropriately.  If the boundary conditions are as in the case of the coupling of $\CLE_{\kappa'}$ described above, then the associated space-filling $\SLE_{\kappa'}$ is the same as the space-filling $\CLE_{\kappa'}$ exploration described in Section~\ref{subsec:cle}. 

Suppose that $h$ is a zero-boundary GFF and $F$ is a harmonic function on $D$. We say that $h+F$ is a GFF on $D$ with boundary conditions $F|_{\partial D}$. For two GFFs with different boundary conditions, their laws when restricted to regions away from where the boundary values are changed are mutually absolutely continuous, and their Radon-Nikodym derivatives have finite moments of all orders. See e.g.\ \cite[Proposition~3.4 and Remark~3.5]{ms2016ig1} for a proof.

\begin{lemma}
\label{lem:rn_derivative}
Fix $M > 0$ and let $h$ (resp.\ $h_0$) be a GFF on $\D$ with boundary conditions bounded by $M$ (resp.\ zero boundary conditions).  Fix $r \in (0,1)$ and let $\CZ_r$ be the Radon-Nikodym derivative of the law of $h|_{B(0,r)}$ with respect to $h_0|_{B(0,r)}$.  For each $q \in \R$ there exists a constant $C \in (0,\infty)$ depending only on $q$, $r$, and $M$ so that
\[ \E[ \CZ_r^q] \leq C.\]
\end{lemma}

\begin{lemma}\label{le:fill_ball}
 Let $D \subseteq \C$ be a simply connected domain, and let $\eta'$ be a space-filling \slekp{} in $D$. Fix $a>0$. For $\delta \in (0,1)$ let $E_\delta$ be the event that for every $s,t$ with $\eta'(s) \in B(0,1) \cap D$, $\dist(\eta'(s),\partial D) \ge \delta$, and $\abs{\eta'(s)-\eta'(t)} \ge \delta$, the segment $\eta'[s,t]$ contains a ball of radius $\delta^{1+a}$. Then $\p[E_\delta^c] = o^\infty(\delta)$ as $\delta\searrow 0$.
\end{lemma}

Note that the event $E_\delta$ depends only on the values of a GFF away from the boundary, therefore it is well-defined even when the boundary values are so that the space-filling SLE starting from the boundary is not well-defined.

\begin{proof}
 The case when $D=\C$ was proved in \cite[Section~3]{ghm2020kpz} and \cite[Theorem~1.6]{hy-sle-regularity}. For general domains, the result follows from the absolute continuity of the GFF in $D$ away from the boundary (modulo additive constant $2\pi\chi$) with respect to a whole-plane GFF.
\end{proof}

In some parts of the paper, we will sample GFF flow lines that enclose some region $U$, and conditionally sample the internal metric $\metapproxres{\epsilon}{U}{\cdot}{\cdot}{\Gamma_U}$ according to its conditional law given $U$. Note that we do not require the metric to be measurable with respect to $\Gamma_U$. If we have two GFF variants that are mutually absolutely continuous, then the absolute continuity continues to hold when additionally the metrics are sampled. This is recorded in the following abstract lemma.

\begin{lemma}\label{le:abs_cont_kernel}
Let $\mu,\wt{\mu}$ be probability measures on a space $X$, and let $\nu[dy \mid x]$ be a probability kernel from $X$ to $Y$. Suppose that $\frac{d\wt{\mu}}{d\mu}$ has $\mu$-moments of order $p>1$. Then
\[
\wt{\mu}\otimes\nu[E] \le \norm*{\frac{d\wt{\mu}}{d\mu}}_{L^p(\mu)} ( \mu\otimes\nu[E] )^{1-1/p}
\]
for any event $E$ (depending on both variables $x,y$).
\end{lemma}

\begin{proof}
By an application of H\"older's inequality we get
\[ \begin{split}
\wt{\mu}\otimes\nu[E] 
&= \int \nu[E \mid x] \wt{\mu}[dx] \\
&= \int \nu[E \mid x] \frac{d\wt{\mu}}{d\mu}(x) \mu[dx] \\
&\le \left( \int \nu[E \mid x]^{1/(1-1/p)} \mu[dx] \right)^{1-1/p} \left( \int \frac{d\wt{\mu}}{d\mu}(x)^{p} \mu[dx] \right)^{1/p} \\
&\le \left( \int \nu[E \mid x] \mu[dx] \right)^{1-1/p} \left( \int \frac{d\wt{\mu}}{d\mu}(x)^{p} \mu[dx] \right)^{1/p} \\
&= ( \mu\otimes\nu[E] )^{1-1/p} \left( \int \frac{d\wt{\mu}}{d\mu}(x)^{p} \mu[dx] \right)^{1/p}
\end{split} \]
where in the second last step we have used that $\nu[E \mid x] \le 1$.
\end{proof}

\subsubsection{Independence across scales}
\label{se:gff}

We introduce some notation that we frequently use in this paper.

Let $h$ be a GFF in a domain $D$ with some choice of boundary values. In case $D=\C$, we regard $h$ as a whole-plane GFF modulo additive constant of $2\pi\chi$. Let $z \in D$ and $r > 0$ so that $B(z,r) \subseteq D$. Let $\CF_{z,r}$ be the $\sigma$-algebra generated by the values of $h$ outside of $B(z,r)$, and let $\Fh_{z,r}$ be the harmonic extension from $\partial B(z,r)$ to $B(z,r)$ of the values of $h$ on $\partial B(z,r)$. Recall that given $\CF_{z,r}$, the conditional law of $h-\Fh_{z,r}$ is that of a zero-boundary GFF in $B(z,r)$.

For $M > 0$, we say that $(z,r)$ is $M$-good for $h$ if
\begin{equation}
\label{eqn:m_good}
\sup_{w \in B(z,7r/8)} |\Fh_{z,r}(z) - \Fh_{z,r}(w)| \leq M.
\end{equation}

We restate the following independence across scales result from \cite{mq2020notsle}.
\begin{lemma}
\label{lem:gff_independence_across_scales}
Suppose that $h$ is a GFF on $D$ with some bounded boundary values.  For every $b>0$ and $q \in (0,1)$ there exists $M > 0$ so that the following is true. For every $z \in D$ and $r > 0$ such that $B(z,r) \subseteq D$, with probability $1-O(e^{-bk})$, the number of $1 \leq j \leq k$ so that $(z,2^{-j}r)$ is $M$-good for $h$ is at least $qk$.
\end{lemma}
\begin{proof}
This result for whole-plane GFF is \cite[Proposition~4.3]{mq2020notsle}. The extension of the result to domains is explained in \cite[Lemma~4.1]{amy-cle-resampling}.
\end{proof}

Let also
\[ C_{z,r} = 2\pi \chi \lfloor \Fh_{z,r}(z) / (2\pi \chi) \rfloor \]
and
\[ h_{z,r} = h|_{B(z,r)} - C_{z,r}.\]
Note that the flow lines of $h_{z,r}$ in $B(z,r)$ agree with those of $h$.

Suppose $\Fg_{z,r}$ is a harmonic function on $B(z,r)$ with some prescribed boundary values. We assume that $\Fg_{z,r}$ are chosen so that they do not depend on $z,r$ except for translation and scaling. Then $h - \Fh_{z,r} + \Fg_{z,r}$ is a GFF with boundary values given by $\Fg_{z,r}$ on $B(z,r)$. We construct an auxiliary field $\wt{h}_{z,r}$ on $B(z,r)$ that interpolates between $h$ and $h - \Fh_{z,r} + \Fg_{z,r}$. Let $\phi$ be a fixed $C_0^\infty$ bump function that is supported on $B(0,7/8)$ with $\phi|_{B(0,3/4)} \equiv 1$, and let $\phi_{z,r}(w) = \phi( (w-z)/r)$ so that $\phi_{z,r}$ is supported on $B(z,7r/8)$ and $\phi_{z,r}|_{B(z,3r/4)} \equiv 1$. Let
\begin{equation}
\label{eqn:wt_h_def}
\wt{h}_{z,r} = h_{z,r} + (1-\phi_{z,r})(\Fg_{z,r} - (\Fh_{z,r} - C_{z,r})).
\end{equation}
Note that
\begin{alignat*}{2}
\wt{h}_{z,r} &= h_{z,r} &&\quad\text{on } B(z,3r/4), \\
\wt{h}_{z,r} &= h - \Fh_{z,r} + \Fg_{z,r} &&\quad\text{on } A(z,7r/8,r) .
\end{alignat*}
On the event that $(z,r)$ is $M$-good for $h$ we have that the Radon-Nikodym derivative of the law of $\wt{h}_{z,r}$ with respect to the law of $h - \Fh_{z,r} + \Fg_{z,r}$ has finite moments of all orders which depend only on $M$ (see Lemma~\ref{lem:rn_derivative}).

A useful consequence of the independence across scales is the following result. It roughly says that if an event is likely to occur for the restriction of a GFF to an annulus, then it is very likely to occur in a large fraction of subsequent annuli.

\begin{lemma}
\label{lem:good_scales_for_event}
Suppose that $h$ is a GFF on $D$ with some bounded boundary values. For $p > 0$, let $G^p$ be an event that is measurable with respect to the restriction of a GFF to $A(0,1/4,3/4)$, and suppose that $\lim_{p \to 0} \p[G^p] = 1$. Let $G^p_{z,r}$ denote the event that $\wt{h}_{z,r}(z+r\cdot) \in G^p$ (so that $G^p_{z,r}$ is an event measurable with respect to the restriction of $\wt{h}_{z,r}$ to $A(z,r/4,3r/4)$).

For any $b>0$ and $q \in (0,1)$ there exist $M>0$, $p > 0$ so that the following holds. Let $z \in D$ and $r>0$ such that $B(z,r) \subseteq D$. For $k\in\N$ let $E$ denote the event that the fraction of $1 \leq j \leq k$ so that the scale $(z,2^{-j}r)$ is $M$-good for $h$ and $G^p_{z,2^{-j}r}$ occurs is at least $q$. Then $\p[E^c] = O(e^{-bk})$.
\end{lemma}

\begin{proof}
We can consider the even and odd scales separately, so let us just show that with probability $1-O(e^{-bk})$ at a large fraction of even scales the event $G^p_{z,2^{-j}r}$ occurs.

For any $b>0$, Lemma~\ref{lem:gff_independence_across_scales} implies that we can choose $M$ sufficiently large so that with probability $1-O(e^{-bk})$, the fraction of $1 \leq j \leq k$ such that $B(z,2^{-2j}r)$ is $M$-good for $h$ is at least $1-(1-q)/2$.

Let $E_{z,k}$ denote the event that the fraction of $1 \leq j \leq k$ such that $B(z,2^{-2j}r)$ is $M$-good for $h$ but $(G^p_{z,2^{-2j}r})^c$ occurs is at least $(1-q)/2$. It suffices to show that by making $p > 0$ sufficiently small we have that $\p[ E_{z,k} ] = O(e^{-bk})$. By the assumption $\lim_{p\to 0} \p[(G^p)^c] = 0$ and the absolute continuity statement described above, there exists for any $\varepsilon>0$ some $p = p(M,u) > 0$ such that
\[
\p[ (G^p_{z,2^{-2j}r})^c \mid \CF_{z,2^{-2j}r} ] \,\one_{\{(z,2^{-2j}r) \text{ is $M$-good for $h$}\}} < \varepsilon .
\]
Observing that the event $G^p_{z,2^{-2j}r}$ is measurable with respect to $\CF_{z,2^{-2(j+1)}r}$, this implies the claim.
\end{proof}

\subsubsection{Flow line merging probabilities}
\label{se:fl_merging}

We will often consider events for interior flow lines within a ball $B(z,r) \subseteq D$. We want to compare the probabilities for such events to corresponding events for flow lines growing from the boundary. This is possible when we have a positive probability bounded from below that the boundary emanating flow lines merge into the given interior flow lines. We describe a sufficient condition to guarantee this.

\begin{figure}[ht]
\centering
\includegraphics[width=0.45\textwidth]{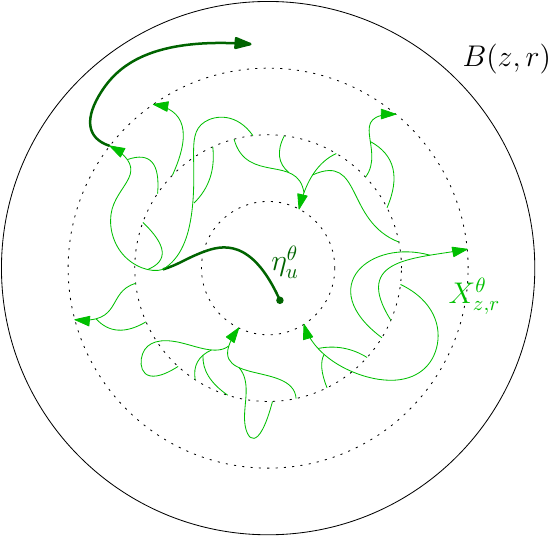}
\caption{Illustration of the set $X_{z,r}^\theta$ and a flow line $\eta_u^\theta$ with angle $\theta$ starting from $u \in B(z,r/4)$.}
\end{figure}

For $w \in \partial B(z,r/2)$ and $\theta\in\R$ we let $\eta^\theta_w$ be the angle $\theta$ flow line of $\wt{h}_{z,r}$ (or equivalently of $h$) starting at $w$ and stopped upon exiting $A(z,r/4,3r/4)$. Let
\begin{equation}\label{eq:fl_annulus}
X^\theta_{z,r} = \ol{\bigcup_{w \in z+re^{i\Q}/2} \eta^\theta_{w}} .
\end{equation}
Then the set $\partial A(z,r/4,3r/4) \cap X^\theta_{z,r}$ is almost surely finite. Indeed, this follows from the almost sure continuity of space-filling $\SLE_{\kappa'}$ (since each distinct point on $\partial A(z,r/4,3r/4) \cap X^\theta_{z,r}$ corresponds to a crossing of a space-filling $\SLE_{\kappa'}$ from $\partial A(z,r/4,3r/4)$ to $\partial B(z,r/2)$).

The sets $X^\theta_{z,r}$ will be useful for describing events where flow lines from the boundary have a sufficient probability of merging into given interior flow lines. Using the independence across scales, we can guarantee with high probability that many annuli satisfy such a condition. We describe a few instances which we will repeatedly use in this paper.

Suppose $\wt{h}$ is a GFF on $\D$ with boundary values so that the flow lines $\eta_1$ (resp.\ $\eta_2$) with angles $\theta_1$ (resp.\ $\theta_2$) from $-i$ to $i$ are defined. Assume that $\theta_1 > \theta_2$. Let $\eta'_1$ (resp.\ $\eta'_2$) be the counterflow lines of $\wt{h}$ with angles $\theta_1+\pi/2$ (resp.\ $\theta_2-\pi/2$) from $i\delta$ to $-i\delta$, so that the right boundary of $\eta'_1$ agrees with $\eta_1$, and the left boundary of $\eta'_2$ agrees with $\eta_2$. Let $X_{0,1}^\theta$ be as in~\eqref{eq:fl_annulus}. For $M,p>0$, let $G$ denote the event that the following hold. (This is an event for the random sets $X_{0,1}^{\theta_1+\pi}, X_{0,1}^{\theta_1}, X_{0,1}^{\theta_2}, X_{0,1}^{\theta_2-\pi}$ and the values of $\wt{h}$ on these sets. In particular, it is measurable with respect to the values of $\wt{h}$ on $A(0,1/4,3/4)$.)
\begin{itemize}
\item Let $\Fh^X$ be the conditional expectation of $\wt{h}$ given its values on $X_{0,1}^{\theta_1+\pi}, X_{0,1}^{\theta_1}, X_{0,1}^{\theta_2}, X_{0,1}^{\theta_2-\pi}$. Then
\[ \sup_{\partial B(0,7/8)} |\Fh^X| \le M . \]
\item The number of points in $\partial A(0,1/4,3/4) \cap X_{0,1}^{\theta_1}$ and $\partial A(0,1/4,3/4) \cap X_{0,1}^{\theta_2}$ is at most $M$.
\item Given $X_{0,1}^{\theta_1+\pi}, X_{0,1}^{\theta_1}, X_{0,1}^{\theta_2}, X_{0,1}^{\theta_2-\pi}$, and $\Fh^X$, for each quadruple of strands ending on $\partial B(0,3/4)$ with height differences $\pi,\theta_1-\theta_2,\pi$ (in counterclockwise order), the conditional probability that $\eta'_1$ merges into the latter two strands and $\eta'_2$ merges into the former two strands before entering $B(0,1/4)$ is at least $p$.
\end{itemize}

Note that the conditional probability described in the last item is not affected if we further condition on the values of $\wt{h}$ within the connected component of $\D \setminus \bigcup\{X^\theta_{0,1}\}$ containing $B(0,1/4)$. Indeed, the merging event is described purely by the values of $\wt{h}$ outside that region (this is the reason why the event requires the merging to happen before entering $B(0,1/4)$). Similar remarks apply to the definition of the events in the Lemmas~\ref{le:good_scales_merging_refl} and~\ref{le:good_scales_merging_multiple} below.

\begin{lemma}\label{le:good_scales_merging}
Suppose that $h$ is a GFF on $D$ with some bounded boundary values. For any $b>0$ and $q \in (0,1)$ there exist $M,p>0$ so that the following is true. 
Let $z \in D$ and $r>0$ such that $B(z,r) \subseteq D$. Let $\wt{h}_{z,r}$ be as defined in Section~\ref{se:gff} with the same boundary values as $\wt{h}$ above (after scaling and translation). Let $G_{z,r}$ be the event that
\begin{itemize}
 \item the scale $(z,r)$ is $M$-good for $h$, and
 \item the event $G$ above occurs for the (scaled and translated) field $\wt{h}_{z,r}$.
\end{itemize}
For $k\in\N$ let $E$ be the event that the fraction of $1 \leq j \leq k$ so that $G_{z,2^{-j}r}$ occurs is at least $q$. Then $\p[E^c] = O(e^{-bk})$.
\end{lemma}

\begin{proof}
This is a consequence of Lemma~\ref{lem:good_scales_for_event}. To apply the lemma, observe that we can pick $M,p$ such that the probability $\p[G^c]$ for $\wt{h}$ is as small as we wish, since for each quadruple of strands with the correct height differences the merging probability is a.s.\ positive, and the sets $\partial A(0,1/4,3/4) \cap X^\theta_{0,1}$ are almost surely finite.
\end{proof}

We need another variant that considers flow lines reflected off another flow line in the opposite direction. Recall that conditionally on a flow line $\eta_1$ of $h$, we can consider the restriction of $h$ to the complement of $\eta_1$ which has the conditional law of a GFF with flow line boundary values along $\eta_1$. We say that a flow line $\eta_2$ is reflected off another flow line $\eta_1$ in the opposite direction if $\eta_2$ is a flow line of the restriction of $h$ to the complement of $\eta_1$, and whenever it hits $\eta_1$, we continue it with the branch towards the opposite direction of $\eta_1$.

For $w \in \partial B(z,3r/16)$ and $w_1,w_2,\ldots \in \partial B(z,r/2) \cup \partial B(z,r/16)$, let $\eta^\theta_{w_i}$ be the angle $\theta$ flow line starting at $w_i$ and stopped upon exiting $A(z,r/32,3r/4)$, and then let $\eta^{\theta;w_1,w_2,\ldots}_w$ be the angle $\theta$ flow line starting at $w$ and reflected off $\{\eta^\theta_{w_i}\}$ in the opposite direction, stopped upon exiting $A(z,r/8,r/4)$. Let
\begin{equation}\label{eq:fl_refl_annulus}
X^{\theta;*}_{z,r} = \ol{\bigcup_{w \in (z+re^{i\Q}/2) \cup (z+re^{i\Q}/16)} \eta^\theta_{w}} \cup \ol{\bigcup_{\substack{w \in z+(3/16)re^{i\Q} \\ w_1,w_2,\ldots \in (z+re^{i\Q}/2) \cup (z+re^{i\Q}/16)}} \eta^{\theta;w_1,w_2,\ldots}_w} .
\end{equation}
As before, the set $\partial A(z,r/8,r/4) \cap X^{\theta;*}_{z,r}$ is almost surely finite. This is because only finitely many of the strands $\eta^\theta_{w_i}$ do not merge before reaching $\partial B(z,r/8)$ (resp.\ $\partial B(z,r/4)$), hence the collection of choices $w_1,w_2,\ldots$ in the union reduces to a finite number.

We now describe an event similar to the one in Lemma~\ref{le:good_scales_merging} that guarantees a sufficient merging probability for a pair of flow lines where $\eta_2$ is reflected off $\eta_1$. The statement is more technical as we need to first observe $\eta_1$ in order to determine $\eta_2$.

Suppose $\wt{h}$ is a GFF on $\D$ with boundary values so that the angle $\theta$ flow line $\eta_1$ from $-i$ to $i$ and the angle $\theta$ flow line $\ol{\eta}_2$ from $i$ to $-i$ in the components of $\D\setminus\eta_1$ to the right of $\eta_1$ (and reflected off $\eta_1$) are defined. Let $X^\theta_{0,1}$, $X^{\theta;*}_{0,1}$ be as in~\eqref{eq:fl_annulus}, \eqref{eq:fl_refl_annulus}. For $M,p>0$, let $G$ denote the event that the following hold.
\begin{itemize}
 \item The number of points in $\partial A(0,1/4,3/4) \cap X^\theta_{0,1}$ and $\partial A(0,1/8,1/4) \cap X^{\theta;*}_{0,1}$ is at most $M$.
\item For each pair of strands $\eta^\theta_{z_1}$, $\eta^\theta_{\wt{z}_1}$ of $X^\theta_{0,1}$ ending on $\partial B(0,1/4)$ (resp.\ $\partial B(0,3/4)$) and each strand $\eta^{\theta-\pi/2}_{w_1}$ of $X^{\theta-\pi/2}_{0,1}$ ending on $\partial B(0,3/4)$, the conditional probability given $X^\theta_{0,1}$, $X^{\theta-\pi/2}_{0,1}$, and the values of $\wt{h}$ on these sets of the following event is either $0$ or at least $p$:
\begin{itemize}
\item $\eta_1$ merges into $\eta^\theta_{z_1}$ before entering $B(0,1/4)$.
\item The extensions of $\eta^\theta_{\wt{z}_1}$ and $\eta^{\theta-\pi/2}_{w_1}$ do not enter $B(0,1/4)$ until they hit $\partial\D$, and they do not intersect each other.
\end{itemize}
\end{itemize}

The meaning of the event described in the last item is to generate the following scenario. Suppose there is a point $w \in B(0,1/2)$ such that the angle $\theta$ flow line $\eta^\theta_w$ intersects the right side of $\eta_1$ with angle difference $0$, and the angle $\theta-\pi/2$ flow line $\eta^{\theta-\pi/2}_w$ hits $\partial\D$ without intersecting $\eta_1$. Then $\ol{\eta}_2$ necessarily merges into $\eta^\theta_w$. (See e.g.\ Lemma~\ref{lem:interior_intersection_left} where this is applied, and the left picture of Figure~\ref{fi:local_event_left} for an illustration.)

\begin{lemma}\label{le:good_scales_merging_refl}
Suppose that $h$ is a GFF on $D$ with some bounded boundary values. For any $b>0$ and $q \in (0,1)$ there exist $M,p>0$ so that the following is true. 
Let $z \in D$ and $r>0$ such that $B(z,r) \subseteq D$. Let $\wt{h}_{z,r}$ be as defined in Section~\ref{se:gff} with the same boundary values as $\wt{h}$ above (after scaling and translation). Let $G_{z,r}$ be the event that
\begin{itemize}
 \item the scale $(z,r)$ is $M$-good for $h$, and
 \item the event $G$ above occurs for the (scaled and translated) field $\wt{h}_{z,r}$.
\end{itemize}
For $k\in\N$ let $E$ be the event that the fraction of $1 \leq j \leq k$ so that $G_{z,2^{-j}r}$ occurs is at least $q$. Then $\p[E^c] = O(e^{-bk})$.
\end{lemma}

\begin{proof}
This is a consequence of Lemma~\ref{lem:good_scales_for_event} and the fact that the sets $\partial A(0,1/4,3/4) \cap X^\theta_{0,1}$, $\partial A(0,1/8,1/4) \cap X^{\theta;*}_{0,1}$, $\partial A(0,1/4,3/4) \cap X^{\theta-\pi/2}_{0,1}$ are almost surely finite.
\end{proof}

We will use another variant where we merge consecutively into several strands. Suppose $\wt{h}$ is a GFF on $\D$ with boundary values so that the flow line $\eta$ with angle $\theta$ from $-i$ to $i$ is defined. Let $X^\theta_{0,1}$ be as in~\eqref{eq:fl_annulus}. For $M,p > 0$, let $G$ denote the following event.
\begin{itemize}
\item The number of points in $\partial A(0,1/4,3/4) \cap X^\theta_{0,1}$ is at most $M$.
\item Given $X^\theta_{0,1}$ and the values of $\wt{h}$ on $X^\theta_{0,1}$, for any choice of strands $\eta^\theta_{w_1;\mathrm{in}},\eta^\theta_{w_2;\mathrm{out}},\eta^\theta_{w_3;\mathrm{in}},\ldots$ alternatingly ending on $\partial B(0,1/4)$ and $\partial B(0,3/4)$, the conditional probability that $\eta$ merges into $\eta^\theta_{w_1;\mathrm{in}}$ and the continuation of each $\eta^\theta_{w_\ell;\mathrm{out}},\ldots$ merges into $\eta^\theta_{w_{\ell+1};\mathrm{in}}$ before entering $B(0,1/4)$ is either $0$ or at least $p$.
\end{itemize}

\begin{lemma}\label{le:good_scales_merging_multiple}
Suppose that $h$ is a GFF on $D$ with some bounded boundary values. For any $b>0$ and $q \in (0,1)$ there exist $M,p>0$ so that the following is true. 
Let $z \in D$ and $r>0$ such that $B(z,r) \subseteq D$. Let $\wt{h}_{z,r}$ be as defined in Section~\ref{se:gff} with the same boundary values as $\wt{h}$ above (after scaling and translation). Let $G_{z,r}$ be the event that
\begin{itemize}
 \item the scale $(z,r)$ is $M$-good for $h$, and
 \item the event $G$ above occurs for the (scaled and translated) field $\wt{h}_{z,r}$.
\end{itemize}
For $k\in\N$ let $E$ be the event that the fraction of $1 \leq j \leq k$ so that $G_{z,2^{-j}r}$ occurs is at least $q$. Then $\p[E^c] = O(e^{-bk})$.
\end{lemma}

\begin{proof}
 The proof is the same as for Lemma~\ref{le:good_scales_merging} and~\ref{le:good_scales_merging_refl}.
\end{proof}

\section{Intersection crossing exponent}
\label{se:intersection_exponent}

\newcommand*{\innexp}{a_1}

The purpose of this section is to show a superpolynomial tail for the $\Fd_\epsilon$-length of crossing the intersection between two loops (i.e., establish~\eqref{eqn:main_concentration}). The precise setup of this section will be the regions bounded between two intersecting SLE$_\kappa$ flow lines with an angle difference given by the double point angle of SLE$_{\kappa'}$ (i.e., $\angledouble$ from \eqref{eq:angle_double}; cf.\ \cite{mw2017intersections}). This is exactly the type of regions that lie between the outer boundaries of two intersecting SLE$_{\kappa'}$ segments. We will define the renormalization factor $\median{\epsilon}$ in the setup of this section which is more convenient to work with than the definition given above Definition~\ref{def:good_approximation}. (It will turn out that the two definitions are comparable for good approximation schemes.) Roughly speaking, we let $\median{\epsilon}$ be the median of the $\Fd_\epsilon$-distance between the first and last intersection point of the two flow lines in this setup. We will show that the probability that the $\Fd_\epsilon$-distance between two intersection points with Euclidean distance $\delta$ exceeds $\median{\epsilon}$ decays superpolynomially as $\delta\searrow 0$.

In Section~\ref{se:bubble_exponent} we will prove an analogous result for the individual bubbles that lie between two intersection points. As we explain below, we will prove these two tail bounds in parallel, and the results from the next section will be used as an intermediate step for the main result of this section.

We now give a brief outline of this section. We describe in Section~\ref{se:intersections_setup} the setup for the intersection crossing exponent, and state a few basic properties. In particular, the relationships to the bubble crossing exponent in Section~\ref{se:bubble_exponent} will be explained. In Section~\ref{se:a_priori}, using (approximate) self-similarity, we prove an a priori estimate which tells us that the $\Fd_\epsilon$-distance across a region with Euclidean diameter $\delta$ exceeds $\median{\epsilon}$ with probability at most $O(\delta^{\ddouble+o(1)})$ as $\delta \to 0$. We will also prove that the metric scales with at least a power of $\delta$ as we decrease the Euclidean scale $\delta$. Starting from our a priori estimate, we then use an inductive argument to show that the probability actually decays faster than any power of $\delta$. Concretely, let $\bestexp$ be the supremum over the exponents $\alpha$ such that this probability is $O(\delta^\alpha)$ as $\delta \to 0$. We assume that $\bestexp < \infty$ and then aim to get a contradiction by showing that it can in fact be improved to an exponent strictly larger than $\bestexp$. For this step, we will first show the result for the probability of crossing an individual bubble. The proof of this will be outsourced to the following Section~\ref{se:bubble_exponent}. Finally, in Section~\ref{se:intersection_exponent_conclude} we combine this with the a priori estimate $\bestexp \ge \ddouble$ to improve the intersection crossing exponent to a value strictly larger than $\bestexp$.

\subsection{Setup}
\label{se:intersections_setup}

Let $h$ be a GFF on $\h$ with boundary conditions given by $-\lambda - \theta_1 \chi$ on~$\R_-$ and by $\lambda-\theta_2 \chi$ on $\R_+$ where
\begin{equation}\label{eq:angles_intersection}
\theta_1 = -\pi/2
\quad\text{and}\quad 
\theta_2 = -\pi/2 - \angledouble \\
\end{equation}
and $\angledouble$ is the double point angle~\eqref{eq:angle_double}. (The reason that we are adding $-\pi/2$ to both angles is that these angles will also be used in Section~\ref{se:bubble_exponent} where the angle $\theta_1 = -\pi/2$ will appear as the right boundary of a counterflow line.)

Let $\eta_1$ (resp.\ $\eta_2$) be the flow line of $h$ from $0$ to $\infty$ with angle $\theta_1$ (resp.\ $\theta_2$). Note that the conditional law of $\eta_2$ given $\eta_1$ is that of an \slekr{\kappa-4}, and vice versa. Finally, we let $\Gamma$ be a collection of conditionally independent $\CLE_{\kappa'}$ in each of the components of $\h \setminus (\eta_1 \cup \eta_2)$ between $\eta_1$, $\eta_2$.

\begin{figure}[ht]
\centering
\includegraphics[width=0.45\textwidth]{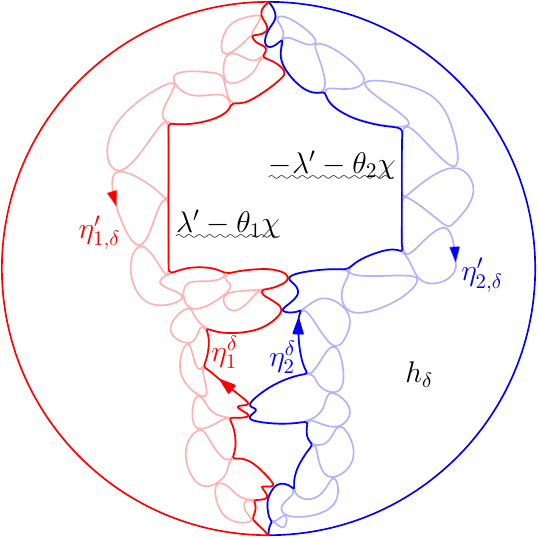}
\caption{The setup of Section~\ref{se:intersection_exponent}.}
\end{figure}

For $\delta \in (0,1)$ let $\varphi_\delta \colon \h \to \delta\D$ be a conformal map that takes $0$ to $-i \delta$ and $\infty$ to $i \delta$.  Let $h_\delta = h \circ \varphi_\delta^{-1} - \chi \arg (\varphi_\delta^{-1})'$, $\eta_i^\delta = \varphi_\delta(\eta_i)$ for $i=1,2$, let $\Gamma_\delta = \varphi_\delta(\Gamma)$.
Let
\begin{equation}\label{eq:pmed_def}
 \pmed = \p[ \eta_1^\delta \cap \eta_2^\delta \cap B(0,\delta/2) \neq \varnothing ] 
\end{equation}
which by scaling does not depend on $\delta$. For $x,y \in \eta_1^\delta \cap \eta_2^\delta \cap \D$, let $U_{x,y}$ denote the region bounded between the segments of $\eta_1^\delta,\eta_2^\delta$ from $x$ to $y$. Let $\intpts{\delta}$ be the set of pairs $(x,y) \in (\eta_1^\delta \cap \eta_2^\delta)^2$ such that $U_{x,y} \subseteq B(0,3\delta/4)$. For an (approximate) \clekp{} metric $\metapprox{\epsilon}{\cdot}{\cdot}{\Gamma}$, we let $\metapproxres{\epsilon}{U_{x,y}}{\cdot}{\cdot}{\Gamma_\delta}$ be the internal metric in $U_{x,y}$. To see that this is well-defined, let $\eta'_{1,\delta}$ (resp.\ $\eta'_{2,\delta}$) be the counterflow line of $h_\delta$ with angle $\theta_1+\pi/2$ (resp.\ $\theta_2-\pi/2$) from $i\delta$ to $-i\delta$. Then $\eta_1^\delta$ (resp.\ $\eta_2^\delta$) agrees with the right boundary of $\eta'_{1,\delta}$ (resp.\ the left boundary of $\eta'_{2,\delta}$), and the conditional law of $\eta'_{1,\delta}$ (resp.\ $\eta'_{2,\delta}$) given $\eta_1^\delta,\eta_2^\delta$ is that of an SLE$_{\kappa'}(\kappa'/2-4;\kappa'/2-2)$ (resp.\ SLE$_{\kappa'}(\kappa'/2-2;\kappa'/2-4)$) in the complementary connected component. Let $\Gamma'_\delta$ consist of the \clekp{} generated by $h$ in each connected component on the right side of $\eta'_{1,\delta}$ and in each connected component on the left side of $\eta'_{2,\delta}$. We now explain that the values of $h_\delta$ in a neighborhood of $U_{x,y}$ determine the configuration of $\Gamma'_\delta$ in some small neighborhood of $U_{x,y}$. This, by the Markovian property and absolute continuity, implies that the internal metrics within $U_{x,y}$ are defined, and due to the triviality of the infinitesimal information of a (Dirichlet-)GFF near the boundary \cite[Proposition~3.2]{ms2016ig1}, are conditionally independent of $\eta'_{1,\delta},\eta'_{2,\delta}, \Gamma'_\delta \setminus \Gamma_\delta$ given $\eta_1^\delta,\eta_2^\delta$.

To see the claim above, let $U$ be a small neighborhood of $U_{x,y}$, and consider the strands of the associated space-filling \slekp{} that intersect $\eta_1^\delta$ (resp.\ $\eta_2^\delta$), extended until exiting $U$. These are determined by the values of $h_\delta$ in $U$, and we know that they are traversed in the reverse order as visited by $\eta_1^\delta$ (resp.\ $\eta_2^\delta$). This determines also the non-space-filling counterflow lines $\eta'_{1,\delta}$ (resp.\ $\eta'_{2,\delta}$) in some smaller neighborhood of $U_{x,y}$, and the \clekp{} configurations in the enclosed regions.

Let $\cserial \ge 0$ be the constant from the compatibility assumption in Section~\ref{se:assumptions}. For $0 \le \epsilon < \delta/(1 \vee 10\cserial)$, let $\intptsapprox{\delta}{\epsilon}$ be the set of tuples $(x',x,y,y') \in (\eta_1^\delta \cap \eta_2^\delta)^4$ in the same order as visited by $\eta_1^\delta,\eta_2^\delta$ and such that $(x,y)\in\intpts{\delta}$, $\dpath[\ol{U}_{x',y'}](x,x') \ge \cserial\epsilon$, $\dpath[\ol{U}_{x',y'}](y,y') \ge \cserial\epsilon$. (The reason we introduce this is so that~\eqref{eq:strong_compatibility_bubbles} below is satisfied.)

For $q \in (0,1)$, we let $\quant[\delta]{q}{\epsilon}$ be the $q$-quantile of the random variable
\begin{equation}\label{eq:quantile_def}
\sup_{(x',x,y,y')\in\intptsapprox{\delta}{\epsilon}}\metapproxres{\epsilon}{U_{x',y'}}{x}{y}{\Gamma_\delta} 
\quad\text{conditioned on the event that $\eta_1^\delta \cap \eta_2^\delta \cap B(0,\delta/2) \neq \varnothing$.}
\end{equation}
We write $\median[\delta]{\epsilon} = \quant[\delta]{1/2}{\epsilon}$ for the median, and $\median{\epsilon} = \median[1]{\epsilon}$, $\quant{q}{\epsilon} = \quant[1]{q}{\epsilon}$. (We allow the quantiles to be zero in case the approximate CLE metric is not almost surely positive.)

Recall the metric $\metapproxacres{\epsilon}{U}{\cdot}{\cdot}{\Gamma}$ defined in~\eqref{eq:shortcutted_metric}. 
The main result of this section is the following proposition which will be proved at the very end of Section~\ref{se:intersection_exponent_conclude}, with a key step outsourced to Section~\ref{se:bubble_exponent}. (We note that the condition $\epsilon < \delta/(1 \vee 10\cserial)$ is not a true restriction since in the case $\delta \le \epsilon$ we automatically have $\metapproxac{\epsilon}{\cdot}{\cdot}{\Gamma} \le \ac{\epsilon}$ by the definition~\eqref{eq:shortcutted_metric}.)

\begin{proposition}
\label{prop:intersection_crossing_exponent}
For any fixed $\epsexp > 0$, we have
\[ 
\p\left[ \sup_{(x',x,y,y')\in\intptsapprox{\delta}{\epsilon}}\metapproxacres{\epsilon}{U_{x',y'}}{x}{y}{\Gamma_\delta} \geq \median{\epsilon}+\epsilon^{-\epsexp}\ac{\epsilon} \right] = o^\infty(\delta) 
\quad\text{as } \delta \searrow 0 .
\]
\end{proposition}

We note that the compatibility assumption for the internal metrics implies that if we have $(x',x,y,y'), (x'',x,y,y'') \in \intptsapprox{\delta}{\epsilon}$, then
\begin{equation}\label{eq:strong_compatibility_bubbles}
 \metapproxres{\epsilon}{U_{x',y'}}{x}{y}{\Gamma_\delta} = \metapproxres{\epsilon}{U_{x'',y''}}{x}{y}{\Gamma_\delta} .
\end{equation}
Further, the series law implies the following monotonicity for the distances between intersection points of $\eta_1^\delta$, $\eta_2^\delta$. If $x,y,x',y' \in \eta_1^\delta \cap \eta_2^\delta$ are such that $U_{x,y} \subseteq U_{x',y'}$, then
\begin{equation}\label{eq:monotonicity_across_bubbles}
 \metapproxres{\epsilon}{U}{x}{y}{\Gamma_\delta} \le \metapproxres{\epsilon}{U}{x'}{y'}{\Gamma_\delta} 
\end{equation}
for any $U \supseteq U_{x',y'}$.
Finally, we note that due to the Markovian property of the CLE metric, we can sample the internal metrics in the regions $U_{x,y}$ without knowing the parts of $\eta_1^\delta$, $\eta_2^\delta$ away from $\partial U_{x,y}$.

In the proofs we will consider the regions bounded between pairs of intersecting flow lines similar to $\eta_1^\delta,\eta_2^\delta$. If $\wt{\eta}_1,\wt{\eta}_2$ is a pair of intersecting flow lines and $x,y \in \wt{\eta}_1 \cap \wt{\eta}_2$, we will also use $U_{x,y}$ to denote the regions bounded between the segments of $\wt{\eta}_1,\wt{\eta}_2$ from $x$ to $y$. We will further denote
\begin{equation}\label{eq:intpts_bubble}
 \intptsapproxbubble{U_{x',y'}}{\epsilon} = \left\{ (x,y) \in (\wt{\eta}_1 \cap \wt{\eta}_2 \cap \ol{U}_{x',y'})^2 \ :\  \dpath[\ol{U}_{x',y'}](x,x') \ge \cserial\epsilon ,\, \dpath[\ol{U}_{x',y'}](y,y') \ge \cserial\epsilon \right\} .
\end{equation}

\subsubsection{Markovian properties}
\label{se:cle_sle_markov}

The following lemma describes the law of a single bubble bounded between $\eta_1$ and $\eta_2$.

\begin{lemma}\label{le:law_bubble}
Let $\eta_1,\eta_2$ be as described in Section~\ref{se:intersections_setup}. Given $\eta_1 \cap \eta_2$, pick two points in $\eta_1(\tau_1), \eta_1(\sigma_1) \in \eta_1 \cap \eta_2$ such that $\eta_1((\tau_1,\sigma_1))$ is a connected component of $\eta_1 \setminus \eta_2$. Let $\tau_2,\sigma_2$ be such that $\eta_1(\tau_1)=\eta_2(\tau_2)$, $\eta_1(\sigma_1)=\eta_2(\sigma_2)$. Then for $i=1,2$ the conditional law of $\eta_i|_{[\tau_i,\sigma_i]}$ given $\eta_{3-i}$ and $\eta_i|_{[0,\tau_i]\cup[\sigma_i,\infty]}$ is an \slek{} in (a connected component of) $\h \setminus (\eta_{3-i} \cup \eta_i([0,\tau_i]\cup[\sigma_i,\infty]))$ from $\eta_i(\tau_i)$ to $\eta_i(\sigma_i)$.
\end{lemma}

\begin{proof}
This follows from \cite[Theorem~5.6]{dms2021mating}.
\end{proof}

We note that Lemma~\ref{le:law_bubble} and \cite[Theorem~4.1]{ms2016ig2} uniquely determine the conditional joint law of $(\eta_1|_{[\tau_1,\sigma_1]}, \eta_2|_{[\tau_2,\sigma_2]})$ given $\eta_1|_{[0,\tau_1]\cup[\sigma_1,\infty]}$ and $\eta_2|_{[0,\tau_2]\cup[\sigma_2,\infty]}$.

Next, we collect the following facts which allow us to discover a \slek{}-\clekp{} pair in certain Markovian ways.

\begin{lemma}
\label{lem:cle_sle_markov}
Suppose that $\eta$ is an $\SLE_\kappa$ in $\h$ from $0$ to $\infty$ and (given $\eta$) $\Gamma$ is a $\CLE_{\kappa'}$ in the component of $\h \setminus \eta$ to the left of $\eta$. Let $-\infty \le a < b \le 0$, and let $\eta'$ be the branch of the exploration tree of $\Gamma$ from $a$ to $b$ that discovers the loops of $\Gamma$ intersecting $[a,b]$, and let $\wt{\eta}$ be the outer boundary of $\eta'$ (we see $\wt{\eta}$ as a path from $0$ to $\infty$, allowed to trace parts of $\R_-$ in case $(a,b)\neq(-\infty,0)$). Given $\wt\eta$ and the set $\wt\eta \cap \eta$, pick two points $\eta(\tau), \eta(\sigma) \in \wt\eta \cap \eta$ such that $\eta((\tau,\sigma))$ is a connected component of $\eta \setminus \wt{\eta}$. Let $U$ be the component of $\h \setminus (\wt{\eta} \cup \eta)$ with $\eta((\tau,\sigma))$ on its right boundary, let $\wt{U}$ the component of $\h \setminus (\wt{\eta} \cup \eta([0,\tau]) \cup \eta([\sigma,\infty)))$ containing $\eta((\tau,\sigma))$, and let $\Gamma_U$ be the loops of $\Gamma$ that are contained in $\ol{U}$. Then the conditional law of $\eta|_{[\tau,\sigma]}$ given $\Gamma \setminus \Gamma_U$ and $\eta|_{[0,\tau]}$, $\eta|_{[\sigma,\infty)}$ is that of an $\SLE_\kappa$ in $\wt{U}$ from $\eta(\tau)$ to $\eta(\sigma)$ and the conditional law of $\Gamma_U$ given $\Gamma \setminus \Gamma_U$ and $\eta$ is that of a $\CLE_{\kappa'}$ in $U$.
\end{lemma}

An equivalent formulation of the conclusion of Lemma~\ref{lem:cle_sle_markov} is that the conditional law of $(\eta|_{[\tau,\sigma]}, \Gamma_U)$ given $\Gamma \setminus \Gamma_U$ and $\eta|_{[0,\tau]}$, $\eta|_{[\sigma,\infty)}$ is equal to that of the conformal image of $(\eta,\Gamma)$ under a conformal map from $\h$ to $\wt{U}$ that sends $0$ to $\eta(\tau)$ and $\infty$ to $\eta(\sigma)$.

\begin{proof}[Proof of Lemma~\ref{lem:cle_sle_markov}]
The statement in the case when $a=\infty$, $b=0$ follows from \cite[Theorem~5.6]{dms2021mating} once we argue that the conditional law of~$\eta$ given~$\eta'$ is that of an $\SLE_\kappa(\kappa-4)$ process from~$0$ to~$\infty$ in the component of $\h \setminus \eta'$ that is to the right of $\eta'$. By \cite[Theorem~5.6]{dms2021mating}, for any component $\eta((\tau,\sigma))$ of $\eta \setminus \wt{\eta}$ as in the statement of the lemma we have that the conditional law of $\eta|_{[\tau,\sigma]}$ given $\eta_{[0,\tau]}$ and $\eta|_{[\sigma,\infty)}$ is that of an $\SLE_\kappa$ in $\wt{U}$ from $\eta(\tau)$ to $\eta(\sigma)$.  Moreover, the conditional law of $\Gamma_U$ given $\eta$, $\eta'$, and $\Gamma \setminus \Gamma_U$ is that of a $\CLE_{\kappa'}$ in $\Gamma_U$.  This proves that the conditional law of $\Gamma_U$ and $\eta|_{[\tau,\sigma]}$ is as described in the lemma statement.

To see that the conditional law of~$\eta$ given~$\eta'$ is an $\SLE_\kappa(\kappa-4)$, we can couple them with a GFF $h$ on $\h$ with boundary conditions given by $-\lambda$ (resp.\ $\lambda$) on $\R_-$ (resp.\ $\R_+$).  Let~$\eta$ be the flow line of $h$ from $0$ to $\infty$ and let~$\eta'$ be the counterflow line of $h+ 2\lambda' - \pi \chi/2$ from $\infty$ to $0$.  Let also $\Gamma$ be the $\CLE_{\kappa'}$ in the component of $\h \setminus \eta$ to the left of $\eta$ and coupled with $h+ 2\lambda' - \pi \chi/2$.  Then $\eta'$ is also the branch of the exploration tree of $\Gamma$ from $\infty$ to $0$ so that the joint law of $\eta$, $\eta'$, $\Gamma$ is exactly as in the statement of the lemma.  In this coupling, the right boundary $\wt{\eta}$ of $\eta'$ is the flow line of~$h$ from~$0$ to~$\infty$ with angle $2\lambda'/\chi - \pi$. The claim thus follows from \cite[Proposition~7.4]{ms2016ig1}

Now let us prove the case for general $a<b$. Let us denote by $\wt{\eta}_{a,b}$ the outer boundary of the exploration path from $a$ to $b$. We pick a stopping time $T$ and a reverse stopping time $S$ for $\eta$ such that $T < S$ and such that the quad $(\h\setminus\eta([0,T]\cup[S,\infty]), \eta(T),\eta(S),\infty^+,0^-)$ is conformally equivalent to $(\h, 0,\infty,a,b)$. Such times exist since the modulus of the quad depends continuously on $T,S$ and attain any value in $[0,\infty]$ as we vary between $T,S = 0$ and $0 < T=S < \infty$. Given $\eta|_{[0,T]\cup[S,\infty]}$, the conditional law of $(\eta|_{[T,S]},\Gamma,\wt{\eta}_{\infty,0})$ is the same as that of the conformal image of $(\eta,\Gamma,\wt{\eta}_{a,b})$. On the other hand, we could have also produced the pair $(\eta([0,T]\cup[S,\infty]), \wt{\eta}_{\infty,0})$ by first sampling $\wt{\eta}_{\infty,0}$ and then sampling $\eta$ conditionally on $\wt{\eta}_{\infty,0}$. Therefore the claim in the general case follows from the special case $(a,b) = (\infty,0)$.
\end{proof}

\begin{lemma}\label{le:cle_sle_domain_markov}
Suppose that $\eta$ is an $\SLE_\kappa$ in $\h$ from $0$ to $\infty$ and (given $\eta$) $\Gamma$ is a $\CLE_{\kappa'}$ in the component of $\h \setminus \eta$ to the left of $\eta$. Let $W \subseteq \ol{\h}$ be a connected, relatively open set that intersects~$\R_-$, and let $\wt{W} \subseteq W$ be the union of connected components of $W \setminus \eta$ that intersect $\R_-$. Let $\Gamma(\wt{W})$ be the the collection of loops in $\Gamma$ that intersect $\wt{W}$. Let $\wt{\eta}$ denote the outer boundary of the union of loops in $\Gamma(\wt{W})$ (as seen from $1$, say). Given $\wt\eta$ and the set $\wt\eta \cap \eta$, pick two points $\eta(\tau), \eta(\sigma) \in \wt\eta \cap \eta \setminus \ol{W}$ such that $\eta((\tau,\sigma))$ is a connected component of $\eta \setminus \wt{\eta}$. Let $U$ be the component of $\h \setminus (\wt{\eta} \cup \eta)$ with $\eta((\tau,\sigma))$ on its right boundary, let $\wt{U}$ the component of $\h \setminus (\wt{\eta} \cup \eta([0,\tau]) \cup \eta([\sigma,\infty)))$ containing $\eta((\tau,\sigma))$, and let $\Gamma_U$ be the loops of $\Gamma$ that are contained in $\ol{U}$. Then the conditional law of $\eta|_{[\tau,\sigma]}$ given $\Gamma \setminus \Gamma_U$ and $\eta_{[0,\tau]}$, $\eta|_{[\sigma,\infty)}$ is that of an $\SLE_\kappa$ in $\wt{U}$ from $\eta(\tau)$ to $\eta(\sigma)$ and the conditional law of $\Gamma_U$ given $\Gamma \setminus \Gamma_U$ and $\eta$ is that of a $\CLE_{\kappa'}$ in $U$.
\end{lemma}

\begin{proof}
This follows by applying Lemma~\ref{lem:cle_sle_markov} repeatedly. Note that every point in $\wt\eta \cap \eta \setminus \ol{W}$ almost surely lies on a loop of $\Gamma(\wt{W})$ (since for any $z \notin \ol{W}$ there are only finitely many loops of diameter at least $\dist(z,\ol{W})$ in a neighborhood of $z$). Moreover (see e.g.\ \cite[Lemma~2.1]{amy-cle-resampling}), every loop in $\Gamma(\wt{W})$ is connected to $\R_- \cap W$ via a finite number of loops $\CL_1,\CL_2,\ldots \in \Gamma(\wt{W})$ such that each pair $\CL_j,\CL_{j+1}$ intersect in $W$. Therefore we can discover the loops in $\Gamma(\wt{W})$ iteratively as follows. Pick a positive interval $J_1 \subseteq \R_- \cap W$ and let $\wt{\eta}_1$ be the outer boundary of the exploration path discovering the loops intersecting $J_1$ as in Lemma~\ref{lem:cle_sle_markov}. Conditionally on $\wt{\eta}_1$ and $\wt{\eta}_1 \cap \eta$, pick $\wt{\tau}_1,\wt{\sigma}_1 \in \wt{\eta}_1 \cap \eta$ such that $\wt{\eta}_1((\wt{\tau}_1,\wt{\sigma}_1))$ is a connected component of $\wt{\eta}_1 \setminus \eta$, and pick a segment $J_2 \subseteq \wt{\eta}_1[\wt{\tau}_1,\wt{\sigma}_1]$. Repeat this procedure to obtain the outer boundary $\wt{\eta}_2$ of the exploration path discovering the loops intersecting $J_2$. Iterating this procedure, Lemma~\ref{lem:cle_sle_markov} implies that at each step, if we let $\tau_j,\sigma_j$ be such that $\eta(\tau_j)=\wt{\eta}_j(\wt{\tau}_j)$, $\eta(\sigma_j)=\wt{\eta}_j(\wt{\sigma}_j)$, and define $U_j$ and $\wt{U}_j$ analogously to the statement of the lemma, then the conditional law of $(\eta\big|_{[\tau_j,\sigma_j]}, \Gamma_{U_j})$ given $\Gamma \setminus \Gamma_{U_j}$ and $\eta_{[0,\tau_j]}$, $\eta|_{[\sigma_j,\infty)}$ is that of an \slek{}-\clekp{} pair in $\wt{U}_j$. In particular, on the event that $(\tau_j,\sigma_j)$ agrees with $(\tau,\sigma)$ from the statement of the lemma, the conclusion follows. This concludes the proof since we can always find one amongst a countable number of such explorations that discovers $\wt{\eta}((\wt{\tau},\wt{\sigma}))$ in a finite number of iterations.
\end{proof}

\subsection{A priori estimates}
\label{se:a_priori}

Consider the setup described in Section~\ref{se:intersections_setup}. The goal of this subsection is to show the a priori estimates Lemma~\ref{le:a_priori_disc} and Lemma~\ref{le:median_scaling}. Let $\ddouble$ be as in~\eqref{eq:dim_double}.

\begin{lemma}\label{le:a_priori_disc}
Fix $q>0$, $\innexp > 0$. Let $E$ denote the event that there exist $(x',x,y,y') \in \intptsapprox{\delta}{\epsilon}$ such that $U_{x',y'} \subseteq B(0,\delta^{1+\innexp})$ and $\metapproxres{\epsilon}{U_{x',y'}}{x}{y}{\Gamma_\delta} \ge \quant{q}{\epsilon}$. Then $\p[E] = O(\delta^{\ddouble})$ as $\delta\searrow 0$ (uniformly in $\epsilon$).
\end{lemma}

We remark that compared to the statement of Lemma~\ref{le:a_priori_int_fl} below, here we do not need an ``$+o(1)$'' term in the exponent since we are working in $B(0,\delta^{1+\innexp})$.

The intuition behind the proof is as follows: The flow lines $\eta_1,\eta_2$ intersect on a set of dimension $\ddouble$, the double point dimension of \slekp{} (cf.\ \cite{mw2017intersections}). Therefore, if we divide space into boxes of sidelength $\delta$, about $\delta^{-\ddouble}$ of them will contain intersection points. After an intersection point, the law of the remainder of $\eta_1,\eta_2$ (after mapping back to $\h$) will be the same as that of $\eta_1,\eta_2$ again, so with probability approximately $\p[E]$ we will create $\Fd_\epsilon$-length of $\quant{q}{\epsilon}$ in the box, and these events are approximately independent in each box. If $\p[E] > \delta^{\ddouble}$, then it becomes very likely to create length $\quant{q}{\epsilon}$ in at least one box. But by the definition of the quantile, this probability is only $1-q$. This means $\p[E]$ must be smaller than $\delta^{\ddouble}$.

In order to make the argument rigorous, we need to create approximately $\delta^{-\ddouble}$ events that are indeed independent, and each one has probability approximately $\p[E]$. The following elementary lemma is useful as it applies to random variables that are not necessarily independent.

\begin{lemma}\label{le:binary_rv_fraction}
 Let $X_1,\ldots,X_N$ be random variables taking values in $\{0,1\}$, and suppose that $\p[X_i=1] \ge p$ for each $i$. Then
 \[ \p\left[ \sum_i X_i \ge rN \right] \ge p-r . \]
\end{lemma}

\begin{proof}
 We have
 \[ pN \le \E\left[\sum_i X_i\right] \le N \,\p\left[ \sum_i X_i \ge rN \right] + rN , \]
 and the result follows.
\end{proof}

We proceed to proving Lemma~\ref{le:a_priori_disc}. We first prove the following similar, more technical statement which applies under an extra condition that we are on the event from Lemma~\ref{le:good_scales_merging} where flow lines are sufficiently likely to merge. To prove Lemma~\ref{le:a_priori_disc}, we then use Lemma~\ref{le:good_scales_merging} to find with high probability such good scales where this additional event occurs, and apply Lemma~\ref{le:a_priori_int_fl} to these good scales.

Let $h_\delta$ be a GFF in $\delta\D$ with boundary values as described in Section~\ref{se:intersections_setup}. We let $\eta'_{1,\delta}$ (resp.\ $\eta'_{2,\delta}$) be the counterflow line of $h_\delta$ with angle $\theta_1+\pi/2$ (resp.\ $\theta_2-\pi/2$) from $i\delta$ to $-i\delta$. Then the right boundary of $\eta'_{1,\delta}$ agrees with $\eta^\delta_1$, and the left boundary of $\eta'_{2,\delta}$ agrees with $\eta^\delta_2$. Let $X_\delta^\theta = X_{0,\delta}^\theta$ be as defined in~\eqref{eq:fl_annulus}.

\begin{figure}[ht]
\centering
\includegraphics[width=0.5\textwidth]{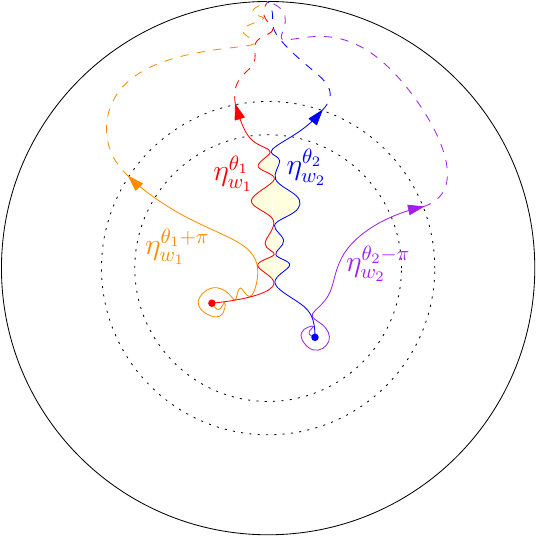}
\caption{The setup of Lemma~\ref{le:a_priori_int_fl}. On the event $G_{0,\delta}$, the conditional probability is at least $p$ that the dashed flow lines exit $\delta\D$ at $i\delta$.}
\label{fi:intersections_localised}
\end{figure}

\begin{lemma}\label{le:a_priori_int_fl}
Fix $q>0$, $M,p > 0$. Let $w_1,w_2 \in B(0,\delta/4)$, and let $E$ denote the event that the following hold.
\begin{itemize}
\item Let $\eta_{w_1}^{\theta_1+\pi}$, $\eta_{w_1}^{\theta_1}$ (resp.\ $\eta_{w_2}^{\theta_2}$, $\eta_{w_2}^{\theta_2-\pi}$) be the flow lines with respective angles starting at $w_1$ (resp.\ $w_2$) and stopped upon exiting $B(0,3\delta/4)$. Then
\begin{itemize}
\item The right side of $\eta_{w_1}^{\theta_1}$ intersects the left side of $\eta_{w_2}^{\theta_2}$ with angle difference $\theta_1-\theta_2$, and they do not intersect in any other way.
\item $\eta_{w_1}^{\theta_1+\pi}$ does not intersect $\eta_{w_2}^{\theta_2}$, $\eta_{w_2}^{\theta_2-\pi}$; and $\eta_{w_2}^{\theta_2-\pi}$ does not intersect $\eta_{w_1}^{\theta_1+\pi}$, $\eta_{w_1}^{\theta_1}$.
\item Sample the internal metrics in the regions bounded between $\eta_{w_1}^{\theta_1}, \eta_{w_2}^{\theta_2}$. There exist $x',x,y,y' \in \eta_{w_1}^{\theta_1} \cap \eta_{w_2}^{\theta_2}$ such that $(x,y) \in \intptsapproxbubble{U_{x',y'}}{\epsilon}$ as in~\eqref{eq:intpts_bubble} and
\[ \metapproxres{\epsilon}{U_{x',y'}}{x}{y}{\Gamma} > \quant{q}{\epsilon} . \] 
\end{itemize}
\item The event $G_{0,\delta}$ from the statement of Lemma~\ref{le:good_scales_merging} occurs for $h_\delta$.
\end{itemize}
Then
\[ \p[E] = O(\delta^{\ddouble+o(1)}) \]
uniformly in the choice of $w_1,w_2$.
\end{lemma}

\begin{proof}
We show that for each $a>0$ we have $\p[E] = O(\delta^{\ddouble-a})$. Throughout the proof, we fix $a>0$ small. Recall the definition of $\quant{q}{\epsilon} = \quant[1]{q}{\epsilon}$, i.e.\ we are in the setup of Section~\ref{se:intersections_setup} with $\delta = 1$. (In the following, the variable $\delta$ will only refer to the setup of the present lemma.) Let $\pmed$ be as in~\eqref{eq:pmed_def}.

\textbf{Step 1.} 
Consider two flow lines $\wt{\eta}_1$ (resp.\ $\wt{\eta}_2$) in $\D$ with angles $\theta_1+a$ (resp.\ $\theta_2-a$). In other words, $\wt{\eta}_1,\wt{\eta}_2$ have slightly wider angle than $\eta_1,\eta_2$. Given $\wt{\eta}_1,\wt{\eta}_2$, the curves $\eta_1$ (resp.\ $\eta_2$) are concatenations of flow lines with angles $\theta_1$ (resp.\ $\theta_2$) within the components of $\D \setminus (\wt{\eta}_1 \cup \wt{\eta}_2)$ that are bounded between $\wt{\eta}_1,\wt{\eta}_2$. By \cite[Lemma~3.11]{dkm-fan-adjacency}, the probability $\p[ \wt{\eta}_1 \cap \wt{\eta}_2 \cap B(0,1/2) \neq \varnothing ] \to \pmed$ as $a \to 0$. We assume that $a$ is small enough so that the difference in probabilities is less than $(q/100)\pmed$.

By \cite[Theorem~1.5]{mw2017intersections}, the set $\wt{\eta}_1 \cap \wt{\eta}_2 \cap \D$ almost surely has Hausdorff dimension $\wt{d} = \ddouble-O(a)$. On the event that $\wt{\eta}_1,\wt{\eta}_2$ intersect in $B(0,1/2)$, the same is true for $\wt{\eta}_1 \cap \wt{\eta}_2 \cap B(0,1/2)$.

For $0 < \delta < 1/2$, we divide $B(0,1/2)$ into $\delta^{-2}$ squares of sidelength $\delta$. For a square $Q$, denote by $I_Q$ the indicator of the event that there exists an intersection point $\wt{\eta}_1(t_1) = \wt{\eta}_2(t_2) \in Q$. The result on the dimension implies
\[ 
\p\left[ \sum_Q I_Q \ge \delta^{-\ddouble+O(a)} \right] \ge (1-q/10)\pmed
\quad\text{for small $\delta$.}
\]
Let $F_\delta$ denote the event $\{ \sum I_Q \ge \delta^{-\ddouble+O(a)}\}$.

\textbf{Step 2.} 
Fix $M,a > 0$ and let $G_{M,a}$ denote the event that for every intersection point $\wt{\eta}_1(t_1) = \wt{\eta}_2(t_2)$, the domain $\D \setminus (\wt{\eta}_1[0,t_1] \cup \wt{\eta}_2[0,t_2])$ is $(M,a)$-good within $B(0,3/4)$ (as defined in Definition~\ref{def:regularity}). By Proposition~\ref{pr:regularity} and Example~\eqref{it:intersections_regularity}, we have $\p[G_{M,a}^c] < (q/10)\pmed$ for sufficiently large $M$.

For $t_1,t_2$ as above, let $\varphi_{t_1}\colon \h \to \D \setminus \Fill(\wt{\eta}_1[0,t_1] \cup \wt{\eta}_2[0,t_2])$ be a conformal transformation with $\varphi_{t_1}(0) = \wt{\eta}_1(t_1)$, $\varphi_{t_1}(\infty) = i$, and normalize it so that $\min\{ y>0 : \dist(\varphi_{t_1}(iy), \wt{\eta}_1[0,t_1] \cup \wt{\eta}_2[0,t_2] \cup \partial\D) = \delta \} = 1$. On the event $G_{M,a}$ we have (with the notation from Definition~\ref{def:regularity})
\[
\dist(\varphi_{t_1}(i), \wt{\eta}_1(t_1)) \lesssim \sum_{k \ge \log_2(\delta^{-1})} \abs{\CZ^{B(0,3/4)}_k}2^{-k} \le \sum_{k \ge \log_2(\delta^{-1})} M2^{-(1-a)k} \lesssim M\delta^{1-a} .
\]

\textbf{Step 3.} 
For each square $Q$, we define now a binary random variable $X_Q$ such that $\p[X_Q=1] \ge 1-(q/10)\pmed$. If $I_Q=0$, we set $X_Q=1$. If $I_Q=1$, let $\wt{\eta}_1(t_{Q,1}) = \wt{\eta}_2(t_{Q,2})$ denote the first intersection point in $Q$. It is not hard to see that the set $\wt{\eta}_1[0,t_{Q,1}] \cup \wt{\eta}_2[0,t_{Q,2}]$ is local for $h$ (since for each closed set $A$, the curve $\wt{\eta}_1$ (resp.\ $\wt{\eta}_2$) up until first exiting $A$ is determined by the values of $h$ on $A$, which by \cite[Lemma~3.9]{ss2013continuumcontour} implies locality).

For some fixed $\wt{r} > 1$, we set $X_Q = 1$ if there exists a component $P_Q$ of $\D \setminus (\wt{\eta}_1[t_{Q,1},\infty)) \cup \wt{\eta}_2[t_{Q,2},\infty))$ such that
\begin{itemize}
\item $\varphi_{t_{Q,1}}^{-1}(P_Q) \subseteq [-\wt{r},\wt{r}]\times[\wt{r}^{-1},\wt{r}]$,
\item there is some point $z_{P_Q}$ on the hyperbolic geodesic in $P_Q$ from its first boundary point to its last boundary point traced by $\wt{\eta}_1$ such that $\dist(z_{P_Q}, \partial P_Q) \ge \wt{r}^{-1}\delta$.
\end{itemize}
We will refer to such $P_Q$ as a ``nice pocket''. 
Since $\abs{\varphi_{t_{Q,1}}'(i)} \asymp \delta$, by Koebe's distortion theorem, such a ``nice pocket'' satisfies $\diam(P_Q) \asymp \delta$ and $\dist(P_Q, \varphi_{t_{Q,1}}(i)) \lesssim \delta$ (up to implicit factors depending on $\wt{r}$).

We claim that for any $c<1$ we can pick $\wt{r}$ such that $\p(X_Q=1) \ge c$ uniformly in $\delta$ and $Q$. Indeed, note that the conditional laws of $(\varphi_{t_{Q,1}}^{-1}(\wt{\eta}_1([t_{Q,1},\infty))), \varphi_{t_{Q,1}}^{-1}(\wt{\eta}_2([t_{Q,2},\infty))))$ are comparable for all choices of $\delta$ and $Q$, the only difference being the location of the force points of weights $O(a)$ at $\varphi_{t_{Q,1}}^{-1}(0^+), \varphi_{t_{Q,1}}^{-1}(0^-)$. Observe that
\begin{enumerate}[1.]
\item $\varphi_{t_{Q,1}}^{-1}(\wt{\eta}_1), \varphi_{t_{Q,1}}^{-1}(\wt{\eta}_2)$ intersect with probability $1$, hence the probability of them intersecting in $B(0,\wt{r})\setminus B(0,\wt{r}^{-1})$ tends to $1$ as $\wt{r} \to \infty$.
\item When they intersect, they will intersect infinitely often in the neighborhood.
\item $\varphi_{t_{Q,1}}^{-1}(\wt{\eta}_1), \varphi_{t_{Q,1}}^{-1}(\wt{\eta}_2)$ do not intersect $\R$ (since $\kappa<4$), hence for any $x>0$ the probability of them intersecting $\{ \abs{\re(z)} \ge x, \im(z) \le \wt{r}^{-1} \}$ tends to $0$ as $\wt{r} \to \infty$.
\end{enumerate}
These facts imply that with probability tending to $1$, the curves $\varphi_{t_{Q,1}}^{-1}(\wt{\eta}_1), \varphi_{t_{Q,1}}^{-1}(\wt{\eta}_2)$ will form a pocket that is contained in $[-\wt{r},\wt{r}]\times[\wt{r}^{-1},\wt{r}]$. Denoting by $\wt P$ the largest such pocket, and $z_{\wt P}$ the point on the hyperbolic geodesic from its first to its last boundary point that maximizes $\dist(z_{\wt P},\partial\wt P)$, we have a.s.\@ $\dist(z_{\wt P},\partial\wt P) > 0$, and therefore $\dist(z_{\wt P},\partial\wt P) > \wt{r}^{-1}$ with probability tending to $1$. Then $P_Q = \varphi_{t_{Q,1}}(\wt P)$ is as desired.

\textbf{Step 4.} 
We apply the previous steps to show that with positive probability, we create $\delta^{-\ddouble+O(a)}$ disjoint ``nice pockets'' of sizes comparable to $\delta$. 

We consider now squares with sidelength $\delta^{1-a}$ instead of $\delta$. Let $t_{Q_1} < t_{Q_2} < \cdots$ each denote the first intersection time in a new square of sidelength $\delta^{1-a}$. (Set $t_{Q_i}=\infty$ when there are no more intersection points.) Let $N_\delta = \delta^{-\ddouble+O(a)}$. By the previous steps and Lemma~\ref{le:binary_rv_fraction} we have
\[ 
\p\left[ G_{M,a} \cap F_{\delta^{1-a}} \cap \left\{ \sum_{i=1}^{N_\delta} X_{Q_i} \ge \frac{q}{10}\pmed N_\delta \right\} \right] 
\ge (1-\frac{4}{10}q)\pmed 
\]
for small enough $\delta$. 
On this event we have $t_{Q_{N_\delta}} < \infty$ and for at least $\frac{q}{10}\pmed N_\delta$ such $Q_i$'s there is a ``nice pocket'' $P_{Q_i}$ of size comparable to $\delta$ within distance $\lesssim \delta^{1-a}$ to $\wt{\eta}_1(t_{Q_i})$ (with implicit factors depending on $M,a,\wt{r}$). Each pocket can be within distance $\delta^{1-a}$ to at most a fixed number of such $Q_i$'s. Therefore at least $\gtrsim N_\delta$ of these ``nice pockets'' must be disjoint.

\textbf{Step 5.} 
Now let $p_\bad = \p[E]$ where $E$ is the event in the lemma statement. We want to show that $p_\bad \lesssim \delta^{\ddouble-O(a)}$.

Recall that for each pocket $P_{Q_i}$ the segments of $\eta_1,\eta_2$ within $P_{Q_i}$ are given by the flow lines of the restriction of $h$ to $P_{Q_i}$. Moreover, the laws of the internal metrics within disjoint pockets are conditionally independent given $\wt{\eta}_1,\wt{\eta}_2$ (due to the Markovian property of the metric). It therefore suffices to show that for each ``nice pocket'', with conditional probability at least $p_\bad^{1+a}$ the distance across the segments of $\eta_1,\eta_2$ within $P_{Q_i}$ is at least $\quant{q}{\epsilon}$. Indeed, this will imply that with conditional probability at least $1-(1-p_\bad^{1+a})^{N_\delta}$ this happens in at least one pocket, i.e.
\begin{equation}\label{eq:larger_than_median}
 \metapproxres{\epsilon}{U_{x',y'}}{x}{y}{\Gamma_1} > \quant{q}{\epsilon} \quad\text{for some }(x',x,y,y')\in\intptsapprox{1}{\epsilon} .
\end{equation}
But by the definition of the quantile $\quant{q}{\epsilon}$ in~\eqref{eq:quantile_def}, the overall probability of~\eqref{eq:larger_than_median} is at most $(1-q)\pmed$. If we had $p_\bad^{1+a}N_\delta \to \infty$ as $\delta\searrow 0$, then the overall probability of~\eqref{eq:larger_than_median} would approach at least $(1-(4/10)q)\pmed$ which is a contradiction. Therefore we must have $p_\bad^{1+a} \lesssim N_\delta^{-1}$ which is the statement we wanted to show.

Let $P=P_{Q_i}$ be one of the ``nice pockets'', and let $z_P \in P$ be as described in Step~3. Let $\varphi_P\colon \delta\D \to P$ be a conformal transformation with $\varphi_P(0) = z_P$. Recall that $|\varphi_P'(0)| \in [\wt{r}_1^{-1},\wt{r}_1]$ for some $\wt{r}_1$ depending only on $\wt{r}$. Pick $c_0>0$ small enough (depending only on $\wt{r}$) so that $|(\varphi_P^{-1})'(w)|/|(\varphi_P^{-1})'(z_P)| \in [0.99,1.01]$ for all $w \in B(z_P,c_0\delta)$.

Let $E_P$ denote the event that the event in the lemma statement occurs for the restriction of $h$ to $B(z_P,c_0\delta)$. This is an event that depends only on the values of $h$ on $B(z_P,3c_0\delta/4)$ and the internal metric in the enclosed regions. By absolute continuity (recall Lemma~\ref{le:abs_cont_kernel}) and the translation invariance, we have $\p[ E_P \mid P ] \gtrsim \p[E]^{1+a}$.

Let $\eta_{w^P_1}^{\theta_1+\pi}$, $\eta_{w^P_1}^{\theta_1}$, $\eta_{w^P_2}^{\theta_2}$, $\eta_{w^P_2}^{\theta_2-\pi}$ denote the flow lines occurring in the event $E_P$. Let $\Fh^X_P$ denote the conditional expectation of $h$ given its values on $X_{z_P,c_0\delta}^{\theta_1+\pi}, X_{z_P,c_0\delta}^{\theta_1}, X_{z_P,c_0\delta}^{\theta_2}, X_{z_P,c_0\delta}^{\theta_2-\pi}$. We argue now that, on the event $E_P$, with positive conditional probability $\wt{p}$ the flow lines $\eta_1,\eta_2$ merge with $\eta_{w^P_1}^{\theta_1}$, $\eta_{w^P_2}^{\theta_2}$ before they trace the region $U_P$ as described in the definition of $E_P$. This will imply that the conditional probability given $P$ that the distance across a region bounded between $\eta_1,\eta_2$ in $P$ is at least $\quant{q}{\epsilon}$ is at least $\wt{p}p_\bad^{1+a}$ which is what we wanted to prove.

It remains to show the claim. Let $\wt{X}_P^\theta$ be defined analogously to $X_\delta^\theta$ but in the conformal annulus 
\[ \wt{A}_P := \varphi\left( A\left(0, \frac{0.8c_0\delta}{|\varphi_P'(0)|}, \frac{0.9c_0\delta}{|\varphi_P'(0)|}\right) \right) . \] 
Recall that the choice of $c_0$ implies that $\wt{A}_P$ sits strictly between $\partial B(z_P,3c_0\delta/4)$ and $\partial B(z_P,0.99c_0\delta)$. Let $\wt{F}_P$ be the following event for $\wt{X}_P^{\theta_1+\pi}, \wt{X}_P^{\theta_1}, \wt{X}_P^{\theta_2}, \wt{X}_P^{\theta_2-\pi}$, and the values of $h$ on these sets: For each quadruple of strands ending on the outer boundary of the annulus with height differences $\pi,\theta_1-\theta_2,\pi$ (in counterclockwise order), the conditional probability that the continuation of the flow line with angle $\theta_1+\pi$ (resp.\ $\theta_2-\pi$) hits the left (resp.\ right) boundary of $P$ is at least $p_1$. By absolute continuity there exists $p_1>0$ (depending on $a,M,p_0$) such that
\[ \p[\wt{F}_P^c \mid \Fh^X_P] \one_{\sup_{\partial B(z_P,7c_0\delta/8)|\Fh^X_P|} \le M} < p_0/2 . \]
In particular,
\[ \p[\wt{F}_P^c \mid \Fh^X_P] \one_{E_P} < p_0/2 . \]

Let $\wt{F}^1_P$ be the event that the flow lines occurring in the event $E_P$ extend to the outer boundary of $\wt{A}_P$ in a way that the continuation of $\eta_{w^P_1}^{\theta_1+\pi}$ does not intersect the continuations of $\eta_{w^P_2}^{\theta_2}$, $\eta_{w^P_2}^{\theta_2-\pi}$, and the continuation of $\eta_{w^P_2}^{\theta_2-\pi}$ does not intersect the continuations of $\eta_{w^P_1}^{\theta_1}$, $\eta_{w^P_1}^{\theta_1+\pi}$. By the last item in the definition of $E_P$ and absolute continuity we have $\p[\wt{F}^1_P \mid E_P] \ge p_0$ where $p_0>0$ depends only on $p$. Hence,
\[ \p[\wt{F}^1_P \cap \wt{F}_P \mid E_P] \ge p_0/2 . \]
Finally, by how $\wt{F}_P$ is defined, the conditional probability given $E_P \cap \wt{F}^1_P \cap \wt{F}_P$ that the continuation of $\eta_{w^P_1}^{\theta_1+\pi}$ (resp.\ $\eta_{w^P_2}^{\theta_2-\pi}$) hits the left (resp.\ right) boundary of $P$ without intersecting $\eta_{w^P_2}^{\theta_2}$, $\eta_{w^P_2}^{\theta_2-\pi}$ (resp.\ $\eta_{w^P_1}^{\theta_1+\pi}$, $\eta_{w^P_1}^{\theta_1}$) is at least $p_1$. Hence, the conditional probability of this given just $E_P$ is at least $p_1 p_0/2$, and then $\eta_1,\eta_2$ necessarily merge with $\eta_{w^P_1}^{\theta_1}$, $\eta_{w^P_2}^{\theta_2}$. This finishes the proof of the claim and the lemma.
\end{proof}

\begin{proof}[Proof of Lemma~\ref{le:a_priori_disc}]
Let $G_{0,2^{-j}}$, for $j = \lceil\log_2(\delta^{-1-\innexp/2})\rceil,\ldots,\lfloor\log_2(\delta^{-1-\innexp})\rfloor$, be the event from Lemma~\ref{le:good_scales_merging}. Let $b>0$ large and let $F^1$ be the event that at least $1/3$ fraction of $G_{0,2^{-j}}$ occur, so that $\p[(F^1)^c] = O(\delta^b)$ for suitable $M,p$. Further, pick a small constant $a>0$ and let $F^2$ be the event that the space-filling SLEs associated with $\eta'_{1,\delta}$ and $\eta'_{2,\delta}$ fill a ball of radius $r^{1+a}$ whenever they travel distance $r \le \delta$ within $B(0,3\delta/4)$. By Lemma~\ref{le:fill_ball} we have $\p[(F^2)^c] = o^\infty(\delta)$.

It therefore suffices to show $\p[E \cap F^1 \cap F^2] = O(\delta^\ddouble)$.

Let $J \in \{ \lceil\log_2(\delta^{-1-\innexp/2})\rceil,\ldots,\lfloor\log_2(\delta^{-1-\innexp})\rfloor \}$ be sampled uniformly at random, and let $w_1,w_2 \in B(0,2^{-J}/4)$ be sampled according to the normalized Lebesgue measure (independently of $h_\delta$ and $\metapprox{\epsilon}{\cdot}{\cdot}{\Gamma_\delta}$). Suppose we are on the event $E \cap F^1 \cap F^2$, and let $U_{x',y'}$ be as described in the event $E$. On the event $F^2$, each $\eta'_{1,\delta}$, $\eta'_{2,\delta}$ will cut off a ball of radius $2^{-J(1+a)}$ after it finishes tracing the left (resp.\ right) boundary of $U_{x',y'}$ and before it exits $B(0,2^{-J}/4)$. Let $\wh{E}$ be the event that $J$ is sampled so that $G_{0,2^{-J}}$ occurs and $w_1,w_2$ are sampled within these respective balls, so that
\[ \p[\wh{E} \mid E \cap F^1 \cap F^2] \ge \delta^{5a} . \]
On the event $\wh{E}$ the flow lines $\eta_{w_1}^{\theta_1}$ (resp.\ $\eta_{w_2}^{\theta_2}$) merge into $\eta^\delta_1$ (resp.\ $\eta^\delta_2$) before they trace the left (resp.\ right) boundary of $U_{x',y'}$.

Note that in order to sample the metric in the region between $\eta^\delta_1$, $\eta^\delta_2$, we can alternatively first sample $J,w_1,w_2$, then $\wt{h}_{0,2^{-J}}$ and $\eta_{w_1}^{\theta_1}, \eta_{w_2}^{\theta_2}$, the internal metrics in their enclosed regions (its law depends only on the values of $\wt{h}_{0,2^{-J}}$ in $B(0,(3/4)2^{-J})$ due to the Markovian property), and finally the remainder of the metric according to its conditional law in case $\wh{E}$ occurs (which is an event for $h$, and does not bias the law of the metric). On the event $\wh{E}$, both procedures yield the same result. In particular, if we let $E_{0,J}$ denote the event described in Lemma~\ref{le:a_priori_int_fl} occurring for $\wt{h}_{0,2^{-J}}$ and $w_1,w_2$, then
\begin{equation}\label{eq:condpr_int_fl_event}
\p[E_{0,J} \cap G_{0,2^{-J}}] \ge \delta^{5a}\p[E \cap F^1 \cap F^2] .
\end{equation}
On the event that the scale $2^{-J}$ is $M$-good, the law of $\wt{h}_{0,2^{-J}}$ is equivalent to that of a GFF with boundary values as $h_\delta$, with Radon-Nikodym derivative having bounded moments of all orders (depending on $M$). By Lemma~\ref{le:abs_cont_kernel}, this is also true for the joint law of $\wt{h}_{0,2^{-J}}$ and the internal metric. Hence, Lemma~\ref{le:a_priori_int_fl} implies $\p[E_{0,J} \cap G_{0,2^{-J}}] = O(\delta^{(1+\innexp/2)\ddouble+o(1)})$. This together with~\eqref{eq:condpr_int_fl_event} shows the result.
\end{proof}

\begin{lemma}
\label{le:median_scaling}
Consider the setup in Section~\ref{se:intersections_setup}. Fix $q,q' \in (0,1)$. Then
\[ \quant[\delta]{q'}{\epsilon} \le \delta^{\ddouble+o(1)}\quant{q}{\epsilon} \quad\text{as } \delta \searrow 0 \]
for $\epsilon < \delta \le 1$.
\end{lemma}

\begin{proof}
We use the setup in the proof of Lemma~\ref{le:a_priori_int_fl}. Since we are considering both the setup in $\D$ and $\delta\D$, we will write $\p_1$ (resp.\ $\p_\delta$) for the law in $\D$ (resp.\ $\delta\D$). By the definition~\eqref{eq:quantile_def} of $\quant[\delta]{q'}{\epsilon}$,
\[
\p_\delta\left[ \sup_{(x',x,y,y')\in\intptsapprox{\delta}{\epsilon}}\metapproxres{\epsilon}{U_{x',y'}}{x}{y}{\Gamma_\delta} \ge \quant[\delta]{q'}{\epsilon} \right] \ge (1-q')\pmed .
\]
Let $F$ be the event that space-filling SLE fills a ball of radius $c\delta$ whenever it travels distance $\delta/4$. Since this event is scale-invariant, we can pick $c$ small enough so that $\p_\delta[F^c] < ((1-q')/4)\pmed$ for all $\delta$.

Sample $w_1,w_2 \in B(0,3\delta/4)$ uniformly at random. Let $E$ be the event defined similarly as in the statement of Lemma~\ref{le:a_priori_int_fl} but with $\quant[\delta]{q'}{\epsilon}$ in place of $\quant{q}{\epsilon}$, and replacing the radii of the balls $B(0,\delta/4)$, $B(0,3\delta/4)$ by $B(0,3\delta/4)$, $B(0,7\delta/8)$. Then the first item in the definition of $E$ (with $\quant{q}{\epsilon}$ replaced by $\quant[\delta]{q'}{\epsilon}$) is satisfied with probability at least $c'(1-q')\pmed$ (where $c'>0$ depends on $c$). The second item, i.e.\ the analogue of the event $G_{0,\delta}$ from Lemma~\ref{le:good_scales_merging}, is an event for the GFF and does not depend on $\delta$. Hence, we can pick $M,p$ such that $\p_\delta[E] \ge c'(1-q')\pmed$.

Now let $a>0$ be arbitrarily small and let us suppose $\quant[\delta]{q'}{\epsilon} > \delta^{\ddouble-a}\quant{q}{\epsilon}$. We argue that this would imply
\[ \p_1\left[ \sup_{(x',x,y,y')\in\intptsapprox{1}{\epsilon}}\metapproxres{\epsilon}{U_{x',y'}}{x}{y}{\Gamma_1} > \quant{q}{\epsilon} \right] > (1-q)\pmed \]
which is a contradiction.

We set up as in the proof of Lemma~\ref{le:a_priori_int_fl}. Suppose we are on the event in Step~4 that there are at least $N_\delta = \delta^{-\ddouble+O(a)}$ disjoint ``nice pockets'' of sizes comparable to $\delta$, and with distances at least $\delta^{1-a}$ from each other (which occurs with probability at least $(1-(4/10)q)\pmed$ when $\delta$ is small). We use the same argument as in Step~5, only that here we have an event $E$ with $\p_\delta[E] \ge c'(1-q')\pmed$. By the argument in Step~5, for each ``nice pocket'', the conditional probability that the distance across it is at least $\quant[c_0\delta]{q'}{\epsilon}$ (where $c_0>0$ is a suitable constant) is at least some fixed constant $p_1>0$.

By the Markovian property of the metric and Cram\'er's theorem, with probability $1-\exp(-cN_\delta+o(N_\delta)) \to 1$ for at least a positive fraction of the $N_\delta$ ``nice pockets'' the distance across is at least $\quant[c_0\delta]{q'}{\epsilon}$. Since the ``nice pockets'' have distances at least $\delta^{1-a}$ from each other, the compatibility~\eqref{eq:strong_compatibility_bubbles} and the (approximate) series law then imply $\metapproxres{\epsilon}{U_{x',y'}}{x}{y}{\Gamma_1} \gtrsim N_\delta \quant[c_0\delta]{q'}{\epsilon}$ for some $(x',x,y,y')\in\intptsapprox{1}{\epsilon}$. This means that we must have $\quant[c_0\delta]{q'}{\epsilon} \le c_1 N_\delta^{-1}\quant{q}{\epsilon}$ for some constant $c_1$, otherwise we would get a contradiction.
\end{proof}

The Lemma~\ref{le:median_scaling} is useful for obtaining stronger bounds on distances such as the probability that it exceeds $\delta^\zeta \median{\epsilon}$ for some $\zeta>0$. This follows by applying the estimates with the metric $\mettapprox{\delta_0^{-1}\epsilon}{\cdot}{\cdot}{\Gamma} = \metapprox{\epsilon}{\delta_0\cdot}{\delta_0\cdot}{\delta_0\Gamma}$ and noting that $\mediant{\delta_0^{-1}\epsilon} = \median[\delta_0]{\epsilon} \le \delta_0^{\ddouble+o(1)}\median{\epsilon}$ (recall Lemma~\ref{le:scaled_metric}). For example, we have the following statement.

\begin{corollary}
Fix $\innexp > 0$. For $0<\delta<\delta_0 \le 1$ and $\epsilon < \delta_0$, let $E$ denote the event that there exist $(x',x,y,y') \in \intptsapprox{\delta}{\epsilon}$ such that $U_{x',y'} \subseteq B(0,\delta^{1+\innexp})$ and $\metapproxres{\epsilon}{U_{x',y'}}{x}{y}{\Gamma_\delta} \ge \delta_0^{\ddouble}\median{\epsilon}$. Then $\p[E] = O((\delta/\delta_0)^{\ddouble})$.
\end{corollary}

We will show later that there is also a lower bound for the scaling behavior of the quantiles, complementing Lemma~\ref{le:median_scaling}. Moreover, we will show that the quantiles are comparable across all scales. These properties are stated as Lemmas~\ref{le:median_scaling_lb} and~\ref{le:quantiles_comparable}.

\subsection{Proof of the intersection crossing exponent}
\label{se:intersection_exponent_conclude}

In this subsection, we conclude the proof of Proposition~\ref{prop:intersection_crossing_exponent} conditionally on the results of Section~\ref{se:bubble_exponent}.

Throughout this subsection, we assume that a variant of Lemma~\ref{le:a_priori_disc} holds for a certain exponent
\[ \bestexp \ge \ddouble . \]
Consider the setup described in Section~\ref{se:intersections_setup}, and recall the metric $\metapproxacres{\epsilon}{U}{\cdot}{\cdot}{\Gamma}$ defined in~\eqref{eq:shortcutted_metric}. Let $\epsexp > 0$. We assume that for any fixed $\innexp > 0$ the following is true. Let $E$ denote the event that there exist $(x',x,y,y') \in \intptsapprox{\delta}{\epsilon}$ such that $U_{x',y'} \subseteq B(0,\delta^{1+\innexp})$ and $\metapproxacres{\epsilon}{U_{x',y'}}{x}{y}{\Gamma_\delta} \ge \median{\epsilon}+\epsilon^{-\epsexp}\ac{\epsilon}$. Then
\begin{equation}\label{eq:a_priori_assumption}
\p[E] \le \delta^{\bestexp+o(1)} \quad\text{as } \delta\searrow 0 .
\end{equation}

Lemma~\ref{le:a_priori_disc} shows that~\eqref{eq:a_priori_assumption} holds with $\bestexp = \ddouble$. Our goal is to prove that the same holds for any exponent $\bestexp > 0$.

\begin{proposition}\label{pr:intersection_exponent_improve}
Consider the setup described at the beginning of Section~\ref{se:intersections_setup}. Assume that~\eqref{eq:a_priori_assumption} holds for some $\bestexp \in [\ddouble, \infty)$ and $\epsexp > 0$. Then there exists $\alpha' > 0$ such that
\[
\p\left[ \sup_{(x',x,y,y')\in\intptsapprox{\delta}{\epsilon}}\metapproxacres{\epsilon}{U_{x',y'}}{x}{y}{\Gamma_\delta} \geq \median{\epsilon}+\epsilon^{-\epsexp}\ac{\epsilon} \right] = O(\delta^{\bestexp+\alpha'})
\quad\text{as } \delta\searrow 0 .
\]
\end{proposition}

\begin{proof}[Proof of Proposition~\ref{prop:intersection_crossing_exponent} given Proposition~\ref{pr:intersection_exponent_improve}]
Let $\epsexp > 0$. From Lemma~\ref{le:a_priori_disc} and Proposition~\ref{pr:intersection_exponent_improve} it follows that~\eqref{eq:a_priori_assumption} holds for any exponent $\bestexp > 0$ (by the form of~\eqref{eq:a_priori_assumption}, if the supremum of the exponents such that~\eqref{eq:a_priori_assumption} holds were finite, it would be a maximum, contradicting Proposition~\ref{pr:intersection_exponent_improve}). This implies that Proposition~\ref{pr:intersection_exponent_improve} holds for any exponent $\bestexp > 0$. This is exactly the statement of Proposition~\ref{prop:intersection_crossing_exponent}.
\end{proof}

In the remainder of this section, we prove Proposition~\ref{pr:intersection_exponent_improve}. We assume the results of Section~\ref{se:bubble_exponent} which are based on the assumption~\eqref{eq:a_priori_assumption}.

The idea is to define an exploration procedure as follows. We explore the region bounded between $\eta_1^\delta$ and $\eta_2^\delta$ until the first point where $\metapproxacres{\epsilon}{U_{x',y'}}{x}{y}{\Gamma_\delta} \ge \median{\epsilon}+\epsilon^{-\epsexp}\ac{\epsilon}$ for some $(x',x,y,y')\in\intptsapprox{\delta}{\epsilon}$. By the definition of the median $\median{\epsilon}$, with probability $1/2$ the exploration will reach $i\delta$ without stopping. If not, we stop after the last bubble we have been exploring and start exploring again. We will show that conditionally on what we have explored so far, we have another positive chance of reaching $i\delta$ within $\Fd_\epsilon$-distance $\median{\epsilon}+\epsilon^{-\epsexp}\ac{\epsilon}$. Iterating, we see that the probability that we stop the exploration more than $N$ times is $O(e^{-cN})$. In case we need at most $N$ steps, the total distance is at most $N(\median{\epsilon}+\epsilon^{-\epsexp}\ac{\epsilon})$ plus the distances across the bubbles in which the exploration steps failed. From the bubble crossing exponent of Section~\ref{se:bubble_exponent}, the probability that the distance across a given bubble of size $\delta_1 < \delta$ is more than $\median{\epsilon}+\epsilon^{-\epsexp}\ac{\epsilon}$ is $O(\delta_1^{\bubbleexp})$ for an exponent $\bubbleexp > \bestexp$. Since there are approximately $(\delta_1/\delta)^{-\ddouble}$ such bubbles and $\bestexp \ge \ddouble$ by Lemma~\ref{le:a_priori_disc}, we can sum over them and conclude that with probability $1-O(\delta^{\bubbleexp})$, each bubble is good. Picking $N \asymp \log(\delta^{-1})$ we get that $\sup_{(x',x,y,y')\in\intptsapprox{\delta}{\epsilon}}\metapproxacres{\epsilon}{U_{x',y'}}{x}{y}{\Gamma_\delta} \le 2N(\median{\epsilon}+\epsilon^{-\epsexp}\ac{\epsilon}) \lesssim \log(\delta^{-1})(\median{\epsilon}+\epsilon^{-\epsexp}\ac{\epsilon})$ with probability $1-O(\delta^{\bubbleexp})$. The result then follows by scaling and Lemma~\ref{le:median_scaling}.

Recall the key property (Lemma~\ref{le:law_bubble}) that the law of a single bubble (a connected component of $\delta\D \setminus (\eta_1^\delta \cup \eta_2^\delta)$ that lies between $\eta_1^\delta$ and $\eta_2^\delta$) is given by the law of the pair $(\eta_0,\eta_2)$ as in Section~\ref{se:bubble_setup} in the connected component of $\delta\D \setminus (\bigcup_{i=1,2}\eta_i^\delta([0,\tau_i] \cup [\sigma_i,\infty]))$ with $\eta_1^\delta(\tau_1),\eta_1^\delta(\sigma_1)$ on its boundary (in the notation of Lemma~\ref{le:law_bubble}). The following lemma is the key step transferring the bubble crossing exponent to the bound described in the paragraph above.

\begin{lemma}\label{le:each_bubble_good}
There exist $\alpha'>0$, $\zeta>0$ such that the following is true. Let $E^\bubble$ denote the event that for every connected component $V$ of $\delta\D \setminus (\eta_1^\delta \cup \eta_2^\delta)$ between $\eta_1^\delta$, $\eta_2^\delta$ with $V \subseteq B(0,3\delta/4)$ and $\diam(V) > \epsilon$, we have
\[ \metapproxacres{\epsilon}{V}{x_V}{y_V}{\Gamma_\delta} \le \diam(V)^\zeta \median{\epsilon}+\diam(V)^{\zeta\epsexp}\epsilon^{-\epsexp}\ac{\epsilon} \]
where $x_V$ (resp.\ $y_V$) denotes the first (resp.\ last) point of $\partial V$ traced by $\eta_1^\delta,\eta_2^\delta$. Then $\p[(E^\bubble)^c] = O(\delta^{\bestexp+\alpha'})$.
\end{lemma}

We will prove Lemma~\ref{le:each_bubble_good} separately for ``typical'' and ``atypical'' shapes of $V$. We begin with the ``typical'' case.

\begin{lemma}
There exists $c_1>0$ such that the conclusion of Lemma~\ref{le:each_bubble_good} holds for $V$ with $\diam(V) \le \abs{x_V-y_V}^{1-c_1}$.
\end{lemma}

\begin{proof}
Throughout, we fix $M,a>0$ where $a$ is chosen sufficiently small (depending on $\epsexp$). Let $\wt{G}_{M,a}$ be the regularity event described in Corollary~\ref{co:regularity_scaled} and Example~\eqref{it:expl_regularity}. That is, for $i=1,2$, $s<t$, and $\CC \subseteq \Gamma$, we let $A_{i,s,t,\CC}$ be the closure of the union of the loops in $\CC$, $\eta_{3-i}^\delta$, and $\eta_i^\delta([0,s]) \cup \eta_i^\delta([t,\infty))$. Let $D_{i,s,t,\CC}$ be the component of $\delta\D \setminus A_{i,s,t,\CC}$ containing $\eta_i^\delta([s,t])$. Then we let $\wt{G}_{M,a}$ be the event that $(D_{i,s,t,\CC}, \eta_i^\delta(s), \eta_i^\delta(t))$ is $(M,a)$-good within $B(0,3\delta/4)$ for every $i=1,2$, $s < t$, and $\CC \subseteq \Gamma$. By Corollary~\ref{co:regularity_scaled}, we have $\p[G_{M,a}^c] = o^\infty(\delta)$.

Further, let $G_\delta$ denote the event that the space-filling SLEs whose right (resp.\ left) boundaries are given by the angle $\theta_1$ (resp.\ $\theta_2$) flow lines fill a ball of radius $(\delta')^{1+a}$ whenever they travel distance $\delta' < \delta$ within $B(0,3\delta/4)$. By Lemma~\ref{le:fill_ball} we have $\p[G_\delta^c] = o^\infty(\delta)$.

For $k \in \N$ and $z \in B(0,3\delta/4) \cap \delta 2^{-k}\Z^2$, let $F_{z,k}$ denote the event that there exists a bubble $V \subseteq B(z,(\delta 2^{-k})^{1-c_1})$ with $x_V \in B(z,\delta 2^{-k})$ and $\abs{x_V-y_V} \in [\delta 2^{-k}, \delta 2^{-k+1}]$, and let $F^\bad_{z,k}$ denote the event that there exists such a bubble with $\metapproxacres{\epsilon}{V}{x_V}{y_V}{\Gamma_\delta} \ge \diam(V)^\zeta \median{\epsilon}+\diam(V)^{\zeta\epsexp}\epsilon^{-\epsexp}\ac{\epsilon}$. By Proposition~\ref{pr:dbl_exponent}, we have $\p[F_{z,k}] = O(2^{-k(2-\ddouble+o(1))})$. Moreover, on the event $G_\delta$ there exist at most $O((\delta 2^{-k})^{-2a})$ such bubbles in $B(z,(\delta 2^{-k})^{1-c_1})$.

Let $(\eta^V_1,\eta^V_2)$ denote the segments of $(\eta^\delta_1,\eta^\delta_2)$ between $x_V,y_V$. By Lemma~\ref{le:law_bubble}, the conditional law of $(\eta^V_1,\eta^V_2)$ given the remaining parts of $\eta_1^\delta,\eta_2^\delta$ is the same as the law described in Section~\ref{se:bubble_setup}. On the event $\wt{G}_{M,a}$ in particular the event $G_{M,a}$ described before the statement of Proposition~\ref{prop:bubble_crossing_exponent} occurs for the pair $(\eta^V_1,\eta^V_2)$. Thus, by Proposition~\ref{prop:bubble_crossing_exponent} we have
\[
\p[F^\bad_{z,k} \cap \wt{G}_{M,a} \cap G_\delta] \lesssim 2^{-k(2-\ddouble+o(1))} (\delta 2^{-k})^{(1-c_1)\bubbleexp-2a} .
\]
We sum up the possible choices of $k\in\N$ and $z \in B(0,3\delta/4) \cap \delta 2^{-k}\Z^2$. Since $\bubbleexp > \bestexp \ge \ddouble$, we can pick $c_1,\alpha'>0$ such that $(1-c_1)\bubbleexp-2a > \bestexp+\alpha'$, and hence
\[
\p[\wt{E}^c \cap \wt{G}_{M,a} \cap G_\delta] \lesssim \sum_{k\in\N} 2^{2k} 2^{-k(2-\ddouble+o(1))} (\delta 2^{-k})^{(1-c_1)\bubbleexp-2a} \asymp \delta^{(1-c_1)\bubbleexp-2a} \le \delta^{\bestexp+\alpha'}
\]
where $\wt{E}$ denotes the event that $\metapproxacres{\epsilon}{V}{x_V}{y_V}{\Gamma_\delta} \le \diam(V)^\zeta \median{\epsilon}+\diam(V)^{\zeta\epsexp}\epsilon^{-\epsexp}\ac{\epsilon}$ for every $V$ with $\diam(V) \le |x_V-y_V|^{1-c_1}$.
\end{proof}

We now consider the bubbles with ``atypical'' shapes. We use the fact that such bubbles are rare as specified in the following lemma. Recall the exponents $\ddouble$ from~\eqref{eq:dim_double} and $\alpha_{4,\kappa} > 2$ from~\eqref{eqn:double_exponent_simple}.

\begin{lemma}\label{le:7arms}
Let $\eta_1^1,\eta_2^1$ be as in Section~\ref{se:intersections_setup}. For $z \in B(0,3/4)$, $r,r' \in (0,1)$, let $E$ denote the event that that $\eta_1^1$ and $\eta_2^1$ both intersect $B(z,r)$ and there exist $\eta_1^1(s),\eta_1^1(t) \in B(z,r)$ with $\abs{\eta_1^1(s)-\eta_1^1(t)} \le r'r$ and $\diam(\eta_1^1[s,t]) \ge r$. Then $\p[E] = O((r')^{\alpha_{4,\kappa}-2+o(1)}r^{2-\ddouble+o(1)})$.
\end{lemma}

\begin{proof}
By Proposition~\ref{pr:dbl_exponent} the probability that both $\eta_1^1,\eta_2^1$ intersect $B(z,2r)$ is $O(r^{2-\ddouble+o(1)})$. Let $\tau_1,\tau_2$ denote their respective hitting times of $B(z,2r)$. Conditionally on $\eta_1^1[0,\tau_1],\eta_2^1[0,\tau_2]$, we apply Proposition~\ref{pr:4arm_simple} for each $z' \in r'r\Z^2 \cap B(z,r)$. Taking a union bound gives the result.
\end{proof}

\begin{lemma}
For every fixed $\wt{a}>0$ there exist $\alpha'>0$, $\zeta>0$ such that the conclusion of Lemma~\ref{le:each_bubble_good} holds for $V$ with $\diam(V) \ge \abs{x_V-y_V}^{1-\wt{a}}$.
\end{lemma}

\begin{proof}
Let $(\eta^V_1,\eta^V_2)$ denote the segments of $(\eta^\delta_1,\eta^\delta_2)$ between $x_V,y_V$. By symmetry, it suffices to consider those bubbles where $\diam(\eta^V_1) \ge \diam(\eta^V_2)$.

Fix $a>0$ sufficiently small (depending on $\wt{a},\epsexp$; the precise condition will become apparent later). Let $G_{M,a}$ denote the event that for any $s<t$, the connected component of $\delta\D \setminus (\eta_1^\delta \cup \eta_2^\delta([0,s] \cup [t,\infty]))$ containing $\eta_2^\delta[s,t]$ is $(M,a)$-good within $B(0,3\delta/4)$. By Corollary~\ref{co:regularity_scaled}, we have $\p[G_{M,a}^c] = o^\infty(\delta)$.

Further, let $G_\delta$ denote the event that the space-filling SLEs whose right (resp.\ left) boundaries are given by the angle $\theta_1$ (resp.\ $\theta_2$) flow lines fill a ball of radius $(\delta')^{1+a}$ whenever they travel distance $\delta' < \delta$ within $B(0,3\delta/4)$. By Lemma~\ref{le:fill_ball} we have $\p[G_\delta^c] = o^\infty(\delta)$.

For $k \in \N$ and $z \in B(0,3\delta/4) \cap \delta 2^{-k}\Z^2$, let $F_{z,k}$ denote the event that there exists a bubble $V \subseteq B(z,\delta 2^{-k+2})$ with $\abs{x_V-y_V} < (\delta 2^{-k})^{1+\wt{a}}$ and $\diam(\eta^V_1) \ge \delta 2^{-k}$. Fix a small constant $\zeta>0$, and let $F^\bad_{z,k}$ denote the event that there exists such a bubble with $\metapproxacres{\epsilon}{V}{x_V}{y_V}{\Gamma_\delta} \ge \diam(V)^\zeta \median{\epsilon}+\diam(V)^{\zeta\epsexp}\epsilon^{-\epsexp}\ac{\epsilon}$. By Lemma~\ref{le:7arms}, we have $\p[F_{z,k}] = O(2^{-k(2-\ddouble+o(1))}(\delta 2^{-k})^{\wt{a}\alpha})$ for some $\alpha > 0$. Moreover, on the event $G_\delta$ there exist at most $O((\delta 2^{-k})^{-2a})$ such bubbles. Therefore, applying Lemma~\ref{le:law_bubble} and Corollary~\ref{co:close_geodesic_small_parts} to $D = \delta\D \setminus (\eta_1^\delta \cup (\eta_2^\delta \setminus \eta^V_2))$ gives us, for a constant $c$,
\[ 
\p[ F^\bad_{z,k} \cap G_{M,a} ] 
\lesssim 2^{-k(2-\ddouble+o(1))}(\delta 2^{-k})^{\wt{a}\alpha}(\delta 2^{-k})^{\bestexp-c\zeta-ca} .
\]
To finish the proof, we just need to sum over all possible choices of $k\in\N$ and $z \in B(0,3\delta/4) \cap \delta 2^{-k}\Z^2$ which gives us (in case $c\zeta+ca < \alpha\wt{a}$)
\[
\sum_{k\in\N} 2^{2k} 2^{-k(2-\ddouble+o(1))}(\delta 2^{-k})^{\wt{a}\alpha}(\delta 2^{-k})^{\bestexp-c\zeta-ca} \asymp \delta^{\bestexp+\wt{a}\alpha-c\zeta-ca}
\]
since $\bestexp \ge \ddouble$.
\end{proof}

We now carry out the second step in the proof of Proposition~\ref{pr:intersection_exponent_improve}. We describe the following exploration procedure. Let $x_0 = -i\delta$. Given $x_{k-1}$, let $x_k \in \eta_1^\delta\cap\eta_2^\delta$ be the first intersection point after $x_{k-1}$ such that there exist intersection points $x',x,y \in \eta_1^\delta\cap\eta_2^\delta$ between $x_{k-1}$ and $x_k$ with $(x',x,y,x_k) \in \intptsapprox{\delta}{\epsilon}$ such that $\metapproxacres{\epsilon}{U_{x',x_k}}{x}{y}{\Gamma_\delta} \ge \median{\epsilon}+\epsilon^{-\epsexp}\ac{\epsilon}$. We set $x_k = i\delta$ if there is no such point. Let $N$ be the smallest integer such that $x_N=i\delta$.

\begin{lemma}\label{le:geometric_exploration}
Let $E^\bubble$ be the event from Lemma~\ref{le:each_bubble_good}. There exists $\zeta>0$ such that
\[ 
\p[ N > n ,\, E^\bubble ] = O(\delta^{n\zeta\epsexp})
\quad\text{as } \delta\searrow 0 . 
\]
\end{lemma}

The proof of Lemma~\ref{le:geometric_exploration} will be given in several steps.

\begin{lemma}\label{le:expl_markov}
For $k\in\N$, let $t_{1,k},t_{2,k}$ be such that $\eta_1^\delta(t_{1,k}) = \eta_2^\delta(t_{2,k}) = x_k$, and let $U$ be the collection of connected components bounded between $\eta_1^\delta[0,t_{1,k}]$ and $\eta_2^\delta[0,t_{2,k}]$. On the event $\{N>k\}$, the conditional law of $(\eta_1^\delta\big|_{[t_{1,k},\infty)}, \eta_2^\delta\big|_{[t_{2,k},\infty)})$ given $(\eta_1^\delta\big|_{[0,t_{1,k}]}, \eta_2^\delta\big|_{[0,t_{2,k}]})$ is given by the conformal image to $\delta\D \setminus \overline{U}$ of an independent pair $(\eta_1,\eta_2)$ as described at the beginning of Section~\ref{se:intersections_setup}. The conditional law of $\Gamma \setminus \Gamma_U$ is that of an independent \clekp{} in each of the remaining components of $\delta\D \setminus (\eta_1^\delta[t_{1,k},\infty] \cup \eta_2^\delta[t_{2,k},\infty])$.
\end{lemma}

\begin{proof}
We claim that the set $\overline{U}$ is local for $h_\delta$. Indeed, if $A \subseteq \delta\D$ is a relatively closed set, then the conditional probability given $h_\delta$ of the event $\overline{U} \subseteq A$ depends only on the segments of $\eta_1^\delta$ (resp.\ $\eta_2^\delta$) until exiting $A$, due to the Markovian property of the metric. By \cite[Lemma~3.9]{ss2013continuumcontour}, this implies that $\ol{U}$ is a local set. Similarly, we see that $\eta_1^\delta\big|_{[t_{1,k},\infty)}$ (resp.\ $\eta_2^\delta\big|_{[t_{2,k},\infty)}$) are the flow lines of $h\big|_{\delta\D\setminus\overline{U}}$.
\end{proof}

\begin{lemma}\label{le:intersections_good_box}
Consider the setup described at the beginning of Section~\ref{se:intersections_setup}. There exists $\zeta > 0$ such that the following is true. Let $x,y \in \eta_1 \cap \eta_2$ be the first intersection point outside $B(0,1)$ (resp.\@ $B(0,2)$), and $\eta^{x,y}_1$ (resp.\@ $\eta^{x,y}_2$) the segments of $\eta_1$ (resp.\@ $\eta_2$) from $x$ to $y$. Then
\[
\p\left[ \eta^{x,y}_1,\eta^{x,y}_2 \subseteq [-u,u] \times [u^{-1},u] \right] = 1-O(u^{-\zeta}) \quad \text{as } u\to\infty .
\]
\end{lemma}

\begin{proof}
The probability that $\eta^{x,y}_1$ or $\eta^{x,y}_2$ comes back close to $\R$ is bounded by $u^{-(8/\kappa-2)}$ (see e.g.\ \cite{sz-boundary-proximity}). To bound the probability that $\eta_1,\eta_2$ do not intersect while crossing the annulus $A(0,2,u) \cap \h$, we can e.g.\ apply an independence across scales argument (see e.g.\ \cite[Section~4.1.6]{amy-cle-resampling} for the boundary version) with the events $E_j$ that there is only one angle $\theta_1$ (resp.\ $\theta_2$) flow line crossing the annulus $A(0,2^j,2^{j+1}) \cap \h$ with the correct height (i.e.\ the height of the first such crossing of a flow line from $0$) and they intersect. Then we have that $\p[\bigcap_{j=2,\ldots,k} E_j^c] = O(e^{-\zeta k})$ for some $\zeta>0$.
\end{proof}

To prove Lemma~\ref{le:geometric_exploration}, we prove that the probability that the next exploration step makes it to $i\delta$ is positive. We consider the following setup. Let $D$ be a simply connected domain with $\operatorname{inrad}(D) \le \delta$, and $x_0,y_0 \in \partial D$ distinct. Suppose that $(D,x_0,y_0)$ is $(M,a)$-good within $B(0,3\delta/4)$ as defined in Definition~\ref{def:regularity}. Let $h$ be a GFF in $D$ with boundary values as described in Section~\ref{se:intersections_setup} (and transformed to $D$). Let $\eta_1 $(resp.\ $\eta_2$) be the flow line of angle $\theta_1$ (resp.\ $\theta_2$) from $x_0$ to $y_0$. Let $U$ be the region bounded between $\eta_1,\eta_2$. Let $\intptsapprox{\delta}{\epsilon} \subseteq (\eta_1 \cap \eta_2)^4$ and $U_{x',y'}$ for $(x',x,y,y')\in\intptsapprox{\delta}{\epsilon}$ be defined analogously.

\begin{lemma}\label{le:exploration_step}
There exists $\zeta > 0$ such that the following is true. Fix $M,a > 0$ where $a>0$ is sufficiently small (depending on $\epsexp$). Let $D \subseteq \C$ be a simply connected domain with $\operatorname{inrad}(D) \le \delta$, $x_0,y_0 \in \partial D$ distinct, and assume $(D,x_0,y_0)$ is $(M,a)$-good within $B(0,3\delta/4)$. Then
\[
\p\left[ \sup_{(x',x,y,y')\in\intptsapprox{\delta}{\epsilon}}\metapproxacres{\epsilon}{U_{x',y'}}{x}{y}{\Gamma} \ge M\median{\epsilon}+M\epsilon^{-\epsexp}\ac{\epsilon} ,\,E^\bubble \right] = O(M\delta^{\zeta\epsexp})
\quad\text{as } \delta\searrow 0 .
\]
\end{lemma}

To prove Lemma~\ref{le:exploration_step}, we divide $U$ into scales and bound the distance across each scale.

Let $\varphi \colon D \to \h$ be a conformal transformation with $\varphi(x_0) = 0$ and $\varphi(y_0) = \infty$. For each $k \in \N$, we let $\CZ^{B(0,3\delta/4)}_k$ as defined in~\eqref{eq:bottleneck_def}. By our assumption, $|\CZ^{B(0,3\delta/4)}_k| \leq M2^{ak}$. Note that the condition $\operatorname{inrad}(D) \leq \delta$ implies $\CZ^{B(0,3\delta/4)}_k = \varnothing$ for $k<\log_2(\delta^{-1})$.

Let $\eta_1(t_{j,1}) = \eta_2(t_{j,2})$ be the first intersection point outside $\varphi^{-1}(B(0,2^j))$, and let $U_j$ denote the region bounded between $\eta_1[t_{j,1},t_{j+2,1}], \eta_2[t_{j,2},t_{j+2,2}]$.

Let $B_k = [-2^{\wt{a}k},2^{\wt{a}k}] \times [2^{-\wt{a}k},2^{\wt{a}k}]$ for some $\wt{a}>0$. For $j\in\CZ^{B(0,3\delta/4)}_k$, let $F_j$ denote the event that $\varphi(\eta_i[t_{j,i},t_{j+2,i}]) \subseteq 2^j B_k$ for $i=1,2$, and $F = \bigcap_{j\in\Z} F_j$. Suppose that $a < \wt{a}\zeta/2$ where $\zeta>0$ is the exponent from Lemma~\ref{le:intersections_good_box}. Then
\[
\p[F^c] \lesssim \sum_{k\ge\log_2(\delta^{-1})} |\CZ^{B(0,3\delta/4)}_k|\,2^{-\zeta\wt{a}k} 
\le \sum_{k\ge\log_2(\delta^{-1})} M2^{-(\zeta/2)\wt{a}k} 
\lesssim M\delta^{\wt{a}\zeta/2} .
\]
Further, for $j\in\CZ^{B(0,3\delta/4)}_k$, let
\[
D_j = \left\{ z \in D \ :\  |z-\varphi^{-1}(i2^j)| \le 2^{-k+6\wt{a}k} ,\ \dist(z,\partial D) \ge 2^{-k-4\wt{a}k} \right\}
\]
and note that $\varphi^{-1}(2^j B_k) \subseteq D_j$. In particular, we have $\eta_i[t_{j,i},t_{j+2,i}] \subseteq D_j$ on the event $F_j$.

\begin{lemma}\label{le:exploration_step_one_scale}
There exists $\zeta>0$ such that if $\wt{a} > 0$ is sufficiently small (depending on $\epsexp$), then
\[
\p\left[ \sup_{(x',x,y,y')\in\intptsapprox{\delta}{\epsilon}\cap\ol{U}_j}\metapproxacres{\epsilon}{U_{x',y'}}{x}{y}{\Gamma} \ge 2^{-\zeta k}\median{\epsilon}+2^{-\zeta\epsexp k}\epsilon^{-\epsexp}\ac{\epsilon} ,\, E^\bubble ,\, F_j \right] = O(2^{-\zeta k})
\]
for every $k < \log_2(\epsilon^{-1})$ and $j\in\CZ^{B(0,3\delta/4)}_k$.
\end{lemma}

\begin{proof}
Throughout the proof, we let $c>0$ denote a fixed constant whose value may change from line to line.

Let $E_1$ be the event that there are at most $2^{c\wt{a}k}$ bubbles $V \subseteq U_j$ with $\diam(V) \ge 2^{-k-6\wt{a}k}$. It follows from Lemma~\ref{le:fill_ball} and Koebe's distortion theorem that $\p[E_1^c \cap F_j] = o^\infty(2^{-\wt{a}k})$. On the event $E^\bubble \cap E_1 \cap F_j$, the sum of distances across all such bubbles is at most $2^{c\wt{a}k}(2^{-\zeta k}\median{\epsilon}+2^{-\zeta\epsexp k}\epsilon^{-\epsexp}\ac{\epsilon})$.

Therefore we only need to consider the remaining regions.

Let $E_2$ be the event that for each $i=1,2$, there are $2^{c\wt{a}k}$ time points $t_{j,i} = s_1<s_2<\cdots= t_{j+2,i}$ such that $\diam(\eta_i[s_l,s_{l+1}]) \le 2^{-k-5\wt{a}k}$ for each $l$. It follows from Lemma~\ref{le:variation_one_scale} that $\p[E_2^c \cap F_j] = o^\infty(2^{-\wt{a}k})$.

On the event $E_1 \cap E_2 \cap F_j$, we can divide the remaining regions of $U_j$ (i.e., after removing the bubbles of diameter at least $2^{-k-6\wt{a}k}$) into at most $2^{c\wt{a}k}$ segments of diameter at most $2^{-k-5\wt{a}k}$. Let us enumerate them by $U_{j,l}$, $l=1,\ldots,L$. Each $U_{j,l}$ is bounded between a segment of $\eta_1$ and a segment of $\eta_2$, and contains only bubbles smaller than $2^{-k-6\wt{a}k}$.

Let $G_{z,r}$ be the event from Lemma~\ref{le:good_scales_merging}, and let $E_3$ be the event that for each $z \in 2^{-k-6\wt{a}k}\Z^2 \cap D_j$, the events $G_{z,2^{-j}}$ occur for at least $9/10$ fraction of scales $j \in \{ \lceil k(1+4\wt{a}) \rceil, \ldots, \lfloor k(1+5\wt{a}) \rfloor \}$. By Lemma~\ref{le:good_scales_merging} there exist $\wt{M},\wt{p}$ (depending on $\wt{a}$) such that $\p[E_3^c] = O(e^{-k})$.

Sample $z \in 2^{-k-6\wt{a}k}\Z^2 \cap D_j$ and a scale $J \in \{ \lceil k(1+4\wt{a}) \rceil, \ldots, \lfloor k(1+5\wt{a}) \rfloor \}$ uniformly at random. Let $E^\bad_{z,J}$ be the event that the event from Lemma~\ref{le:a_priori_int_fl} occurs for $\wt{h}_{z,2^{-J}}$ and the metric $\metapprox{\epsilon}{2^{-k/2}\cdot}{2^{-k/2}\cdot}{2^{-k/2}\Gamma}$. By the proof of Lemma~\ref{le:a_priori_disc}, we have
\[\begin{split} 
\p\left[ \sup_{(x',x,y,y')\in\intptsapprox{\delta}{\epsilon}\cap\ol{U}_{j,l}}\metapproxacres{\epsilon}{U_{x',y'}}{x}{y}{\Gamma} \ge \median[2^{-k/2}]{\epsilon} \text{ for some $l$} ,\, E_3 \right] 
&\lesssim 2^{12\wt{a}k} \p[E^\bad_{z,J} \cap G_{z,2^{-J}}] \\
&= O(2^{-k\ddouble/2}) .
\end{split}\]
By Lemma~\ref{le:median_scaling}, we have $\median[2^{-k/2}]{\epsilon} \lesssim 2^{-k\ddouble/2+\wt{a}k}\median{\epsilon}$. Summing over $l$ yields the bound $2^{-k\ddouble/2+c\wt{a}k}\median{\epsilon}$ for the distances across all small bubbles, off an event with probability $O(2^{-k\ddouble/2})$. The remaining parts from $x'$ to $x$ (resp.\ from $y$ to $y'$) consist of bubbles of total size at most $\cserial\epsilon$ plus at most one larger bubble. Since we are considering the metric defined in~\eqref{eq:shortcutted_metric}, we can combine the bounds for the big and for the small bubbles and use the compatibility assumption for the internal metrics to conclude.
\end{proof}

\begin{proof}[Proof of Lemma~\ref{le:exploration_step}]
Fix $\wt{a}>0$ sufficiently small, and suppose that $a < \wt{a}\zeta/2$. Let $F$ be the event defined below the lemma statement, so that $\p[F^c] = O(M\delta^{\wt{a}\zeta/2})$. It therefore suffices to show that
\[
\p\left[ \sup_{(x',x,y,y')\in\intptsapprox{\delta}{\epsilon}}\metapproxacres{\epsilon}{U_{x',y'}}{x}{y}{\Gamma} \ge M\median{\epsilon}+M\epsilon^{-\epsexp}\ac{\epsilon} ,\,E^\bubble ,\, F \right] = O(M\delta^\zeta) .
\]
For each $(x',x,y,y')\in\intptsapprox{\delta}{\epsilon}$, we divide the region $U_{x',y'}$ into the scales $U_{x',y'} \cap U_j$ and apply Lemma~\ref{le:exploration_step_one_scale} to each scale. For $k < \log_2(\epsilon^{-1})$, by Lemma~\ref{le:exploration_step_one_scale}, we have
\[
\p\left[ \sup_{(x'_j,x_j,y_j,y'_j)\in\intptsapprox{\delta}{\epsilon}\cap\ol{U}_j}\metapproxacres{\epsilon}{U_{x'_j,y'_j}}{x_j}{y_j}{\Gamma} \ge 2^{-\zeta k}\median{\epsilon}+2^{-\zeta\epsexp k}\epsilon^{-\epsexp}\ac{\epsilon} ,\, E^\bubble ,\, F_j \right] = O(2^{-\zeta k}) .
\]
Summing in $\log_2(\delta^{-1}) \le k < \log_2(\epsilon^{-1})$ and $j \in \CZ^{B(0,3\delta/4)}_k$ and using the assumption $|\CZ^{B(0,3\delta/4)}_k| \leq M2^{ak}$ yields the union bound
\[
\sum_{k\ge\log_2(\delta^{-1})} |\CZ^{B(0,3\delta/4)}_k|\,2^{-\zeta k} 
\le \sum_{k\ge\log_2(\delta^{-1})} M2^{ak-\zeta k} 
\lesssim M\delta^{\zeta-a} .
\]
The remaining parts from $x'_j$ to $x_j$ (resp.\ from $y_j$ to $y'_j$) consist of bubbles of total size at most $\cserial\epsilon$ plus at most one larger bubble. By the compatibility assumption for the internal metrics and the definition of $\metapproxac{\epsilon}{\cdot}{\cdot}{\Gamma}$ in~\eqref{eq:shortcutted_metric}, we conclude the following. Off an event with probability $O(M\delta^{\zeta-a})$, on the event $E^\bubble$, we have, denoting $x_j$ (resp.\ $y_j$) the first (resp.\ last) point of $\ol{U}_{x,y} \cap \ol{U}_j$,
\[ \begin{split}
&\sup_{(x',x,y,y')\in\intptsapprox{\delta}{\epsilon}} \sum_{\log_2(\delta^{-1}) \le k < \log_2(\epsilon^{-1})} \sum_{j\in\CZ^{B(0,3\delta/4)}_k} \metapproxacres{\epsilon}{U_{x',y'}}{x_j}{y_j}{\Gamma} \\
&\quad \le \sum_{\log_2(\delta^{-1}) \le k < \log_2(\epsilon^{-1})} |\CZ^{B(0,3\delta/4)}_k|\,(2^{-\zeta k}\median{\epsilon}+2^{-\zeta\epsexp k}\epsilon^{-\epsexp}\ac{\epsilon}) \\
&\quad \lesssim M\delta^{\zeta-a}\median{\epsilon}+M\delta^{\zeta\epsexp-a}\epsilon^{-\epsexp}\ac{\epsilon} 
\end{split} \]
where we assume that $a < \zeta\epsexp$.

It remains to consider the scales $k \ge \log_2(\epsilon^{-1})$. (When $\epsilon=0$, this case is not needed, and the conclusion follows from the continuity of $\met{\cdot}{\cdot}{\Gamma}$.) In this case, the distance across is bounded in the exact same way as in the proof of Lemma~\ref{le:close_geodesic_good_eta}. By the exact same proof, it follows from the assumption $|\CZ^{B(0,3\delta/4)}_k| \leq M2^{ak}$ and the definition of $\metapproxac{\epsilon}{\cdot}{\cdot}{\Gamma}$ in~\eqref{eq:shortcutted_metric} that with probability $1-o^\infty(\epsilon)$, if $j_1,j_1+1,\ldots,j_2 \in \bigcup_{k \ge \log_2(\epsilon^{-1})} \CZ^{B(0,3\delta/4)}_k$, then $\metapproxacres{\epsilon}{U_{x',y'}}{x}{y}{\Gamma} \le M\epsilon^{-a-c\wt{a}}\ac{\epsilon}$ for every $x,y \in \eta_1\cap\eta_2[t_{j_1,2},t_{j_2+1,2}]$ on the event $F$. Moreover, we can pick a collection of such $(j_1,j_2)$ so that all of $\bigcup_{k \ge \log_2(\epsilon^{-1})} \CZ^{B(0,3\delta/4)}_k$ is covered in this way and
\[ \sum_{(j_1,j_2)} \sup_{x,y \in \eta_1\cap\eta_2[t_{j_1,2},t_{j_2+1,2}]}\metapproxacres{\epsilon}{U_{x',y'}}{x}{y}{\Gamma} \le M\epsilon^{-a-c\wt{a}}\ac{\epsilon} . \]
\end{proof}

\begin{proof}[Proof of Lemma~\ref{le:geometric_exploration}]
Fix $M,a>0$ where $a$ is sufficiently small (depending on $\epsexp$), and let $G_{M,a}$ denote the event that for every $t_1,t_2$ with $\eta_1^\delta(t_1) = \eta_2^\delta(t_2)$, the domain $\delta\D \setminus (\eta_1^\delta[0,t_1] \cup \eta_2^\delta[0,t_2])$ is $(M,a)$-good within $B(0,3\delta/4)$. By Corollary~\ref{co:regularity_scaled} and Example~\eqref{it:intersections_regularity}, we have $\p[G_{M,a}^c] = o^\infty(\delta)$.

For every $k<N$, the following holds by Lemma~\ref{le:expl_markov} and the Markovian property of the metric. Conditionally on the segments of $\eta_1^\delta,\eta_2^\delta$ until $x_k$, the remainder of the curves and loop ensembles (after mapping to $\h$) have the same law as $\eta_1,\eta_2,\Gamma$ described at the beginning of Section~\ref{se:intersections_setup}, and the internal metrics in the remainder are conditionally independent of the internal metrics we have explored so far. Therefore, by Lemma~\ref{le:exploration_step},
\[
\p[ N>k+1 ,\, E^\bubble ,\, G_{M,a} \mid N>k ] = O(\delta^{\zeta\epsexp}) 
\]
which implies the result.
\end{proof}

We now put the Lemmas~\ref{le:each_bubble_good} and~\ref{le:geometric_exploration} together and prove Proposition~\ref{pr:intersection_exponent_improve}.

\begin{proof}[Proof of Proposition~\ref{pr:intersection_exponent_improve}]
By Lemma~\ref{le:each_bubble_good}, we have $\p[(E^\bubble)^c] = O(\delta^{\bestexp+\alpha'})$, so that it remains to show
\[ \p\left[ \sup_{(x',x,y,y')\in\intptsapprox{\delta}{\epsilon}}\metapproxacres{\epsilon}{U_{x',y'}}{x}{y}{\Gamma_\delta} \ge \median{\epsilon}+\epsilon^{-\epsexp}\ac{\epsilon} ,\, E^\bubble \right] = O(\delta^{\bestexp+\alpha'}) . \]
Pick $n > \bestexp/(\zeta\epsexp)$ where $\zeta>0$ is the exponent in Lemma~\ref{le:geometric_exploration}. Then, with probability $1-O(\delta^{n\zeta\epsexp})$, the exploration takes at most $n$ steps. Recall that we have defined the exploration so that if $(x',x,y,y')\in\intptsapprox{\delta}{\epsilon}$ lie between $x_{k-1}$ and $x_k$, then $\metapproxacres{\epsilon}{U_{x',y'}}{x}{y}{\Gamma_\delta}$ is bounded by $\median{\epsilon}+\epsilon^{-\epsexp}\ac{\epsilon}$ plus the distance across the bubble terminating at $y$. The remaining parts from $x'$ to $x$ (resp.\ from $y$ to $y'$) consist of bubbles of total size at most $\cserial\epsilon$ plus at most one larger bubble. On the event $E^\bubble$ the distance across each of the bubbles is at most $\median{\epsilon}+\epsilon^{-\epsexp}\ac{\epsilon}$. Since we are considering the metric defined in~\eqref{eq:shortcutted_metric}, we see that for general $(x',x,y,y')\in\intptsapprox{\delta}{\epsilon}$, by the compatibility assumption for the internal metrics, we have $\metapproxacres{\epsilon}{U_{x',y'}}{x}{y}{\Gamma_\delta} \le 3n(\median{\epsilon}+\epsilon^{-\epsexp}\ac{\epsilon})$. The result follows by considering a rescaled metric and using Lemma~\ref{le:median_scaling} (with a constant scaling factor depending on $n$).
\end{proof}

\section{Bubble crossing exponent}
\label{se:bubble_exponent}

In this section, we complete the missing step in the proof of Proposition~\ref{pr:intersection_exponent_improve} and ultimately Proposition~\ref{prop:intersection_crossing_exponent}. We consider an individual bubble of the region described in Section~\ref{se:intersection_exponent}.  Recall that the conditional law of a single bubble is explicitly described in Section~\ref{se:cle_sle_markov}. Our main goal is to show that assuming an a priori estimate for the intersection crossing exponent $\bestexp$ from Section~\ref{se:intersection_exponent}, we can find a bubble crossing exponent $\bubbleexp > \bestexp$ such that the $\Fd_\epsilon$-distance across a bubble of Euclidean size $\delta$ will exceed $\median{\epsilon}$ with probability at most $O(\delta^{\bubbleexp})$. This serves as a key step in the inductive proof in Section~\ref{se:intersection_exponent_conclude} which then shows that also the intersection crossing exponent can be improved to a number strictly larger than $\bestexp$.

We define in Section~\ref{se:bubble_setup} the setup for the bubble crossing exponent, and state a few basic properties. The basic idea of the proof is to consider neighborhoods of the left and the right boundary, and show that we have almost an independent chance that the distance across either of the two sides of the bubble is short. The proof will be divided into two main parts; we use two different localization procedures to deal with the regions between the intersections of large loops with the boundary and with the regions between the large loops. In Section~\ref{subsec:flow_line_intersection}, we will show that the regions underneath large loops that intersect the two sides of the bubble are almost independent using a localization procedure with the GFF. In Section~\ref{subsec:map_in} we will prove a useful lemma that allows us to bound with high probability (i.e., $1-O(\delta^{\bestexp+o(1)})$) the distances in regions near one boundary of the bubble. Using a Markovian exploration of the \clekp{} loops (see Section~\ref{se:cle_sle_markov}), we will have an independent trial of bounding distances in the regions between the large loops on the other side. We combine the two steps in Section~\ref{se:bubble_exponent_conclude} to conclude the proof of the bubble exponent.

\subsection{Setup}
\label{se:bubble_setup}

Throughout this section, we assume that the a priori estimate 
\eqref{eq:a_priori_assumption} holds for a given exponent $\bestexp \ge \ddouble$ and $\epsexp > 0$. We recall the metric $\metapproxacres{\epsilon}{V}{\cdot}{\cdot}{\Gamma}$ defined in~\eqref{eq:shortcutted_metric}.

Note that once Proposition~\ref{prop:intersection_crossing_exponent} has been proved, it implies that all results in this section are valid for arbitrarily large exponents $\bestexp > 0$.

\begin{figure}[ht]
\centering
\includegraphics[width=0.45\textwidth]{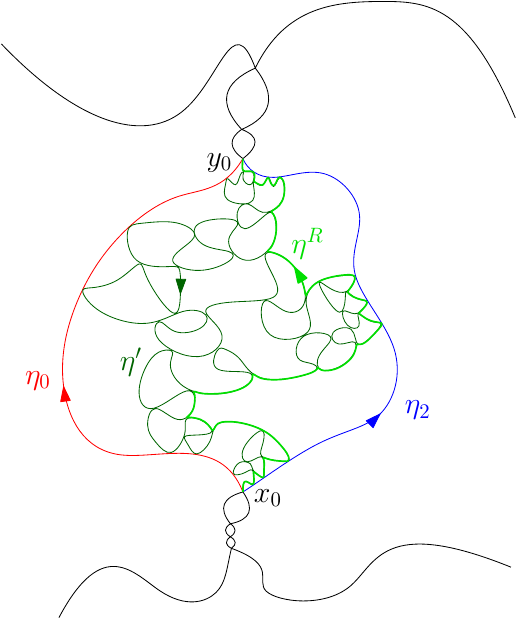}
\caption{The setup of the main result of Section~\ref{se:bubble_exponent}. Note that the right outer boundary $\eta^R$ of $\eta'$ intersects $\eta_0$ if and only if $\kappa' < 6$.}
\end{figure}

Let $\delta>0$, and let $D \subseteq \C$ be a simply connected domain, $x_0,y_0\in\partial D$ with $\dist(x_0,y_0) < \delta$. All results in this section will be uniform in the choice of $(D,x_0,y_0)$ as we let $\delta\searrow 0$.

Let $h_\h$ be a GFF on $\h$ with boundary values given by $-\lambda - \theta_0 \chi$ on~$\R_-$ and $\lambda-\theta_2 \chi$ on~$\R_+$ where
\begin{equation}\label{eq:angles_bubble}
\theta_0 = 	3\pi/2 \quad\text{and}\quad 
\theta_2 = -\pi/2 - \angledouble ,
\end{equation}
and $\angledouble$ is as in~\eqref{eq:angle_double}.

Let $\varphi \colon \h \to D$ be a conformal map that takes $0$ (resp.\ $\infty$) to $x_0$ (resp.\ $y_0$).  Let $h = h_\h \circ \varphi^{-1} - \chi \arg (\varphi^{-1})'$. Let $\eta_0$ (resp.\ $\eta_2$) be the flow line of $h$ from $x_0$ to $y_0$ with angle $\theta_0$ (resp.\ $\theta_2$).  We also let $\eta'$ be the counterflow line of $h$ from $y_0$ to $x_0$.  We note that the conditional law of~$\eta_2$ given~$\eta_0$ is an SLE$_\kappa$ and vice versa, and the conditional law of~$\eta'$ given $\eta_0$, $\eta_2$ is that of an $\SLE_{\kappa'}(\kappa'-6)$ from~$y_0$ to~$x_0$ with its force point located immediately to the counterclockwise side of~$y_0$. Note also that the right outer boundary of $\eta'$ viewed from~$x_0$ to~$y_0$ is given by the flow line of~$h$ with angle $-\pi/2$, and intersects $\eta_2$ with the same angle difference as the flow lines considered in Section~\ref{se:intersections_setup}. Finally, we let $V$ be the connected component of $D \setminus (\eta_0 \cup \eta_2)$ between $\eta_0$, $\eta_2$, and let $\Gamma$ be the $\CLE_{\kappa'}$ that is generated by the restriction of $h$ to $V$ so that $\eta'$ is the branch of the exploration tree associated with $\Gamma$ from $y_0$ to $x_0$. The internal metric $\metapproxres{\epsilon}{V}{\cdot}{\cdot}{\Gamma}$ is defined by an analogous discussion as in Section~\ref{se:intersections_setup}.

Note that the law of $(\eta_0,\eta_2)$ is that of a single bubble of the intersecting flow lines in Section~\ref{se:intersections_setup} in the domain given by their remainder (cf.\ Lemma~\ref{le:law_bubble}).

We state the main result of this section. The results holds on the following regularity event. Recall the Definition~\ref{def:regularity} of $(M,a)$-good domains. For $i=0,2$, $0 \le s < t \le \infty$, and $\CC \subseteq \Gamma$, we let $A_{i,s,t,\CC}$ be the closure of the union of the loops in $\CC$, $\eta_{2-i}$, and $\eta_i([0,s]) \cup \eta_i([t,\infty))$.  Let $D_{i,s,t,\CC}$ be the component of $D \setminus A_{i,s,t,\CC}$ containing $\eta_i([s,t])$. Let $G_{M,a}$ be the event that $(D_{i,s,t,\CC}, \eta_i(s), \eta_i(t))$ is $(M,a)$-good for every $i=0,2$, $s < t$, and $\CC \subseteq \Gamma$. (The argument given below Proposition~\ref{pr:regularity} shows that this is a measurable event.)

\begin{proposition}
\label{prop:bubble_crossing_exponent}
Assume that~\eqref{eq:a_priori_assumption} holds for some $\bestexp \ge \ddouble$ and $\epsexp > 0$. 
There exist $\bubbleexp > \bestexp$ and $\zeta,c>0$ (not depending on $M,a$) such that the following is true. Fix $M,a > 0$ where $a$ is sufficiently small (depending on $\epsexp$). 
Let $D \subseteq \C$ be a simply connected domain. Let $G_{M,a}$ be the event defined above. Let $E_\delta$ denote the event that $\diam(\eta_i) \le \delta$ for $i=0,2$. Then
\[
\p[\metapproxacres{\epsilon}{V}{x_0}{y_0}{\Gamma} \geq \delta^\zeta \median{\epsilon}+\delta^{\zeta}\epsilon^{-\epsexp}\ac{\epsilon} ,\, E_\delta,\ G_{M,a}] = O(M^c \delta^{\bubbleexp})
\quad\text{as } \delta\searrow 0 
\]
uniformly in the choice of the domain and $\epsilon < \delta$.
\end{proposition}

We note that it suffices to consider $\delta > \epsilon$ since in the case $\delta \le \epsilon$ we automatically have $\metapproxac{\epsilon}{\cdot}{\cdot}{\Gamma} \le \ac{\epsilon}$ by the definition~\eqref{eq:shortcutted_metric}.

\begin{remark}\label{rm:bubble_exponent_scaling}
We note that by scaling, it suffices to prove the statement of Proposition~\ref{prop:bubble_crossing_exponent} with
\begin{equation}\label{eq:bubble_exp_unscaled}
\p[\metapproxacres{\epsilon}{V}{x_0}{y_0}{\Gamma} \geq \median{\epsilon}+\epsilon^{-\epsexp}\ac{\epsilon} ,\, E_\delta,\ G_{M,a}] = O(M^{c}\delta^{\bubbleexp}) .
\end{equation}
Indeed, consider the metric $\mettapprox{\delta^{-\alpha}\epsilon}{\cdot}{\cdot}{\Gamma} = \metapprox{\epsilon}{\delta^\alpha\cdot}{\delta^\alpha\cdot}{\delta^\alpha\Gamma}$ where $0 < \alpha < (\bubbleexp-\bestexp)/(4\bubbleexp)$ (cf.\ Lemma~\ref{le:scaled_metric}). Note that $\act{\delta^{-\alpha}\epsilon} = \ac{\epsilon}$, and $\mediant{\delta^{-\alpha}\epsilon} = \median[\delta^\alpha]{\epsilon} \lesssim \delta^{\zeta\alpha}\median{\epsilon}$ for some $\zeta>0$ by Lemma~\ref{le:median_scaling}. Recall from Lemma~\ref{le:good_domain_scaling} that if $D$ is $(M,a)$-good, then $\delta^{-\alpha}D$ is $(O(\delta^{-a\alpha}M),a)$-good. Therefore applying~\eqref{eq:bubble_exp_unscaled} with the rescaled metric and $\wt{\delta} = \delta^{1-\alpha}$ yields (for small $a$)
\[ \begin{split} 
\p[\metapproxacres{\epsilon}{V}{x_0}{y_0}{\Gamma} \geq \delta^{\zeta\alpha}\median{\epsilon}+\delta^{\alpha\epsexp}\epsilon^{-\epsexp}\ac{\epsilon} ,\, E_\delta,\ G_{M,a}] 
&= O((\delta^{-a\alpha}M)^{c}\delta^{(1-\alpha)\bubbleexp}) \\
&= O(M^{c}\delta^{(\bubbleexp+\bestexp)/2}) . 
\end{split} \]
\end{remark}

We note that once Proposition~\ref{prop:intersection_crossing_exponent} has been proved, this implies (for fixed $M$)
\[
\p[\metapproxacres{\epsilon}{V}{x_0}{y_0}{\Gamma} \geq \median{\epsilon}+\epsilon^{-\epsexp}\ac{\epsilon} ,\, E_\delta,\ G_{M,a}] = o^\infty(\delta) .
\]

The main idea of the proof is that the intersections of the outer boundary of a loop $\CL \in \Gamma$ with $\eta_0$ (resp.\ $\eta_2$) look locally like the intersections of the flow lines considered in Section~\ref{se:intersection_exponent}. The assumption~\eqref{eq:a_priori_assumption} implies that the probability that the $\Fd_\epsilon$-distance underneath $\CL$ exceeds $\median{\epsilon}+\epsilon^{-\epsexp}\ac{\epsilon}$ is at most $O(\delta^\bestexp)$. We will show that these events are approximately independent on the two sides of the bubble (if they were perfectly independent, the probability that the distances on both sides exceed $\median{\epsilon}+\epsilon^{-\epsexp}\ac{\epsilon}$ would be $O(\delta^{2\bestexp})$, implying the result). For this, we localize these events and show that we can find disjoint neighborhoods of $\eta_0$ (resp.\ $\eta_2$) in which we can bound the distances \emph{conditionally} on the loop configuration and the internal metrics outside the neighborhood. We will treat separately the regions underneath loops exceeding a certain size (Section~\ref{subsec:flow_line_intersection}) and the regions between such loops (Section~\ref{subsec:map_in}).

We begin by collecting a few variations and consequences of the assumption~\eqref{eq:a_priori_assumption}. We remark that compared to the statement of~\eqref{eq:a_priori_assumption}, here we do not need an ``$+o(1)$'' term in the exponent since we are working in $B(0,\delta^{1+\innexp})$ in the lemmas below.

\begin{lemma}\label{le:a_priori_fl_general_boundary}
Fix $M,\innexp >0$. Let $h_\delta$ be a GFF on $\delta\D$ with some boundary values that are bounded by $M$ and so that the flow lines $\eta^\delta_1$ (resp.\ $\eta^\delta_2$) with angles $\theta_1 = -\pi/2$ (resp.\ $\theta_2 = -\pi/2 - \angledouble$) from $-i\delta$ to $i\delta$ are defined. Given $\eta^\delta_1,\eta^\delta_2$, let $\Gamma$ be a \clekp{} independently in each connected component bounded between $\eta^\delta_1,\eta^\delta_2$. Let $E$ denote the event that there exist $x',x,y,y' \in \eta^\delta_1 \cap \eta^\delta_2$ such that $(x,y) \in \intptsapproxbubble{U_{x',y'}}{\epsilon}$ as in~\eqref{eq:intpts_bubble} and $U_{x',y'} \subseteq B(0,\delta^{1+\innexp})$ and $\metapproxacres{\epsilon}{U_{x',y'}}{x}{y}{\Gamma} \ge \median{\epsilon}+\epsilon^{-\epsexp}\ac{\epsilon}$. Then $\p[E] = O(\delta^{\bestexp})$ as $\delta\searrow 0$.
\end{lemma}

\begin{proof}
This follows from the assumption~\eqref{eq:a_priori_assumption} by the exact same argument as the proof of Lemma~\ref{le:a_priori_disc}. Indeed, let the fields $\wt{h}_{0,2^{-j}}$ (as defined in Section~\ref{se:gff}) have boundary values as in the setup of~\eqref{eq:a_priori_assumption}. Sample $J \in \{ \lceil\log_2(\delta^{-1-\innexp/4})\rceil,\ldots,\lfloor\log_2(\delta^{-1-\innexp/2})\rfloor \}$ uniformly at random. Let $\wt{E}_J$ be the event that the event in~\eqref{eq:a_priori_assumption} occurs for $\wt{h}_{0,2^{-J}}$ (with $\innexp/2$ in place of $\innexp$). We have seen in the proof of Lemma~\ref{le:a_priori_disc} that on a suitable event $F$ with $\p[F^c] = O(\delta^b)$ we have $\p[\wt{E}_J ,\, (0,2^{-J}) \text{ is } M\text{-good for } h ] \gtrsim \delta^a \p[E \cap F]$ where $a>0$ can be chosen arbitrarily small and $b>0$ can be chosen arbitrarily large. We conclude $\p[\wt{E}_J ,\, (0,2^{-J}) \text{ is } M\text{-good for } h ] \le \delta^{(1+\innexp/4)\bestexp+o(1)}$ by~\eqref{eq:a_priori_assumption} and absolute continuity.
\end{proof}

\begin{lemma}\label{le:a_priori_fl_reflected}
Fix $M,\innexp >0$. Let $h_\delta$ be a GFF on $\delta\D$ with some boundary values that are bounded by $M$ and so that the flow line $\eta^\delta_1$ from $-i\delta$ to $i\delta$ and the flow line $\ol{\eta}^\delta_2$ from $i\delta$ to $-i\delta$ in the components of $\delta\D \setminus \eta^\delta_1$ to the right of $\eta^\delta_1$ (and reflected off $\eta^\delta_1$) are defined. Given $\eta^\delta_1,\ol{\eta}^\delta_2$, let $\Gamma$ be a \clekp{} independently in each connected component bounded between $\eta^\delta_1,\ol{\eta}^\delta_2$. Let $E$ denote the event that there exist $x',x,y,y' \in \eta^\delta_1 \cap \ol{\eta}^\delta_2$ such that $(x,y) \in \intptsapproxbubble{U_{x',y'}}{\epsilon}$ as in~\eqref{eq:intpts_bubble} and $U_{x',y'} \subseteq B(0,\delta^{1+\innexp})$ and $\metapproxacres{\epsilon}{U_{x',y'}}{x}{y}{\Gamma} \ge \median{\epsilon}+\epsilon^{-\epsexp}\ac{\epsilon}$. Then $\p[E] = O(\delta^{\bestexp})$ as $\delta\searrow 0$.

The same is true when $\ol{\eta}^\delta_2$ is the (reflected) flow line in the components to the left of $\eta^\delta_1$.
\end{lemma}

\begin{proof}
In case the boundary values of $h_\delta$ are as in Lemma~\ref{le:reflected_fl_law}, this follows from Lemma~\ref{le:a_priori_fl_general_boundary} by reversal symmetry. The result transfers to the case of general boundary values via Lemma~\ref{le:good_scales_merging_refl} and a similar argument as before. As we do not need the general statement right now, we refer to the first part of the proof of Lemma~\ref{le:a_priori_loop_intersecting_one_side} where the argument is spelled out.
\end{proof}

We collect another useful property.

\begin{lemma}
\label{lem:loop_fill_in_ball}
Fix $a > 0$ and let $E$ be the event that for every $r \in (0,\delta)$ and every segment $\ell$ of a loop in $\Gamma$ with $\ell \subseteq B(x_0,1) \cap D$, $\dist(\ell,\partial D) \ge r$, and $\diam(\ell) \ge r$, the segment $\ell$ disconnects a ball of radius $r^{1+a}$ from $\infty$. Then $\p[E^c] = o^\infty(\delta)$ as $\delta \to 0$ (uniformly in the choice of $D,x_0,y_0$).
\end{lemma}

\begin{proof}
When $\ell$ does not intersect $\eta_0,\eta_2$, this follows immediately from Lemma~\ref{le:fill_ball} (the uniformity in the choice of $D,x_0,y_0$ follows from the absolute continuity between the GFFs). When $\ell$ intersects $\eta_0$, since $\eta_0$ has the same angle as the right boundary of $\eta'$, the result follows as well. When $\ell$ intersects $\eta_2$, the result follows by symmetry.
\end{proof}

\subsection{Flow line intersection lemma}
\label{subsec:flow_line_intersection}

\newcommand*{\excexp}{a_3}

The aim of this subsection is to show that it is unlikely that for both $i=0,2$ there exists a large loop of $\Gamma$ that intersects $\eta_i$ such that the $\Fd_\epsilon$-distance under the loop exceeds $\median{\epsilon}+\epsilon^{-\epsexp}\ac{\epsilon}$. We state the main result of this subsection in Section~\ref{subsubsec:independence_statement}. For its proof, we will first explain in Section~\ref{subsubsec:detect_loop} how one can detect intersections of loops of $\Gamma$ with $\eta_0$, $\eta_2$ using GFF flow lines. This will lead us to give the definition of a localized version of the event that we have such an intersection, and we will then obtain upper bounds for the localized event probabilities in Section~\ref{subsubsec:localized_event_probabilities}. We conclude the proof in Section~\ref{subsec:approx_both_sides} using the approximate independence of the localized events.

\subsubsection{Main statement}
\label{subsubsec:independence_statement}

Fix $\excexp > 0$. For $i=0,2$, let $\wh{\Gamma}_i^\delta \subseteq \Gamma$ be the collection of loops $\CL$ such that $\CL \subseteq B(x_0,\delta)$, $\diam(\CL) \ge \delta^{1+\excexp}$, and $\CL$ intersects $\eta_i$.

Let $\CL \in \wh{\Gamma}_i^\delta$. We denote by $a_{i,l}^\CL, b_{i,l}^\CL \in \CL \cap \eta_i$, $l=1,\ldots,L$, the collection of intersection points of $\CL$ with $\eta_i$ such that
\begin{itemize}
\item every intersection point in $\CL \cap \eta_i$ that lies between $a_{i,l}^\CL$ and $b_{i,l}^\CL$ has distance at least $\delta^{1+\excexp}$ to $\eta_{2-i} \cup \partial D$,
\item there is at least one intersection point in $\CL \cap \eta_i$ that lies between $a_{i,l}^\CL$ and $b_{i,l}^\CL$ and has distance at least $2\delta^{1+\excexp}$ to $\eta_{2-i} \cup \partial D$,
\item if $a',b' \in \CL \cap \eta_i$ are two consecutive intersection points between $a_{i,l}^\CL$ and $b_{i,l}^\CL$, then both the segment of the outer boundary of $\CL$ and the segment of $\eta_i$ from $a'$ to $b'$ have diameter at most $\delta^{1+4\excexp}$,
\item the segment from $a_{i,l}^\CL$ to $b_{i,l}^\CL$ is maximal in the sense that it is not strictly contained in another segment with the properties above.
\end{itemize}
For each $l$ we let $U_{i,l}^\CL$ be the union of those components of $D \setminus (\eta_i \cup \CL)$ whose boundary consists of one arc of $\eta_i$ between $a_{i,l}^\CL$ and $b_{i,l}^\CL$ and one arc of the outer boundary of $\CL$. Let $\wt{a}_{i,l}^\CL$ (resp.\ $\wt{b}_{i,l}^\CL$)  in $\CL \cap \eta_i$ be the first (resp.\ last) point such that $(\wt{a}_{i,l}^\CL, \wt{b}_{i,l}^\CL) \in \intptsapproxbubble{U_{i,l}^\CL}{\epsilon}$ as in~\eqref{eq:intpts_bubble}, and
\begin{equation}
\label{eqn:d_i_delta_definition}
D_i^\delta = \sum_{\CL \in \wh{\Gamma}_i^\delta} \sum_l \metapproxacres{\epsilon}{U_{i,l}^\CL}{\wt{a}_{i,l}^\CL}{\wt{b}_{i,l}^\CL}{\Gamma} .
\end{equation}

\begin{proposition}
\label{prop:intersect_loops_large}
There exists $\alpha' > 0$ (not depending on $\excexp,\epsexp$) such that the following is true. Suppose that we have the setup from Section~\ref{se:bubble_setup}, and $\excexp > 0$ is sufficiently small (depending on $\epsexp$). Then
\[ \p[ D_i^\delta \geq \median{\epsilon}+\epsilon^{-\epsexp}\ac{\epsilon} \ \text{for both}\ i=0,2] = O(\delta^{\bestexp + \alpha'}).\]
\end{proposition}

\subsubsection{Detecting intersections using flow lines}
\label{subsubsec:detect_loop}

We begin by explaining how we can localize the event that a loop of $\Gamma$ intersects either $\eta_0$ or $\eta_2$. We will divide the outer boundaries of these loops into smaller segments which we call ``excursions'' from $\eta_0$ (resp.\ $\eta_2$).

Fix $\excexp > 0$. Suppose that we have the setup described in Section~\ref{se:bubble_setup}. Let $\wh{\Gamma}^\delta_0$, $\wh{\Gamma}^\delta_2$ be as defined before Proposition~\ref{prop:intersect_loops_large} (recall that we only consider loops that are contained in $B(x_0,\delta)$). For each $\CL \in \wh{\Gamma}_i^\delta$, $i=0,2$, we define a collection $\CE^\delta_i(\CL)$ of excursions of size $\delta^{1+4\excexp}$ that $\CL$ makes from $\eta_i$ and are at distance at least $\delta^{1+\excexp}$ from $\eta_{2-i}$ and the boundary.

Concretely, suppose that $i = 0,2$ and $\CL \in \wh{\Gamma}_i^\delta$. Let $e$ be a segment of the outer boundary of $\CL$ (between its first and last intersection with $\eta_i$) that starts and ends on $\eta_i$, and let $U_e$ denote the region that is bounded between $e$ and $\eta_i$. Let $x_e\in e\cap\eta_i$ (resp.\ $y_e\in e\cap\eta_i$) denote the first (resp.\ last) point such that $(x_e,y_e) \in \intptsapproxbubble{U_e}{\epsilon}$ as in~\eqref{eq:intpts_bubble}. We say $e \in \CE^\delta_i(\CL)$ if
\begin{itemize}
\item $\diam(U_e) \le \delta^{1+4\excexp}$, and
\item $\dist(U_e, \eta_{2-i} \cup \partial D) \ge \delta^{1+\excexp}$.
\end{itemize}
We let $\CE^{\delta,\bad}_i(\CL)$ be the collection of excursions $e \in \CE^\delta_i(\CL)$ for which $\metapproxacres{\epsilon}{U_e}{x_e}{y_e}{\Gamma} \geq \median{\epsilon}+\epsilon^{-\epsexp}\ac{\epsilon}$.

The following lemmas explain how we can detect excursions by observing the flow lines in a small neighborhood. By this we mean that we can find a pair of flow lines that contain the left and right boundary of $U_e$, respectively, as subsegments. Due to the asymmetric nature of the exploration tree, there is also an asymmetry in the intersections of loops with $\eta_0$ vs.\ $\eta_2$. See Figures~\ref{fig:detect_left} and~\ref{fig:detect_right} for illustrations of the statements.

We point out that in the statement of Lemma~\ref{lem:detect_intersection_left}, the flow lines $\eta^R_{w_1}$ and $\ol{\eta}^R$ need to be reflected off $\eta_0$ (in the sense as defined in Section~\ref{se:fl_merging}). The reason is that the exploration tree of $\Gamma$ is coupled as counterflow lines of the \emph{restriction of} $h$ to the component $V$ between $\eta_0,\eta_2$ which we denote by $h_V$ in the following. In particular, it branches whenever it hits $\eta_0$ or $\eta_2$. For each $z \in V$, if we consider the branch $\eta'_z$ targeting $z$, by the duality the left (resp.\ right) boundary of $\eta'_z$ seen from $z$ are described by the flow lines $\eta^L_z$ (resp.\ $\eta^R_z$) of $h_V$ with angles $\pi/2$ (resp.\ $-\pi/2$) starting from $z$. Since they are flow lines of $h_V$, they are reflected off $\eta_0$ and $\eta_2$. In particular, when $\eta^R_z$ intersects $\eta_0$ as in Figure~\ref{fig:detect_left}, it reflects off in the opposite direction of $\eta_0$. In this situation, if we viewed $\eta^R_z$ in the same direction as $\eta_0$, their angle difference at the intersection point would be $3\pi/2-\theta_0 = 0$. Consequently, if we viewed $\eta^R_z$ as a flow line of $h$ instead of $h_V$, then it would merge with $\eta_0$. Therefore, in order to detect the right boundary of $\eta'_z$, we need to \emph{first} detect $\eta_0$ and then reflect $\eta^R_z$ off $\eta_0$. This issue is not present for the intersections of $\eta^R_z$ with $\eta_2$ in Lemma~\ref{lem:detect_intersection_right} since it reflects in the same direction as $\eta_2$ (see Figure~\ref{fig:detect_right}).

\begin{figure}[ht]
\centering
\includegraphics[width=0.33\textwidth]{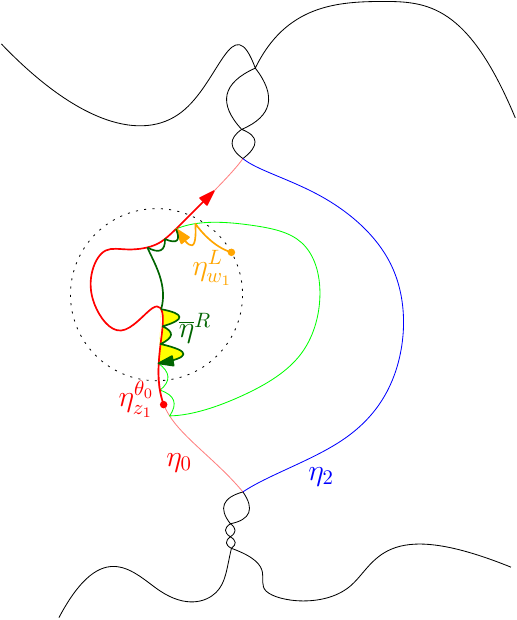}\includegraphics[width=0.33\textwidth]{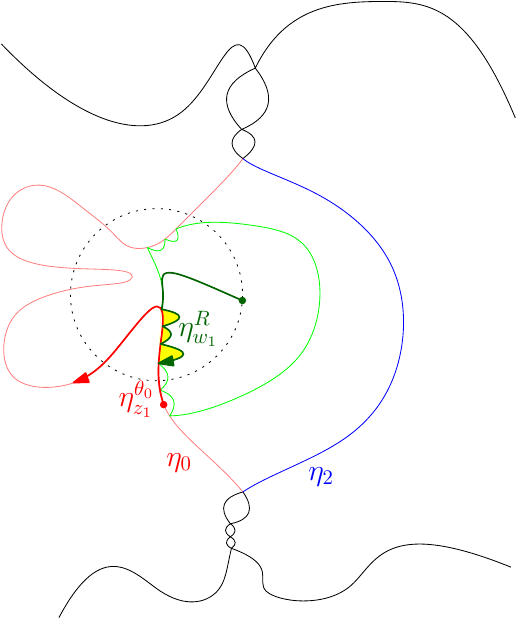}\includegraphics[width=0.33\textwidth]{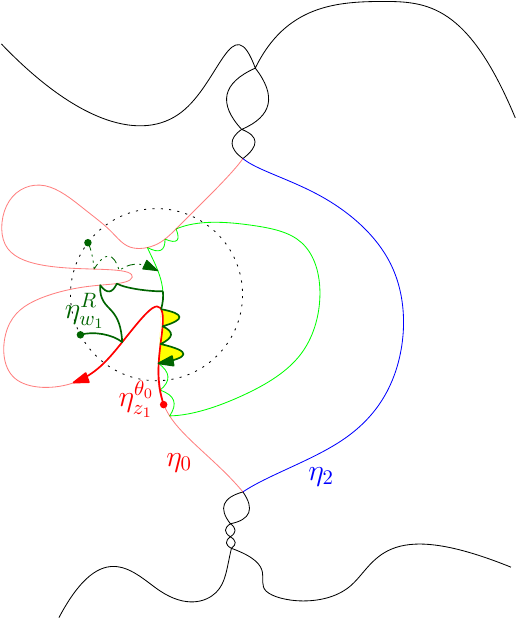}
\caption{The setup described in Lemma~\ref{lem:detect_intersection_left} for the intersections with $\eta_0$. The light green loop represents the outer boundary of $\CL$. Shown are the three cases in the proof of the lemma. In each case, an angle $-\pi/2$ flow line merges into the part of the boundary of $\CL$ that intersects $\eta_0$. In Cases~2a and 2b (center and right), the flow line can come from $\partial B(z,r/4)$. In Case~1 (left), the flow line starts only within $B(z,r/4)$, but its starting point can be detected by an angle $+\pi/2$ flow line coming from $\partial B(z,r/4)$.}
\label{fig:detect_left}
\end{figure}

\begin{figure}[ht]
\centering
\includegraphics[width=0.33\textwidth]{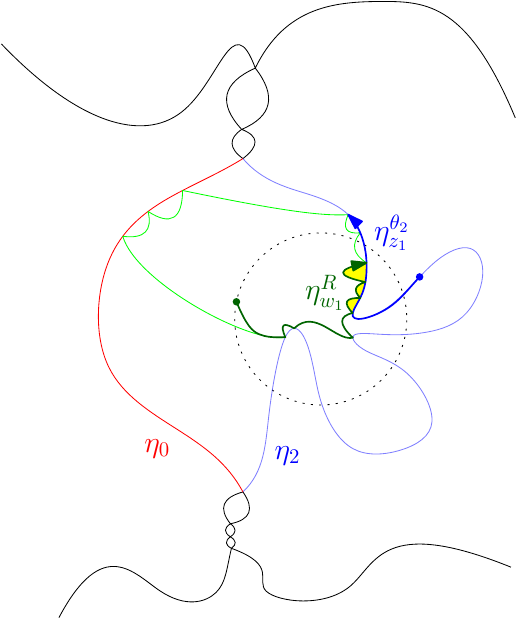}\includegraphics[width=0.33\textwidth]{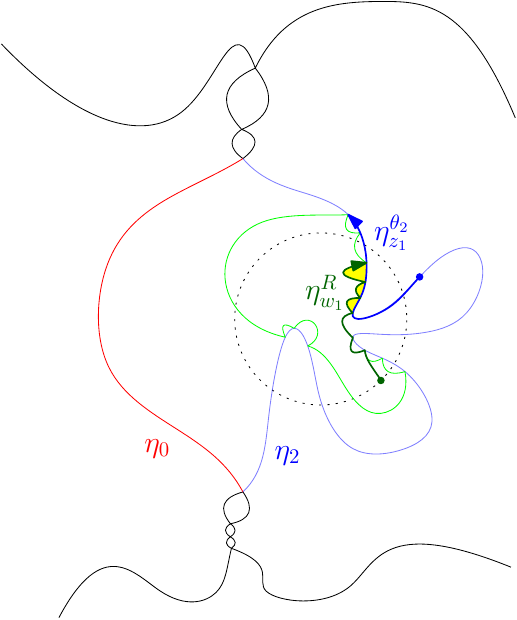}
\includegraphics[width=0.33\textwidth]{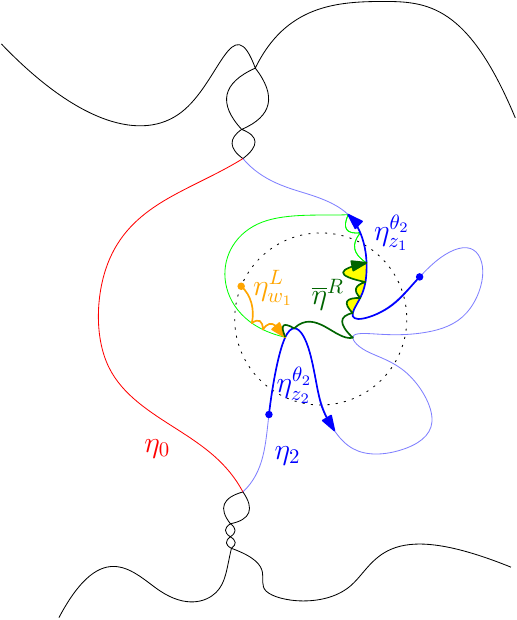}
\caption{The setup described in Lemma~\ref{lem:detect_intersection_right} for the intersections with $\eta_2$. The light green loop represents the outer boundary of $\CL$. Shown are the three cases in the proof of the lemma.}
\label{fig:detect_right}
\end{figure}

\begin{lemma}
\label{lem:detect_intersection_left}
Let $z\in D$ and $r>0$ such that $2r < \delta^{1+\excexp}$ and $B(z,r) \subseteq D$. Let $e \in \CE^\delta_0(\CL)$ for some $\CL \in \wh{\Gamma}_0^\delta$, and suppose that $U_e \subseteq B(z,r/4)$. Then there exist $z_1 \in \partial B(z,r/2)$, $w_1 \in \partial B(z,r/4)$ such that one of the following is true.
\begin{itemize}
\item Let $\eta^{\theta_0}_{z_1}$ be the angle $\theta_0$ flow line of $h$ starting at $z_1$ and stopped upon exiting $B(z,3r/4)$. Let $\eta^R_{w_1}$ be the angle $-\pi/2$ flow line starting at $w_1$, reflected off $\eta^{\theta_0}_{z_1}$ in the opposite direction, and stopped upon exiting $B(z,r/2)$. Then $\partial U_e \subseteq \eta^R_{w_1} \cup \eta^{\theta_0}_{z_1}$.
\item Let $\eta^{\theta_0}_{z_1}$ be the angle $\theta_0$ flow line of $h$ starting at $z_1$ and stopped upon exiting $B(z,3r/4)$. Let $\eta_{w_1}^L$ be the angle $+\pi/2$ flow line starting at $w_1$. Then it hits $\eta^{\theta_0}_{z_1}$ before exiting $B(z,r/2)$. Let $\ol{w}$ denote the first hitting point, and let $\ol{\eta}^R$ be the angle $-\pi/2$ flow line starting at $\ol{w}$, reflected off $\eta^{\theta_0}_{z_1}$ in the opposite direction, and stopped upon exiting $B(z,r/2)$. Then $\partial U_e \subseteq \ol{\eta}^R \cup \eta^{\theta_0}_{z_1}$.
\end{itemize}
\end{lemma}

\begin{proof}
Let $z_1 \in \partial B(z,r/2)$ be the point last visited by $\eta_0$ before intersecting $e$. Then the left boundary of $U_e$ is contained in $\eta^{\theta_0}_{z_1}$. We need to find $w_1$ such that $e$ is contained in either $\eta^R_{w_1}$ or $\ol{\eta}^R$.

Recall that if $\eta'_w$ denotes the branch of the exploration tree of $\Gamma$ from $y_0$ targeted at $w$, then its right (resp.\ left) boundary almost surely agrees with $\eta_w^R$ (resp.\ $\eta_w^L$).

Let $u_\CL \in \CL \cap \eta_0$ be the last intersection point visited by $\eta_0$ (equivalently first visited by $\eta'$). Consider the branch $\eta'_\CL$ of the exploration tree of $\Gamma$ that starts at $y_0$ and traces $\CL$ completely. This is obtained by starting with the branch $\eta'$ from $y_0$ to $x_0$ and branching towards $u_\CL$ when it separates it from $x_0$. Then $e$ is contained in the right boundary of $\eta'_\CL$.

Let $u_e$ be the last point on $e$ visited by $\eta'_\CL$. We distinguish several cases (see Figure~\ref{fig:detect_left}).

\textbf{Case 1:} Both the segments of $\eta_0$ and of $\eta'_\CL$ between $u_e$ and $u_\CL$ are contained in $B(z,3r/4)$. In that case, the segment of $\eta_0$ is also contained in $\eta^{\theta_0}_{z_1}$. Since we assumed $2r < \delta^{1+\excexp} \le \diam(\CL)$, we see that $\eta'_\CL$ must exit $B(z,r/4)$ after first visiting $u_\CL$ (otherwise $\CL$ would be contained in $B(z,3r/4)$). Then if we let $w_1 \in \partial B(z,r/4)$ be the point first visited by $\eta'_\CL$ after $u_\CL$, then $\eta_{w_1}^L$ hits $\eta^{\theta_0}_{z_1}$ at $\ol{w} = u_\CL$, and $e \subseteq \ol{\eta}^R$.

\textbf{Case 2:} The segment of $\eta_0$ between $u_e$ and $u_\CL$ or the segment of $\eta'_\CL$ between $u_e$ and $u_\CL$ exits $B(z,3r/4)$. Then at least one of the following must happen after $\eta'_\CL$ visits $u_e$.\\
(a) $\eta'_\CL$ exits $B(z,r/4)$.\\
(b) $\eta'_\CL$ intersects $\eta_0$ at a point $u \in \eta_0 \setminus \eta^{\theta_0}_{z_1}$ (i.e.\ $\eta_0$ exits $B(z,3r/4)$ between $u_e$ and $u$).\\
We distinguish between whether (a) or (b) occurs first.

\textbf{Case 2a:} Let $w_1 \in \partial B(z,r/4)$ be the first exit point. Then $\eta^R_{w_1}$ intersects $\eta_0$ only on $\eta^{\theta_0}_{z_1}$ until $u_e$, and therefore $e \subseteq \eta^R_{w_1}$.

\textbf{Case 2b:} Consider the branch of $\eta'$ that branches into the component between $\eta_0$ and $\eta'_\CL$ after hitting $u$. Let $\wt{w}_1$ be the first point when it exits $B(z,r/4)$. In case $\eta^R_{\wt{w}_1}$ intersects $\eta_0$ in the opposite direction at a point $\wt{u} \in \eta_0 \setminus \eta^{\theta_0}_{z_1}$, we repeat the procedure by branching into that component after hitting $\wt{u}$. Eventually, we find a point $w_1 \in \partial B(z,r/4)$ such that $\eta^R_{w_1}$ intersects $\eta_0$ only in the same direction or on $\eta^{\theta_0}_{z_1}$ until $u_e$. We do not need to reflect flow lines that intersect in the same direction, therefore $e \subseteq \eta^R_{w_1}$.
\end{proof}

\begin{lemma}
\label{lem:detect_intersection_right}
Let $z\in D$ and $r>0$ such that $2r < \delta^{1+\excexp}$ and $B(z,r) \subseteq D$. Let $e \in \CE^\delta_2(\CL)$ for some $\CL \in \wh{\Gamma}_2^\delta$, and suppose that $U_e \subseteq B(z,r/4)$. Then there exist $z_1,z_2 \in \partial B(z,r/2)$, $w_1 \in \partial B(z,r/4)$ (where $z_1=z_2$ is allowed) such that one of the following is true.
\begin{itemize}
\item Let $\eta^{\theta_2}_{z_1}$ be the angle $\theta_2$ flow line of $h$ starting at $z_1$ and stopped upon exiting $B(z,3r/4)$. Let $\eta^R_{w_1}$ be the angle $-\pi/2$ flow line starting at $w_1$ and stopped upon exiting $B(z,r/2)$. Then $\partial U_e \subseteq \eta^R_{w_1} \cup \eta^{\theta_2}_{z_1}$.
\item Let $\eta^{\theta_2}_{z_1}$ (resp.\ $\eta^{\theta_2}_{z_2}$) be the angle $\theta_2$ flow lines of $h$ starting at $z_1$ (resp.\ $z_2$) and stopped upon exiting $B(z,3r/4)$. Let $\eta_{w_1}^L$ be the angle $+\pi/2$ flow line starting at $w_1$. Then it hits $\eta^{\theta_2}_{z_2}$ before exiting $B(z,r/2)$. Let $\ol{w}$ denote the first hitting point, and let $\ol{\eta}^R$ be the angle $-\pi/2$ flow line starting at $\ol{w}$ and stopped upon exiting $B(z,r/2)$. Then $\partial U_e \subseteq \ol{\eta}^R \cup \eta^{\theta_2}_{z_1}$.
\end{itemize}
\end{lemma}

\begin{proof}
The proof is very similar to the proof of Lemma~\ref{lem:detect_intersection_left}, but due to the opposite orientation of $\eta_0$ and $\eta_2$ (which influences how flow lines need to be reflected) we have to distinguish the cases slightly differently.

We pick $z_1 \in \partial B(z,r/2)$ as before so that the right boundary of $U_e$ is contained in $\eta^{\theta_2}_{z_1}$. We need to find $w_1$ such that $e$ is contained in either $\eta^R_{w_1}$ or $\ol{\eta}^R$.

Consider the branch $\eta'_\CL$ of the exploration tree of $\Gamma$ that starts at $y_0$ and traces $\CL$ completely. Let $u_\CL \in \CL$ be the first (equivalently last) point visited by $\eta'_\CL$. We have $u_\CL \in \eta_0$ in case $\CL$ intersects both $\eta_0,\eta_2$, and $u_\CL \in \eta_2$ in case $\CL$ intersects only $\eta_2$. Let $u_e$ be the last point on $e$ visited by $\eta'_\CL$ (equivalently first visited by $\eta_2$).

We distinguish the following cases (see Figure~\ref{fig:detect_right}).

\textbf{Case 1:} $\CL$ intersects both $\eta_0,\eta_2$. In this case, the right boundary of $\CL$ lies on the branch $\eta'$ from $y_0$ to $x_0$. Letting $w_1$ be the next point where $\eta'$ exits $B(z,r/4)$ after $u_e$, we then have $e \subseteq \eta^R_{w_1}$.

\textbf{Case 2:} $\CL$ intersects only $\eta_2$, and $\eta'_\CL$ exits $B(z,r/4)$ after $u_e$. Let $w_1$ be the first exit point. Then $\eta^R_{w_1}$ intersects $\eta_2$ only in the same direction until $u_e$, and therefore $e \subseteq \eta^R_{w_1}$.

\textbf{Case 3:} $\CL$ intersects only $\eta_2$, and $\eta'_\CL$ stays within $B(z,r/4)$ between $u_e$ and $u_\CL$. Since we assumed $2r < \delta^{1+\excexp} \le \diam(\CL)$, we see that $\eta'_\CL$ must exit $B(z,r/4)$ after first visiting $u_\CL$ and before visiting $e$ (otherwise $\CL$ would be contained in $B(z,r/4)$). It might be that $u_\CL \in \eta_2 \setminus \eta^{\theta_2}_{z_1}$ (i.e.\ $\eta_2$ exits $B(z,3r/4)$ between $u_\CL$ and $u_e$), in which case we pick another $z_2 \in \partial B(z,r/2)$ such that $u_\CL \in \eta^{\theta_2}_{z_2}$. Then if we let $w_1 \in \partial B(z,r/4)$ be the point first visited by $\eta'_\CL$ after $u_\CL$, then $\eta_{w_1}^L$ hits $\eta^{\theta_2}_{z_2}$ at $\ol{w} = u_\CL$, and $e \subseteq \ol{\eta}^R$.
\end{proof}

\subsubsection{Localized event probabilities}
\label{subsubsec:localized_event_probabilities}

In the Lemmas~\ref{lem:interior_intersection_left},~\ref{lem:interior_intersection_left_first},~\ref{lem:interior_intersection_right},~\ref{lem:interior_intersection_right_first}, we are going to define events that we will use to localize the event that a large loop $\CL$ of $\Gamma$ intersects $\eta_0$ or $\eta_2$ such that the regions disconnected by the intersection lead to $\metapproxac{\epsilon}{\cdot}{\cdot}{\Gamma}$-distances exceeding $\median{\epsilon}+\epsilon^{-\epsexp}\ac{\epsilon}$. We will then show that if there is such a loop $\CL$ of $\Gamma$, such events are likely to occur with overwhelming probability. We have seen in the Lemmas~\ref{lem:detect_intersection_left} and~\ref{lem:detect_intersection_right} that there will be an asymmetry in the local versions of the events for the intersections of the loop with $\eta_0$ vs.\ $\eta_2$. Also, we need to distinguish the excursions that are either close or far from the first point where the loop intersects $\eta_i$.

Let $X_\delta^\theta = X_{0,\delta}^\theta$ be as defined in~\eqref{eq:fl_annulus}. The events defined below in Lemmas~\ref{lem:interior_intersection_left},~\ref{lem:interior_intersection_left_first},~\ref{lem:interior_intersection_right},~\ref{lem:interior_intersection_right_first} will be measurable with respect to a collection of $X_\delta^\theta$ with certain values of $\theta$, the values of $h$ on these sets plus the regions they separate from $\partial (\delta\D)$, and the internal metrics in certain regions determined by flow lines within $B(0,3\delta/4)$. They are all measurable with respect to the restriction of $h$ to $B(0,3\delta/4)$ and the internal metrics in regions that are determined by the restriction of $h$ to $B(0,3\delta/4)$. In particular, they do not depend on the values of $h$ outside $B(0,3\delta/4)$.

Note that since $X_\delta^\theta$ separates $0$ from $\partial(\delta\D)$, the values of $h$ in the two complementary components are conditionally independent given its values on $X_\delta^\theta$. In particular, the conditional probability of any event depending only on its values in the exterior component (such as the merging events described in the Lemmas~\ref{le:good_scales_merging},~\ref{le:good_scales_merging_refl},~\ref{le:good_scales_merging_multiple}) is not affected by further conditioning on its values in the interior component. This will be useful later when we try to find good scales on which the localized events described in the Lemmas~\ref{lem:interior_intersection_left},~\ref{lem:interior_intersection_left_first},~\ref{lem:interior_intersection_right},~\ref{lem:interior_intersection_right_first} hold because we can reduce it to finding good events for annuli for which the independence across scales applies.

In order to bound the probability of the event defined in Lemma~\ref{lem:interior_intersection_left_first} below, we use another variant of the assumption~\eqref{eq:a_priori_assumption}. This variant is formulated purely in terms of the CLE, and does not depend on how it is coupled with a GFF. (The proof however does use the coupling with a GFF.)

\begin{lemma}\label{le:a_priori_loop_intersecting_both_sides}
Fix $\innexp > 0$. Consider the setup in Section~\ref{se:bubble_setup} with $(D,x_0,y_0) = (\delta\D,-i\delta,i\delta)$. Let $E$ be the event that there exists a loop $\CL \in \Gamma$ intersecting both $\eta_0$, $\eta_2$ and for $i=0$ or $i=2$ there exist $x',x,y,y' \in \eta_i \cap \CL$ such that if $U$ denotes the region bounded between (the outer boundary of) $\CL$ and the segment of $\eta_i$ from $x'$ to $y'$, then $(x,y) \in \intptsapproxbubble{U}{\epsilon}$ as in~\eqref{eq:intpts_bubble} and $U \subseteq B(0,\delta^{1+\innexp})$ and $\metapproxacres{\epsilon}{U}{x}{y}{\Gamma} \ge \median{\epsilon}+\epsilon^{-\epsexp}\ac{\epsilon}$. Then $\p[E] = O(\delta^{\bestexp})$ as $\delta\searrow 0$.
\end{lemma}

\begin{proof}
This is a direct consequence of Lemma~\ref{le:a_priori_fl_general_boundary}. Indeed, consider the coupling with $h$. Since $\CL$ intersects both $\eta_0$, $\eta_2$, the segment of the right boundary of $\CL$ from $x$ to $y$ is part of the flow line with angle $-\pi/2$, so that (in the case $i=2$) $U$ is part of the region bounded between two flow lines with angle difference $\angledouble$. Similarly, the statement for $i=0$ follows by coupling the CLE with $h$ such that the CLE exploration path is the counterflow line with angle $\theta_2+3\pi/2$, and then the segment of the left boundary of $\CL$ is part of the flow line with angle $\theta_0-\angledouble$. (Note that we are not assuming or using reflection symmetry of $\metapprox{\epsilon}{}{}{}$.)
\end{proof}

\begin{figure}[ht]
\centering
\includegraphics[width=0.45\textwidth]{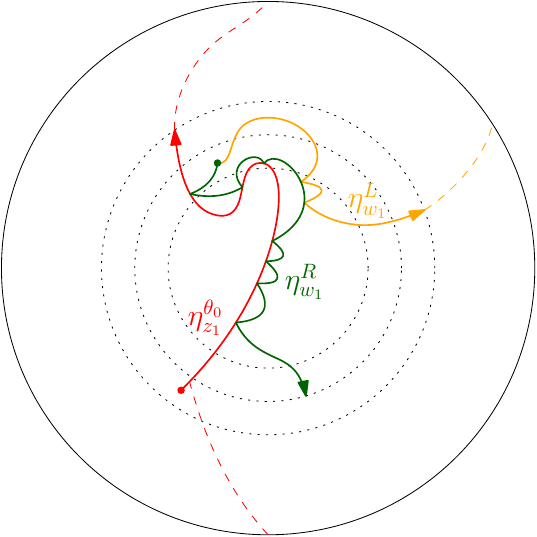}\hspace{0.05\textwidth}\includegraphics[width=0.45\textwidth]{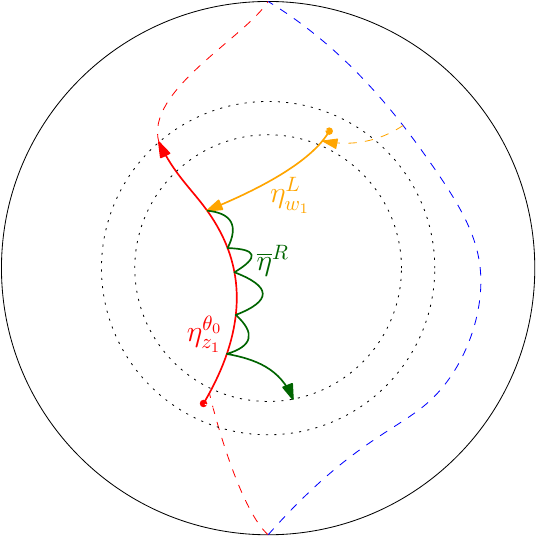}
\caption{Illustration of the setup of Lemma~\ref{lem:interior_intersection_left} (left) and~\ref{lem:interior_intersection_left_first} (right). Note that modulo $2\pi$, the flow lines $\eta_{w_1}^R$ (resp.\ $\ol{\eta}^R$) and $\eta_{z_1}^{\theta_0}$ have the same angle. Rather than merging into $\eta_{z_1}^{\theta_0}$ when hitting it, $\eta_{w_1}^R$ (resp.\ $\ol{\eta}^R$) reflects off $\eta_{z_1}^{\theta_0}$ in the opposite direction. The dashed curves occur on the events described in the last bullet points. In Lemma~\ref{lem:interior_intersection_left} (left), they ensure that $\ol{\eta}_0$ merges into $\eta_{w_1}^R$. In Lemma~\ref{lem:interior_intersection_left_first} (right), they ensure that we draw a loop that also intersects $\eta_2$.}
\label{fi:local_event_left}
\end{figure}

The following lemma deals with the excursions of loops from $\eta_0$ that are away from the first point where the loop intersects $\eta_0$.

\begin{lemma}
\label{lem:interior_intersection_left}
Fix $\innexp >0$, $p \in (0,1)$. Suppose $h$ is a GFF on $\delta\D$ with some bounded boundary values so that the angle $\theta_0$ flow line $\eta_0$ from $-i\delta$ to $i\delta$ and the angle $\theta_0$ flow line $\ol{\eta}_0$ from $i\delta$ to $-i\delta$ in the components of $\delta\D \setminus \eta_0$ to the right of $\eta_0$ (and reflected off $\eta_0$) are defined. Let $z_1 \in \partial B(0,\delta/2)$, $w_1 \in \partial B(0,3\delta/16)$. Let $\eta_{z_1}^{\theta_0}$ be the angle $\theta_0$ flow line of $h$ starting from $z_1$ and stopped upon exiting $B(0,3\delta/4)$. Let $\eta_{w_1}^R$ be the angle $-\pi/2$ flow line starting from $w_1$, reflecting off $\eta_{z_1}^{\theta_0}$ in the opposite direction, and stopped upon exiting $B(0,\delta/4)$. Let $\eta_{w_1}^L$ be the angle $+\pi/2$ flow line starting from $w_1$ and stopped upon exiting $B(0,3\delta/4)$. In case $\eta_{w_1}^R$ and $\eta_{z_1}^{\theta_0}$ intersect, sample the internal metric in the regions bounded between them. 
Let $E^p$ be the event that the following hold.
\begin{itemize}
\item $\eta_{w_1}^R$ intersects the right side of $\eta_{z_1}^{\theta_0}$ with an angle difference of $3\pi/2 - \theta_0 = 0$ (before reflecting off), and they do not intersect in any other way.
\item $\eta_{w_1}^L$ does not intersect $\eta_{z_1}^{\theta_0}$.
\item There exist $x',x,y,y' \in \eta_{z_1}^{\theta_0} \cap \eta_{w_1}^R$ such that $(x,y) \in \intptsapproxbubble{U_{x',y'}}{\epsilon}$ as in~\eqref{eq:intpts_bubble} and $U_{x',y'} \subseteq B(0,\delta^{1+\innexp})$ and $\metapproxacres{\epsilon}{U_{x',y'}}{x}{y}{\Gamma} \geq \median{\epsilon}+\epsilon^{-\epsexp}\ac{\epsilon}$.
\item The following event has conditional probability at least $p$ given $\eta_{z_1}^{\theta_0}$, $\eta_{w_1}^R$, $\eta_{w_1}^L$, $X_\delta^{\theta_0}$, $X_\delta^{+\pi/2}$, and the values of $h$ on these sets:
\begin{itemize}
\item The flow line $\eta_0$ merges into $\eta_{z_1}^{\theta_0}$ before entering $B(0,\delta/4)$.
\item The extension of $\eta_{w_1}^L$ does not intersect $\eta_0$ until it hits $\partial(\delta\D)$.
\end{itemize}
\end{itemize}
Then $\p[E^p] = O(\delta^{\bestexp})$ as $\delta \to 0$.
\end{lemma}

\begin{proof}
Let $F$ denote the event described in Lemma~\ref{le:a_priori_fl_reflected}. We claim that we can couple the internal metrics so that $\p[F \mid E^p] \ge p$, which then implies
\[
\p[E^p] \le \frac{1}{p}\p[F] = O(\delta^{\bestexp}) .
\]
Suppose that we are on the event $E^p$ and also the event described in the last bullet point occurs. Then we see that $\ol{\eta}_0$ necessarily merges into $\eta_{w_1}^R$ before $\eta_{w_1}^R$ intersects $\eta_0$. In particular, $U$ is part of the region bounded between $\eta_0$ and $\ol{\eta}_0$.

Finally, note that we can sample the metric in the region between $\eta_0$, $\ol{\eta}_0$ by first sampling the internal metric in the regions bounded between $\eta_{w_1}^R$, $\eta_{z_1}^{\theta_0}$ (which depend only on the values of $h$ in $B(0,3\delta/4)$ due to the Markovian property), and then sample the remainder of the metric according to its conditional law given the components that match the ones bounded between $\eta_{w_1}^R$, $\eta_{z_1}^{\theta_0}$. Then, if we are on the event $E^p$ and also the event described in the last bullet point occurs, then also $F$ occurs.
\end{proof}

The following lemma deals with the excursions of loops from $\eta_0$ that are near the first point where the loop intersects $\eta_0$.

\begin{lemma}
\label{lem:interior_intersection_left_first}
Fix $\innexp >0$, $p \in (0,1)$. Suppose that we have the setup of Section~\ref{se:bubble_setup} with $(D,x_0,y_0) = (\delta\D,-i\delta,i\delta)$. Let $z_1,w_1 \in \partial B(0,\delta/2)$. Let $\eta_{z_1}^{\theta_0}$ be the angle $\theta_0$ flow line starting from $z_1$, stopped upon exiting $B(0,3\delta/4)$. Let $\eta_{w_1}^L$ be the angle $+\pi/2$ flow line starting from $w_1$, stopped upon hitting $\eta_{z_1}^{\theta_0}$ or exiting $B(0,3\delta/4)$ (whichever occurs first). In case $\eta_{w_1}^L$ intersects $\eta_{z_1}^{\theta_0}$ in $B(0,\delta/4)$, let $\ol{w}$ denote the intersection point, and let $\ol{\eta}^R$ be the angle $-\pi/2$ flow line starting from $\ol{w}$, reflecting off $\eta_{z_1}^{\theta_0}$ in the opposite direction, and stopped upon exiting $B(0,\delta/4)$. Sample the internal metric in the regions bounded between $\ol{\eta}^R$ and $\eta_{z_1}^{\theta_0}$. 
Let $E^p$ be the event that the following hold.
\begin{itemize}
\item $\eta_{w_1}^L$ intersects the right side of $\eta_{z_1}^{\theta_0}$ with an angle difference of $\pi/2 - \theta_0$, and they do not intersect in any other way.
\item There exist $x',x,y,y' \in \eta_{z_1}^{\theta_0} \cap \ol{\eta}^R$ such that $(x,y) \in \intptsapproxbubble{U_{x',y'}}{\epsilon}$ as in~\eqref{eq:intpts_bubble} and $U_{x',y'} \subseteq B(0,\delta^{1+\innexp})$ and $\metapproxacres{\epsilon}{U_{x',y'}}{x}{y}{\Gamma} \geq \median{\epsilon}+\epsilon^{-\epsexp}\ac{\epsilon}$.
\item The following event has conditional probability at least $p$ given $\eta_{z_1}^{\theta_0}$, $\eta_{w_1}^L$, $\ol{\eta}^R$, $X_\delta^{\theta_0}$, $X_\delta^{+\pi/2}$, and the values of $h$ on these sets:
\begin{itemize}
\item The flow line $\eta_0$ merges into $\eta_{z_1}^{\theta_0}$ before entering $B(0,\delta/4)$.
\item There exists a flow line with angle $+\pi/2$ starting from a point on $\eta_2$ that merges into $\eta_{w_1}^L$ before intersecting $\eta_0$.
\end{itemize}
\end{itemize}
Then $\p[E^p] = O(\delta^{\bestexp})$ as $\delta \to 0$.
\end{lemma}

\begin{proof}
Let $F$ denote the event described in Lemma~\ref{le:a_priori_loop_intersecting_both_sides}. Similarly as in the proof of Lemma~\ref{lem:interior_intersection_left}, we will couple the internal metrics so that $\p[F \mid E^p] \ge p$, which then implies
\[
\p[E^p] \le \frac{1}{p}\p[F] = O(\delta^{\bestexp}) .
\]
Suppose we are on the event $E^p$ and also the event described in the last bullet point occurs. The fact that an angle $+\pi/2$ flow line from a point on $\eta_2$ hits $\eta_1$ at $\ol{w}$ means that the counterflow line $\eta'$ intersects $\eta_2$ after visiting $\ol{w}$ and before intersecting $\eta_0$ again. Therefore, if we consider the branch $\eta'_{\ol{w}}$ of the exploration tree of $\Gamma$ that targets back towards $\ol{w}$ after its first visit, it traces a loop $\CL$ starting at $\ol{w}$ and intersecting both $\eta_0$, $\eta_2$. The right boundary of $\eta'_{\ol{w}}$ is given by (the extension of) $\ol{\eta}^R$. Hence, $U_{x',y'}$ is part of the region bounded between $\eta_0$ and $\CL$.
\end{proof}

\begin{figure}[ht]
\centering
\includegraphics[width=0.45\textwidth]{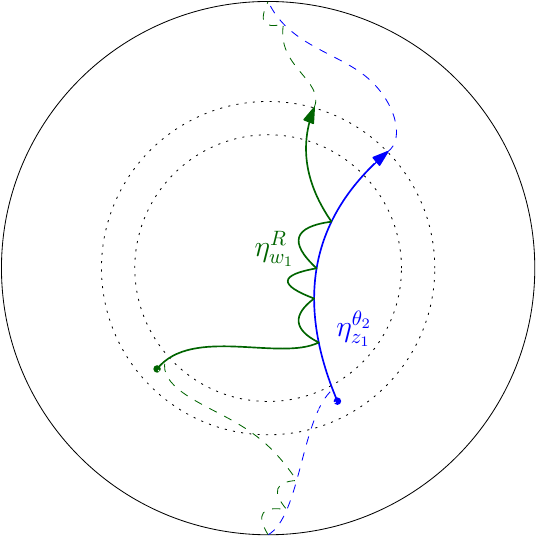}\hspace{0.05\textwidth}\includegraphics[width=0.45\textwidth]{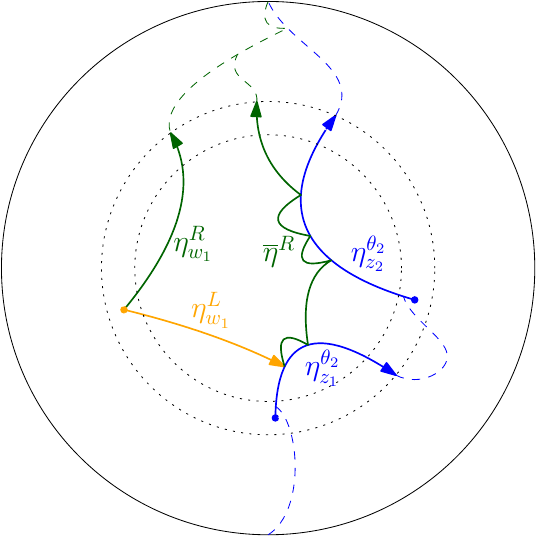}
\caption{Illustration of the setup of Lemma~\ref{lem:interior_intersection_right} (left) and~\ref{lem:interior_intersection_right_first} (right). The dashed curves occur on the events described in the last bullet points. In Lemma~\ref{lem:interior_intersection_right_first} (right), the behavior of the flow lines $\eta_{w_1}^R$, $\ol{\eta}^R$ ensure that we draw a loop that exits $B(0,3\delta/4)$.}
\end{figure}

We now consider the excursions of loops intersecting $\eta_2$. In order to bound the probability of the event defined in Lemma~\ref{lem:interior_intersection_right_first} below, we use another variant of Lemma~\ref{le:a_priori_loop_intersecting_both_sides} for intersections with a large loop that does not necessarily intersect both sides. The proof of it uses the Lemmas~\ref{lem:interior_intersection_left} and~\ref{lem:interior_intersection_left_first}.

\begin{lemma}\label{le:a_priori_loop_intersecting_one_side}
Fix $\innexp > 0$. Consider the setup in Section~\ref{se:bubble_setup} with $(D,x_0,y_0) = (\delta\D,-i\delta,i\delta)$. Let $E$ be the event that there exists a loop $\CL \in \Gamma$ with $\diam(\CL) \ge \delta^{1+\innexp/2}$ and for $i=0$ or $i=2$ there exist $x',x,y,y' \in \eta_i \cap \CL$ such that if $U$ denotes the region bounded between (the outer boundary of) $\CL$ and the segment of $\eta_i$ from $x'$ to $y'$, then $(x,y) \in \intptsapproxbubble{U}{\epsilon}$ as in~\eqref{eq:intpts_bubble} and $U \subseteq B(0,\delta^{1+\innexp})$ and $\metapproxacres{\epsilon}{U}{x}{y}{\Gamma} \ge \median{\epsilon}+\epsilon^{-\epsexp}\ac{\epsilon}$. Then $\p[E] = O(\delta^{\bestexp})$ as $\delta\searrow 0$.
\end{lemma}

\begin{proof}
We consider the setup in Section~\ref{se:bubble_setup} with $(D,x_0,y_0) = (\delta\D,-i\delta,i\delta)$. It suffices to argue for the case $i=0$. The case $i=2$ follows by a symmetrical argument where we consider a CLE generated by the counterflow line with angle $\theta_2+3\pi/2$ instead of $0$.

Recall the notation introduced in Section~\ref{se:gff}. Let $E^1_j$ (resp.\ $E^2_j$) denote the event that the event described in Lemma~\ref{lem:interior_intersection_left} (resp.~\ref{lem:interior_intersection_left_first}) occurs for $\wt{h}_{0,2^{-J}}$ and some choice of $z_1,w_1$ (when the internal metrics are coupled so that they agree in $U$, which can be done in the analogous way as in the proof of Lemma~\ref{lem:interior_intersection_left}).

Let $J \in \{ \lceil\log_2(\delta^{-1-\innexp/2})\rceil,\ldots,\lfloor\log_2(\delta^{-1-3\innexp/4})\rfloor \}$ be sampled uniformly at random. We argue that for some $p>0$ the conditional probability given $E$ that either $E^1_J$ or $E^2_J$ occurs is at least $1/2$.

Recall that $U$ can be detected as described in Lemma~\ref{lem:detect_intersection_left} (and note that the assumption $\dist(U_e, \eta_{2}) \ge \delta^{1+\excexp}$ in the definition of $\CE^\delta_0(\CL)$ was not used in the proof of the lemma). Depending on which case of the lemma statement occurs, we find $z_1,w_1$ so that the first two items in either $E^1_j$ or $E^2_j$ are satisfied for each $j$ (i.e.\ except for the condition on the merging probabilities). It remains to argue that we can also realize (with high probability) the condition on the positive merging probabilities. We argue for the two cases separately.

Consider the case of $E^1_J$. Let $G^1_r$ be the event defined in Lemma~\ref{le:good_scales_merging_refl}, and let $G^1$ be the event that at least $3/4$ fraction of $G^1_{2^{-j}}$, $j \in \{ \lceil\log_2(\delta^{-1-\innexp/2})\rceil,\ldots,\lfloor\log_2(\delta^{-1-3\innexp/4})\rfloor \}$ occur. By Lemma~\ref{le:good_scales_merging_refl}, we can find for any $b>0$ some $M,p>0$ such that $\p[(G^1)^c] = O(\delta^b)$.

Suppose $J$ is sampled so that $G^1_{2^{-J}}$ occurs. We claim that in the former case of Lemma~\ref{lem:detect_intersection_left} the event $E^1_J$ occurs. Consider the last time $t_1$ before getting to $U$ when $\eta_0$ enters the regions that are separated from $\infty$ by $X^{\theta_0}_{2^{-J}} \cup X^{+\pi/2}_{2^{-J}}$. We can choose $z_1\in \partial B(0,2^{-J}/2)$ so that $\eta_{z_1}^{\theta_0}$ is the first strand of $X^{\theta_0}_{2^{-J}}$ that $\eta_0$ traces after $t_1$. Since there is an angle $\theta_0$ flow line from the outside of the annulus merging into $\eta_{z_1}^{\theta_0}$, the conditional probability given the values of $\wt{h}_{z,r}$ on $X^{\theta_0}_{2^{-J}} \cup X^{+\pi/2}_{2^{-J}}$ that an angle $\theta_0$ flow line from $-i2^{-J}$ merges into $\eta_{z_1}^{\theta_0}$ is almost surely positive. Similarly, we can pick $w_1 \in \partial B(0,(3/16)2^{-J})$ such that $\partial U \subseteq \eta^R_{w_1} \cup \eta^{\theta_0}_{z_1}$ and $\eta^L_{w_1}$ does not intersect $\eta_{z_1}^{\theta_0}$. Since the flow lines $\eta^L_{w_1}$, $\eta^R_{w_1}$, $\eta_{z_1}^{\theta_0}$ intersect with the correct height difference, the conditional probability that the extensions of $\eta^L_{w_1}$ and $\eta_{z_1}^{\theta_0}$ get to $\partial B(0,2^{-J})$ without intersecting is positive. Since we are on the event $G^1_{2^{-J}}$, the probability is at least $p$. This shows the claim.

On the event $G^1_{2^{-J}}$ there are at most $M$ different choices for $z_1,w_1$. Therefore taking a union bound and using absolute continuity (together with Lemma~\ref{le:abs_cont_kernel}) we get $\p[E^1_J \cap G^1_{2^{-J}}] = O(\delta^{(1+\innexp/2)\bestexp+o(1)})$ from Lemma~\ref{lem:interior_intersection_left}.

We now argue for the case of $E^2_J$. Let $\wt{\eta}_0,\wt{\eta}_2$ denote the corresponding flow lines of $\wt{h}_{0,2^{-J}}$. Let $G^2_r$ be the event that the following hold.
\begin{itemize}
\item The scale $(0,r)$ is $M$-good for $h$.
\item The number of points in $\partial A(0,r/4,3r/4) \cap X^{\theta_0}_{r}$ and $\partial A(0,r/4,3r/4) \cap X^{+\pi/2}_{r}$ is at most $M$.
\item For each pair of strands of $X^{\theta_0}_{r}$, $X^{+\pi/2}_{r}$, respectively, that end on $\partial B(0,r/4)$, the conditional probability given $X^{\theta_0}_{r}$, $X^{+\pi/2}_{r}$, and the values of $\wt{h}_{0,r}$ on these sets is either $0$ or at least $p$ that
\begin{itemize}
\item the flow line $\wt{\eta}_0$ merges into the former strand before entering $B(0,r/4)$,
\item there is an angle $+\pi/2$ flow line starting from a point on $\wt{\eta}_2$ that merges into the latter strand before intersecting $\wt{\eta}_0$ or entering $B(0,r/4)$.
\end{itemize}
\end{itemize}
Let $G^2$ be the event that at least $3/4$ fraction of $G^2_{2^{-j}}$, $j \in \{ \lceil\log_2(\delta^{-1-\innexp/2})\rceil,\ldots,\lfloor\log_2(\delta^{-1-3\innexp/4})\rfloor \}$ occur. By Lemma~\ref{lem:good_scales_for_event} (and using the fact that almost surely only finitely many strands end on different points of $\partial A(0,r/4,3r/4)$), we can find for any $b>0$ some $M,p>0$ such that $\p[(G^2)^c] = O(\delta^b)$.

Then, arguing in the same way as for $E^1_J$, we see that on $G^1_{2^{-J}}$, in the latter case of Lemma~\ref{lem:detect_intersection_left} also $E^2_J$ occurs. Again, taking a union bound and using absolute continuity we get $\p[E^2_J \cap G^2_{2^{-J}}] = O(\delta^{(1+\innexp/2)\bestexp+o(1)})$ from Lemma~\ref{lem:interior_intersection_left_first}.

Altogether, we conclude that $\p[E] \le \p[E \cap G^1 \cap G^2] + O(\delta^b) \lesssim \p[E^1_J \cap G^1_{2^{-J}}] + \p[E^2_J \cap G^2_{2^{-J}}] + O(\delta^b) = O(\delta^\bestexp)$.
\end{proof}

The following lemma deals with the excursions of loops from $\eta_2$ that are away from the first point where the loop intersects $\eta_2$.

\begin{lemma}
\label{lem:interior_intersection_right}
Fix $\innexp >0$, $p \in (0,1)$. Suppose that we have the setup of Section~\ref{se:bubble_setup} with $(D,x_0,y_0) = (\delta\D,-i\delta,i\delta)$. Let $z_1,w_1 \in \partial B(0,\delta/2)$. Let $\eta_{z_1}^{\theta_2}$ be the angle $\theta_2$ flow line starting from $z_1$, stopped upon exiting $B(0,3\delta/4)$. Let $\eta_{w_1}^R$ be the angle $-\pi/2$ flow line starting from $w_1$, stopped upon exiting $B(0,3\delta/4)$. In case $\eta_{w_1}^R$ and $\eta_{z_1}^{\theta_2}$ intersect, sample the internal metric in the regions bounded between them. 
Let $E^p$ be the event that the following hold.
\begin{itemize}
\item The right side of $\eta_{w_1}^R$ and the left side of $\eta_{z_1}^{\theta_2}$ intersect with an angle difference of $-\pi/2 - \theta_2$, and they do not intersect in any other way.
\item There exist $x',x,y,y' \in \eta_{z_1}^{\theta_2} \cap \eta_{w_1}^R$ such that $(x,y) \in \intptsapproxbubble{U_{x',y'}}{\epsilon}$ as in~\eqref{eq:intpts_bubble} and $U_{x',y'} \subseteq B(0,\delta^{1+\innexp})$ and $\metapproxacres{\epsilon}{U_{x',y'}}{x}{y}{\Gamma} \geq \median{\epsilon}+\epsilon^{-\epsexp}\ac{\epsilon}$.
\item The following event has conditional probability at least $p$ given $\eta_{z_1}^{\theta_2}$, $\eta_{w_1}^R$, $X_\delta^{\theta_2}$, $X_\delta^{-\pi/2}$, and the values of $h$ on these sets:
\begin{itemize}
\item The flow line $\eta_2$ merges into $\eta_{z_1}^{\theta_2}$ before entering $B(0,\delta/4)$.
\item The angle $-\pi/2$ flow line of $h$ from $-i\delta$ to $i\delta$ merges into $\eta_{w_1}^R$ before entering $B(0,\delta/4)$.
\end{itemize}
\end{itemize}
Then $\p[E^p] = O(\delta^{\bestexp})$ as $\delta \to 0$.
\end{lemma}

\begin{proof}
Let $F$ denote the event described in Lemma~\ref{le:a_priori_fl_general_boundary}. As in the proof of Lemma~\ref{lem:interior_intersection_left}, we can couple the internal metrics so that $\p[F \mid E^p] \ge p$, which then implies
\[
\p[E^p] \le \frac{1}{p}\p[F] = O(\delta^{\bestexp}) .
\]
\end{proof}

The following lemma deals with the excursions of loops from $\eta_2$ that are near the first point where the loop intersects $\eta_2$.

\begin{lemma}
\label{lem:interior_intersection_right_first}
Fix $\innexp >0$, $p \in (0,1)$. Suppose that we have the setup of Section~\ref{se:bubble_setup} with $(D,x_0,y_0) = (\delta\D,-i\delta,i\delta)$. Let $z_1,z_2,w_1 \in \partial B(0,\delta/2)$ (where $z_1=z_2$ is allowed). Let $\eta_{z_1}^{\theta_2}$ (resp.\ $\eta_{z_2}^{\theta_2}$) be the angle $\theta_2$ flow lines starting from $z_1$ (resp.\ $z_2$), stopped upon exiting $B(0,3\delta/4)$. Let $\eta_{w_1}^R$ be the angle $-\pi/2$ flow line starting from $w_1$, stopped upon exiting $B(0,3\delta/4)$. Let $\eta_{w_1}^L$ be the angle $+\pi/2$ flow line starting from $w_1$, stopped upon hitting $\eta_{z_1}^{\theta_2}$ or exiting $B(0,3\delta/4)$ (whichever occurs first). In case $\eta_{w_1}^L$ intersects $\eta_{z_1}^{\theta_2}$ in $B(0,\delta/4)$, let $\ol{w}$ denote the intersection point, and let $\ol{\eta}^R$ be the angle $-\pi/2$ flow line starting from $\ol{w}$, stopped upon exiting $B(0,\delta/4)$. In case $\ol{\eta}^R$ and $\eta_{z_2}^{\theta_2}$ intersect, sample the internal metric in the regions bounded between them. 
Let $E^p$ be the event that the following hold.
\begin{itemize}
\item $\eta_{w_1}^L$ intersects the left side of $\eta_{z_1}^{\theta_2}$ with an angle difference of $-3\pi/2 - \theta_2$, and they do not intersect in any other way.
\item The right side of $\ol{\eta}^R$ and the left side of $\eta_{z_2}^{\theta_2}$ intersect with an angle difference of $-\pi/2 - \theta_2$, and they do not intersect in any other way.
\item There exist $x',x,y,y' \in \eta_{z_2}^{\theta_2} \cap \ol{\eta}^R$ such that $(x,y) \in \intptsapproxbubble{U_{x',y'}}{\epsilon}$ as in~\eqref{eq:intpts_bubble} and $U_{x',y'} \subseteq B(0,\delta^{1+\innexp})$ and $\metapproxacres{\epsilon}{U_{x',y'}}{x}{y}{\Gamma} \geq \median{\epsilon}+\epsilon^{-\epsexp}\ac{\epsilon}$.
\item The following event has conditional probability at least $p$ given $\eta_{z_1}^{\theta_2}$, $\eta_{z_2}^{\theta_2}$, $\eta_{w_1}^R$, $\eta_{w_1}^L$, $\ol{\eta}^R$, $X_\delta^{\theta_2}$, $X_\delta^{-\pi/2}$, $X_\delta^{+\pi/2}$, and the values of $h$ on these sets:
\begin{itemize}
\item The flow line $\eta_2$ merges into $\eta_{z_1}^{\theta_2}$ before entering $B(0,\delta/4)$, and then (in case $\eta_{z_1}^{\theta_2}, \eta_{z_2}^{\theta_2}$ do not merge) into $\eta_{z_2}^{\theta_2}$ before entering $B(0,\delta/4)$ the next time.
\item The extension of $\ol{\eta}^R$ merges into the right side of the extension of $\eta_{w_1}^R$ after both exit $B(0,3\delta/4)$.
\end{itemize}
\end{itemize}
Then $\p[E^p] = O(\delta^{\bestexp})$ as $\delta \to 0$.
\end{lemma}

\begin{proof}
Let $F$ denote the event described in Lemma~\ref{le:a_priori_loop_intersecting_one_side}. Similarly as in the proof of Lemma~\ref{lem:interior_intersection_left}, we will couple the internal metrics so that $\p[F \mid E^p] \ge p$, which then implies
\[
\p[E^p] \le \frac{1}{p}\p[F] = O(\delta^{\bestexp}) .
\]
Suppose we are on the event $E^p$ and also the event described in the last bullet point occurs. Let $\eta'_{\ol{w}}$ be the branch of the exploration tree of $\Gamma$ that targets back towards $\ol{w}$ after its first visit. It traces a loop $\CL$ clockwise starting at $\ol{w}$. The segment of (the extension of) $\ol{\eta}^R$ until it merges into the right side of the extension of $\eta_{w_1}^R$ is visited by $\eta'_{\ol{w}}$ after $\ol{w}$, hence is part of the outer boundary of $\CL$. In particular, $U_{x',y'}$ is part of the region bounded between $\eta_2$ and $\CL$. Since $\ol{\eta}^R$, $\eta_{w_1}^R$ merge only outside of $B(0,3\delta/4)$, we see that $\CL$ exits $B(0,3\delta/4)$.
\end{proof}

\subsubsection{Approximate independence of bad large loops hitting both sides}
\label{subsec:approx_both_sides}

We are now going to show that if we have a large loop $\CL$ of $\wh{\Gamma}_0^\delta$ (resp.\ $\wh{\Gamma}_2^\delta$) as defined before Proposition~\ref{prop:intersect_loops_large}, then we can detect the intersection with overwhelming probability using the events described in Section~\ref{subsubsec:detect_loop}. Recall that we are only considering loops that are contained in $B(x_0,\delta)$.

Throughout the section we will denote
\[
D_{\excexp} = \{ z \in D : \dist(z,\partial D) > \delta^{1+\excexp} \} \cap B(x_0,\delta) 
\]
and
\[ \CD_r = D_{\excexp} \cap r\Z^2 \quad\text{for } 0<r<\delta^{1+\excexp}. \]

Recall the notation introduced in Section~\ref{se:gff}. Here we let the fields $\wt{h}_{z,r}$ on $B(z,r)$ have boundary values as the fields described in Section~\ref{subsubsec:localized_event_probabilities}.

Let $E^{0;p}_{z,r}$ be the event that either the event from Lemma~\ref{lem:interior_intersection_left} or Lemma~\ref{lem:interior_intersection_left_first} occurs for the field $\wt{h}_{z,r}$, some $z_1,w_1$, and the internal metrics in the regions involved. Let $E^{2;p}_{z,r}$ be the event that either the event from Lemma~\ref{lem:interior_intersection_right} or Lemma~\ref{lem:interior_intersection_right_first} occurs for the field $\wt{h}_{z,r}$, some $z_1,z_2,w_1$, and the internal metrics in the regions involved. Recall also that the events $E^{i;p}_{z,r}$ depend only on the flow lines inside $B(z,3r/4)$, and hence are the same for $h_{z,r}$ and $\wt{h}_{z,r}$. Moreover, using the Markovian property, we can couple the internal metrics in the involved regions $U_{x',y'}$ so that they agree with the internal metrics for $\Gamma$ in case the regions agree with corresponding regions for $\Gamma$.

\begin{lemma}
\label{lem:detect_bad_large_loop}
For any $b>0$ there exist $M>0$, $p > 0$ such that the following is true. There exists an event $F$ for $h$ (equivalently $\Gamma$) with $\p[F^c] = O(\delta^b)$ and the following holds. 

Sample $z \in \CD_{\delta^{1+4\excexp}}$ and $J \in \{ \lceil\log_2(\delta^{-1-\excexp})\rceil,\ldots,\lfloor\log_2(\delta^{-1-2\excexp})\rfloor \}$ uniformly at random, and let $R = 2^{-J}$. Then on the event $F$, for every $e \in \CE^{\delta,\bad}_i(\CL)$, $\CL \in \wh{\Gamma}_i^\delta$, $i=0,2$, the following occurs with conditional probability at least $\delta^{8\excexp}$:
\begin{itemize}
\item The scale $(z,R)$ is $M$-good for $h$.
\item The number of points in $\partial A(z,R/4,3R/4) \cap X^{\theta_0}_{z,R}$, $\partial A(z,R/8,R/4) \cap X^{\theta_0;*}_{z,R}$, $\partial A(z,R/4,3R/4) \cap X^{\theta_2}_{z,R}$, $\partial A(z,R/4,3R/4) \cap X^{-\pi/2}_{z,R}$, $\partial A(z,R/4,3R/4) \cap X^{+\pi/2}_{z,R}$ is at most $M$.
\item $U_e \subseteq B(z,R^{1+\excexp})$ and $E^{i;p}_{z,R}$ occurs with $U_e$ for some $z_1,z_2,w_1$, supposing the internal metrics in $U_e$ are coupled so that they agree with the ones for $\Gamma$.
\end{itemize}
\end{lemma}

\begin{proof}[Proof of Lemma~\ref{lem:detect_bad_large_loop}]
Recall that $\diam(U_e) \le \delta^{1+4\excexp}$ by our definition of excursions. Hence, with probability comparable to $\delta^{8\excexp}$ the point $z$ is picked such that $U_e \subseteq B(z,2\delta^{1+4\excexp}) \subseteq B(z,R^{1+\excexp})$.

In the following, we denote by $\wt{\eta}_i$, $i=0,2$, the angle $\theta_i$ flow line of $\wt{h}_{z,r}$ from $z-ir$ to $z+ir$. 
As in Lemma~\ref{lem:detect_intersection_left} and~\ref{lem:detect_intersection_right}, we distinguish between several cases.

The case for $i=0$ is exactly the same as in the proof of Lemma~\ref{le:a_priori_loop_intersecting_one_side}. Define the events $G_{z,r} = G^1_{z,r} \cap G^2_{z,r}$ in the same way but on $B(z,r)$. Then taking a union bound over $z \in \CD_{\delta^{1+4\excexp}}$ we see that off an event with probability $O(\delta^b)$ for $h$, for every $z \in \CD_{\delta^{1+4\excexp}}$ the conditional probability of picking $R$ such that $G_{z,R}$ occurs is at least $3/4$. Moreover, we have seen there that when $U_e \subseteq B(z,R^{1+\excexp})$ and $G_{z,R}$ occurs, then also $E^{0;p}_{z,R}$ occurs.

We now turn to the case $i=2$. This is done similarly by applying Lemma~\ref{lem:detect_intersection_right} and defining similar annulus events $G_{z,r}$. We distinguish between the two cases in the statement of Lemma~\ref{lem:detect_intersection_right}.

For the former case in Lemma~\ref{lem:detect_intersection_right}, consider the event $G_{z,r}$ such that
\begin{itemize}
\item The scale $(z,r)$ is $M$-good for $h$.
\item The number of points in $\partial A(z,r/4,3r/4) \cap X^{\theta_2}_{z,r}$ and $\partial A(z,r/4,3r/4) \cap X^{-\pi/2}_{z,r}$ is at most $M$.
\item For every pair of strands of $X^{\theta_2}_{z,r}$, $X^{-\pi/2}_{z,r}$, respectively, that end on $\partial B(z,r/4)$, the conditional probability given $X^{\theta_2}_{z,r}$, $X^{-\pi/2}_{z,r}$, and the values of $\wt{h}_{z,r}$ on these sets that $\wt{\eta}_2$ merges into the former strand and the angle $-\pi/2$ flow line from $z-ir$ merges into the latter strand before entering $B(z,r/4)$ is either $0$ or at least $p$.
\end{itemize}
Let $G$ be the event that for each $z \in \CD_{\delta^{1+4\excexp}}$, at least $3/4$ fraction of the $G_{z,2^{-j}}$ with $j \in \{ \lceil\log_2(\delta^{-1-\excexp})\rceil,\ldots,\lfloor\log_2(\delta^{-1-2\excexp})\rfloor \}$ occur. As the number of strands that end on different points of $\partial A(0,r/4,3r/4)$ is almost surely finite, we can apply Lemma~\ref{lem:good_scales_for_event} and a union bound over $z \in \CD_{\delta^{1+4\excexp}}$ to find $M,p>0$ such that $\p[G^c] = O(\delta^b)$.

Now suppose that we have the former case in Lemma~\ref{lem:detect_intersection_right}, and are on the events $U_e \subseteq B(z,R^{1+\excexp})$ and $G_{z,R}$. Since the flow lines $\eta^R_{w_1}$, $\eta^{\theta_2}_{z_1}$ intersect with the correct angle difference and there are flow lines from the outside of the annulus that merge into $\eta^{\theta_2}_{z_1}$ resp.\ $\eta^R_{w_1}$, the conditional probability of the merging event described in the last item of Lemma~\ref{lem:interior_intersection_right} is almost surely positive. Since we are on the event $G_{z,R}$, the conditional probability is at least $p$. This proves the claim this case.

For the latter case in Lemma~\ref{lem:detect_intersection_right}, consider the event $G_{z,r}$ such that
\begin{itemize}
\item The scale $(z,r)$ is $M$-good for $h$.
\item The number of points in $\partial A(z,r/4,3r/4) \cap X^{\theta_2}_{z,r}$ and $\partial A(z,r/4,3r/4) \cap X^{+\pi/2}_{z,r}$ is at most $M$.
\item For every triple of strands $\eta^{\theta_2}_{z_1}, \eta^{\theta_2}_{\wt{z}_1}, \eta^{\theta_2}_{z_2}$ of $X^{\theta_2}_{z,r}$ that end, respectively, on $\partial B(z,r/4)$, $\partial B(z,3r/4)$, $\partial B(z,r/4)$, and for each pair of strands $\eta^{-\pi/2}_{w_1}, \eta^{-\pi/2}_{w_2}$ of $X^{-\pi/2}_{z,r}$ that end on $\partial B(z,3r/4)$, the conditional probability given $X^{\theta_2}_{z,r}$, $X^{-\pi/2}_{z,r}$, $X^{+\pi/2}_{z,r}$, and the values of $\wt{h}_{z,r}$ on these sets is either $0$ or at least $p$ that
\begin{itemize}
\item the flow line $\wt{\eta}_2$ merges into $\eta^{\theta_2}_{z_1}$ before entering $B(z,r/4)$,
\item the extension of $\eta^{\theta_2}_{\wt{z}_1}$ merges into $\eta^{\theta_2}_{z_2}$ before entering $B(z,r/4)$,
\item the extensions of $\eta^{-\pi/2}_{w_1}, \eta^{-\pi/2}_{w_1}$ merge before entering $B(z,r/4)$.
\end{itemize}
\end{itemize}
Let $G$ be the event that for each $z \in \CD_{\delta^{1+4\excexp}}$, at least $3/4$ fraction of the $G_{z,2^{-j}}$ with $j \in \{ \lceil\log_2(\delta^{-1-\excexp})\rceil,\ldots,\lfloor\log_2(\delta^{-1-2\excexp})\rfloor \}$ occur. As the number of strands that end on different points of $\partial A(0,r/4,3r/4)$ is almost surely finite, we can apply Lemma~\ref{lem:good_scales_for_event} and a union bound over $z \in \CD_{\delta^{1+4\excexp}}$ to find $M,p>0$ such that $\p[G^c] = O(\delta^b)$.

Now suppose that we have the latter case in Lemma~\ref{lem:detect_intersection_right}, and are on the events $U_e \subseteq B(z,R^{1+\excexp})$ and $G_{z,R}$. Then $\eta^R_{w_1}$ exits $B(0,3R/4)$ before merging with $\ol{\eta}^R$. Since all the flow lines involved in the event of Lemma~\ref{lem:interior_intersection_right_first} intersect with the correct angle difference and there are angle $\theta_2$ flow lines that merge into $\eta^{\theta_2}_{z_1}$ resp.\ $\eta^{\theta_2}_{z_2}$ from the outside of the annulus, the conditional probability of the merging event described in the last item of Lemma~\ref{lem:interior_intersection_right_first} is almost surely positive. Since we are on the event $G_{z,R}$, the conditional probability is at least $p$. This proves the claim in this case.
\end{proof}

We are now ready to complete the proof of Proposition~\ref{prop:intersect_loops_large}.

\begin{proof}[Proof of Proposition~\ref{prop:intersect_loops_large}]
Let $b>\bestexp$ and
\begin{itemize}
\item $F_1$ be the event from Lemma~\ref{lem:loop_fill_in_ball},
\item $F_2$ be the event from Lemma~\ref{lem:detect_bad_large_loop}.
\end{itemize}
Let $F = F_1 \cap F_2$. Then $\p[F^c] = O(\delta^b)$. Therefore it suffices to prove the result on $F$.

Consider the notation introduced in Section~\ref{subsubsec:independence_statement}. We begin by observing that on the event $F_1$ we can find for $i=0,2$ a collection of $\delta^{-10\excexp}$ excursions $e \in \CE^\delta_i(\CL)$, $\CL \in \wh{\Gamma}_i^\delta$, such that all the regions $U_{i,l}^\CL$ that define $D_i^\delta$ are contained in the regions $U_e$. Indeed, whenever a loop $\CL$ travels distance $\delta^{1+4\excexp}$, it disconnects a ball of radius $\delta^{1+5\excexp}$. In particular, there can be at most $\delta^{-10\excexp}$ loops $\CL \in \wh{\Gamma}_i^\delta$ (recall that we are considering only loops that are contained in $B(x_0,\delta)$). By the definition of $a_{i,l}^\CL$ and $b_{i,l}^\CL$ we can split the region $U_{i,l}^\CL$ into sub-segments of diameter at most $\delta^{1+4\excexp}$ such that they satisfy the definition of $\CE^\delta_i(\CL)$. Again there can be at most $\delta^{-10\excexp}$ many since a ball of radius $\delta^{1+5\excexp}$ will be disconnected whenever $\eta_i$ or (the outer boundary of) $\CL$ travels distance $\delta^{1+4\excexp}$.

Let $E$ be the event that $D_i^\delta \geq \median{\epsilon}+\epsilon^{-\epsexp}\ac{\epsilon}$ for both $i=0,2$. On $E \cap F$, by the paragraph above and the compatibility of the internal metrics, there must exist $e_i \in \CE^\delta_i(\CL_i)$, $\CL_i \in \wh{\Gamma}_i^\delta$ such that $\metapproxacres{\epsilon}{U_{e_i}}{x_{e_i}}{y_{e_i}}{\Gamma} \gtrsim \delta^{10\excexp} (\median{\epsilon}+\epsilon^{-\epsexp}\ac{\epsilon})$. By considering the rescaled metric $\mettapprox{\wt{\epsilon}}{\cdot}{\cdot}{\Gamma} = \metapprox{\epsilon}{\delta^{10\excexp/\epsexp}\cdot}{\delta^{10\excexp/\epsexp}\cdot}{\delta^{10\excexp/\epsexp}\Gamma}$ where $\wt{\epsilon} = \delta^{-10\excexp/\epsexp}\epsilon$, this means (recalling Lemma~\ref{le:scaled_metric} and~\ref{le:median_scaling}) that $\mettapprox{\wt{\epsilon}}{\delta^{-10\excexp/\epsexp}x_{e_i}}{\delta^{-10\excexp/\epsexp}y_{e_i}}{\delta^{-10\excexp/\epsexp}\Gamma} \ge \mediant{\wt{\epsilon}}+\wt{\epsilon}^{-\epsexp}\act{\wt{\epsilon}}$.

Sample $(\wt{z}_i,R_i)$ for $i=0,2$ independently from $h$ as in Lemma~\ref{lem:detect_bad_large_loop} and restricted to the event that $B(\wt{z}_0,R_0)$ and $B(\wt{z}_2,R_2)$ are disjoint. Recall that $U_{e_0}$ and $U_{e_2}$ have distance at least $\delta^{1+\excexp}$ from each other. By the Markovian property of the metric, we can couple the internal metrics so that they are conditionally independent in the two disjoint balls, and agree with the ones for $\Gamma$ in case the regions agree.

Let $G_{z,R}$ be the event that the first two items in the statement of Lemma~\ref{lem:detect_bad_large_loop} occur. On the event $F_2$, then the conditional probability that all the events $G_{\wt{z}_0,R_0}$, $G_{\wt{z}_2,R_2}$, $E^{0;p}_{\wt{z}_0,R_0}$, $E^{2;p}_{\wt{z}_2,R_2}$ occur (for $\mettapprox{\wt{\epsilon
}}{\cdot}{\cdot}{\delta^{-10\excexp/\epsexp}\Gamma}$) is at least $\delta^{16\excexp}$. Hence,
\[ \p[ E \cap F] \leq \delta^{-16\excexp} \p[ E^{0;p}_{\wt{z}_0,R_0} \cap G_{\wt{z}_0,R_0} \cap E^{2;p}_{\wt{z}_2,R_2} \cap G_{\wt{z}_2,R_2} ] . \]
It therefore suffices to show that there exists $\alpha' > 100\excexp >0$ so that $\p[ E^{0;p}_{\wt{z}_0,R_0} \cap G_{\wt{z}_0,R_0} \cap E^{2;p}_{\wt{z}_2,R_2} \cap G_{\wt{z}_2,R_2} ] = O(\delta^{\bestexp + \alpha'})$.

Reversing the sampling procedure, we estimate the probability of $E^{i;p}_{\wt{z}_i,R_i} \cap G_{\wt{z}_i,R_i}$ given $(\wt{z}_i,R_i)$. Recall that on the event $G_{\wt{z}_i,R_i}$ there are at most $M$ different choices for $z_1,z_2,w_1$, and that the scale $(\wt{z}_i,R_i)$ is $M$-good. By the Lemmas~\ref{lem:interior_intersection_left}--\ref{lem:interior_intersection_right_first}, translation invariance, absolute continuity (together with Lemma~\ref{le:abs_cont_kernel}), and a union bound we have
\[
\p[ E^{i;p}_{\wt{z}_i,R_i} \cap G_{\wt{z}_i,R_i} \mid \CF_{\wt{z}_i,R_i} ] = O(\delta^{(1-10\excexp/\epsexp)\bestexp+o(1)}) .
\]
Using that $B(\wt{z}_0,R_0)$ and $B(\wt{z}_2,R_2)$ are disjoint, we have
\[ \begin{split}
\p[ E^{0;p}_{\wt{z}_0,R_0} \cap G_{\wt{z}_0,R_0} \cap E^{2;p}_{\wt{z}_2,R_2} \cap G_{\wt{z}_2,R_2} ] 
&= \E\left[ \p[ E^{0;p}_{\wt{z}_0,R_0} \cap G_{\wt{z}_0,R_0} \mid \CF_{\wt{z}_0,R_0} ] \one_{E^{2;p}_{\wt{z}_2,R_2} \cap G_{\wt{z}_2,R_2}} \right] \\
&= O(\delta^{2(1-10\excexp/\epsexp)\bestexp+o(1)})
\end{split} \]
If $\excexp > 0$ is sufficiently small so that $(1-20\excexp/\epsexp)\bestexp > 100\excexp$, we see altogether that this implies that result.
\end{proof}

\subsection{Distances along an SLE boundary}
\label{subsec:map_in}

In this subsection we consider an \slek{} curve $\eta$ in a good domain, and a \clekp{} in its left component. This is the object that arises in the Markovian exploration of CLE in Section~\ref{se:cle_sle_markov}. The main result is that with probability $1-O(\delta^{\bestexp})$ we can find a region close to $\eta$ such that the distance across is bounded by $\median{\epsilon}+\epsilon^{-\epsexp}\ac{\epsilon}$. We would like to deduce this result from the assumption~\eqref{eq:a_priori_assumption}. A challenge here is that the configurations in different domains can only be compared within regions away from the domain boundaries. But the outer boundary of the \clekp{} exploration path may get far away from $\eta$ (in the case $\kappa'<6$ it will even intersect the left domain boundary). To overcome this, we use the resampling argument from \cite{amy-cle-resampling} to show that we can always find a configuration comparable to the one in~\eqref{eq:a_priori_assumption} locally around each point inside the domain.

Throughout this subsection, we let $D \subseteq \C$ be a simply connected domain, and $x,y \in \partial D$ distinct. Let $\eta$ be an \slek{} in $(D,x,y)$, and let $D_\eta$ be the component of $D\setminus\eta$ to the left of $\eta$. Given $\eta$, let $\Gamma$ be a \clekp{} in $D_\eta$. (We remark that by the reversibility of $\eta$, the results of this section hold also when $D_\eta$ is chosen to be the component to the right of $\eta$.)

The main result of this subsection is the following lemma (and its variants Lemma~\ref{le:close_geodesic},~\ref{le:close_geodesic_general}, Corollary~\ref{co:close_geodesic_small_parts}). Recall the Definition~\ref{def:regularity} of $(M,a)$-good domains.

\begin{lemma}
\label{le:close_geodesic_old}
There exists $c>0$ (not depending on $M,a,\epsexp$) such that the following is true. Fix $M,a > 0$ such that $a$ is sufficiently small (depending on $\epsexp$). 
Let $\wh{G}_{M,a}$ be the event that for every $0 \le s < t \le \infty$, the domain $D \setminus \eta([0,s]\cup[t,\infty))$ is $(M,a)$-good. Let $\wh{E}_\delta$ denote the event that $\sup_t\dist(\eta(t),\partial D) \le \delta$.

Let $G$ denote the event that $\metapproxacres{\epsilon}{D_\eta}{x}{y}{\Gamma} \le M\delta^{-ca}(\median{\epsilon}+\epsilon^{-\epsexp}\ac{\epsilon})$. 
Then
\[ \p[ G^c \cap \wh{G}_{M,a} \cap \wh{E}_\delta ] = O(M\delta^{\bestexp-ca}) \quad\text{as}\quad \delta \to 0.\]
\end{lemma}

We remark that the statement of Lemma~\ref{le:close_geodesic_old} is only interesting when the distance between $x,y \in \partial D$ is at most $\delta^{1-ca}$ for some fixed $c>0$. Indeed, if $\abs{x-y} > \delta^{1-ca}$, then since $(D,x,y)$ is supposed to be $(M,a)$-good, the hyperbolic geodesic in $D$ from $x$ to $y$ contains points $w$ with $\dist(w,\partial D) > \delta^{1-(c-1)a}$. In that case, $\eta$ would need to stay close to $\partial D$ in order to meet the event $\wh{G}_{M,a} \cap \wh{E}_\delta$. The probability of this is easily estimated (see Lemma~\ref{lem:get_too_close_ubd}).

The proof will require several steps. We will first show the statement of Lemma~\ref{le:close_geodesic} on the event when $\eta$ does not get too close to $\partial D$. To obtain the general result, we will divide $\eta$ into overlapping segments and bound the distance across each segment. In case we have a geodesic metric, then the geodesics connecting the endpoints of the segments necessarily cross, and we can just concatenate them. For metrics satisfying the generalized parallel law, we will find suitable separation points and concatenate at those points (see Figure~\ref{fi:concatenation_loop_chain} below). The separation points will also be important in order to make the regions satisfy the monotonicity~\eqref{eq:approx_monotonicity_axiom}.

The precise definitions of these events are rather technical, but unavoidable. The key input is the following lemma. As indicated at the beginning of the subsection, we will prove it using resampling arguments applied in a small neighborhood of each segment of $\eta$, and then apply the assumption~\eqref{eq:a_priori_assumption}. The proof is outsourced to Section~\ref{se:map_in_resampling}.

\begin{lemma}
\label{le:dist_across_chain}
Fix $\innexp > 0$. There exists $N\in\N$ such that the following is true. Let $\wt{z}\in D$ and $\delta>0$ such that $B(\wt{z},\delta) \subseteq D$, and let $\lambda \in (0,1]$. Let $\wt{G}_{\wt{z},\delta,\lambda}$ be the event that there is a set of at most $N$ points $z_1,z_2,\ldots \in A(\wt{z},\delta^{1+\innexp},\delta)$ with $\abs{z_{i_1}-z_{i_2}} \ge \delta^{1+\innexp}$ for each pair $i_1 \neq i_2$ and such that whenever $\eta \cap B(\wt{z},\delta^{1+\innexp}) \neq \varnothing$, the following hold.
\begin{itemize}
\item The set $\{z_1,z_2,\ldots\}$ separates $\eta \cap B(\wt{z},\delta^{1+\innexp})$ from $\partial B(\wt{z},\delta)$ in $\Upsilon_\Gamma$.
\item Let $K'$ be any connected component of $\Upsilon_\Gamma \setminus \{z_1,z_2,\ldots\}$ adjacent to $\eta \cap B(\wt{z},\delta^{1+\innexp})$, and let $i_1,i_2$ so that $z_{i_1},z_{i_2} \in \partial K'$. Let $U'_1 \in \metregions$ be such that $\ol{U'_1}$ contains all \emph{simple} admissible paths within $K'$ from $z_{i_1}$ to $z_{i_2}$. Then
\[ \metapproxacres{\epsilon}{U'_1}{z_{i_1}}{z_{i_2}}{\Gamma} \le \lambda^{\ddouble}\median{\epsilon}+\lambda^{\epsexp}\epsilon^{-\epsexp}\ac{\epsilon} . \]
\end{itemize}
Then $\p[\wt{G}_{\wt{z},\delta,\lambda}^c] = O((\delta/\lambda)^{\bestexp})$.
\end{lemma}

For the purpose of this section, one can just think of $\ol{U'_1}$ being the minimal simply connected set containing $K'$ (i.e.\ the union of $K'$ and the bounded connected components of $\C \setminus K'$). It will be convenient in Section~\ref{se:tightness_proof} that we can consider smaller regions where some of the ``dead ends'' in $K'$ are removed. It suffices to prove Lemma~\ref{le:dist_across_chain} in the case $\lambda=1$ (due to scaling and Lemma~\ref{le:median_scaling}), but it will be convenient to introduce the parameter $\lambda$ in the definition of the event because we will consider a sequence of events $\wt{G}_{\wt{z}_l,\wt{\delta}_l,\lambda_l}$ on various scales, and sum the corresponding bounds on $\metapprox{\epsilon}{\cdot}{\cdot}{\Gamma}$. (See~\eqref{eq:sum_across_intervals} in the definition of the events $G^U_{s,t}(r)$ below.)

Note that when $\delta^{1+\innexp} \ge \cserial\epsilon$ where $\cserial$ is the constant in the assumption~\eqref{it:mon_large_loops} of the monotonicity~\eqref{eq:approx_monotonicity_axiom}, then the monotonicity property~\eqref{eq:approx_monotonicity_axiom} applies to the region $U'_1$ in the statement of Lemma~\ref{le:dist_across_chain}. More precisely, let $\ol{U'_1}$ be the minimal simply connected set containing all simple admissible paths within $K'$ from $z_{i_1}$ to $z_{i_2}$. Let $\ol{U'}$ be the minimal simply connected set containing $K'$. Then, by~\eqref{eq:approx_monotonicity_axiom} and~\eqref{eq:shortcutted_metric}, we have
\[ \metapproxacres{\epsilon}{V}{x}{y}{\Gamma} \le \metapproxacres{\epsilon}{V \cap U'}{x}{y}{\Gamma}+2\ac{\epsilon} \le \metapproxacres{\epsilon}{U'_1}{x}{y}{\Gamma}+4\ac{\epsilon} ,\quad x,y \in \ol{U'_1} \cap \Upsilon_\Gamma \]
for every $V \supseteq U'_1$.

\begin{figure}[ht]
\centering
\includegraphics[width=0.6\textwidth]{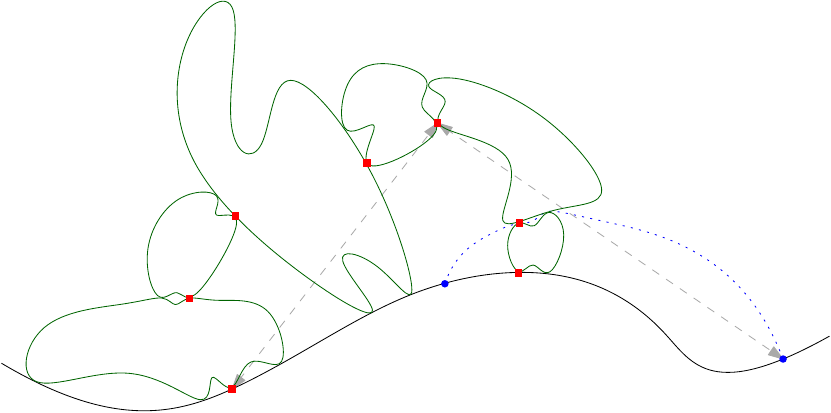}
\caption{Illustration of the concatenation argument used in the proofs of Section~\ref{subsec:map_in}. If the distance between the blue points is bounded, then by the generalized parallel law for at least one of the red points the distance across the right grey arrow is also bounded (up to a constant).}
\label{fi:concatenation_loop_chain}
\end{figure}

We will cover the curve $\eta$ by balls $B(\wt{z}_l,\wt{\delta}_l)$ on which the events from Lemma~\ref{le:dist_across_chain} hold. By the generalized parallel law, this allows us to estimate $\metapproxac{\epsilon}{x}{y}{\Gamma}$ by (a constant times) the sum of the distances across these balls (this is explained in Lemma~\ref{le:concatenate_overlapping_intervals} and illustrated in Figure~\ref{fi:concatenation_loop_chain}). Moreover, we will choose these balls so that $\wt{\delta}_l \le \delta^{1+\innexp}$ for some $\innexp > 0$. This will imply that the event depends only on the internal metrics within the $\delta^{1+\innexp}$-neighborhood of $\eta$, allowing us to use spacial independence properties with other events that depend on regions away from $\eta$.

We now formally define the event that a part of $\eta$ is covered by balls on which the events from Lemma~\ref{le:dist_across_chain} hold. We need to define this event carefully because we will need to use it to bound the internal metric in larger regions, and recall that the monotonicity property~\eqref{eq:approx_monotonicity_axiom} only holds under an additional condition on the regions. While the regions discussed in the paragraph above satisfies the additional condition, the union of overlapping regions does not necessarily satisfy the condition any more. Therefore we keep track of the individual balls $B(\wt{z}_l,\wt{\delta}_l)$ in the definition of the event below. Finally, note that on scales smaller than $\epsilon$, we will ultimately bound the distances by $\ac{\epsilon}$ using~\eqref{eq:approx_error_asymp} which are asymptotically negligible due to~\eqref{eq:eps_bound_ass}. Note that~\eqref{eq:approx_error_asymp} is not a deterministic bound, and moreover the regions smaller than $\epsilon$ do not satisfy the monotonicity~\eqref{eq:approx_monotonicity_axiom}. This is why we keep track of these segments using the definition of $\metapproxac{\epsilon}{\cdot}{\cdot}{\Gamma}$ in~\eqref{eq:shortcutted_metric}.

Fix $\innexp > 0$, $N \in \N$, and let $\wt{G}_{\wt{z},\wt{\delta},\lambda}$ be the event in Lemma~\ref{le:dist_across_chain}. Let $U \subseteq D_\eta$ and $0 \le s \le t \le \infty$. We let $G^U_{s,t}(r)$ be the event that there are times $s = u_1 < u_2 < \cdots < u_L = t$ and $\wt{z}_l,\wt{\delta}_l,\lambda_l$ (random) such that for each $l=1,\ldots,L-1$ we have either
\begin{enumerate}[(a)]
\item\label{it:short_segment} $\diam(\eta[u_l,u_{l+1}]) < \epsilon$ or
\item\label{it:segment_under_chain} $\wt{\delta_l}^{1+\innexp} \ge \cserial\epsilon$, $\eta[u_l,u_{l+1}] \subseteq B(\wt{z}_l,\wt{\delta}_l^{1+\innexp})$, $B(\wt{z}_l,\wt{\delta}_l) \subseteq U$, and $\wt{G}_{\wt{z}_l,\wt{\delta}_l,\lambda_l}$ occurs,
\end{enumerate}
and we have
\begin{equation}\label{eq:sum_across_intervals}
 \sum_l (1+\lambda_l^{\ddouble}+\lambda_l^{\epsexp}) \le r .
\end{equation}
We further write $G^U_{\wh{s},\wh{t}}(r)$ (resp.\ $G^U_{\wh{s},t}(r)$, $G^U_{s,\wh{t}}(r)$) if we can choose the time points and balls so that $\eta(u) \notin B(\wt{z}_l,\wt{\delta}_l)$ for $u \notin (s,t)$ (resp.\ $u \le s$, $u \ge t$) and each $l$.

In the definition above, the requirement $B(\wt{z}_l,\wt{\delta}_l) \subseteq U$ means that the event depends only on the internal metrics within $U$. The reason we distinguish between $G^U_{s,t}(r)$, $G^U_{\wh{s},\wh{t}}(r)$ is that we can bound $\metapproxac{\epsilon}{\eta(s)}{\eta(t)}{\Gamma}$ on the event $G^U_{\wh{s},\wh{t}}(r)$ whereas on $G^U_{s,t}(r)$ we can only bound $\metapproxac{\epsilon}{\eta(s')}{\eta(t')}{\Gamma}$ for some $s' \le s$, $t' \ge t$.

\begin{lemma}\label{le:concatenate_overlapping_intervals}
 There is a constant $c$ depending on $N$ such that the following is true. Suppose that we are on the event $G^U_{\wh{s},\wh{t}}(r)$ defined in the paragraph above. If $V_1 \in \metregions$ is such that every \emph{simple} admissible path within $\ol{U}$ from $\eta(s)$ to $\eta(t)$ is contained in $\ol{V}_1$, then
 \[ \metapproxacres{\epsilon}{V_1}{\eta(s)}{\eta(t)}{\Gamma} \le c\, r(\median{\epsilon}+\epsilon^{-\epsexp}\ac{\epsilon}) . \]
\end{lemma}

See Figure~\ref{fi:concatenation_loop_chain} for an illustration.

\begin{proof}[Proof of Lemma~\ref{le:concatenate_overlapping_intervals}]
Let $l \in \{1,\ldots,L-1\}$. In case~\eqref{it:short_segment} we just apply the definition of $\metapproxac{\epsilon}{\cdot}{\cdot}{\Gamma}$ in~\eqref{eq:shortcutted_metric}. We assume that we are in case~\eqref{it:segment_under_chain}, in particular $\wt{\delta_l}^{1+\innexp} \ge \cserial\epsilon$. Suppose that we are on the event $\wt{G}_{\wt{z}_l,\wt{\delta}_l,\lambda_l}$, and let $z_1,z_2,\ldots \in A(\wt{z}_l,\wt{\delta}_l^{1+\innexp},\wt{\delta}_l)$ be the corresponding points. There must be some $s'_l < u_l$, $t'_l > u_{l+1}$ such that $\eta(s'_l),\eta(t'_l) \in \{z_1,z_2,\ldots\}$ (otherwise there would be an admissible path from $\eta(u_l)$ to $\partial D$ avoiding $\{z_1,z_2,\ldots\}$). Note that the choice of $V_1$ implies that it contains a region $U'_1$ as in the definition of the event $\wt{G}_{\wt{z}_l,\wt{\delta}_l,\lambda_l}$. Recall that $U'_1$ is defined so that the monotonicity~\eqref{eq:approx_monotonicity_axiom} applies. Together with the definition of $\metapproxac{\epsilon}{\cdot}{\cdot}{\Gamma}$ in~\eqref{eq:shortcutted_metric}, we have
 \[ \metapproxacres{\epsilon}{V_1}{\eta(s'_l)}{\eta(t'_l)}{\Gamma} \le \lambda_l^{\ddouble}\median{\epsilon}+\lambda_l^{\epsexp}\epsilon^{-\epsexp}\ac{\epsilon}+2\ac{\epsilon} . \]
Let $s'_{l+1},t'_{l+1}$ be the analogous points for $\wt{G}_{\wt{z}_{l+1},\wt{\delta}_{l+1},\lambda_{l+1}}$, and suppose that $s'_l < s'_{l+1} < t'_l < t'_{l+1}$ (in all other cases we can just ignore $l$ or $l+1$). Then $\{z_1,z_2,\ldots\} \cap \ol{V}_1$ separates $\eta(s'_{l+1})$ from $\eta(t'_{l+1})$ in $\ol{V}_1 \cap \Upsilon_\Gamma$. By the generalized parallel law (which also applies to the metric in~\eqref{eq:shortcutted_metric}), there is some $z_i$ such that
 \[ \metapproxacres{\epsilon}{V_1}{z_i}{\eta(t'_{l+1})}{\Gamma} \le \cparallel(N)\metapproxacres{\epsilon}{V_1}{\eta(s'_{l+1})}{\eta(t'_{l+1})}{\Gamma} \le \cparallel(N)(\lambda_{l+1}^{\ddouble}\median{\epsilon}+\lambda_{l+1}^{\epsexp}\epsilon^{-\epsexp}\ac{\epsilon}) . \]
For every such $z_i$, the corresponding sets $U'_1$ in the definition of $\wt{G}_{\wt{z}_l,\wt{\delta}_l,\lambda_l}$ is also contained in $V_1$. Hence,
\[ \begin{split}
\metapproxacres{\epsilon}{V_1}{\eta(s'_l)}{\eta(t'_{l+1})}{\Gamma} 
&\le \metapproxacres{\epsilon}{V_1}{\eta(s'_l)}{z_i}{\Gamma} + \metapproxacres{\epsilon}{V_1}{z_i}{\eta(t'_{l+1})}{\Gamma} \\
&\le (\lambda_l^{\ddouble}\median{\epsilon}+\lambda_l^{\epsexp}\epsilon^{-\epsexp}\ac{\epsilon}) + \cparallel(N)(\lambda_{l+1}^{\ddouble}\median{\epsilon}+\lambda_{l+1}^{\epsexp}\epsilon^{-\epsexp}\ac{\epsilon}) . 
\end{split} \]
By repeating this argument, we conclude.
\end{proof}

We now phrase the stronger, more technical version of Lemma~\ref{le:close_geodesic_old} which we will use later in this paper.

\begin{lemma}
\label{le:close_geodesic}
There exists $c>0$ (not depending on $M,a,a_1,\epsexp$) such that the following is true. Fix $M,a,a_1 > 0$ such that $a,a_1$ are sufficiently small (depending on $\epsexp$). 
Let $\wh{G}_{M,a}$ be the event that for every $0 \le s < t \le \infty$, the domain $D \setminus \eta([0,s]\cup[t,\infty))$ is $(M,a)$-good. Let $\wh{E}_\delta$ denote the event that $\sup_t\dist(\eta(t),\partial D) \le \delta$. Let $U \subseteq D_\eta$ be the $\delta^{1+a_1}$-neighborhood of $\eta$. Then
\[ \p\left[ G^U_{\wh{0},\wh{\infty}}(M\delta^{-c(a+a_1)})^c \cap \wh{G}_{M,a} \cap \wh{E}_\delta \right] = O(M\delta^{\bestexp-ca-ca_1}) \quad\text{as}\quad \delta \to 0.\]
\end{lemma}

In the remainder of the subsection, we will use the following notation. Let $\varphi \colon D \to \h$ be a conformal transformation with $\varphi(x) = 0$ and $\varphi(y) = \infty$. For $a>0$, $m\in\R$, $k>0$ we denote
\begin{itemize}
\item $A_m = A(0,2^{m-1},2^m) \cap \h$,
\item $\tau_m = \inf\{ t>0 : \varphi(\eta(t)) \in \partial B(0,2^m) \}$,
\item $\sigma_m = \sup\{ t>0 : \varphi(\eta(t)) \in \partial B(0,2^m) \}$,
\item $B_{ak} = \{ z : d_{\h,\mathrm{hyp}}(z,i) \le 10ak \}$.\\
Note that $[-2^{4ak},2^{4ak}]\times[2^{-4ak},2^{4ak}] \subseteq B_{ak} \subseteq [-2^{15ak},2^{15ak}]\times[2^{-15ak},2^{15ak}]$ (but our definition of $B_{ak}$ has the advantage of being symmetric with respect to inversion).
\end{itemize}
Further, let
\[
D_{j,ak} = \left\{ z \in D \ :\ \parbox{.45\linewidth}{$|z-\varphi^{-1}(i2^j)| \le 2^{30ak}\dist(\varphi^{-1}(i2^j),\partial D) ,\\ \dist(z,\partial D) \ge 2^{-15ak}\dist(\varphi^{-1}(i2^j),\partial D)$} \right\}
\]
and note that $\varphi^{-1}(2^{j}B_{ak}) \subseteq D_{j,ak}$.

Let $F_{j,a,k}$ be the event that $\varphi(\eta([\tau_{j-1},\sigma_{j+1}])) \subseteq 2^{j}B_{ak}$. Let $F_a = \bigcap_{k\in\N,j\in\CZ_k} F_{j,a,k}$ where $\CZ_k$ is as defined in~\eqref{eq:bottleneck_def}.

\begin{lemma}
\label{le:close_geodesic_good_eta}
There exists $c>0$ (not depending on $M,a,a_1,\epsexp$) such that the following is true. Fix $M,a,a_1 > 0$ such that $a,a_1$ are sufficiently small (depending on $\epsexp$). 
Suppose $(D,x,y)$ is $(M,a)$-good and let $\CZ_k$ be as defined in~\eqref{eq:bottleneck_def}. Suppose also that $\CZ_k = \varnothing$ for $j < \log_2(\delta^{-1})$. Let $U \subseteq D_\eta$ be the $\delta^{1+a_1}$-neighborhood of $\eta$. 
Then
\[ \p\left[ G^U_{\wh{0},\wh{\infty}}(M\delta^{-c(a+a_1)})^c \cap F_a\right] = O(M\delta^{\bestexp-ca-ca_1}) . \]
\end{lemma}

To prove Lemma~\ref{le:close_geodesic_good_eta}, we bound the distances across each scale $j$ separately. Without loss of generality, we take $j=0$. The following result follows by taking a union bound over the events from Lemma~\ref{le:dist_across_chain}.

\begin{lemma}\label{le:close_geodesic_one_scale}
There exists $c>0$ (not depending on $a,a_1,\epsexp$) such that the following is true. Fix $a,a_1 > 0$. Let $\delta = \dist(\varphi^{-1}(i),\partial D)$, and let $U \subseteq D_\eta$ be the $\delta 2^{-(a+a_1)k}$-neighborhood of $\eta$. 
Then
\[
\p\left[ G^U_{\tau_{-1},\tau_{1}}(\delta^{-2a_1}2^{c(a+a_1)k})^c \cap F_{0,a,k} \right] = O(\delta^{\bestexp-ca_1} 2^{c(a+a_1)k})
.
\]
\end{lemma}

\begin{proof}
Throughout the proof, $c$ denotes a constant whose value may change from line to line.

Recall that on the event $F_{0,a,k}$ we necessarily have $\eta[\tau_{-1},\tau_1] \subseteq D_{0,ak}$. Let $\wt{G}_{\wt{z},\wt{\delta},1}$ denote the event from Lemma~\ref{le:dist_across_chain} with $\wt{\delta} = \delta 2^{-(ca+a_1)k}$. We take a union bound over $\wt{z} \in D_{0,ak} \cap \wt{\delta}^{1+2a_1}\Z^2$, so that
\[ \sum_{\wt{z} \in D_{0,ak} \cap \wt{\delta}^{1+2a_1}\Z^2} \p[(\wt{G}_{\wt{z},\wt{\delta},1})^c] \lesssim \wt{\delta}^{\bestexp-ca_1}2^{c(a+a_1)k} . \]
The curve segment $\eta[\tau_{-1},\tau_1]$ is covered by the balls $B(\wt{z},\wt{\delta}^{1+a_1})$. We want to divide it into a bounded number of sub-segments $\eta([s_i,t_i])$ each contained in some $B(\wt{z},\wt{\delta}^{1+a_1})$. To this end, let $E$ be the event that there are $\wt{\delta}^{-2a_1}2^{cak}$ time points $\tau_{-1} = u_1 < u_2 < \cdots < u_{\wt{\delta}^{-2a_1}2^{cak}} = \tau_{1}$ such that $\diam(\eta[u_l,u_{l+1}]) \le \wt{\delta}^{1+a_1}$ for each $l$. By Lemma~\ref{le:variation_one_scale}, we have $\p[E^c \cap F_{0,a,k}] = o^\infty(\wt{\delta})$.

On the event $E$, for each $l$, if we find the closest $\wt{z} \in D_{ak} \cap \wt{\delta}^{1+2a_1}\Z^2$ to $\eta(u_l)$, then $\eta[u_l,u_{l+1}] \subseteq B(\wt{z},\wt{\delta}^{1+a_1})$.
\end{proof}

\begin{proof}[Proof of Lemma~\ref{le:close_geodesic_good_eta}]
Write $k_\delta = \log_2(\delta^{-1})$. We handle the two cases $k_\delta \le k < \log_2(\epsilon^{-1})$ and $k \ge \log_2(\epsilon^{-1})$ separately (when $\epsilon=0$, the latter case is not needed, and the continuity of $\met{\cdot}{\cdot}{\Gamma}$ is used instead). Throughout the proof, $c$ denotes a constant whose value may change from line to line.

Let $k_\delta \le k < \log_2(\epsilon^{-1})$ and $j \in \CZ_k$. On the scale $j$, we apply Lemma~\ref{le:close_geodesic_one_scale} to the approximate metric $\mettapprox{\lambda^{-1}\epsilon}{\cdot}{\cdot}{\Gamma} = \metapprox{\epsilon}{\lambda\cdot}{\lambda\cdot}{\lambda\Gamma}$ where $\lambda = 2^{-(k-k_\delta)/2}$. Note that $\act{\lambda^{-1}\epsilon} = \ac{\epsilon}$, and $\mediant{\lambda^{-1}\epsilon} = \median[\lambda]{\epsilon} \le 2^{-\zeta(k-k_\delta)}\median{\epsilon}$ for some $\zeta>0$ by Lemma~\ref{le:scaled_metric} and~\ref{le:median_scaling}.

Let
\[ G_j = G^U_{\tau_{j-1},\tau_{j+1}}(2^{c(a+a_1)k}(\lambda^{\ddouble}+\lambda^{\epsexp})) . \]
By Lemma~\ref{le:close_geodesic_one_scale} (applied with $\delta = \lambda^{-1}2^{-k} = 2^{-(k+k_\delta)/2}$) we have the estimate $\p[G_j^c \cap F_{j,a,k}] = O(2^{-\bestexp(k+k_\delta)/2+c(a+a_1)k})$ and hence
\begin{equation}\label{eq:pr_scale_j_sum}
\sum_{k_\delta \le k < \log_2(\epsilon^{-1})} \sum_{j\in \CZ_k} \p[ G_j^c \cap F_{a} ] 
\le \sum_{k_\delta \le k < \log_2(\epsilon^{-1})} |\CZ_k| O(2^{-\bestexp (k+k_\delta)/2 +c(a+a_1)k}) 
= O(M\delta^{\bestexp-c(a+a_1)})
\end{equation}
where we have used the assumption $|\CZ_k| \leq M2^{ak}$.

We now bound the distances. We have
\begin{equation} \begin{split}\label{eq:len_scale_j_sum}
&\sum_{k_\delta \le k < \log_2(\epsilon^{-1})} \sum_{j\in \CZ_k} 2^{c(a+a_1)k}(\lambda^{\ddouble}+\lambda^{\epsexp}) \\
&\quad \le \sum_{k_\delta \le k < \log_2(\epsilon^{-1})} M 2^{c(a+a_1)k}(2^{-\zeta(k-k_\delta)}+2^{-\epsexp(k-k_\delta)/2}) \\
&\quad \lesssim M\delta^{-c(a+a_1)} .
\end{split} \end{equation}
where we are now assuming $c(a+a_1) < \epsexp/2 \wedge \zeta$.

It remains to handle $k \ge \log_2(\epsilon^{-1})$. Let $E_j$ be the event that there are $2^{cak}$ time points $\tau_{j-1} = t_1 < t_2 < \cdots < t_{2^{cak}} = \tau_{j+1}$ such that $\diam(\eta[t_l,t_{l+1}]) \le 2^{-k}$ for each $l$. By Lemma~\ref{le:variation_one_scale}, $\p[E_j^c \cap F_a] = o^\infty(2^{-ak})$ and hence $\sum_{k \ge k_\delta} \sum_{j\in\CZ_k} \p[E_j^c \cap F_a] \le M o^\infty(\delta)$.

Recall that we assumed $|\CZ_{\log_2(\epsilon^{-1})}| \le M\epsilon^{-a}$. Pick $j_1,j_2 \in \CZ_{\log_2(\epsilon^{-1})}$ such that $j_1+1,j_1+2,\ldots,j_2-1 \in \bigcup_{k \ge \log_2(\epsilon^{-1})} \CZ_k$, and consider the segment $\eta[\tau_{j_1},\tau_{j_2}]$. On the event $\bigcap_{j=j_1,\ldots,j_2} E_j$ we can find $M\epsilon^{-ca}$ time points $\tau_{j_1} = t_1 < t_2 < \cdots < t_{M\epsilon^{-ca}} = \tau_{j_2}$ such that $\diam(\eta[t_l,t_{l+1}]) \le \epsilon$ for each $l$. We are in case~\eqref{it:short_segment} in the definition of the event $G^U_{0,\infty}(r)$, and in particular we have $G^U_{\wh{0},\wh{\infty}}(r)$. Considering all such pairs of $(j_1,j_2)$, the total number of such intervals is bounded by $M\epsilon^{-ca}$. Combining this with the previous part~\eqref{eq:len_scale_j_sum}, we conclude that $G^U_{\wh{0},\wh{\infty}}(M\delta^{-c(a+a_1)})$ holds.
\end{proof}

Now we prove the general statement of Lemma~\ref{le:close_geodesic}. The idea is roughly as follows. We stop the SLE curve $\eta$ whenever the condition for event $F_a$ from Lemma~\ref{le:close_geodesic_good_eta} is violated, and consider the remaining part of the SLE curve. On the event $\wh{G}_{M,a}$, the remaining domain is still $(M,a)$-good. There is a positive conditional probability that the remainder of the curve meets the condition $F_a$, and the distance across that part can be bounded. We repeat the same for the time reversal of $\eta$. Suppose again that the distance across the remaining part is bounded. If the two ``good'' parts of $\eta$ overlap, we can concatenate them. In order make sure the two parts overlap, we will need to stop the SLE curve a bit more often.

We define a stopping time $T$ for $\eta$ as follows. With the notation $\tau_m$, $\sigma_m$ defined at the beginning of the section, we let $T \ge 0$ be the first time where one of the following happens for some $k\in\N$ and $j\in\CZ_k$.
\begin{itemize}
\item $\sigma_{j-ak} > \tau_{j-1}$, i.e.\@ $\varphi(\eta)$ revisits $\partial B(0,2^{j-ak})$ after hitting $\partial B(0,2^{j-1})$,
\item $\sigma_{j+ak} > \tau_{j+2ak}$, i.e.\@ $\varphi(\eta)$ revisits $\partial B(0,2^{j+ak})$ after hitting $\partial B(0,2^{j+2ak})$,
\item $\varphi(\eta)$ exits $2^j B_{ak}$ during $[\tau_{j-1},\tau_{j+2ak}]$.
\end{itemize}
We let $T^R$ denote the time corresponding to the stopping time for the time-reversal of $1/\varphi(\eta)$. More precisely, let $\ol{\eta}$ be the time-reversal of $\eta$ so that $\ol{\eta}$ is an $\SLE_\kappa$ in $(D,y,x)$. Let $\ol{T}$ be the first time where one of the items above occur for $1/\varphi(\ol{\eta})$, and let $T^R$ denote the time where $\eta(T^R) = \ol\eta(\ol{T})$.

Note that when the event $F_a^c$ occurs, then we must have $T \le T^R$. Indeed, when $\varphi(\eta[\tau_{j-1},\sigma_{j+1}])$ exits $2^j B_{ak}$, the following scenarios can occur:
\begin{itemize}
\item[(a)] $\sigma_{j+1} > \tau_{j+2ak}$. In this case $\varphi(\eta)$ will first revisit $\partial B(0,2^{j+ak})$ after $\partial B(0,2^{j+2ak})$, and $\varphi(\ol{\eta})$ will revisit $\partial B(0,2^{j+ak})$ after $\partial B(0,2^{j+1})$ before the corresponding time (i.e.\ $1/\varphi(\ol{\eta})$ revisits $\partial B(0,2^{-j-ak})$ after $\partial B(0,2^{-j-1})$).
\item[(b)] $\sigma_{j-2ak} > \tau_{j-1}$. By symmetry, this is case (a) for the time-reversal $\ol{\eta}$.
\item[(c)] In case $\sigma_{j+1} \le \tau_{j+2ak}$ and $\sigma_{j-2ak} \le \tau_{j-1}$, then $\varphi(\eta[\tau_{j-1},\tau_{j+2ak}])$ exits $2^j B_{ak}$. Then $T$ occurs before the first time when $\varphi(\eta[\tau_{j-1},\sigma_{j+1}])$ exits $2^j B_{ak}$. Since the definition of $F_a$ is symmetric for $\eta$ and $\ol{\eta}$, the time $T^R$ will be at or after the last time where $\varphi(\eta[\tau_{j-1},\sigma_{j+1}])$ exits $2^j B_{ak}$.
\end{itemize}

\begin{lemma}
\label{lem:get_too_close_ubd}
With the notation above, there exists $\zeta >0$ such that $\p[T<\infty] = O(\delta^{a\zeta})$.
\end{lemma}
\begin{proof}
This follows from \cite{sz-boundary-proximity} and \cite[Theorem~1.1]{fl-sle-transience}.
\end{proof}

\begin{proof}[Proof of Lemma~\ref{le:close_geodesic}]
We are going to prove the result by iteratively using Lemma~\ref{le:close_geodesic_good_eta} and Lemma~\ref{lem:get_too_close_ubd}.

For each $n \in \N$ we let $\CW_n = \{0,1\}^n$ be the set of binary words $\omega$ of length $n$ and let $\CW = \cup_n \CW_n$.  Given $\omega \in \CW$, we define an arc of $\eta$ which we will call $\eta_\omega$ inductively as follows. 
\begin{itemize}
\item Let us first give the definition of $\eta_0$.  Let $T$ be the stopping time described in the paragraph above. If $T=\infty$, we take $\eta_0 = \varnothing$. If $T<\infty$, then we let $\eta_0 = \eta\big|_{[T,\infty]}$.
\item Let us next give the definition of $\eta_1$. Let $T^R$ denote the reverse stopping time described in the paragraph above. If $T^R = 0$ (equivalently $\ol{T} = \infty$), we take $\eta_1 = \varnothing$. If $T^R > 0$, then we let $\eta_1 = \eta\big|_{[0,T^R]}$.
\item We now inductively define $\eta_\omega$ for general $\omega \in \CW$. Suppose that we have defined $\eta_\omega$ for all $\omega \in \CW_n$. We then set $\eta_{\omega 0} = (\eta_\omega)_0$ and $\eta_{\omega 1} = (\eta_\omega)_1$.
\end{itemize}
Further, we set $D_\omega = D \setminus (\eta \setminus \eta_\omega)$.

We note that Lemma~\ref{lem:get_too_close_ubd} implies there exists a constant $\zeta > 0$ so that for each $\omega \in \CW_n$ we have that
\begin{equation}
\label{eqn:eta_omega_non_empty}
\p[ \eta_\omega \neq \varnothing] = O(\delta^{a\zeta n}).
\end{equation}
Let $N$ be the largest $n \geq 0$ so that there exists $\omega \in \CW_n$ with $\eta_\omega \neq \varnothing$.  By taking a union bound over $\omega \in \CW_n$, we see from~\eqref{eqn:eta_omega_non_empty} that
\begin{equation}
\label{eqn:eta_omega_w_n_non_empty}
\p[ N \geq n] = O(\delta^{a\zeta n}).
\end{equation}

We are going to prove by induction on $n \geq 0$ that (up to multiplicative constants depending on $n$)
\begin{equation}
\label{eqn:g_and_n_bound}
\p[G^U_{\wh{0},\wh{\infty}}(M\delta^{-c(a+a_1)})^c, \wh{G}_{M,a}, \wh{E}_\delta, N \leq n] = O(M\delta^{\bestexp-ca-ca_1}) \quad\text{as}\quad \delta \to 0.
\end{equation}
We note that this indeed proves the result by taking $n \geq \bestexp/(a\zeta)$ as then~\eqref{eqn:eta_omega_w_n_non_empty} and~\eqref{eqn:g_and_n_bound} together imply that
\begin{align*}
\p[G^U_{\wh{0},\wh{\infty}}(M\delta^{-c(a+a_1)})^c \cap \wh{G}_{M,a} \cap \wh{E}_\delta] 
&\leq \p[G^U_{\wh{0},\wh{\infty}}(M\delta^{-c(a+a_1)})^c, \wh{G}_{M,a}, \wh{E}_\delta, N \leq n] + \p[N > n] \\
&= O(M\delta^{\bestexp-ca-ca_1}) \quad\text{as}\quad \delta \to 0.
\end{align*}

We now turn to prove~\eqref{eqn:g_and_n_bound}. In case $n=0$, we have seen in the discussion above that $\{N=0\} \subseteq F_a$. Hence, the result in the case that $n=0$ is Lemma~\ref{le:close_geodesic_good_eta}. 

Suppose that we have proved~\eqref{eqn:g_and_n_bound} for some $n \geq 0$.  We now aim to show that~\eqref{eqn:g_and_n_bound} holds with $n+1$ in place of $n$. In case we are on $F_a$, then we resort again to Lemma~\ref{le:close_geodesic_good_eta}. In case we are on $F_a^c$, we know from the discussion above that $0 < T \le T^R < \infty$. Now $\eta_0$ (resp.\ $\eta_1$) is an \slek{} in $D_0$ (resp.\ $D_1$). On the event $\wh{G}_{M,a}$ the $(M,a)$-good condition continues to hold for $\eta_0$ and $\eta_1$. Note also that $\sup_t\dist(\eta_0(t),\partial D_0) \le \sup_t\dist(\eta(t),\partial D) \le \delta$ on the event $\wh{E}_\delta$ (and likewise for $\eta_1$). Hence, the induction hypothesis implies that
\[
\p\left[ G^U_{\wh{T},\wh{\infty}}(M\delta^{-c(a+a_1)})^c ,\, \wh{G}_{M,a} ,\, \wh{E}_\delta ,\, N\le n+1 \right] = O(M\delta^{\bestexp-ca-ca_1}) .
\]
and
\[
\p\left[ G^U_{\wh{0},\wh{T^R}}(M\delta^{-c(a+a_1)})^c ,\, \wh{G}_{M,a} ,\, \wh{E}_\delta ,\, N\le n+1 \right] = O(M\delta^{\bestexp-ca-ca_1}) .
\]
Since $T \le T^R$, we have
\[
G^U_{\wh{T},\wh{\infty}}(M\delta^{-c(a+a_1)}) \cap G^U_{\wh{0},\wh{T^R}}(M\delta^{-c(a+a_1)}) \subseteq 
G^U_{\wh{0},\wh{\infty}}(2M\delta^{-c(a+a_1)})
\]
which finishes the proof.
\end{proof}

\begin{remark}
The proof shows actually that in the setup of Lemmas~\ref{le:close_geodesic} and~\ref{le:close_geodesic_good_eta}
\[ \p\left[ G^U_{\wh{0},\wh{\infty}}(M\delta^{-c(a+a_1)})^c \cap \wh{G}_{M,a} \cap \wh{E}_\delta \cap F_a^c \right] = O(M\delta^{\bestexp+a\zeta-ca-ca_1}) \quad\text{as}\quad \delta \to 0 \]
where $\zeta>0$ is the exponent from Lemma~\ref{lem:get_too_close_ubd}.
\end{remark}

We state the following variant of Lemma~\ref{le:close_geodesic} which replaces the event $\wh{E}_\delta$ by a weaker condition. Namely, the proof shows that when $\wh{E}_\delta$ is relaxed, we can still cover the parts of $\eta$ that are within Euclidean distance $\delta$ to $\partial D$. We will also use it later to get improved bounds for smaller regions of the domain such as the initial and final segments as well as narrow areas (see Corollary~\ref{co:close_geodesic_small_parts} below).

\begin{lemma}\label{le:close_geodesic_general}
There exists $c>0$ (not depending on $M,a,a_1,\epsexp$) such that the following is true. Fix $M,a,a_1 > 0$ such that $a,a_1$ are sufficiently small (depending on $\epsexp$). 
Let $\wh{G}_{M,a}$ be the event that for every $0 \le s < t \le \infty$, the domain $D \setminus \eta([0,s]\cup[t,\infty))$ is $(M,a)$-good. Let $U \subseteq D_\eta$ be the $\delta^{1+a_1}$-neighborhood of $\eta$.

Let $G_{\delta}$ denote the event that there exists a finite collection of intervals $[s_i,t_i] \subseteq [0,\infty]$ such that
\begin{itemize}
\item if $\dist(\eta(t),\partial D) \le \delta$, then $t \in \bigcup_i [s_i,t_i]$,
\item $\bigcap_i G^U_{s_i,t_i}(r_i)$ holds for some $(r_i)$ with
\[
\sum_i r_i \le M\delta^{-c(a+a_1)} .
\]
\end{itemize}
Then
\[ \p[ G_{\delta}^c \cap \wh{G}_{M,a} ] = O(M\delta^{\bestexp-ca-ca_1}) \quad\text{as}\quad \delta \to 0.\]
\end{lemma}

\begin{proof}
This follows from a slight adjustment of our proofs of Lemmas~\ref{le:close_geodesic_good_eta} and~\ref{le:close_geodesic}. We define an event $F_a$ very similar to Lemma~\ref{le:close_geodesic_good_eta}, but replace $k$ by $|k-k_\delta|+k_\delta$ where $k_\delta = \log_2(\delta^{-1})$. That is, let $F_a = \bigcap_{k\in\N,j\in\CZ_k} F_{j,a,\abs{k-k_\delta}+k_\delta}$.

Then, with the corresponding adjustments in the definition of $T$, Lemma~\ref{lem:get_too_close_ubd} still holds. Observe that in the proof of Lemma~\ref{le:close_geodesic_good_eta}, it suffices to sum~\eqref{eq:pr_scale_j_sum} and~\eqref{eq:len_scale_j_sum} over $k\ge (1-2ca)k_\delta$ for some $c>1$. Indeed, for $k < (1-2ca)k_\delta$, on the event $F_{j,a,2k_\delta-k}$ we have $\eta[\tau_{j-1},\tau_{j+1}] \subseteq D_{j,a(2k_\delta-k)}$, so that $\dist(\eta[\tau_{j-1},\tau_{j+1}]), \partial D) \ge 2^{-ca(2k_\delta-k)}2^{-k} > \delta$. Therefore the interval $[\tau_{j-1},\tau_{j+1}]$ is not required for the event $G_\delta$ in this lemma.

The final part of the proof of Lemma~\ref{le:close_geodesic} (the induction over $\omega \in \CW$) applies here as well since $\dist(\eta(t),\partial D_\omega) \le \dist(\eta(t),\partial D)$ which means that the result for $\eta_\omega$ continues to hold for $\eta$ in the definition of $G_\delta$.
\end{proof}

Note that (cf.\ Lemma~\ref{le:good_domain_scaling}) if $D$ is $(M,a)$-good and $\lambda>0$, then $\lambda D$ is $((2\lambda)^a M,a)$-good. Using Lemma~\ref{le:median_scaling}, we get the following corollary of Lemma~\ref{le:close_geodesic_general}.

\begin{corollary}\label{co:close_geodesic_small_parts}
There exist $c>0$ and $\zeta>0$ (not depending on $M,a,a_1,\epsexp$) such that the following is true. Fix $M,a,a_1 > 0$ such that $a,a_1$ are sufficiently small (depending on $\epsexp$). 
Let $\wh{G}_{M,a}$ be the event that for every $0 \le s < t \le \infty$, the domain $D \setminus \eta([0,s]\cup[t,\infty))$ is $(M,a)$-good. Let $U \subseteq D_\eta$ be the $\delta_1^{1+a_1}$-neighborhood of $\eta$.

For $0 < \delta_1 < \delta$, let $G_{\delta_1,\delta}$ denote the event that there exists a finite collection of intervals $[s_i,t_i] \subseteq [0,\infty]$ such that
\begin{itemize}
\item if $\dist(\eta(t),\partial D) \le \delta_1$, then $t \in \bigcup_i [s_i,t_i]$,
\item $\bigcap_i G^U_{s_i,t_i}(r_i)$ holds for some $(r_i)$ with
\[
\sum_i r_i \le M(\delta_1/\delta)^{\zeta\epsexp} .
\]
\end{itemize}
Then
\[ \p[ G_{\delta_1,\delta}^c \cap \wh{G}_{M,a} ] = O(M\delta^{\bestexp-c(1+\epsexp^{-1})(a+a_1)}(\delta_1/\delta)^{\zeta\bestexp}) \quad\text{as}\quad \delta \to 0.\]
\end{corollary}

\begin{proof}
We can assume that $\delta_1 > \epsilon^2$ (say), otherwise we have already shown the bound $M\epsilon^{-\epsexp}\ac{\epsilon}$ in the proof of Lemma~\ref{le:close_geodesic_general}.

Apply Lemma~\ref{le:close_geodesic_general} with $\lambda^{-1}\delta_1$ in place of $\delta$ to $\lambda^{-1} D$ and the approximate metric $\metapprox{\epsilon}{\lambda\cdot}{\lambda\cdot}{\lambda\Gamma}$ where $\lambda = (\delta_1/\delta)^{1/2} \delta^{c(2+\epsexp^{-1})(a+a_1)} > \epsilon$. By Lemma~\ref{le:good_domain_scaling}, the domain $\lambda^{-1} D$ is $(O(\lambda^{-a} M),a)$-good. By Lemma~\ref{le:median_scaling} we have $\median[\lambda]{\epsilon} \le \lambda^{\ddouble+o(1)}\median{\epsilon}$. This implies the result where we also see that $\zeta$ can be picked close to $(1-ca-ca_1)\ddouble/2$.
\end{proof}

\subsubsection{Resampling arguments}
\label{se:map_in_resampling}

\newcommand*{\Ecrossings}{E^1}
\newcommand*{\ncrossings}{n_1}
\newcommand*{\Eseploops}{E^2}
\newcommand*{\Fseploops}{F^{\mathrm{sep}}}
\newcommand*{\pseploops}{p_2}
\newcommand*{\Gammasep}{\Gamma^{\mathrm{sep}}}
\newcommand*{\Emerge}{E^3}
\newcommand*{\pmerge}{p_3}
\newcommand*{\Elinkgff}{E^4}
\newcommand*{\Flinkgff}{F^\resampled}
\newcommand*{\plinkdisc}{p_4}
\newcommand*{\plinkmerge}{p_5}

The purpose of this subsection is to prove Lemma~\ref{le:dist_across_chain}. We assume the same setup and notation as in Section~\ref{subsec:map_in}.

The idea of proving Lemma~\ref{le:dist_across_chain} is as follows. We would like to compare the law of the CLE configuration near a point $\eta(u)$ to the object in Section~\ref{se:intersections_setup} and then apply the assumption~\eqref{eq:a_priori_assumption}. With the resampling procedure from Section~\ref{se:cle_resampling} in mind, this is not surprising as we can imagine a neighborhood of $\eta(u)$ being under one single loop of $\Gamma$ whose outer boundary intersects $\eta$ in the same way as the flow lines $\eta_1,\eta_2$ in Section~\ref{se:intersections_setup} intersect (see Figure~\ref{fi:outline_bubble} for an illustration).

The rigorous argument involves several steps. One challenge that needs to be taken into account is that $\eta$ may have bottlenecks in $B(\wt{z},\delta)$ that go outside $B(\wt{z},\delta)$. To deal with this, we pretend that $\eta$ is the outer boundary of another \clekp{} loop. In that case we can apply the resampling procedure from Section~\ref{se:cle_resampling} to link more loops to $\eta$ that block the bottlenecks. To make this comparison, note that if we have a CLE in $\h$, the left outer boundary of its exploration path from $0$ to $\infty$ is an \slekr{\kappa-4;2-\kappa}. We can compare the laws of $(\eta,\Gamma)$ under \slek{} vs.\ \slekr{\kappa-4;2-\kappa} locally by formulating an event for a GFF in a neighborhood of $\wt{z}$. Although the CLE configuration is not locally determined by the GFF, the configuration in a region under a CLE loop can be recovered from the GFF locally. We use again the resampling procedure to create such ``localizing'' loops with positive probability. Finally, for a CLE in $\h$ we know from Lemma~\ref{lem:disconnect_boundary} that with high probability there is a random scale in which there is a chain of loops separating a region $U$ from the outside. They will yield the desired points $\{z_i\}$. In order to bound the distances in $U$, we resample again in order to compare its law to a region as in Section~\ref{se:intersections_setup}.

We now proceed to the proof of Lemma~\ref{le:dist_across_chain}. We first state an auxiliary lemma.

\begin{lemma}\label{le:Ncrossings}
Let $h$ be a GFF in $D$ with bounded boundary values. For any $b>0$ there exists $n \in \N$ such that the following is true. Let $z \in D$ and $\delta > 0$ such that $B(z,\delta) \subseteq D$. For $r \in (0,1)$ let $N_{z,\delta,r}$ denote the number of flow lines that cross the annulus $A(z,r\delta,\delta)$. Then $\p[N_{z,\delta,r} > n] = O(r^b)$.
\end{lemma}

\begin{proof}
 If there are $n$ disjoint crossings of $A(z,r\delta,\delta)$, then there are at least $n$ disjoint crossings of $A(z,2^{-j}\delta,2^{-j+1}\delta)$ for every $1 \le j \le \log_2(r^{-1})$. By Lemma~\ref{le:good_scales_merging} (we just need the bound on the number of strands), for any $b>0$ there exists $n \in \N$ such that the probability of this is bounded by $O(r^b)$.
\end{proof}

Throughout the proof, let 
\[ b > \bestexp \]
be fixed. To simplify notation, we assume without loss of generality (due to translation invariance) that $\wt{z}=0$. We can also assume $\lambda = 1$, otherwise we consider the rescaled metric as in Lemma~\ref{le:scaled_metric} and apply Lemma~\ref{le:median_scaling}. Moreover, we are going to show the statement of Lemma~\ref{le:dist_across_chain} with $\delta^{1+6\innexp}$ in place of $\delta^{1+\innexp}$. In all the considerations below, we implicitly restrict to the event $\{ \eta \cap B(0,\delta^{1+6\innexp}) \neq \varnothing \}$.

\textbf{Step 1.} Resample and create a local event for a GFF.

Let $\Ecrossings_{\delta}$ be the event that $\eta$ crosses $A(0,\delta^{1+2\innexp},\delta^{1+\innexp})$ or $A(0,\delta^{1+4\innexp},\delta^{1+3\innexp})$ at most $\ncrossings$ times where $\ncrossings$ is picked large enough so that $\p[(\Ecrossings_{\delta})^c] \lesssim \delta^b$ by Lemma~\ref{le:Ncrossings}. We claim that off an event of probability $O(\delta^b)$ we can resample the CLE within $A(0,\delta^{1+4\innexp},\delta^{1+\innexp}) \cap D_\eta$ so that with positive conditional probability we create a collection of loops that separate the inner from the outer parts of $A(0,\delta^{1+4\innexp},\delta^{1+\innexp}) \cap D_\eta$.

More precisely, suppose we are on the event $\Ecrossings_{\delta}$. We claim that there exists a collection of at most $\ncrossings$ arcs $C'_1,C'_2,\ldots$ of $\partial B(0,\delta^{1+3\innexp}) \cap D_\eta$ such that every path from a point in $D_\eta \cap B(0,\delta^{1+4\innexp})$ to a point in $D_\eta \setminus B(0,\delta^{1+\innexp})$ needs to pass through some $C'_i$. Similarly, there exists a collection of at most $\ncrossings$ arcs $C''_1,C''_2,\ldots$ of $\partial B(0,\delta^{1+2\innexp}) \cap D_\eta$ such that every path from a point in $D_\eta \cap B(0,\delta^{1+4\innexp})$ to a point in $D_\eta \setminus B(0,\delta^{1+\innexp})$ needs to pass through some $C''_j$. (See Figure~\ref{fi:local_event_resample}.)

\begin{figure}[ht]
\centering
\includegraphics[width=0.45\textwidth]{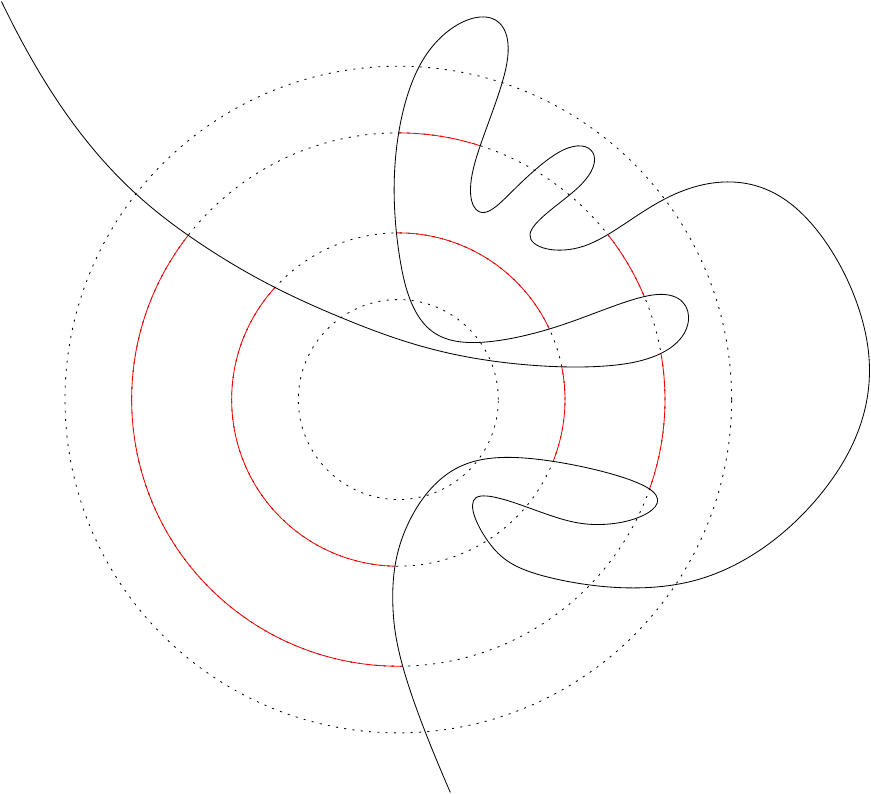}\hspace{0.05\textwidth}\includegraphics[width=0.45\textwidth]{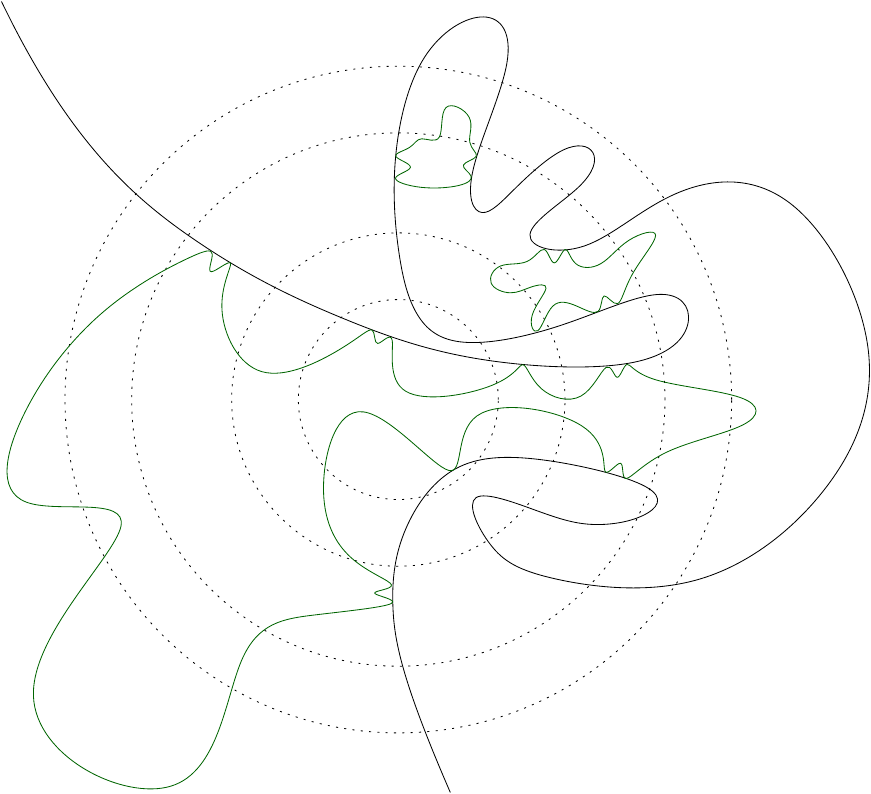}
\caption{Creating an event on which the CLE configuration is locally determined by the GFF.}
\label{fi:local_event_resample}
\end{figure}

Indeed, let $\CC$ be the collection of arcs of $\partial B(0,\delta^{1+2\innexp}) \cap D_\eta$. Suppose $\gamma_1$ (resp.\ $\gamma_2$) are two disjoint crossings of $A(0,\delta^{1+4\innexp},\delta^{1+\innexp})$ by paths in $D_\eta$, and let $C_1$ (resp.\ $C_2$) be the first arc in $\CC$ that it passes. Suppose $C_1 \neq C_2$. There are two possibilities: Either the segments of $\gamma_1,\gamma_2$ until then are connected within $A(0,\delta^{1+2\innexp},\delta^{1+\innexp}) \cap D_\eta$ or there must be a crossing of $A(0,\delta^{1+2\innexp},\delta^{1+\innexp})$ by $\eta$ between $\gamma_1,\gamma_2$. If there is no crossing of $A(0,\delta^{1+2\innexp},\delta^{1+\innexp})$ by $\eta$ between $\gamma_1,\gamma_2$, then there must be a crossing of $A(0,\delta^{1+4\innexp},\delta^{1+2\innexp})$ by $\eta$ between $\gamma_1,\gamma_2$ (otherwise, since $C_1 \neq C_2$, there would be a segment of $\eta$ between $\gamma_1,\gamma_2$ that is disconnected from $\partial D$). This shows that on the event $\Ecrossings_{\delta}$ there can be at most $\ncrossings$ such paths $\gamma$ that pass different arcs of $\CC$.

The same argument applies to the arcs of $\partial B(0,\delta^{1+3\innexp}) \cap D_\eta$. This shows the claim above.

Pick a pair $C''_i,C'_j$, and let $x \in \partial D_\eta \cap B(0,\delta^{1+4\innexp})$, $y \in \partial D_\eta \setminus B(0,\delta^{1+\innexp})$ be two boundary points that are separated by $C''_i,C'_j$. Let $\psi\colon D_\eta \to \h$ be a conformal transformation with $\psi(x) = 0$, $\psi(y) = \infty$. By Koebe's distortion theorem there are at least $(\innexp/100)\log(\delta^{-1})$ annuli $A_m$ such that sit between $\psi(C''_i),\psi(C'_j)$ and such that $\psi^{-1}(i2^{m}) \in A(0,\delta^{1+2.1\innexp},\delta^{1+2.9\innexp})$. By Lemma~\ref{le:bounded_crosscut}, for each such $A_m$ there is a crosscut $\Xi \subseteq A_m$ connecting $[-2^m,-2^{m-1}]$ to $[2^m,2^{m-1}]$ such that $\psi^{-1}(\Xi) \subseteq A(0,\delta^{1+3\innexp},\delta^{1+2\innexp})$. It follows that (on the event $\Ecrossings_{\delta}$) we have $\psi^{-1}(A_m) \subseteq A(0,\delta^{1+4\innexp},\delta^{1+\innexp})$ for all but at most $\ncrossings$ such $A_m$.

We select a scale $m$ uniformly at random and resample $\Gamma$ in $\psi^{-1}(A_m)$ using the procedure in Lemma~\ref{lem:disconnect_boundary}. On an event with probability $1-O(\delta^b)$, for at least $9/10$ fraction of scales $m$ the conditional probability of creating a single separating loop as described in Lemma~\ref{lem:disconnect_boundary} is at least a constant $\pseploops > 0$ (depending on $\innexp,b$).

We repeat this procedure for each pair $C''_i,C'_j$. As a result, we obtain a coupling of $\Gamma$ with another CLE $\Gammasep$ with the following properties:
\begin{itemize}
\item The configurations of $\Gamma$ and $\Gammasep$ differ only within $A(0,\delta^{1+4\innexp},\delta^{1+\innexp}) \cap D_\eta$.
\item Let $\Fseploops_{\delta}$ is the event that $\Gammasep$ contains a collection of at most $\ncrossings^2$ loops $\{\wt{\CL}_i\}$ such that any path from a point in $D_\eta \cap B(0,\delta^{1+4\innexp})$ to a point in $D_\eta \setminus B(0,\delta^{1+\innexp})$ needs to either cross some $\wt{\CL}_i$ or intersect $\eta$. There exists an event $\Eseploops_{\delta}$ for $\Gamma$ with $\p[(\Eseploops_{\delta})^c] = O(\delta^b)$ such that
\[ \p[\Fseploops_{\delta} \mid \Gamma] \one_{\Eseploops_{\delta}} \ge \pseploops^{\ncrossings^2} \]
where $\pseploops > 0$ is a suitable constant depending on $\innexp,b$.
\end{itemize}
We also couple the internal metrics so that $(\metapproxres{\epsilon}{V}{\cdot}{\cdot}{\Gamma})_{V \in \metregions[B(0,\delta^{1+4\innexp})]} = (\metapproxres{\epsilon}{V}{\cdot}{\cdot}{\Gammasep})_{V \in \metregions[B(0,\delta^{1+4\innexp})]}$. This shows the following.

\begin{lemma}
Let $\wt{E}$ be an event that depends on $\eta \cap B(0,\delta^{1+4\innexp})$ and the internal metrics within $B(0,\delta^{1+4\innexp}) \cap \ol{D_\eta}$. Let $\Ecrossings_{\delta}$, $\Fseploops_{\delta}$ be as defined above. Then
\[ \p[\wt{E}] \lesssim \p[\wt{E} \cap \Ecrossings_{\delta} \cap \Fseploops_{\delta}] + \delta^b . \]
\end{lemma}

Suppose we are on the event $\Fseploops_{\delta}$. Consider a segment of $\wh{\eta} \subseteq \eta \cap B(0,\delta^{1+\innexp})$ that starts and ends on $\partial B(0,\delta^{1+\innexp})$ and enters $B(0,\delta^{1+4\innexp})$. (Recall that on the event $\Ecrossings_{\delta}$ there are at most $\ncrossings/2$ such segments.) Let $U$ be the $\dpath$-connected component of $\Upsilon_\Gamma \cap B(0,\delta^{1+4\innexp})$ adjacent to $\wh{\eta}$ (in other words, $U$ is the set of points that are connected to $\wh{\eta}$ via an admissible path within $B(0,\delta^{1+4\innexp})$). Then there exist some $\{ \wt{\CL}_{i_l} \} \subseteq \{ \wt{\CL}_i \}$ and intersection points $x_l,y_l \in \wt{\CL}_{i_l} \cap \eta \cap A(0,\delta^{1+4\innexp},\delta^{1+\innexp})$ such that the following is true. Let $\whwt{\CL}_{i_l} \subseteq \wt{\CL}_{i_l}$ be the counterclockwise segment from $x_l$ to $y_l$. Let $\wh{\eta}_l$ be the segment of $\eta$ from $x_l$ to $y_{l+1}$ (where cyclic indexing is used). Let $\wt{U}$ be the region that is bounded by the right outer boundaries of $\{ \whwt{\CL}_{i_l} \}$ and the left sides of $\{ \wh{\eta}_l \}$. Then $U \subseteq \wt{U} \subseteq B(0,\delta^{1+\innexp})$. (See the left picture of Figure~\ref{fi:local_gff_loops}.)

\begin{figure}[ht]
\centering
\includegraphics[width=0.45\textwidth]{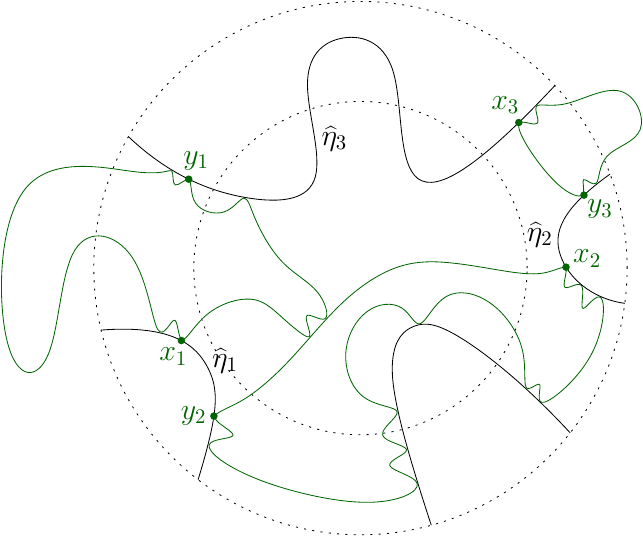}\hspace{0.05\textwidth}\includegraphics[width=0.45\textwidth]{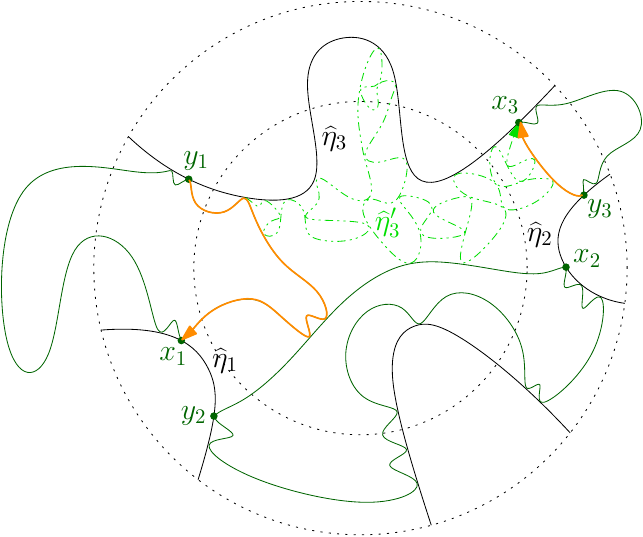}
\caption{The CLE configuration in the regions bounded between the dark green loops can be recovered from the values of a GFF in $B(0,\delta^{1+\innexp})$.}
\label{fi:local_gff_loops}
\end{figure}

Let $h$ be a GFF generating $(\eta,\Gammasep)$ in such a way that $\eta$ is the flow line with angle $-\theta_0=-3\pi/2$ and $\Gammasep$ is constructed from the branching counterflow line which goes clockwise around the boundary (reflecting off $\eta$) and traces the CLE loops counterclockwise. We claim that on the event $\Fseploops_{\delta}$ the CLE configuration within $\wt{U}$ is determined by $\{ \wh{\eta}_l \}$ and the values of $h$ in $B(0,\delta^{1+\innexp})$.

Indeed, suppose that the indices are chosen so that the $\wh{\eta}_l$ are visited in chronological order. The right\footnote{Here, ``right'' is viewed from the direction in which the counterflow line is drawn. It would be ``left'' when viewed from the direction in which the flow lines are drawn.} boundary of $\whwt{\CL}_{i_1}$ can be obtained as the angle $\pi/2$ flow line from $y_1$ to $x_1$ (reflecting off $\{ \wh{\eta}_l \}$). Although this information does not tell us whether the flow line belongs to the boundary of one or several loops, it does tell us that each loop whose boundary shares part of the flow line exits $B(0,\delta^{1+\innexp})$ (otherwise the entire loop would be detectable by the values of $h$ within $B(0,\delta^{1+\innexp})$).

Given $\{ \wh{\eta}_l \}$ and $\whwt{\CL}_{i_1}$, the remaining $\whwt{\CL}_{i_l}$ can be discovered inductively as follows. Run a counterflow line from the first intersection point of $\whwt{\CL}_{i_{l+1}} \cap \wh{\eta}_{l}$ targeting $y_l$. In case it exits $B(0,\delta^{1+\innexp})$ before reaching $y_l$, let $\wh{\eta}'_l$ denote the portion until the last point where it intersected $\wh{\eta}_{l}$. The right boundary of $\whwt{\CL}_{i_l}$ is then obtained as the angle $\pi/2$ flow line from $y_l$ to $x_l$, reflecting off $\wh{\eta}'_l$. (See the right picture of Figure~\ref{fi:local_gff_loops}.)

\textbf{Step 2.} Compare to another GFF which generates a \clekp{} in a ball.

Consider the notation in Section~\ref{se:gff} for the GFF $h$ defined above. Let $\wt{h}$ be a GFF on $B(0,2^{-j})$ with boundary values $\lambda(3-\kappa)+\theta_0 \chi-\chi\arg\varphi'$ where $\varphi\colon B(0,2^{-j}) \to \h$ is a conformal transformation with $\varphi(-i2^{-j}) = 0$, $\varphi(i2^{-j}) = \infty$. The boundary values are chosen so that they are compatible with a \clekp{} on $B(0,2^{-j})$. Namely, we let $\wt{\Gamma}$ be the \clekp{} constructed from the branching counterflow line with angle $-\theta_0-\pi/2$ starting from $i2^{-j}$ (which goes around the domain boundary clockwise and draws the CLE loops counterclockwise). Let $\wt{\eta}$ be the angle $-\theta_0$ flow line from $-i2^{-j}$ to $i2^{-j}$. It agrees with the left boundary of the branch of the exploration tree targeting $-i2^{-j}$. The law of $\wt{\eta}$ is an \slekr{\kappa-4;2-\kappa}.

Let $\wt{h}_{0,2^{-j}}$ be the field defined in Section~\ref{se:gff} with its boundary values as $\wt{h}$. For each $j \in \{ \lceil\log_2(\delta^{-1})\rceil,\ldots,\lfloor\log_2(\delta^{-1-\innexp})\rfloor \}$ we consider a variant of the event described in Lemma~\ref{le:good_scales_merging_multiple}. Concretely, let $\Emerge_{j}$ be the event that
\begin{itemize}
\item the scale $2^{-j}$ is $M$-good for $h$,
\item for any sequence of strands of $X^{-\theta_0}_{0,2^{-j}}$ as described above Lemma~\ref{le:good_scales_merging_multiple}, the conditional probability is either $0$ or at least $\pmerge > 0$ that $\wt{\eta}$ merges into the strands in the correct order and additionally, given $\wt{\eta}$, the exploration path from $-i2^{-j}$ to $i2^{-j}$ stays in the $2^{-j}/100$-neighborhood of the left boundary arc of $\partial B(0,2^{-j})$.
\end{itemize}
Let $\Emerge$ be the event that at least $9/10$ fraction of $\Emerge_{j}$, $j \in \{ \lceil\log_2(\delta^{-1})\rceil,\ldots,\lfloor\log_2(\delta^{-1-\innexp})\rfloor \}$, occur. By the Lemmas~\ref{lem:gff_independence_across_scales} and~\ref{lem:good_scales_for_event}, we have $\p[(\Emerge)^c] = O(\delta^b)$ for suitable values of $M,\pmerge > 0$.

Pick a scale $J \in \{ \lceil\log_2(\delta^{-1})\rceil,\ldots,\lfloor\log_2(\delta^{-1-\innexp})\rfloor \}$ uniformly at random. Then
\[ \p[\wt{E} \cap \Emerge_{J}] \ge (9/10) \p[\wt{E} \cap \Emerge] \]
for any event $\wt{E}$.

Suppose $\wt{E}$ is an event that depends on $\eta \cap B(0,\delta^{1+4\innexp})$ and the internal metrics within the $\dpath$-connected components of $B(0,\delta^{1+4\innexp}) \cap \ol{D_\eta}$ adjacent to $\eta$. For such an event $\wt{E}$, we consider an event $\wt{E}^\GFF_{j}$ for the restriction of $\wt{h}_{0,2^{-j}}$ to $B(0,(3/4)2^{-j})$ and an internal metric as defined below.

Suppose $\eta^{-\theta_0}_{w_1},\eta^{-\theta_0}_{w_2},\ldots$ are strands of $X^{-\theta_0}_{0,2^{-j}}$. We extend each of the flow lines until they exit $B(0,(3/4)2^{-j})$. If we pretend that $\wt{\eta}$ merges into $\eta^{-\theta_0}_{w_1},\eta^{-\theta_0}_{w_2},\ldots$ in exactly this order, we have seen in Step~1 above that on the event $\Fseploops_{\delta}$ the CLE configuration within the $\dpath$-connected components of $B(0,\delta^{1+4\innexp}) \cap \ol{D_{\wt{\eta}}}$ adjacent to $\wt{\eta}$ is determined by the values of $\wt{h}_{0,2^{-j}}$ on $B(0,(3/4)2^{-j})$. In this case, sample the internal metric in that region. Let $\wt{E}^\GFF_{j}$ be the event that $\wt{E}$ occurs for this internal metric. (This event depends also on the choice of $\{\eta^{-\theta_0}_{w_l}\}$.)

Let $\wt{\p}$ denote the law of $\wt{h}$. By absolute continuity (and Lemma~\ref{le:abs_cont_kernel}), we have
\[ \p[ \wt{E}^\GFF_{j} \mid \CF_{0,2^{-j}} ] \one_{\Emerge_{j}} \lesssim \wt{\p}[ \wt{E}^\GFF_{j} ]^{1-o(1)} \one_{\Emerge_{j}} . \]
On $\wt{E}^\GFF_{j} \cap \Emerge_{j}$ (for any choice of $\{\eta^{-\theta_0}_{w_l}\}$), the conditional probability is at least $\pmerge$ that $\wt{\eta}$ merges exactly into $\{\eta^{-\theta_0}_{w_l}\}$ and the branching counterflow line after tracing $\wt{\eta}$ goes clockwise around within the $2^{-j}/100$-neighborhood of the left boundary arc. In this case, the CLE configurations of $\Gamma$ and $\wt{\Gamma}$ within the $\dpath$-connected components of $B(0,\delta^{1+4\innexp}) \cap \ol{D_\eta}$ (resp.\ $B(0,\delta^{1+4\innexp}) \cap \ol{D_{\wt{\eta}}}$) adjacent to $\eta$ (resp.\ $\wt{\eta}$) are determined by the values of $\wt{h}_{0,2^{-j}}$ (resp.\ $\wt{h}$ on $B(0,(3/4)2^{-j})$) in the same way. Let $\wt{E}_{j}$ be the event that $\wt{E}$ occurs for the internal metrics in $B(0,\delta^{1+4\innexp}) \cap \ol{D_{\wt{\eta}}}$ (in place of $B(0,\delta^{1+4\innexp}) \cap \ol{D_\eta}$). Then
\[
\wt{\p}[ \wt{E}_{j}] \ge \pmerge \wt{\p}[ \wt{E}^\GFF_{j} \cap \Emerge_{j}] .
\]
Take a union bound over all choices of $\{\eta^{-\theta_0}_{w_l}\}$, we conclude the following.

\begin{lemma}
Let $\wt{E}$ be an event that depends on $\eta \cap B(0,\delta^{1+4\innexp})$ and the internal metrics within the $\dpath$-connected components of $B(0,\delta^{1+4\innexp}) \cap \ol{D_\eta}$ adjacent to $\eta$. Sample the number $J \in \{ \lceil\log_2(\delta^{-1})\rceil,\ldots,\lfloor\log_2(\delta^{-1-\innexp})\rfloor \}$ uniformly at random. Let $\wt{h}$ be a GFF in $B(0,2^{-J})$ with boundary values described above. Let $\wt{\eta}$ be the angle $-\theta_0$ flow line from $-i2^{-J}$ to $i2^{-J}$, and let $\wt{\Gamma}$ be the \clekp{} generated by $\wt{h}$. Sample the internal metrics in $B(0,\delta^{1+4\innexp}) \cap \ol{D_{\wt{\eta}}}$. Let $\wt{E}_{J}$ be the event that $\wt{E}$ occurs for $\wt{\eta}, (\metapproxres{\epsilon}{V}{\cdot}{\cdot}{\wt{\Gamma}})_{V \in \metregions[B(0,\delta^{1+4\innexp}) \cap \ol{D_{\wt{\eta}}}]}$ in place of $\eta, (\metapproxres{\epsilon}{V}{\cdot}{\cdot}{\Gamma})_{V \in \metregions[B(0,\delta^{1+4\innexp}) \cap \ol{D_{\eta}}]}$. Then
\[ \p[\wt{E} \cap \Ecrossings_{\delta} \cap \Fseploops_{\delta} \cap \Emerge] \lesssim \p[\wt{E}_{J}]^{1-o(1)} . \]
\end{lemma}

In the final step of the proof of Lemma~\ref{le:dist_across_chain} (with $\delta^{1+6\innexp}$ in place of $\delta^{1+\innexp}$), we estimate the probability of $\p[\wt{E}_{j}]$ for each $j \in \{ \lceil\log_2(\delta^{-1})\rceil,\ldots,\lfloor\log_2(\delta^{-1-\innexp})\rfloor \}$ with a suitable event $\wt{E} \supseteq \wt{G}^c$. Now $\wt{E}_{j}$ is an event for a \clekp{} in $B(0,\wt{\delta}) \defeq B(0,2^{-j})$ and $(\metapproxres{\epsilon}{V}{\cdot}{\cdot}{\wt{\Gamma}})_{V \in \metregions[D_{\wt{\eta}}]}$.

\textbf{Step 3.} Resample the CLE in $B(0,\wt{\delta})$ and compare the region to the one from the a priori estimate.

Working with $\wt{\eta}$ has now the advantage that we can resample the CLE $\wt{\Gamma}$ in $B(0,\wt{\delta})$ so that the regions under consideration can be compared to the ones from Section~\ref{se:intersections_setup}. We will apply a similar resampling procedure as in the one considered in Lemma~\ref{lem:disconnect_interior}. A technical complication is that we will need to find a good scale where flow lines are easy to merge into, but \emph{for the GFF generating the resampled CLE}. That is, we need to find with high probability a scale where we have a separating chain of loops with good resampling probabilities and such that \emph{after successful resampling} the flow lines of the new GFF have good merging probabilities. In general, resampling the CLE at one place changes the GFF and the flow lines also at other places (e.g.\ it may change the first point from which a loop is traced, and the merging behavior of flow lines). We will build an event that is robust enough under the change of the GFF. This explains the complicated looking formulation of the event that we define below.

\begin{figure}[ht]
\centering
\includegraphics[width=0.45\textwidth]{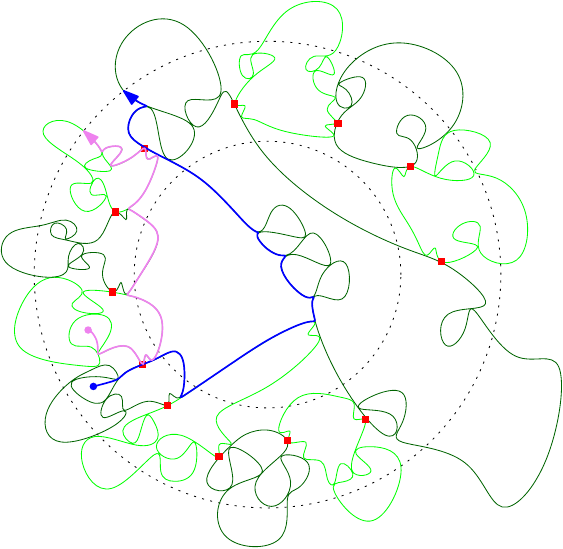}
\caption{Shown in pink and blue is one pair of $\wt{\eta}^{\theta_1}_{i_1,i_2}, \wt{\eta}^{\theta_2}_{i_1,i_2}$ in the definition of the event $\Elinkgff_{j}$.}
\label{fi:sep_chain_merge}
\end{figure}

We define the event $\Elinkgff_{j}$ (depending on the parameters $N^*,\plinkdisc,s,M,\plinkmerge$ which will be chosen later) that
\begin{enumerate}[(i)]
\item\label{it:sep_chain} There exists a collection of at most $N^*$ points $z_1,z_2,\ldots \in A(0,2^{-j-1},2^{-j})$ with $\abs{z_{i_1}-z_{i_2}} \ge s2^{-j}$ for each $i_1 \neq i_2$ and such that every point in $\Upsilon_{\wt\Gamma} \cap B(0,2^{-j-1})$ is separated from $\partial B(0,2^{-j})$ in $\Upsilon_{\wt\Gamma} \setminus \{z_1,z_2,\ldots\}$. Let $\CK$ be the collection of CLE loops that intersect at least one of these points.

\item\label{it:link_prob} The successful resampling probability in $A(0,2^{-j-1},2^{-j})$ for every choice of target pivotals with at most $N^*$ points as in Proposition~\ref{prop:resampling} is at least $\plinkdisc > 0$.

\item\label{it:break_prob} The conditional probability of successfully breaking any choice of crossings of $A(0,2^{-j},2^{-j+1})$ as in Lemma~\ref{lem:break_loops} is at least $\plinkdisc > 0$.

\item\label{it:swallow_ball} For each $z_i$, each of the two loop strands in $\CK$ that intersect at $z_i$ swallows a ball $B_i^+$ (resp.\ $B_i^-$) of radius $s 2^{-j}$ after passing $z_i$ and before exiting $B(0,2^{-j})$.

\item\label{it:merging_prob} For each pair $z_{i_1},z_{i_2}$ the following holds. Let $\wt{\eta}^{\theta_1}_{i_1,i_2}, \wt{\eta}^{\theta_2}_{i_1,i_2}$ be defined as follows. Pretend that the loops in $\CK$ are linked together at the points $z_i$, $i\neq i_1,i_2$. Let $\wt{\eta}^{\theta_2}_{i_1,i_2}$ be a path starting from a point in $B_{i_1}^+$ and behaves like a flow line that traces the outer boundary of the exploration path that traces the loops in $\CK$ (after the linking) counterclockwise (so that $\wt{\eta}^{\theta_2}_{i_1,i_2}$ traces the outer boundaries of the loops clockwise), then passes through $z_{i_2}$ and is stopped before exiting $B(0,2^{-j})$. Let $\wt{\eta}^{\theta_1}_{i_1,i_2}$ be defined analogously, starting in $B_{i_1}^-$ and tracing the outer boundaries of the loops counterclockwise. Then for a GFF in $B(0,2^{-j}) \setminus (\wt{\eta}^{\theta_1}_{i_1,i_2} \cup \wt{\eta}^{\theta_2}_{i_1,i_2})$ with flow line boundary values of angle $\theta_1$ resp.\ $\theta_2$ and winding height at most $M$ on $\wt{\eta}^{\theta_1}_{i_1,i_2}$ resp.\ $\wt{\eta}^{\theta_2}_{i_1,i_2}$, the probability that the respective flow lines from $-i2^{-j}$ merge into them before $z_{i_1}$ is at least $\plinkmerge > 0$. (See Figure~\ref{fi:sep_chain_merge}.)

\item\label{it:winding} The outer boundary of any loop that crosses $A(0,2^{-j-1},2^{-j+1})$ winds within $A(0,2^{-j-1},2^{-j+1})$ at most $M$ times.
\end{enumerate}

Let $\Elinkgff$ be the event that $\Elinkgff_{j}$ occurs for at least $9/10$ fraction of scales $2^{-j} \in [\wt{\delta}^{1+5\innexp},\wt{\delta}^{1+4\innexp}]$. We claim that we can pick the parameters so that $\p[(\Elinkgff)^c] = O(\wt{\delta}^b)$. This follows from the same arguments used in \cite{amy-cle-resampling} to prove Lemmas~\ref{lem:disconnect_interior} and~\ref{lem:break_loops}. We sketch the argument and refer the reader to \cite{amy-cle-resampling} for details. The main input is the independence across scales for \clekp{} \cite[Proposition~4.7]{amy-cle-resampling}. We need to argue that the parameters can be chosen so that\\
(a) The event $\Elinkgff_{j}$ is guaranteed by the CLE configuration inside the annulus $A(0,2^{-j-1},2^{-j+1})$, and its probability can be made as close to $1$ as we like.\\
(b) The conditions for the event are insensitive under slight conformal perturbations.

In the setup of the proof of Lemma~\ref{lem:disconnect_interior} in \cite[Section~5.4]{amy-cle-resampling}, the condition~\eqref{it:sep_chain} is called a chain of loops intersecting in $A(0,2^{-j-1},2^{-j})$ whose inner strand separates $\Upsilon_{\wt\Gamma} \cap B(0,2^{-j-1})$ from $\partial B(0,2^{-j})$. The conditions~\eqref{it:link_prob}, \eqref{it:break_prob} correspond to the events in Proposition~\ref{prop:resampling} and Lemma~\ref{lem:break_loops}. The conditions~\eqref{it:swallow_ball}, \eqref{it:winding} hold almost surely for some $s,M>0$ and therefore can be made to hold with arbitrarily high probability by choosing the parameters suitably. Finally, we comment on the condition~\eqref{it:merging_prob}. To guarantee a positive merging probability $\plinkmerge$, we can consider the event that the flow lines from $-i2^{-j}$ follow a narrow tunnel that avoids coming close to the other loops. By the argument from \cite[Lemma~2.5]{mw2017intersections}, the probability of this event is positive, and is insensitive to the shape of $\wt{\eta}^{\theta_1}_{i_1,i_2}, \wt{\eta}^{\theta_2}_{i_1,i_2}$ inside $B(0,2^{-j-1})$.

We therefore have the following statement.

\begin{lemma}
Let $\wt{E}$ be an event that depends on $\wt{\eta} \cap B(0,\wt{\delta}^{1+5\innexp})$ and the internal metrics within $B(0,\wt{\delta}^{1+5\innexp}) \cap \ol{D_{\wt{\eta}}}$. Pick $J \in \{ \lceil\log_2(\wt{\delta}^{-1-4\innexp})\rceil,\ldots,\lfloor\log_2(\wt{\delta}^{-1-5\innexp})\rfloor \}$ uniformly at random. Let $\Elinkgff_{J}$ be as defined above. Then
\[ \p[\wt{E}] \lesssim \p[\wt{E} \cap \Elinkgff_{J}] + \wt{\delta}^b . \]
\end{lemma}

In the remainder of the proof, we use the following notation. Let $\wt{\Gamma}^+ \subseteq \wt{\Gamma}$ be the collection of loops that intersect the right boundary arc from $-i\wt{\delta}$ to $i\wt{\delta}$, and let $\wh{\Gamma} = \wt{\Gamma}\setminus\wt{\Gamma}^+$. Let $\wt{\eta}'$ be the branch of the exploration path from $i\wt{\delta}$ to $-i\wt{\delta}$ that discovers the loops in $\wt{\Gamma}^+$, and let $\wt{\eta}$ be its left outer boundary from $-i\wt{\delta}$ to $i\wt{\delta}$. Let $D_{\wt{\eta}}$ the region to the left of $\wt{\eta}$.

Suppose that we are on the event $\Elinkgff_{J}$. We now define events $\wt{E}_{i_1,i_2}$ for certain pairs of indices $i_1,i_2$ so that $\wt{G}^c \subseteq \bigcup_{i_1,i_2} \wt{E}_{i_1,i_2}$ where $\wt{G}$ is the event in the statement of Lemma~\ref{le:dist_across_chain}.

\begin{figure}[ht]
\centering
\includegraphics[width=0.45\textwidth]{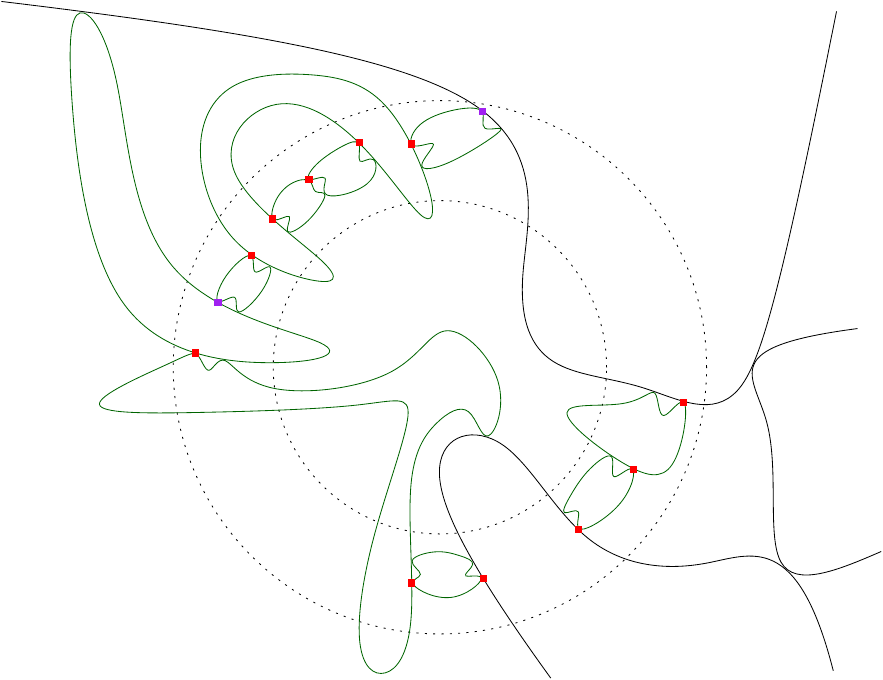}\hspace{0.05\textwidth}\includegraphics[width=0.45\textwidth]{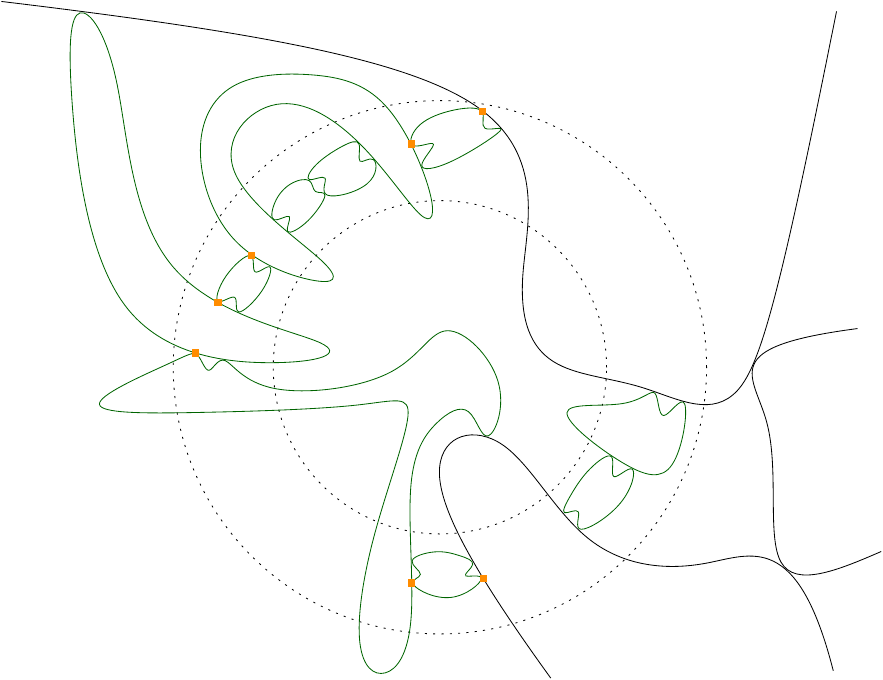}
\caption{Loops in $\wt{\Gamma}^+$ are shown in black, loops in $\wh{\Gamma}$ are shown in dark green. \textbf{Left:} Shown in purple is a pair $z_{i_1},z_{i_2}$ with $(i_1,i_2) \in \CI$. We will link the loops at the red points. \textbf{Right:} The orange points are the set $\{ z_i : i \in I'' \}$.}
\label{fi:allowed_linking}
\end{figure}

Let $(z_i)$ and $\CK$ be as in the definition of $\Elinkgff_{J}$. Consider the region in $D_{\wt{\eta}}$ that is disconnected from $\partial B(0,\wt{\delta})$ by $\CK$ and the points $z_i$. Let $U'$ be one connected component that is adjacent to $\wt{\eta}$ (equivalently $\wt{\Gamma}^+$). Let $\CK' \subseteq \CK$ be the loops in the chain that are adjacent to $\partial U'$, and let $I'$ be the indices such that $z_i \in \partial U'$. Let $\CI \subseteq \{1,\ldots,N^*\} \times \{1,\ldots,N^*\}$ be the set of pairs $(i_1,i_2)$ where $i_1,i_2 \in I'$ for some component $U'$ as above and such that the following hold (see the left picture of Figure~\ref{fi:allowed_linking}).
\begin{itemize}
 \item Let $\partial^1 U' \cup \partial^2 U' = \partial U'$ be the two boundary arcs divided at $z_{i_1},z_{i_2}$. Let $\CK'_1 \subseteq \CK'$ (resp.\ $\CK'_2 \subseteq \CK'$) be the loops adjacent to $\partial^1 U'$ (resp.\ $\partial^2 U'$). Then $\CK'_1 \cap \wt{\Gamma}^+ = \varnothing$ while $\CK'_2 \cap \wt{\Gamma}^+ \neq \varnothing$. 
 \item If there is a loop that belongs to both $\CK'_1$ and $\CK'_2$, its two corresponding segments adjacent to $\partial^1 U'$ (resp.\ $\partial^2 U'$) are not connected within $B(0,2^{-J+1})$.
 \item $z_{i_1},z_{i_2}$ are not separated from $\partial B(0,2^{-J+1})$ in $D_{\wt{\eta}}$ by the points $\{ z_i \in \partial^2 U' \} \setminus \{ z_{i_1},z_{i_2} \}$.
\end{itemize}

We claim that if $\metapproxacres{\epsilon}{U'}{z_{i_1}}{z_{i_2}}{\wh{\Gamma}} \le \median{\epsilon}+\epsilon^{-\epsexp}\ac{\epsilon}$ for all pairs $(i_1,i_2) \in \CI$ (and corresponding $U'$), then the event $\wt{G}$ in the statement of Lemma~\ref{le:dist_across_chain} occurs (with $\delta^{1+6\innexp}$ in place of $\delta^{1+\innexp}$).

It suffices to argue for each connected component $U'$ as above. Among the points $\{ z_i \in \partial U' \}$, let $z_{i_0}$ (resp.\ $z_{i_2}$) be the first (resp.\ last) one visited by $\wt{\eta}$. Let $\partial^L U'$ (resp.\ $\partial^R U'$) be the clockwise (resp.\ counterclockwise) arc of $\partial U'$ from $z_{i_0}$ to $z_{i_2}$. Note that the loops in $\CK$ incident to $\partial^L U'$ all belong to $\wh{\Gamma}$ and there is no single loop in $\CK$ adjacent to both $\partial^L U'$ and $\partial^R U'$ (otherwise $U'$ would not be one single connected component). Let $I''$ be the set of indices $i_1$ such that $z_{i_1} \in \partial^L U'$ and if there is a single loop in $\CK$ that is adjacent to both the arc from $z_{i_0}$ to $z_{i_1}$ and the arc from $z_{i_1}$ to $z_{i_2}$, then the two corresponding segments are not connected within $B(0,2^{-J+1})$ (see the right picture of Figure~\ref{fi:allowed_linking}). Then $(i_1,i_2) \in \CI$ for every $i_1 \in I'' \setminus \{ i_2 \}$.

Let $C_J$ be the clockwise arc of $\partial B(0,2^{-J+1})$ from the last point to the next point visited by $\wt{\eta}$. Then the points $\{ z_i : i \in I'' \}$ separate $U'$ from $C_J$ in $D_{\wt{\eta}}$. Recall from Step~1 that the loop configuration for $\wt{\Gamma}$ and the original $(\eta,\Gamma)$ are identical within $B(0,2^{-J+1})$ (more precisely, the components under the loops $\{\wt{\CL}_i\}$ adjacent to $\eta$). Therefore the points $\{ z_i : i \in I'' \}$ also separate $U'$ from $\partial D$ in the setup of Lemma~\ref{le:dist_across_chain}. This shows the claim above.

To finish the proof of Lemma~\ref{le:dist_across_chain}, we estimate the probability of $\metapproxacres{\epsilon}{U'}{z_{i_1}}{z_{i_2}}{\wh{\Gamma}} > \median{\epsilon}+\epsilon^{-\epsexp}\ac{\epsilon}$ using another resampling procedure, and take a union bound over $(i_1,i_2) \in \CI$.

Apply the resampling procedure described in Proposition~\ref{prop:resampling}. Let $\wt{\Gamma}^{\resampled}$ be the CLE arising from resampling in $A(0,2^{-J-1},2^{-J})$ and then in $A(0,2^{-J},2^{-J+1})$. Let $\wt{\Gamma}^{\resampled,+}$, $\wh{\Gamma}^{\resampled}$, $\wt{\eta}^{\prime,\resampled}$, $\wt{\eta}^{\resampled}$, $D_{\wt{\eta}^{\resampled}}$ be defined from $\wt{\Gamma}^{\resampled}$ analogously as for $\wt{\Gamma}$.

For $(i_1,i_2) \in \CI$ with corresponding $U'$, $\CK'$, we let $\Flinkgff_{i_1,i_2}$ be the event that the loops in $\CK'$ get hooked up together at all $z_i$, $i\in I' \setminus \{i_1,i_2\}$,\footnote{Strictly speaking, the CLE should be resampled in a small neighborhood of an intersection point close to $z_i$ but with positive distance to $U'$; see \cite[Section~5.4]{amy-cle-resampling} for details.} and further all crossings of $A(0,2^{-J},2^{-J+1})$ by loops in $\wh{\Gamma}$ are broken up. In other words, the loops in $\CK'_2$ are linked together to a single loop in $\wt{\Gamma}^{\resampled,+}$, and the loops in $\CK'_1$ are linked together to a single loop in $\wh{\Gamma}^{\resampled}$. When the breaking up of crossings is also successful, the condition on $\CI$ implies that $U'$ still lies to the left of $\wt{\eta}^{\prime,\resampled}$ (although it does not necessarily lie to the left of $\wt{\eta}^{\resampled}$).

On the event $\Flinkgff_{i_1,i_2}$, we couple the internal metrics for $\wh{\Gamma}$ and $\wh{\Gamma}^{\resampled}$ so that they agree on $\ol{U'}$. In particular,
\[
\p\left[ \metapproxacres{\epsilon}{U''}{z_{i_1}}{z_{i_2}}{\wh{\Gamma}^{\resampled}} > \median{\epsilon}+\epsilon^{-\epsexp}\ac{\epsilon} ,\, \Elinkgff_{J}, \Flinkgff_{i_1,i_2} \right] 
\ge \plinkdisc \, \p\left[ \metapproxacres{\epsilon}{U''}{z_{i_1}}{z_{i_2}}{\wh{\Gamma}} > \median{\epsilon}+\epsilon^{-\epsexp}\ac{\epsilon} ,\, \Elinkgff_{J} \right] .
\]
It remains to show
\begin{equation}\label{eq:bad_chain_linked}
\p\left[ \metapproxacres{\epsilon}{U''}{z_{i_1}}{z_{i_2}}{\wh{\Gamma}^{\resampled}} > \median{\epsilon}+\epsilon^{-\epsexp}\ac{\epsilon} ,\, \Elinkgff_{J}, \Flinkgff_{i_1,i_2} \right] 
= O(\wt{\delta}^\bestexp) .
\end{equation}

We couple $\wt{\Gamma}^{\resampled}$ with a GFF $h^{\resampled}$ (in a slightly different way than in Step~2). Let $\theta_1,\theta_2$ be the angles in~\eqref{eq:angles_intersection}. Let $h^{\resampled}$ be a GFF on $B(0,\wt{\delta})$ such that its angle $\theta_2-\pi/2$ counterflow line from $i\wt{\delta}$ to $-i\wt{\delta}$ is the exploration path $\wt{\eta}^{\prime,\resampled}$ which discovers the loops in $\wt{\Gamma}^{\resampled,+}$. Given $\wt{\eta}^{\prime,\resampled}$, let the loops in $\wh{\Gamma}^{\resampled}$ be discovered by a branching angle $\theta_1-3\pi/2$ counterflow line $\wt{\eta}^{\prime,\resampled}_1$ in each of the connected components to the left of $\wt{\eta}^{\prime,\resampled}$ (and are reflected upon $\wt{\eta}^{\prime,\resampled}$). Note the the loops in $\wt{\Gamma}^{\resampled,+}$ are traced by $\wt{\eta}^{\prime,\resampled}$ counterclockwise, and the loops in $\wh{\Gamma}^{\resampled}$ are traced by $\wt{\eta}^{\prime,\resampled}_1$ clockwise.

Let $B_{i_1}^+,B_{i_1}^-$ be as in~\eqref{it:swallow_ball} in the definition of the event $\Elinkgff_{J}$. On the event $\Flinkgff_{i_1,i_2}$, if $w_1 \in B_{i_1}^-$, $w_2 \in B_{i_1}^+$, then the flow lines $\eta^{\theta_1}_{w_1}, \eta^{\theta_2}_{w_2}$ will agree with $\wt{\eta}^{\theta_1}_{i_1,i_2}, \wt{\eta}^{\theta_2}_{i_1,i_2}$ from the definition of the event.

Let $\tau_J$ be the first time when $\wt{\eta}^{\prime,\resampled}$ hits $B(0,2^{-J+1})$. Conditionally on $\wt{\eta}^{\prime,\resampled}[0,\tau_J]$, the conditional law of $h^{\resampled}$ on $B(0,\wt{\delta}) \setminus \wt{\eta}^{\prime,\resampled}[0,\tau_J]$ is a GFF plus a harmonic function with flow line boundary values along $\wt{\eta}^{\prime,\resampled}[0,\tau_J]$. Due to the condition~\eqref{it:winding} in the event $\Elinkgff_J$, the winding heights of $h^{\resampled}$ along $\eta^{\theta_1}_{w_1}, \eta^{\theta_2}_{w_2}$ differ to the height of the harmonic function in $B(0,2^{-J})$ by at most $M$. In particular, the zero-boundary GFF has bounded winding height along $\eta^{\theta_1}_{w_1}, \eta^{\theta_2}_{w_2}$.

By absolute continuity, we can compare the law of the restriction of this GFF to $B(0,2^{-J})$ to that of a GFF $\wtwt{h}$ on $B(0,c2^{-J})$ where $c>1$ is a suitable constant. We pick the boundary values of $\wtwt{h}$ in a way that the flow lines $\wtwt{\eta}_1,\wtwt{\eta}_2$ with respective angles $\theta_1,\theta_2$ starting from $-ic2^{-J}$ are defined.

Consider the following event $\wt{E}$. Sample $w_1,w_2 \in B(0,2^{-J})$ according to Lebesgue measure, and let $\eta^{\theta_1}_{w_1}$, $\eta^{\theta_2}_{w_2}$ be the flow lines with respective angles stopped upon exiting $B(0,2^{-J})$. In case the two flow lines intersect with angle difference $\theta_1-\theta_2$, sample the internal metric in the region $U$ bounded between them. Let $\wt{E}$ be the event that $\metapproxacres{\epsilon}{U}{x}{y}{\Gamma_U} \gtrsim \median{\epsilon}+\epsilon^{-\epsexp}\ac{\epsilon}$ for some $x,y \in \eta^{\theta_1}_{w_1} \cap \eta^{\theta_2}_{w_2}$. Let $F$ be the event that given $\eta^{\theta_1}_{w_1}$, $\eta^{\theta_2}_{w_2}$ and the values of $\wtwt{h}$ along them, the conditional probability that $\wtwt{\eta}_1$, $\wtwt{\eta}_2$ merge into $\eta^{\theta_1}_{w_1}$, $\eta^{\theta_2}_{w_2}$ is at least $\plinkmerge > 0$ where $\plinkmerge$ is as in~\eqref{it:merging_prob}. Using the same argument as in the proof of Lemma~\ref{le:a_priori_int_fl} together with the assumption~\eqref{eq:a_priori_assumption} (or equivalently Lemma~\ref{le:a_priori_fl_general_boundary}), we see that $\p[ \wt{E} \cap F ] = O(\wt{\delta}^{\bestexp})$. 

We have argued previously that on the event $\{ \metapproxacres{\epsilon}{U''}{z_{i_1}}{z_{i_2}}{\wh{\Gamma}^{\resampled}} > \median{\epsilon}+\epsilon^{-\epsexp}\ac{\epsilon} \} \cap \Elinkgff_{J} \cap \Flinkgff_{i_1,i_2}$, the conditional probability given $\wt{\eta}^{\prime,\resampled}[0,\tau_J]$ that the restriction of a GFF to $B(0,2^{-J})$ satisfies $\wt{E} \cap F$ is at least $s^4$ (namely when $w_1 \in B_{i_1}^-$, $w_2 \in B_{i_1}^+$). This implies also
\[
\p\left[ \metapproxacres{\epsilon}{U''}{z_{i_1}}{z_{i_2}}{\wh{\Gamma}^{\resampled}} > \median{\epsilon}+\epsilon^{-\epsexp}\ac{\epsilon} ,\, \Elinkgff_{J}, \Flinkgff_{i_1,i_2} \mmiddle| \wt{\eta}^{\prime,\resampled}[0,\tau_J] \right] 
\lesssim \p[ \wt{E} \cap F ]^{1-o(1)}
= O(\wt{\delta}^\bestexp)
\]
and concludes~\eqref{eq:bad_chain_linked} and finally Lemma~\ref{le:dist_across_chain}.

\subsection{Proof of the bubble crossing exponent}
\label{se:bubble_exponent_conclude}

We now combine the results from the last two sections and complete the proof of Proposition~\ref{prop:bubble_crossing_exponent}. Consider the setup in Section~\ref{se:bubble_setup}, and recall the events $G_{M,a}$ and $E_\delta$ defined in the statement of Proposition~\ref{prop:bubble_crossing_exponent}.

We proceed as outlined below the statement of Proposition~\ref{prop:bubble_crossing_exponent}. Recall the collection of loops $\wh{\Gamma}^\delta_0,\wh{\Gamma}^\delta_2 \subseteq \Gamma$ defined in Section~\ref{subsubsec:independence_statement}. We will divide $\eta_2$ into several pieces, and estimate the $\Fd_\epsilon$-distance across each piece. We will distinguish several cases which we outline here to ease the reading:\\
1. The points $\eta_2(t)$ with small distance to $\eta_0 \cup \partial D$.\\
2. The points that are sufficiently far away from $\eta_0 \cup \partial D$ and lie beneath excursions of loops in $\wh{\Gamma}^\delta_2$.\\
3. The points that are sufficiently far away from $\eta_0 \cup \partial D$ and do not lie beneath such excursions. Here we will distinguish the cases when $\eta_2$ does or does not form a bottleneck around the point.

Throughout the proof, $c$ denotes a constant whose value may change from line to line. We suppose that $a>0$ is sufficiently small, and we fix $\excexp > 0$ (depending on $\epsexp$). Let $a'>0$ be a small constant. We will work on several ``good'' events which we define now.

Let $E_2$ denote the event from Lemma~\ref{lem:loop_fill_in_ball}. Then $\p[E_2^c] = o^\infty(\delta)$. Note also that on the event $E_2 \cap E_\delta$, there are at most $\delta^{-3\excexp}$ loops in $\wh{\Gamma}_0^\delta$, $\wh{\Gamma}_2^\delta$.

Let $E_3$ denote the event from Corollary~\ref{co:close_geodesic_small_parts} with $\delta_1 = 2\delta^{1+\excexp}$ (and $a_1=a$, say). Namely, $E_3$ is the event that (with the notation used in Section~\ref{subsec:map_in}) that there exists a finite collection of intervals $[s_i,t_i]$ such that $t \in \bigcup_i [s_i,t_i]$ for every $t$ with $\dist(\eta_2(t), \eta_0 \cup \partial D) \le 2\delta^{1+\excexp}$, and $\bigcap_i G^U_{s_i,t_i}(r_i)$ holds for some $(r_i)$ with
\begin{equation}\label{eq:narrow_parts}
\sum_i r_i \le M .
\end{equation}
Then $\p[E_3^c \cap G_{M,a}] = O(M\delta^{\bestexp+\zeta\excexp})$ for some constant $\zeta>0$ (provided that $a$ is sufficiently small depending on $\epsexp,\excexp$).

In other words, on the event $E_3$, we can get across the parts of $\eta_2$ that are at most $2\delta^{1+\excexp}$ far from $\eta_0 \cup \partial D$. In the remainder of the proof, we will focus on the parts of $\eta_2$ with distance at least $\delta^{1+\excexp}$ from $\eta_0 \cup \partial D$.

Let $\alpha_{4,\kappa} > 2$ be as in~\eqref{eqn:double_exponent_simple}, and let
\[ c_1 = \frac{8\alpha_{4,\kappa}}{\alpha_{4,\kappa}-2} . \]
Let $E_4$ denote the event that for any $s,t$ with $\eta_2(s) \in B(x_0,\delta)$, $\dist(\eta_2(s),\partial D) \ge \delta^{1+\excexp}$, and $\abs{\eta_2(s)-\eta_2(t)} \le \delta^{1+c_1\excexp}$ we have $\diam(\eta_2[s,t]) \le \delta^{1+4\excexp}$. By Proposition~\ref{pr:4arm_simple} and a union bound we have $\p[ E_4^c \cap E_\delta ] = O(\delta^{\zeta'\excexp})$ where $\zeta' = (c_1-4)\alpha_{4,\kappa}-a'-2c_1 > 3\alpha_{4,\kappa} > 0$. Applying Lemma~\ref{le:close_geodesic_old} to the region of $D \setminus \eta_2$ to the left of $\eta_2$, we see that
\[
\p[ \metapproxacres{\epsilon}{V}{x_0}{y_0}{\Gamma} \ge M\delta^{-ca}(\median{\epsilon}+\epsilon^{-\epsexp}\ac{\epsilon}) ,\, E_4^c \cap E_\delta \cap G_{M,a} ] = O(M\delta^{\bestexp+\zeta'\excexp}) .
\]
To conclude Proposition~\ref{prop:bubble_crossing_exponent}, we will show
\begin{equation}\label{eq:bubble_pf_good_case}
\p[ \metapproxacres{\epsilon}{V}{x_0}{y_0}{\Gamma} \ge M\delta^{-ca}(\median{\epsilon}+\epsilon^{-\epsexp}\ac{\epsilon}) ,\, E_2 \cap E_3 \cap E_4 \cap E_\delta \cap G_{M,a} ] = O(M^2\delta^{\bestexp+\alpha'}) .
\end{equation}
By the scaling argument described in Remark~\ref{rm:bubble_exponent_scaling}, this will conclude the proof of Proposition~\ref{prop:bubble_crossing_exponent}.

We apply Proposition~\ref{prop:intersect_loops_large} with $\excexp>0$. Let $D_i^\delta$, $i=0,2$, be as in~\eqref{eqn:d_i_delta_definition}. Then Proposition~\ref{prop:intersect_loops_large} implies that there exists $\alpha' > 0$ so that
\begin{equation}
\label{eqn:under_the_big_loops}
\p[ D_i^\delta \geq \median{\epsilon}+\epsilon^{-\epsexp}\ac{\epsilon} \ \text{for both}\ i=0,2] = O(\delta^{\bestexp + \alpha'}).
\end{equation}
To complete the proof, we need to estimate the probability that $\median{\epsilon}+\epsilon^{-\epsexp}\ac{\epsilon}$ is exceeded in any of the regions between the excursions of loops in $\wh{\Gamma}_i^\delta$.

Let $W$ be the $\delta^{1+\excexp}$-neighborhood of $\eta_0$ in the component of $D \setminus (\eta_0 \cup \eta_2)$ between $\eta_0,\eta_2$. Let $F_1 = G^W_{\wh{0},\wh{\infty}}(M\delta^{-ca})$. This event is measurable with respect to the internal metrics within $\ol{W}$, and on $F_1$ we have by Lemma~\ref{le:concatenate_overlapping_intervals}
\[ \metapproxacres{\epsilon}{V}{x_0}{y_0}{\Gamma} \lesssim M\delta^{-ca}(\median{\epsilon}+\epsilon^{-\epsexp}\ac{\epsilon}) . \]
By Lemma~\ref{le:close_geodesic} (and scaling), we have $\p[ F_1^c \cap E_\delta \cap G_{M,a} ] = O(M\delta^{\bestexp-ca-c\excexp/\epsexp})$.

It remains to argue that conditionally on $F_1^c$ it is still likely that the $\Fd_\epsilon$-distances in the neighborhood of $\eta_2$ are short. Let $\Gamma(\wt{W}) \subseteq \Gamma$ and $\wt{\eta}$ be as in the statement of Lemma~\ref{le:cle_sle_domain_markov}. Note that the connected component of $W$ adjacent to $\eta_0$ is contained in the regions bounded by $\eta_0$ and the loops in $\Gamma(\wt{W})$. Therefore, by the Markovian property of the metric, the conditional law of $\eta_2$ and the internal metrics in the regions bounded between $\eta_2,\wt{\eta}$ given $\Gamma(\wt{W})$, $\wt{\eta} \cap \eta_2$, $F_1^c$ is as described in Lemma~\ref{le:cle_sle_domain_markov}.

Pick $\wt{\eta}(\wt{\tau}), \wt{\eta}(\wt{\sigma}) \in \wt{\eta} \cap \eta_2$ as in the statement of Lemma~\ref{le:cle_sle_domain_markov} such that additionally one of the following hold.\\
(i) $\wt{\eta}(\wt{\tau})$ or $\wt{\eta}(\wt{\sigma})$ is at distance more than $2\delta^{1+\excexp}$ to $\eta_0$, and they lie on different loops of $\Gamma(\wt{W})$.\\
(ii) $\diam(\wt{\eta}[\wt{\tau},\wt{\sigma}]) \ge \delta^{1+c_1\excexp}$.\\
(iii) $\diam(\wt{\eta}[\wt{\tau},\wt{\sigma}]) \ge \epsilon$ and $\diam(\wt{\eta}[\wt{\sigma}',\wt{\tau}]) \le \cserial\epsilon$ (or $\diam(\wt{\eta}[\wt{\sigma},\wt{\tau}']) \le \cserial\epsilon$) for some $\wt{\sigma}'$ (or $\wt{\tau}'$) as in (i) or (ii) above.\\
Note that if $\CL \in \Gamma(\wt{W})$ is the loop containing the point $\wt{\eta}(\wt{\tau})$ (resp.\ $\wt{\eta}(\wt{\sigma})$) with distance more than $2\delta^{1+\excexp}$ to $\eta_0$, then (since $\CL$ also intersects $W$) we have $\diam(\CL) \ge \delta^{1+\excexp}$ and $\CL \in \wh{\Gamma}_2^\delta$. On the event $E_2 \cap E_\delta$, there are at most $\delta^{-3c_1\excexp}$ such choices of $\wt{\tau},\wt{\sigma}$ satisfying either of the conditions above. 

For each such pair $\wt{\tau},\wt{\sigma}$ and corresponding times $\tau,\sigma$ for $\eta_2$, we apply Lemma~\ref{le:close_geodesic} together with scaling (Lemma~\ref{le:median_scaling}) to the segment $\eta_2[\tau,\sigma]$ (conditionally on $\Gamma(\wt{W})$, $\wt{\eta} \cap \eta_2$, $F_1^c$, and $\eta_2 \setminus \eta_2[\tau,\sigma]$). We find an event $F_2$ with $\p[ F_2^c \mid F_1^c ] = O(M\delta^{\bestexp-ca-c\excexp/\epsexp})$ and such that on the event $F_2 \cap E_2 \cap E_\delta \cap G_{M,a}$ we have
\[
\metapproxacres{\epsilon}{V}{\eta_2(\tau)}{\eta_2(\sigma)}{\Gamma} \le M\delta^{-ca}\delta^{3c_1\excexp}(\median{\epsilon}+\epsilon^{-\epsexp}\ac{\epsilon})
\]
for every pair $\tau,\sigma$ as above, due to the monotonicity of the internal metrics (see the discussion below the statement of Lemma~\ref{le:dist_across_chain}), and hence
\begin{equation}\label{eq:between_big_loops}
\sum_{\tau,\sigma} \metapproxacres{\epsilon}{V}{\eta_2(\tau)}{\eta_2(\sigma)}{\Gamma} \le M\delta^{-ca}(\median{\epsilon}+\epsilon^{-\epsexp}\ac{\epsilon}) .
\end{equation}

Together we have
\begin{equation}\label{eq:both_sides_bad}
\p[ (F_1 \cup F_2)^c \cap E_2 \cap E_\delta \cap G_{M,a} ] = O(M^2 \delta^{2\bestexp-ca-c\excexp/\epsexp}) . 
\end{equation}
This controls the distances across all such segments $\eta_2[\tau,\sigma]$.

We argue that all parts of $\eta_2$ are covered by~\eqref{eq:narrow_parts}, \eqref{eqn:under_the_big_loops}, and~\eqref{eq:between_big_loops}. Suppose that $\dist(\eta_2(t), \eta_0 \cup \partial D) > 2\delta^{1+\excexp}$. Let $\tau,\sigma$ be two consecutive points of $\wt{\eta} \cap \eta_2$ with $\tau < t < \sigma$. In case $\diam(\wt{\eta}[\wt{\tau},\wt{\sigma}]) \ge \delta^{1+c_1\excexp}$, this interval is covered in~\eqref{eq:between_big_loops} with $\tau,\sigma$ as in~(ii). Otherwise, since we are on the event $E_4$, we have $\diam(\eta_2[\tau,\sigma]) \le \delta^{1+4\excexp}$. If $\eta_2(\tau),\eta_2(\sigma)$ lie on two different loops, then this interval is covered in~\eqref{eq:between_big_loops} with $\tau,\sigma$ as in~(i). If they lie on one single loop, then it is covered in~\eqref{eqn:under_the_big_loops} unless they lie between the points $a_{i,l}^\CL,\wt{a}_{i,l}^\CL$ or $b_{i,l}^\CL,\wt{b}_{i,l}^\CL$ defined in~\eqref{eqn:d_i_delta_definition} in which case they are covered in~\eqref{eq:between_big_loops} with $\tau,\sigma$ as in~(iii) or by the definition of $\metapproxac{\epsilon}{\cdot}{\cdot}{\Gamma}$ in~\eqref{eq:shortcutted_metric}.

\begin{figure}[ht]
\centering
\includegraphics[width=0.45\textwidth]{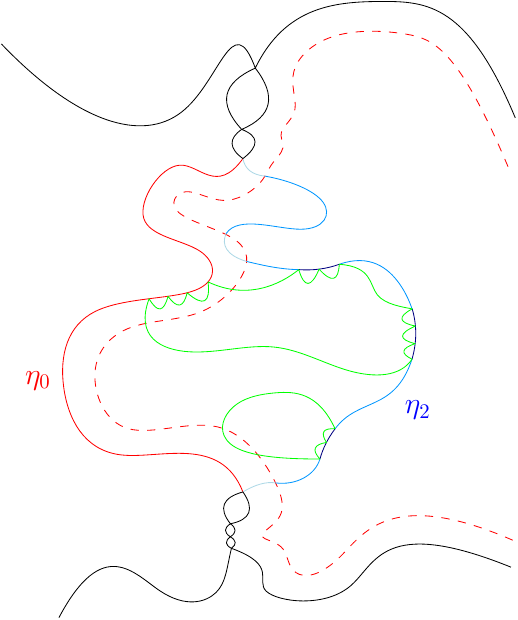}
\caption{In the proof of Proposition~\ref{prop:bubble_crossing_exponent}, we distinguish between the various parts of $\eta_2$ which are shown in different shades of blue.}
\label{fi:bubble_parts}
\end{figure}

To summarize, (on $E_4$) each point of $\eta_2$ satisfies at least one of the following (see Figure~\ref{fi:bubble_parts} for an illustration):
\begin{itemize}
\item $\dist(\eta_2(t), \eta_0 \cup \partial D) \le 2\delta^{1+\excexp}$,
\item $\eta_2(t)$ lies between some $a_{2,l}^\CL$, $b_{2,l}^\CL$ where $\CL \in \wh{\Gamma}_2^\delta$,
\item $\eta_2(t)$ lies between two loops in $\Gamma(\wt{W})$,
\item $\eta_2(t)$ lies between two intersection points $\wt{\eta}(\wt{\tau}), \wt{\eta}(\wt{\sigma})$ where $\diam(\wt{\eta}[\wt{\tau},\wt{\sigma}]) \ge \delta^{1+c_1\excexp}$.
\end{itemize}
By symmetry, the same estimates~\eqref{eq:narrow_parts}, \eqref{eq:both_sides_bad} apply with the roles of $\eta_0$ and $\eta_2$ swapped. We conclude that there is an event $F$ with $\p[F^c] = O(M^2\delta^{\bestexp+\alpha'})$ such that on the event $F \cap E_\delta \cap G_{M,a}$ the following hold.
\begin{itemize}
\item Either $D_0^\delta \le \median{\epsilon}+\epsilon^{-\epsexp}\ac{\epsilon}$ or $D_2^\delta \le \median{\epsilon}+\epsilon^{-\epsexp}\ac{\epsilon}$. Let us suppose the latter (the other case is the same by symmetry).
\item Either $F_1$ or $F_2$ occurs.
\begin{itemize}
\item On $F_1$ we have $\metapproxacres{\epsilon}{V}{x_0}{y_0}{\Gamma} \le M\delta^{-ca}(\median{\epsilon}+\epsilon^{-\epsexp}\ac{\epsilon})$.
\item On $F_2 \cap \{ D_2^\delta \le \median{\epsilon}+\epsilon^{-\epsexp}\ac{\epsilon} \} \cap E_3$ we can partition $\eta_2$ into segments $\tau_1 < \tau_2 < \cdots$ such that for each $k$ either
\begin{itemize}
 \item $(\eta_2(\tau_k),\eta_2(\tau_{k+1})) = (\wt{a}_{2,l}^\CL,\wt{b}_{2,l}^\CL)$ for some $\CL \in \wh{\Gamma}_2^\delta$ or 
 \item $(\tau_k,\tau_{k+1}) = (\tau,\sigma)$ for some $(\tau,\sigma)$ as in~\eqref{eq:between_big_loops} or 
 \item $s_i \le \tau_k < \tau_{k+1} \le t_i$ for some $(s_i,t_i)$ as in~\eqref{eq:narrow_parts} or
 \item $\dpath[\ol{V}](\eta_2(\tau_k),\eta_2(\tau_{k+1})) < \epsilon$.
\end{itemize}
\end{itemize}
\end{itemize}
By Lemma~\ref{le:concatenate_overlapping_intervals} we conclude that in either case we have
\[ \metapproxacres{\epsilon}{V}{x_0}{y_0}{\Gamma} \lesssim M\delta^{-ca}(\median{\epsilon}+\epsilon^{-\epsexp}\ac{\epsilon}) . \]
This proves~\eqref{eq:bubble_pf_good_case}.

\section{Proof of tightness}
\label{se:tightness_proof}

The purpose of this section is to complete the proof of the tightness statement in  Theorem~\ref{th:tightness_metrics}.  The proof will proceed using the following strategy.  First, in Section~\ref{subsec:tightness_at_the_boundary} we will prove a tightness result in the case that we restrict $\metapproxac{\epsilon}{\cdot}{\cdot}{\Gamma}$ to the boundary of a domain with $\SLE_\kappa$-type boundary which arises by considering an infinite volume setup.  Next, in Section~\ref{subsec:tightness_interior_flow_lines} we will use absolute continuity arguments in order to transfer the results of Section~\ref{subsec:tightness_at_the_boundary} to obtain tightness along the boundary of a domain that is formed using flow lines of the GFF starting from interior points.  We will then use these results in order to extend our tightness statements from the boundary to the interior in Section~\ref{subsec:interior_tightness}.  This will proceed by fragmenting our domain into successively smaller components that are bounded by finite collections of CLE loops. The boundaries of these regions look like flow lines starting from interior points, and by a resampling argument we can transfer the tightness results from Section~\ref{subsec:tightness_at_the_boundary} to the boundaries of these regions. Finally, in Section~\ref{subsec:thm1_proof} we will complete the proof of the tightness of the entire CLE metrics.

We assume that $\metapprox{\epsilon}{\cdot}{\cdot}{\Gamma}$ is an approximate \clekp{} metric in the sense defined in Section~\ref{se:assumptions}. Recall the metric $\metapproxacres{\epsilon}{V}{\cdot}{\cdot}{\Gamma}$ defined in~\eqref{eq:shortcutted_metric}. From now on, we will also assume that~\eqref{eq:eps_bound_ass} holds and we fix $\epsexp > 0$ such that
\begin{equation}\label{eq:epsexp2}
\ac{\epsilon} \le \epsilon^{2\epsexp}\median{\epsilon} .
\end{equation}
In particular, in all the statements from Sections~\ref{se:intersection_exponent} and~\ref{se:bubble_exponent} we have the bound
\[ \median{\epsilon}+\epsilon^{-2\epsexp}\ac{\epsilon} \le 2\median{\epsilon} . \]
Let us note that by the main result of Section~\ref{se:intersection_exponent}, all statements in the Sections~\ref{se:intersection_exponent_conclude} and~\ref{se:bubble_exponent} are valid for arbitrarily large $\bestexp > 0$.

\subsection{Tightness at the boundary}
\label{subsec:tightness_at_the_boundary}

We start by proving a tightness result when we restrict to points on the boundary of domains with $\SLE_\kappa$-type boundary conditions. It will be convenient to work in the half-planar setup as we have exact translation invariance. That is, we let $\eta$ be a two-sided whole-plane $\SLE_\kappa$ from $\infty$ to $\infty$ through $0$ normalized so that $\eta(0) = 0$ and equipped with the natural parameterization.  Let $\CH$ be the component of $\C \setminus \eta$ which is to the left of $\eta$ and let $\Gamma$ be a $\CLE_{\kappa'}$ in $\CH$.

The results from Sections~\ref{se:intersection_exponent} and~\ref{se:bubble_exponent} are used to show the following key lemma. Recall the definitions of $G^{U}_{s,t}(r)$, $G^{U}_{\wh{s},\wh{t}}(r)$, etc., above Lemma~\ref{le:concatenate_overlapping_intervals}.

\begin{lemma}
\label{lem:bdry_chaining_tail}
There exists $\zeta > 0$ such that the following is true. Assume~\eqref{eq:epsexp2}. Then
\begin{equation}
\label{eq:bdry_chaining_tail}
\p\left[ G^{\wt{U}}_{\wh{s},\wh{t}}\bigl(\abs{s-t}^\zeta+\epsilon^{\epsexp}\bigr)^c \right] = o^\infty(\abs{s-t}) \quad\text{as } \abs{s-t} \to 0 
\end{equation}
where $\wt{U} \subseteq \CH$ is the $\abs{s-t}^\zeta$-neighborhood of $\eta[s,t]$.
\end{lemma}

In particular, combining Lemma~\ref{lem:bdry_chaining_tail} with Lemma~\ref{le:concatenate_overlapping_intervals} we have
\[
\p\left[ \median{\epsilon}^{-1}\metapproxacres{\epsilon}{V}{\eta(s)}{\eta(t)}{\Gamma} \geq \abs{s-t}^\zeta+\epsilon^{\epsexp} \right] = o^\infty(\abs{s-t}) \quad\text{as } \abs{s-t} \to 0 
\]
for any $V \supseteq \wt{U}$.

\begin{proof}[Proof of Lemma~\ref{lem:bdry_chaining_tail}]
Throughout the proof, we fix a small constant $a>0$. By the stationarity of the two-sided whole-plane \slek{} (\cite[Corollary~4.7]{zhan-sle-loop}) and the translation-invariance of the metric, it suffices to assume $s=0$. We are going to show
\[
\p\left[ G^{\wt{U}}_{\wh{0},\wh{t}}\bigl(t^\zeta+\epsilon^{\epsexp}\bigr)^c \right] = O(t^b)
\]
for any $b>0$.

By~\eqref{eq:wpsle_ub}, we have $\p[ \diam(\eta[0,t]) \ge t^{1/\dsle-a} ] = o^\infty(t)$. In the remainder, we restrict to the event $\{ \diam(\eta[0,t]) \le t^{1/\dsle-a} \}$.

Let $\sigma = \sup\{ u \ge 0 : \abs{\eta(u)} = t^{1/\dsle-a} \}$. Note that $\eta[0,\sigma]$ is an \slek{} given its complement. Let $\wh{G}_{M,a}$ be the event that for every $0 \le s' < t' \le \sigma$ the domain $\C \setminus \eta([-\infty,s'] \cup [t',\infty])$ is $(M,a)$-good within $B(0,t^{1/\dsle-a})$. Then $\p[\wh{G}_{M,a}^c] = o^\infty(t)$ by Corollary~\ref{co:regularity_wpsle}.

Applying Corollary~\ref{co:close_geodesic_small_parts} with $\delta_1 = t^{1/\dsle-a}$, $\delta=\delta_1^{1/2}$ (and with sufficiently large $\bestexp$) we get
\[ \p\left[G^{\wt{U}}_{\wh{0},t}\bigl(t^\zeta+\epsilon^{\epsexp}\bigr)^c \cap \wh{G}_{M,a} \cap \{ t \le \sigma \} \right] = O(t^b) . \]
Similarly, applying the same argument to the time-reversal of $\eta$ from time $t$ to $\ol{\sigma} = \inf\{ u : \abs{\eta(u)-\eta(t)} = t^{1/\dsle-a} \}$, we see that
\[ \p\left[G^{\wt{U}}_{0,\wh{t}}\bigl(t^\zeta+\epsilon^{\epsexp}\bigr)^c \cap \wh{G}_{M,a} \cap \{ \ol{\sigma} \le 0 \} \right] = O(t^b) . \]
Since $G^{\wt{U}}_{\wh{0},t}\bigl(t^\zeta+\epsilon^{\epsexp}\bigr) \cap G^{\wt{U}}_{0,\wh{t}}\bigl(t^\zeta+\epsilon^{\epsexp}\bigr) \subseteq G^{\wt{U}}_{\wh{0},\wh{t}}\bigl(t^\zeta+\epsilon^{\epsexp}\bigr)$, this concludes the proof.
\end{proof}

\begin{proposition}
\label{prop:tightness_at_the_boundary}
There exists $\zeta > 0$ such that the following is true. Assume~\eqref{eq:epsexp2}. For $L>0$ let $G_\delta^L$ be the event that for each $s,t\in[0,1]$ with $\abs{s-t}\le\delta$ the event
\[ G^{\wt{U}^{L^{-\zeta}}_{s,t}}_{\wh{s},\wh{t}}\bigl( \delta^\zeta+\epsilon^{\epsexp}+L\delta \bigr) \]
occurs where $\wt{U}^{L^{-\zeta}}_{s,t} \subseteq \CH$ is the $L^{-\zeta}$-neighborhood of $\eta[s,t]$. Then
\[
 \p[ (G_\delta^L)^c ] = o^\infty(L^{-1}) \quad\text{as}\quad L \to \infty . 
\]
\end{proposition}

In particular, by Lemma~\ref{le:concatenate_overlapping_intervals} we have
\[ 
\p\left[ \sup_{\substack{s,t \in [0,1]\\ \abs{s-t} \le \delta}} \median{\epsilon}^{-1} \metapproxacres{\epsilon}{V}{\eta(s)}{\eta(t)}{\Gamma} \geq \delta^\zeta+\epsilon^{\epsexp}+L\delta \right] = o^\infty(L^{-1}) \quad\text{as}\quad L \to \infty 
\]
for any $V \supseteq \wt{U}^{L^{-\zeta}}_{0,1}$.

\begin{proof}[Proof of Proposition~\ref{prop:tightness_at_the_boundary}]
This follows from~\eqref{eq:bdry_chaining_tail} via a standard chaining argument. More precisely, if we let
\[ E_{j,k} = G^{\wt{U}^{2^{-k\zeta}}_{(j-1) 2^{-k}, j 2^{-k}}}_{\wh{(j-1) 2^{-k}}, \wh{j 2^{-k}}}(2^{-\zeta k}+\epsilon^{\epsexp}) , \]
then $\p[E_{j,k}^c] \lesssim 2^{-bk}$ for any $b>0$. 
In particular, for any $k_0 \in \N$ we have
\[
\p\left[ \bigcup_{k = k_0}^\infty \bigcup_{j=1}^{2^k} E_{j,k}^c \right] \lesssim 2^{-(b-1) k_0}.
\]
On the event $\bigcap_{k=k_0}^\infty \bigcap_{j=1}^{2^k} E_{j,k}$, for any $s,t \in [0,1]$ with $\abs{s-t} \le 2^{-k_0}$ we have
\[
 G^{\wt{U}^{2^{-k_0\zeta}}_{s,t}}_{\wh{s},\wh{t}}(c(\abs{s-t}^\zeta+\epsilon^{\epsexp})) .
\]
On the other hand, if $\abs{s-t} \ge 2^{-k_0}$, we get
\[
 G^{\wt{U}^{2^{-k_0\zeta}}_{s,t}}_{\wh{s},\wh{t}}\left( c\frac{\abs{s-t}}{2^{-k_0}}(2^{-\zeta k_0}+\epsilon^{\epsexp}) \right) 
 \subseteq G^{\wt{U}^{2^{-k_0\zeta}}_{s,t}}_{\wh{s},\wh{t}}( c2^{k_0}\abs{s-t} ) .
\]
\end{proof}

We will obtain tightness on the boundary in the setting of a bounded domain $D$ as a consequence of Proposition~\ref{prop:tightness_at_the_boundary}. For this, we consider the setup as described at the beginning of this subsection. Further, let $\eta'$ be the branch of the exploration tree of $\Gamma$ from $\infty$ to $0$ so that its force point (when standing at $\infty$ and looking towards $0$) is located infinitesimally to the right of $\infty$. Let $\eta^R$ be the right (when looking from $0$) outer boundary of $\eta'$.

Note that we can couple the objects described above with a whole-plane GFF as follows. Let $h^w$ be a whole-plane GFF, and let $h = h^w - \alpha \arg(\cdot)$ with values considered modulo $2\pi(\chi+\alpha)$ where $\alpha = \frac{\kappa \lambda}{2\pi} = \frac{\sqrt{\kappa}}{2}$. Let $\theta_1,\theta_2$ be as in~\eqref{eq:angles_intersection}. Then we can view $\eta|_{[0,\infty)}$ (resp.\ $\eta|_{(-\infty,0]}$) as being the flow line of $h$ with angle $\theta_2$ (resp.\ $\theta_2+\pi(1+\alpha/\chi)$) from $0$ to $\infty$. We can then view $\eta'$ as the counterflow line of $h$ from $\infty$ to $0$, and $\eta^R$ as the flow line of $h$ with angle $\theta_1 = -\pi/2$ from $0$ to $\infty$.

\begin{proposition}\label{pr:tightness_bubble_bdry}
There exists $\zeta > 0$ such that the following is true. Assume~\eqref{eq:epsexp2}. For $\delta>0$, let $E_\delta$ be the event that there exist $x,y\in \eta[0,\infty] \cap \eta^R$ such that if $U_{x,y}$ denotes the region bounded between the segments of $\eta[0,\infty]$ and $\eta^R$ from $x$ to $y$, then
\begin{itemize}
\item $U_{x,y} \subseteq B(0,1)$,
\item $\diam(U_{x,y}) \le \delta$,
\item $\dist(U_{x,y},\eta[-\infty,0]) \ge \delta^\zeta$, and
\item $\sup_{u,v \in \partial U_{x,y} \cap \eta} \median{\epsilon}^{-1} \metapproxacres{\epsilon}{U_{x,y}}{u}{v}{\Gamma} \ge L\delta^\zeta+\epsilon^{\epsexp}$.
\end{itemize}
Then $\p[E_\delta] = o^\infty(L^{-1}\delta)$.
\end{proposition}

\begin{proof}
Let $b>0$ be given. Let $F_1$ be the event that $\eta$ does not revisit $B(0,1)$ after time $\delta^{-b}$. Then $\p[F_1^c] = O(\delta^{\zeta b})$ for some $\zeta>0$ by \cite{fl-sle-transience}. Let $F_2$ be the event that $\diam(\eta[t,t+\delta^{\dsle-a}]) > \delta$ for any $t \in [0,\delta^{-b}]$. Then $\p[F_2^c] = o^\infty(\delta)$ by~\eqref{eq:wpsle_lb}.

On $E_\delta \cap F_1 \cap F_2$, we see that $\partial U_{x,y} \cap \eta \subseteq \eta[t,t+\delta^{\dsle-a}]$ for some $t \in [0,\delta^{-b}]$. By Proposition~\ref{prop:tightness_at_the_boundary} and a union bound we have
\[
\p\left[ \bigcup_{t\in[0,\delta^{-b}]} \bigcup_{s\in[t,t+\delta^{\dsle-a}]} G^{\wt{U}^{\delta^\zeta}_{t,s}}_{\wh{t},\wh{s}}(\delta^\zeta+\epsilon^{\epsexp}+L\delta^{1/2})^c \right] = O(L^{-b}\delta^b) .
\]
If $\dist(\wt{U}^{\delta^\zeta}_{t,s}, \eta[-\infty,0]) > 0$, then there cannot be two disjoint admissible paths within $\wt{U}^{\delta^\zeta}_{t,s}$ from any point in $\wt{U}^{\delta^\zeta}_{t,s} \setminus U_{x,y}$ to $U_{x,y}$ (since otherwise there would be a segment of $\eta^R$ whose incident loops cannot intersect $\eta[-\infty,0]$). Therefore $U_{x,y}$ satisfies the requirement for $U_1$ in Lemma~\ref{le:concatenate_overlapping_intervals}.
\end{proof}

\subsection{Tightness between interior flow lines}
\label{subsec:tightness_interior_flow_lines}

We now turn to establish tightness results on the boundary of collections of bubbles that are bounded between intersecting pairs of GFF flow lines.  The main result (Proposition~\ref{prop:tightness_in_nice_bubble}) is a bound on the $\Fd_\epsilon$-diameters of the boundaries of such bubbles in terms of their Euclidean diameter. This will be used in the next subsection to bound the distances between components that are formed by finite collections of $\CLE_{\kappa'}$ loops.  The strategy of the proof will be to use the results proved in the previous section in the half-planar setup and absolute continuity arguments in order to compare the interior flow line case to the half-planar case.

Throughout this subsection, we write
\begin{equation}\label{eq:bdry_diam}
D_\epsilon^{\partial U} = \sup_{u,v \in \partial U} \median{\epsilon}^{-1} \metapproxacres{\epsilon}{U}{u}{v}{\Gamma}.
\end{equation}

Consider the following setup. Let $\Gamma_\D$ be a nested $\CLE_{\kappa'}$ in $\D$, let $\CL$ be the outermost loop of $\Gamma_\D$ that surrounds~$0$, and let $D$ be the regions surrounded by $\CL$. Let $\Gamma_D$ be the loops of $\Gamma_\D$ contained in $\ol{D}$, and let $\Gamma = \{\CL\} \cup \Gamma_D$. Let $\Upsilon_\Gamma$ be the gasket of $\Gamma_D$.  Consider the exploration tree that starts at $-i$ and traces the outermost loops of $\Gamma_\D$ counterclockwise, and then the outermost loops of $\Gamma_D$ clockwise. For each $z \in \D$ we let $\eta_z'$ be the branch of the exploration tree from $-i$ to $z$. Suppose that $\Gamma$ is coupled with a GFF $h$ so that $\eta_z'$ is given by the counterflow line of $h$ starting at $-i$ and targeting $z$.

\begin{figure}[ht]
\centering
\includegraphics[width=0.5\textwidth]{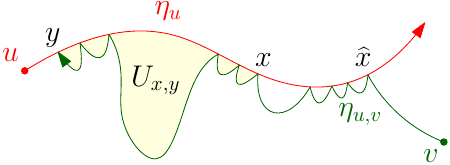}
\caption{The setup of Proposition~\ref{prop:tightness_in_nice_bubble}.}
\label{fi:tightness_int_fl_setup}
\end{figure}

\begin{proposition}
\label{prop:tightness_in_nice_bubble}
There exists $\zeta > 0$ such that the following is true. Fix $0<r<1$. 
Suppose that $h$ is a GFF on $\D$ with zero boundary conditions. Let $u,v \in B(0,r)$, and let $\eta_u$ be the flow line of $h$ starting from $u$. Given $\eta_u$, let $\eta_{u,v}$ be the flow line of $h$ starting from $v$ and reflected off $\eta_u$ in the opposite direction. Let $x,y \in \eta_{u,v} \cap \eta_u$ be chosen such that the segment of $\eta_{u,v}$ from $v$ to $y$ intersects $\eta_u$ only on the right side of $\eta_u$ with angle difference $0$ (before reflecting off), and $x$ is on the segment between $v$ and $y$. Let $U_{x,y}$ be the region bounded between the segments of $\eta_{u,v}$ and $\eta_u$ from $x$ to $y$. Let $\Gamma_{x,y}$ be the collection of conditionally independent $\CLE_{\kappa'}$ naturally coupled with $h$ in each of the components of $U_{x,y}$. Let $\Upsilon_{\Gamma_{x,y}}$ denote the gasket of $\Gamma_{x,y}$. 

Let $E_\delta$ be the event that there exist $u,v,x,y$ as above such that
\begin{itemize}
 \item $U_{x,y} \subseteq B(0,r)$,
 \item $\Upsilon_{\Gamma_{x,y}} \subseteq \Upsilon_\Gamma$,
 \item $\diam(U_{x,y}) \le \delta$,
 \item $D_\epsilon^{\partial U_{x,y}} \ge \delta^\zeta+\epsilon^{\epsexp}$.
\end{itemize}
Let $\wh{x}$ be the first point where $\eta_{u,v}$ hits $\eta_u$. Fix $c_1>0$, and let $E^1_\delta$ be the event that $E_\delta$ occurs and additionally the segment of $\eta_u$ from $x$ to $\wh{x}$ has diameter at least $\delta^{1+c_1}$.
Then
\[ \p[E^1_\delta] = o^\infty(\delta) \quad\text{as } \delta \searrow 0 . \]
\end{proposition}

We will prove Proposition~\ref{prop:tightness_in_nice_bubble} by reducing it to Proposition~\ref{pr:tightness_bubble_bdry}. This is done in several steps. In each of the following lemmas, we consider the boundaries of regions bounded by flow lines of various kinds. Recall the double point angle $\angledouble$ defined in~\eqref{eq:angle_double}.

\begin{lemma}\label{le:tightness_bdry_bubble_fl}
There exists $\zeta > 0$ such that the following is true. 
Fix $M >0$, $\innexp>0$. Let $\delta^{1-\innexp} \le r \le 1$, and let $h$ be a GFF on $r\D$ with some boundary values that are bounded by $M$ and so that the angle $\angledouble$ (resp.\ $0$) flow lines $\eta_1,\eta_2$ from $-ir$ are defined. Given $\eta_1,\eta_2$, sample a \clekp{} and the internal metric in the region bounded between them. Let $E_\delta$ be the event that there exist $x,y \in \eta_1 \cap \eta_2$ such that if $U$ denotes the region bounded between the segments of $\eta_1$ and $\eta_2$ from $x$ to $y$, then $U \subseteq B(0,3r/4)$, $\diam(U) \le \delta$, and
\[ D_\epsilon^{\partial U} \ge \delta^\zeta+\epsilon^{\epsexp} . \]
Then $\p[E_\delta] = o^\infty(\delta)$ as $\delta\searrow 0$.
\end{lemma}

This will follow from the next lemma in the same way Lemma~\ref{le:a_priori_disc} follows from Lemma~\ref{le:a_priori_int_fl}. See Figure~\ref{fi:intersections_localised} for an illustration of the setup.

\begin{lemma}
\label{le:tightness_bdry_good_fl}
There exists $\zeta>0$ such that the following is true. Fix $M,p > 0$, $r \in (0,1)$, and suppose $\delta < r$. Let $h$ be a GFF on $r\D$ with boundary values as in Section~\ref{se:intersections_setup}, and let $\theta_1,\theta_2$ be as in~\eqref{eq:angles_intersection}. Let $w_1,w_2 \in B(0,r/4)$, and let $E_\delta$ denote the event that the following hold.
\begin{itemize}
\item Let $\eta_{w_1}^{\theta_1+\pi}$, $\eta_{w_1}^{\theta_1}$, $\eta_{w_2}^{\theta_2}$, $\eta_{w_2}^{\theta_2-\pi}$ be the flow lines with respective angles starting at $w_1$ (resp.\ $w_2$) and stopped upon exiting $B(0,3r/4)$. Then
\begin{itemize}
\item The right side of $\eta_{w_1}^{\theta_1}$ intersects the left side of $\eta_{w_2}^{\theta_2}$ with angle difference $\theta_1-\theta_2$, and they do not intersect in any other way.
\item $\eta_{w_1}^{\theta_1+\pi}$ does not intersect $\eta_{w_2}^{\theta_2}$, $\eta_{w_2}^{\theta_2-\pi}$; and $\eta_{w_2}^{\theta_2-\pi}$ does not intersect $\eta_{w_1}^{\theta_1+\pi}$, $\eta_{w_1}^{\theta_1}$.
\item Let $U$ be the region bounded between $\eta_{w_1}^{\theta_1}, \eta_{w_2}^{\theta_2}$, and sample the internal metric in $U$. There exist $x,y \in \eta^{\theta_1}_{w_1} \cap \eta^{\theta_2}_{w_2}$ such that if $U_{x,y} \subseteq U$ denotes the region between $x$ and $y$, then $\diam(U_{x,y}) \le \delta$ and
\[ D_\epsilon^{\partial U_{x,y}} \ge \delta^\zeta+\epsilon^{\epsexp} . \]
\end{itemize}
\item The event $G_{0,r}$ from the statement of Lemma~\ref{le:good_scales_merging} occurs for $h$.
\end{itemize}
Then
\[ \p[E_\delta] = o^\infty(\delta) \]
uniformly in the choice of $w_1,w_2$.
\end{lemma}

\begin{proof}
\textbf{Step 1.} We first prove the result in the case $r=\delta^\zeta$ where $\zeta$ is the constant in Proposition~\ref{pr:tightness_bubble_bdry}.

By symmetry, it suffices to bound
\[ \sup_{u,v \in \partial U_{x,y} \cap \eta_{w_2}^{\theta_2}} \median{\epsilon}^{-1} \metapproxacres{\epsilon}{U_{x,y}}{u}{v}{\Gamma}.\]

The proof is similar to Step~5 of Lemma~\ref{le:a_priori_int_fl}, so we shall be brief. Let $\eta$ be a two-sided whole-plane \slek{} and let $F$ be the event that there exists a point $z$ on the hyperbolic geodesic in $(\C\setminus\eta[-\infty,0], 0, \infty)$ such that $\abs{z} \asymp \delta^\zeta$ and $\dist(z,\eta[-\infty,0]) \asymp \delta^\zeta$. By scale-invariance, $\p[F] > 0$ does not depend on $\delta$.

On the event $F$, let $\varphi$ be the conformal transformation from $(\D,-i,i)$ to $(\C\setminus\eta[-\infty,0], 0, \infty)$ with $\varphi(0) = z$. Let $c_0>0$ be sufficiently small such that $\abs{(\varphi^{-1})'(w)}/\abs{(\varphi^{-1})'(z)} \in [0.99,1.01]$ for all $w \in B(z,c_0\delta^\zeta)$. Let $h$ be a GFF in $\C\setminus\eta[-\infty,0]$ with boundary values such that its angle $\theta_2$ flow line is $\eta[0,\infty]$ and its counterflow line is $\eta'$. Let also $\eta'_{\theta_2-\pi/2}$ be the angle $\theta_2-\pi/2$ counterflow line so that its left boundary agrees with $\eta[0,\infty]$.

Let $\wt{E}$ be the event that the event in the lemma statement occurs for the restriction of $h$ to $B(z,c_0\delta^\zeta)$. Then $\p[\wt{E} \mid F] \gtrsim \p[E_{c_0\delta}]^{1+a}$ by absolute continuity and Lemma~\ref{le:abs_cont_kernel}. By the same argument as in Step~5 of the proof of Lemma~\ref{le:a_priori_int_fl}, the conditional probability given $\wt{E} \cap F$ that the counterflow lines $\eta'$, $\eta'_{\theta_2-\pi/2}$ merge into the objects involved in the event $\wt{E}$ is bounded from below by some $\wt{p}>0$.

This implies that if $E$ is the event from Proposition~\ref{pr:tightness_bubble_bdry}, then
\[ \p[E \mid \wt{E} \cap F] \ge \wt{p} \]
and hence
\[
\p[E_{c_0\delta}] \lesssim \p[E]^{1/(1+a)} = o^\infty(\delta) .
\]

\textbf{Step 2.} We now prove the result in the case $r=\delta$. Suppose $\wt{h}$ is a whole-plane GFF modulo additive constant $2\pi\chi$, and let $\wt{h}_{0,\delta}$ be as in Section~\ref{se:gff}. Let $\wt{E}_\delta$ denote the event $E_\delta$ occurring for $\wt{h}_{0,\delta}$ (with $r=\delta$) and let $G_{0,\delta}$ be the event from Lemma~\ref{le:good_scales_merging} for $\wt{h}_{0,\delta}$. Then $\p[E_\delta] \lesssim \p[\wt{E}_\delta]^{1+o(1)}$ by absolute continuity (and Lemma~\ref{le:abs_cont_kernel}). We are going to show $\p[\wt{E}_\delta] = o^\infty(\delta)$.

For the whole-plane GFF $\wt{h}$, we can extend the flow lines until they exit $B(0,1)$; let us denote them by $\wt{\eta}_{w_1}^{\theta_1+\pi}$, $\wt{\eta}_{w_1}^{\theta_1}$, $\wt{\eta}_{w_2}^{\theta_2}$, $\wt{\eta}_{w_2}^{\theta_2-\pi}$. On the events $\wt{E}_\delta$ and $G_{0,\delta}$, the conditional probability that $\wt{\eta}_{w_1}^{\theta_1+\pi}$ does not intersect $\wt{\eta}_{w_2}^{\theta_2}$, $\wt{\eta}_{w_2}^{\theta_2-\pi}$ and $\wt{\eta}_{w_2}^{\theta_2-\pi}$ does not intersect $\wt{\eta}_{w_1}^{\theta_1+\pi}$, $\wt{\eta}_{w_1}^{\theta_1}$ is at least $\delta^c$ where $c$ is a constant depending on $M,p$.

Let $b>0$ be given. Let $G_{0,r}$ be the event from Lemma~\ref{le:good_scales_merging} for $\wt{h}$, and $F$ the event that at least $9/10$ fraction of $G_{0,2^{-j}}$, $j \in \{ \lceil\log_2(\delta^{-\zeta/2})\rceil,\ldots,\lfloor\log_2(\delta^{-\zeta})\rfloor \}$ occur. We can pick $M,p$ such that $\p[F^c] = O(\delta^b)$.

Now sample $J \in \{ \lceil\log_2(\delta^{-\zeta/2})\rceil,\ldots,\lfloor\log_2(\delta^{-\zeta})\rfloor \}$ uniformly at random. Let $\wt{E}_{2^{-J}}$ be the event in the lemma statement occurring for $\wt{h}_{0,2^{-J}}$ for $\delta$ and $r=2^{-J}$. We have just argued that
\[ \p[\wt{E}_\delta] \lesssim \p[F^c]+\p[\wt{E}_\delta \cap F] \lesssim \delta^b+\delta^{-c}\p[\wt{E}_{2^{-J}} \cap G_{0,2^{-J}}] . \]
Now we can apply the result from Step~1 and conclude $\p[\wt{E}_{2^{-J}} \cap G_{0,2^{-J}}] = o^\infty(\delta)$.

\textbf{Step 3.} The result for general $r>\delta$ follows from the result for $r=\delta$. In fact, when $r>\delta^{1-\innexp}$ for fixed $\innexp>0$, the result holds also without requiring the event $G_{0,r}$. The proof of this is exactly the same as for Lemma~\ref{le:tightness_bdry_bubble_fl}, and we will refer the reader to that proof below. (Just note that for the statements are identical except that the flow lines come from the interior vs.\ from the boundary. This does not change anything in the proof.)
\end{proof}

\begin{proof}[Proof of Lemma~\ref{le:tightness_bdry_bubble_fl}]
The proof is analogous to Lemma~\ref{le:a_priori_disc}, so we shall we brief. Let $b>0$ be given.

Let $G_{z,r}$ be the event from Lemma~\ref{le:good_scales_merging}. Let $F^1$ be the event that for every $z \in B(0,3r/4) \cap \delta\Z^2$ at least $1/3$ fraction of $G_{z,2^{-j}}$, $j = \lceil\log_2(\delta^{-1+\innexp})\rceil,\ldots,\lfloor\log_2(\delta^{-1})\rfloor$ occur, so that $\p[(F^1)^c] = O(\delta^b)$ for suitable $M,p$. Further, let $F^2$ be the event that space-filling SLE fills a ball of radius $\delta^{2}$ whenever it travels distance $\delta$ within $B(0,3r/4)$. By Lemma~\ref{le:fill_ball} we have $\p[(F^2)^c] = o^\infty(\delta)$.

It therefore suffices to show $\p[E_\delta \cap F^1 \cap F^2] = o^\infty(\delta)$.

Let $z \in B(0,3r/4) \cap \delta\Z^2$ and $J \in \{ \lceil\log_2(\delta^{-1+\innexp})\rceil,\ldots,\lfloor\log_2(\delta^{-1})\rfloor \}$ be sampled uniformly at random, and let $\wt{w}_1,\wt{w}_2 \in B(0,2^{-J}/4)$ be sampled according to Lebesgue measure (independently of $h$ and $\metapproxacres{\epsilon}{U}{\cdot}{\cdot}{\Gamma}$). Then, as argued in the proof of Lemma~\ref{le:a_priori_disc}, we have
\[
\p[E_{z,J} \cap G_{z,2^{-J}}] \ge \delta^{10}\p[E_\delta \cap F^1 \cap F^2] .
\]
where $E_{z,J}$ denote the event described in Lemma~\ref{le:tightness_bdry_good_fl} occurring for $\wt{h}_{z,2^{-J}}$ and $\wt{w}_1,\wt{w}_2$. By absolute continuity we have $\p[E_{z,J} \cap G_{z,2^{-J}}] = o^\infty(\delta)$. This proves the result.
\end{proof}

\begin{lemma}\label{le:tightness_bdry_fl_refl}
There exists $\zeta > 0$ such that the following is true. 
Fix $M >0$, $\innexp>0$. Let $\delta^{1-\innexp} \le r \le 1$, and let $h$ be a GFF on $r\D$ with some boundary values that are bounded by $M$ and so that the flow line $\eta_1$ from $-ir$ to $ir$ and the flow line $\ol{\eta}_2$ from $ir$ to $-ir$ in the components of $\D \setminus \eta_1$ to the right of $\eta_1$ (and reflected off $\eta_1$) are defined. Given $\eta_1,\ol{\eta}_2$, sample a \clekp{} and the internal metric in the region bounded between them. Let $E_\delta$ be the event that there exist $x,y \in \eta_1 \cap \ol{\eta}_2$ such that if $U$ denotes the region bounded between the segments of $\eta_1$ and $\ol{\eta}_2$ from $x$ to $y$, then $U \subseteq B(0,3r/4)$, $\diam(U) \le \delta$, and
\[ D_\epsilon^{\partial U} \ge \delta^\zeta+\epsilon^{\epsexp} . \]
Then $\p[E_\delta] = o^\infty(\delta)$ as $\delta\searrow 0$.

The same is true when $\ol{\eta}_2$ is the (reflected) flow line in the components to the left of $\eta_1$.
\end{lemma}

\begin{proof}
In case the boundary values of $h$ are as in Lemma~\ref{le:reflected_fl_law}, this follows from Lemma~\ref{le:tightness_bdry_bubble_fl} by reversal symmetry. To transfer the result to the case of general boundary values, we can apply a similar argument as for Lemma~\ref{le:tightness_bdry_bubble_fl} where here we use Lemma~\ref{le:good_scales_merging_refl} and Lemma~\ref{le:tightness_bdry_good_fl_refl} below. (See also the first part of the proof of Lemma~\ref{le:a_priori_loop_intersecting_one_side} where the same argument is used.)
\end{proof}

\begin{lemma}\label{le:tightness_bdry_good_fl_refl}
There exists $\zeta>0$ such that the following is true. Fix $M,p > 0$, $r \in (0,1)$, and suppose $\delta < r$. Let $h$ be a GFF on $r\D$ with boundary values as in Lemma~\ref{le:reflected_fl_law}. Let $\eta_1,\ol{\eta}_2$ be as defined in Lemma~\ref{le:reflected_fl_law}. Let $z_1 \in \partial B(0,r/2)$, $w_1 \in B(0,3r/16)$, and let $E_\delta$ denote the event that the following hold. 
Let $\eta_{z_1}$ be the flow line starting at $z_1$ and stopped upon exiting $B(0,3r/4)$. Let $\eta_{w_1}$ be the flow line starting at $w_1$, reflected off $\eta_{z_1}$ in the opposite direction, and stopped upon exiting $B(0,r/4)$. Then
\begin{itemize}
 \item $\eta_{w_1}$ intersects the right side of $\eta_{z_1}$ with angle difference $0$ (before reflecting off), and they do not intersect in any other way.
 \item Let $U$ be the region bounded between $\eta_{w_1}, \eta_{z_1}$, and sample the internal metric in $U$. There exist $x,y \in \eta_{w_1} \cap \eta_{z_1}$ such that if $U_{x,y} \subseteq U$ denotes the region between $x$ and $y$, then $\diam(U_{x,y}) \le \delta$ and
\[ D_\epsilon^{\partial U_{x,y}} \ge \delta^\zeta+\epsilon^{\epsexp} . \]
 \item The angle $\pi$ flow line $\eta^{\pi}_{w_1}$ starting from $w_1$ exits $B(0,3r/4)$ without intersecting $\eta_{z_1}$.
 \item The conditional probability given $\eta_{z_1}$, $\eta_{w_1}$, $\eta^{\pi}_{w_1}$, $X_r^0$, $X_r^{\pi}$, and the values of $h$ on these sets is at least $p$ that
 \begin{itemize}
  \item $\eta_1$ merges into $\eta_{z_1}$ before entering $B(0,r/4)$.
  \item The continuation of $\eta^{\pi}_{w_1}$ does not intersect $\eta_1$ before it hits $\partial B(0,r)$.
 \end{itemize}
\end{itemize}
Then
\[ \p[E_\delta] = o^\infty(\delta) \]
uniformly in the choice of $z_1,w_1$.
\end{lemma}

\begin{proof}
 This follows from the case of Lemma~\ref{le:tightness_bdry_fl_refl} when the boundary values of $h$ are as in Lemma~\ref{le:reflected_fl_law}. Indeed, if we let $F_\delta$ denote the event from Lemma~\ref{le:tightness_bdry_fl_refl}, then $\p[F_\delta] \ge p\p[E_\delta]$ (this is the exact same argument as in the proof of Lemma~\ref{lem:interior_intersection_left}).
\end{proof}

We turn towards proving Proposition~\ref{prop:tightness_in_nice_bubble}. We need to distinguish bubbles that are small compared to their distance to $u,v$, and the bubbles that are big. We consider the small bubbles in the following lemma.

\begin{lemma}\label{le:tightness_long_fl}
 Consider the setup of Proposition~\ref{prop:tightness_in_nice_bubble}. Fix $\innexp>0$. Let $E'_\delta$ be the event defined similarly to $E_\delta$ with the additional requirement $\dist(U_{x,y},\{u,v\}) \ge \delta^{1-\innexp}$. Then $\p[E'_\delta] = o^\infty(\delta)$ as $\delta \searrow 0$.
\end{lemma}

\begin{proof}
 This follows from Lemma~\ref{le:tightness_bdry_good_fl_refl} and Lemma~\ref{le:good_scales_merging_refl}. Indeed, let $b>0$, and let $G_{z,r}$ be the event from Lemma~\ref{le:good_scales_merging_refl} (with suitable parameters $M,p$). Let $G$ be the event that for each $z \in \delta\Z^2 \cap B(0,r)$ the fraction of $j \in \{ \lceil\log_2(\delta^{-1+\innexp})\rceil,\ldots,\lfloor\log_2(\delta^{-1})\rfloor \}$ where $G_{z,2^{-j}}$ occurs is at least $9/10$. Then we can pick $M,p$ such that $\p[G^c] = O(\delta^b)$.
 
 Sample $z \in \delta\Z^2 \cap B(0,r)$ and $J \in \{ \lceil\log_2(\delta^{-1+\innexp})\rceil,\ldots,\lfloor\log_2(\delta^{-1})\rfloor \}$ uniformly at random. Note that on the event that $U_{x,y} \subseteq B(z,\delta)$ both flow lines $\eta_u$, $\eta_{u,v}$ begin and end outside $B(z,2^{-J})$. Therefore, if we denote by $E'_{z,2^{-J}}$ the event from Lemma~\ref{le:tightness_bdry_good_fl_refl} occurring for $\wt{h}_{z,2^{-J}}$, then by translation invariance and absolute continuity
 \[ \p[E'_\delta \cap G] \lesssim \delta^{-2}\p[E'_{z,2^{-J}} \cap G_{z,2^{-J}}] = o^\infty(\delta) . \]
\end{proof}

The next lemma deals with the big bubbles.

\begin{lemma}\label{le:tightness_bdry_big_bubbles}
 Consider the setup of Proposition~\ref{prop:tightness_in_nice_bubble}. Let $E''_\delta$ be the event defined similarly to $E^1_\delta$ with the additional requirement that $U_{x,y}$ is a single connected component of $\D\setminus (\eta_{u,v}\cup\eta_u)$. Then $\p[E''_\delta] = o^\infty(\delta)$ as $\delta \searrow 0$.
\end{lemma}

The challenge here is that the flow lines may not be long enough to immediately give us sufficient merging probabilities as in the proof of Lemma~\ref{le:tightness_long_fl}. To prove Lemma~\ref{le:tightness_bdry_big_bubbles}, we will employ the following trick. We consider a flow line $\eta_{-i}$ coming from $\partial B(0,1)$. We argue that for any $U = U_{x,y}$ we can pick the flow lines $\eta_u$, $\eta_{u,v}$ (stopped at suitable stopping times) so that with conditional probability at least some power of $\delta$, the flow line $\eta_{-i}$ hits $\partial U$. When this happens, it merges into $\eta_u$ or $\eta_{u,v}$. Suppose $y$ is the last point on $\partial U$ visited by $\eta_{u,v}$. If we do not reflect $\eta_{u,v}$ at $y$, it will merge into $\eta_u$. This means that $\eta_{-i}$ and $\eta_{-i,v}$ together also detect the same bubble $U$. See Figure~\ref{fi:big_bubble} for an illustration.

\begin{figure}[ht]
\centering
\includegraphics[width=0.5\textwidth]{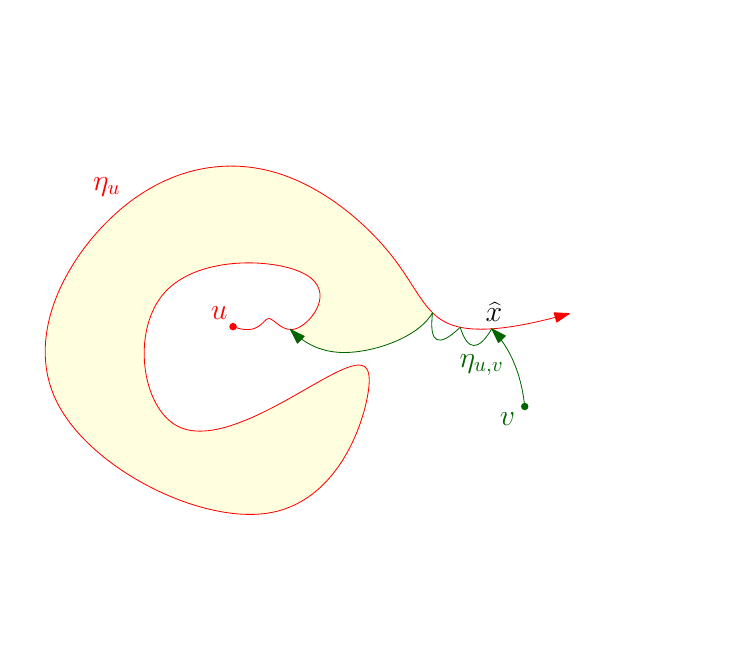}\includegraphics[width=0.5\textwidth]{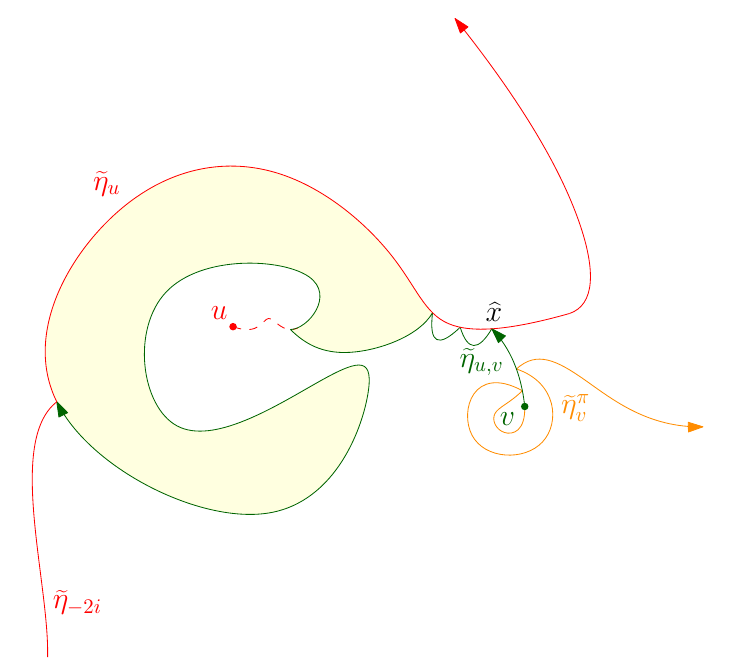}
\caption{The setup and proof of Lemma~\ref{le:boundary_normal_big_bubble}. If the right side of $\wh{\eta}_{u,v}$ does not have too small harmonic measure, then there is a sufficiently high chance that the bubble can be detected and compared to the setup of Lemma~\ref{le:tightness_bdry_good_fl_refl}.}
\label{fi:big_bubble}
\end{figure}

We split the proof of Lemma~\ref{le:tightness_bdry_big_bubbles} into several steps. For the next few lemmas, we consider the following setup. Let $u,v \in B(0,r)$. Consider a flow line $\eta_u$ starting from $u$ and a stopping time for $\eta_u$ that occurs before exiting $B(0,r)$. (We will abstain from introducing an extra notation for the stopping time, and directly assume that $\eta_u$ is stopped at that stopping time.) Given $\eta_u$, let $\eta_{u,v}$ be the flow line starting from $v$ and reflected off $\eta_u$ in the opposite direction. Suppose that $\eta_{u,v}$ intersects $\eta_u$ only on the right side of $\eta_u$ with angle difference $0$ (before reflecting off). Stop $\eta_{u,v}$ at the first time when there is a connected component $U$ of $\C \setminus (\eta_{u,v} \cup \eta_u)$ bounded between $\eta_{u,v}$ and $\eta_u$ with $\diam(U) \in [\delta/2,\delta]$. Sample the internal metric in $U$.

Let $E_{u,v}$ be the event that the above occurs and
\begin{itemize}
 \item $\eta_u \cup \eta_{u,v} \subseteq B(0,r)$,
 \item the segment of $\eta_{u,v}$ before tracing $\partial U$ and the segment of $\eta_u$ after tracing $\partial U$ both have diameter smaller than $\diam(U)$,
 \item $D_\epsilon^{\partial U} \ge \delta^\zeta+\epsilon^{\epsexp}$.
\end{itemize}
(The dependence on the stopping time for $\eta_u$ is implicitly assumed.)

We will bound the probability of $E_{u,v}$ under some additional constraints. Let $\wh{x}$ be the first intersection point of $\eta_{u,v}$ with $\eta_u$. Let $\eta_{\wh{x}}$ denote the segment of $\eta_u$ from $\wh{x}$ to the point where it is stopped. Let $\wh{\eta}_{u,v}$ denote the segment of $\eta_{u,v}$ from $v$ to $\wh{x}$.

\begin{lemma}\label{le:boundary_normal_big_bubble}
 Consider the setup described in the paragraphs above. Fix $c_0>0$. Let $F_0$ be the event that the harmonic measure of (the right side of) $\wh{\eta}_{u,v}$ in $\wh{\C} \setminus (\eta_u \cup \eta_{u,v})$ seen from $\infty$ is at least $\delta^{1+c_0}$. (See Figure~\ref{fi:big_bubble}.) Then $\p[ E_{u,v} \cap F_0 ] = o^\infty(\delta)$.
\end{lemma}

\begin{proof}
 We would like to compare the probability of $E_{u,v} \cap F_0$ to the probability to a similar event as in Lemma~\ref{le:tightness_bdry_good_fl_refl}. For this, let $\wt{h}$ be a whole-plane GFF modulo additive constant $2\pi\chi$. Let $\wt{E}_{u,v}$, $\wt{F}_0$ denote the corresponding events for $\wt{h}$. Then $\p[E_{u,v} \cap F_0 ] \lesssim \p[\wt{E}_{u,v} \cap \wt{F}_0 ]^{1+o(1)}$ by absolute continuity (and Lemma~\ref{le:abs_cont_kernel}). It therefore suffices to show $\p[\wt{E}_{u,v} \cap \wt{F}_0 ] = o^\infty(\delta)$.
 
 The idea is that on the event $\wt{F}_0$, with positive conditional probability of at least some power of $\delta$, the flow lines extend to $\partial B(0,2)$, and on that event we can compare the probability of $\wt{E}_{u,v}$ to the probability of the event in Lemma~\ref{le:tightness_bdry_good_fl_refl}.
 
 Let $\wt{\eta}_{-2i}$ be the flow line of $\wt{h}$ starting at $-2i$. Let $\wt{\eta}^{\pi}_v$ be the angle $\pi$ flow line starting at $v$. Let $\wt{F}_2$ be the event that $\wt{\eta}_{-2i}$ hits $\partial U$ (upon which it merges into $\wt{\eta}_u$ (resp.\ $\wt{\eta}_{u,v}$)), and $\wt{\eta}^{\pi}_v$ exits $\partial B(0,2)$ without intersecting $\wt{\eta}_{-2i}$ (see Figure~\ref{fi:big_bubble}). By iteratively applying \cite[Section~2.2]{mw2017intersections}, we see that there exists $c' > 0$ (depending on $c_0$) so that $\p[\wt{F}_2 \mid \wt{F}_0] \gtrsim \delta^{c'}$. Therefore it suffices to show $\p[ \wt{E}_{u,v} \cap \wt{F}_2] = o^\infty(\delta)$. But this follows from Lemma~\ref{le:tightness_bdry_good_fl_refl} and Lemma~\ref{le:good_scales_merging_refl} by the same proof as Lemma~\ref{le:tightness_long_fl}.
\end{proof}

\begin{figure}[ht]
\centering
\includegraphics[width=0.33\textwidth]{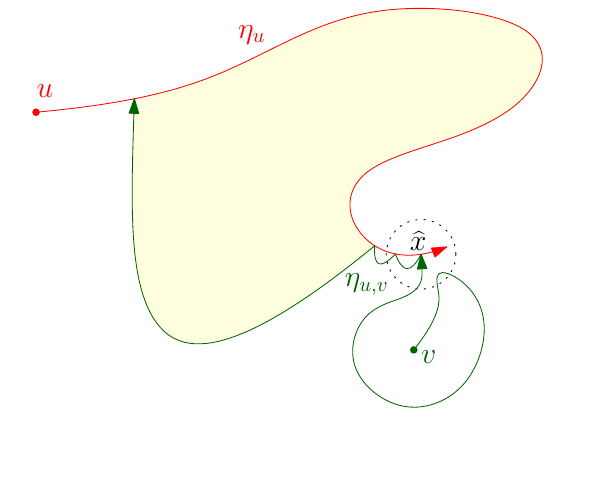}\includegraphics[width=0.33\textwidth]{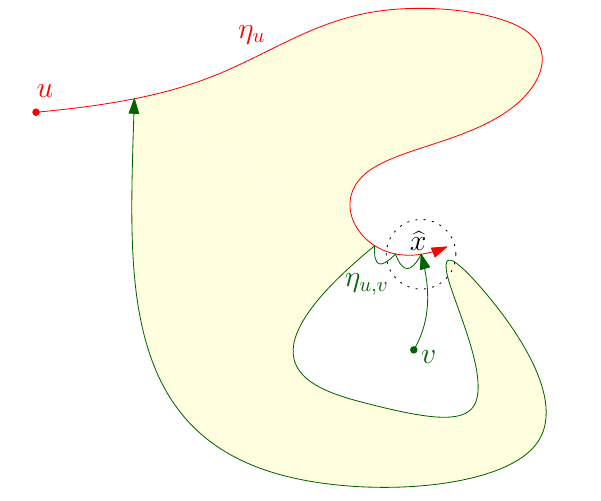}\includegraphics[width=0.33\textwidth]{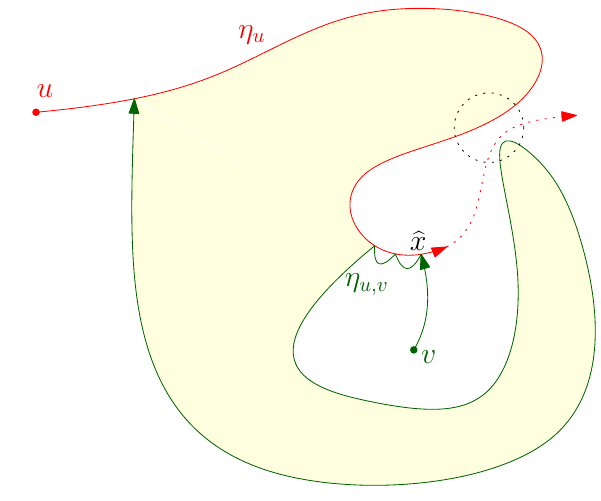}\caption{Several atypical shapes of bubbles that are ruled out in Lemma~\ref{le:strange_big_bubbles}. The first two scenarios are ruled out in the event $F_3$, the third scenario would create an approximate double point and is therefore unlikely.}
\label{fi:strange_big_bubble}
\end{figure}

\begin{lemma}\label{le:strange_big_bubbles}
 Consider the setup of Lemma~\ref{le:boundary_normal_big_bubble}. For any $c_2>c_3>0$ and $b>0$ there exists $c_0 > 0$ such that the following is true. Let $F_3$ be the event that $\abs{\wh{x}-v} \ge \delta^{1+c_3}$, $\diam(\eta_{\wh{x}}) \le \delta^{1+c_2}$, and $\eta_{u,v}$ does not contain another crossing of $A(\wh{x},2\delta^{1+c_2},\delta^{1+c_3})$ to the right of its crossing passing $\wh{x}$ (see Figure~\ref{fi:strange_big_bubble}). Let $F_0$ be the event in Lemma~\ref{le:boundary_normal_big_bubble} with $c_0$. Then $\p[F_3 \cap F_0^c] = O(\delta^b)$.
\end{lemma}

\begin{proof}
 Let $c_0'>c_2$ and let $E^2$ be the event that $\eta_{u,v}$ gets $\delta^{1+c_0'}$ close to the left side of $\eta_u$ somewhere outside of $B(\wh{x},\delta^{1+c_2})$. Note that if this happens, say at $B(z,\delta^{1+c_0'})$ for $z\in\delta^{1+c'_0}\Z^2\cap B(0,r)\setminus  B(\wh{x},\delta^{1+c_2})$, then the extension of $\eta_u$ necessarily passes through $B(z,\delta^{1+c_0'})$ again. As a result, we have four crossings of $A(z,\delta^{1+c_0'},\delta^{1+c_2})$ by $\eta_u$, and by Proposition~\ref{pr:4arm_fl} we have $\p[E^2] = O(\delta^{(c_0'-c_2)\alpha_{4,\kappa}-2c_0'-2})$. When $c_0'$ is chosen sufficiently large, this is $O(\delta^b)$. To conclude, observe that $F_3 \setminus E^2 \subseteq F_0$ for some $c_0$ depending on $c_0'$. (See Figure~\ref{fi:strange_big_bubble}.)
\end{proof}

\begin{lemma}\label{le:consecutive_intersections}
 For each $b>0$ there exists $c>0$ such that the following is true. Let $h$ be a GFF in $D$ with bounded boundary conditions. Let $E$ be the event that there are $u,v \in B(0,r)$, flow lines $\eta_u,\eta_{u,v}$ as defined in Proposition~\ref{prop:tightness_in_nice_bubble}, and $x,y \in \eta_{u,v} \cap \eta_u$ such that the segments of $\eta_u$ from $u$ to $y$ and from $y$ to $x$ both have diameter at least $\delta$, and all consecutive intersection points of $\eta_{u,v} \cap \eta_u$ between $x$ and $y$ have distance at most $\delta^c$. Then $\p[E] = O(\delta^b)$.
\end{lemma}

\begin{proof}
 Let $c'>1$ sufficiently large. There are two possible scenarios. In case $\abs{x-y} \le \delta^{c'}$, then $\eta_u$ would make an approximate double point, and the probability of this is bounded in Proposition~\ref{pr:4arm_fl}. In case $\abs{x-y} \ge \delta^{c'}$, this follows from the intersection exponents proved in \cite{mw2017intersections}. Indeed, from \cite[Lemma~4.3(i)]{mw2017intersections}, the reversibility Lemma~\ref{le:reflected_fl_law}, and our good scales argument (see the proof of Proposition~\ref{pr:4arm_fl}, using Lemma~\ref{le:good_scales_merging_refl} as input), it follows that for each $b>0$ there is $c>0$ sufficiently large so that for each $z \in B(0,r)$ the probability that there exist flow lines $\eta_u,\eta_{u,v}$ as above, both crossing $A(z,\delta^{2c'},\delta^{c'})$, such that all their consecutive intersection points within $B(z,2\delta^{2c'})$ have distance at most $\delta^{cc'}$, is at most $O(\delta^b)$. Taking a union bound over $z \in \delta^{2c'}\Z^2 \cap B(0,r)$ then yields the desired result.
 \end{proof}

\begin{figure}[ht]
\centering
\includegraphics[width=0.5\textwidth]{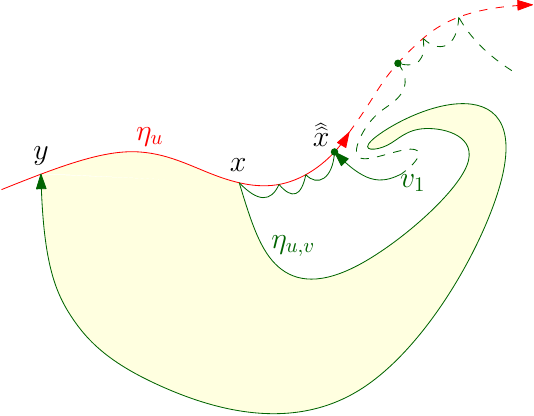}
\caption{An unlikely event in the proof of Lemma~\ref{le:tightness_bdry_big_bubbles}.}
\label{fi:big_bubble_conclude}
\end{figure}

\begin{proof}[Proof of Lemma~\ref{le:tightness_bdry_big_bubbles}]
 Recall the setup of Proposition~\ref{prop:tightness_in_nice_bubble}. We can assume that $\diam(U_{x,y}) \in [\delta/2,\delta]$ and $\abs{u-x} \ge \delta/4$.
 
 Fix $b>0$. Let $G_1$ be the event that the event in Lemma~\ref{le:consecutive_intersections} does not occur for any $\delta' < \delta$. We choose $c$ large enough so that $\p[G_1^c] = O(\delta^b)$. Let $G_2$ be the event that space-filling SLE fills a ball of radius $(\delta')^2$ whenever it travels distance $\delta'$ within $B(0,r)$ for any $\delta' < \delta$. By Lemma~\ref{le:fill_ball} we have $\p[G_2^c] = o^\infty(\delta)$. Let $G_3$ be the event that there is no flow line of $h$ that crosses in and out of an annulus $A(w,(\delta')^c,\delta') \subseteq B(0,r)$ with $\delta' < \delta$ twice. By Proposition~\ref{pr:4arm_fl} we have $\p[G_3^c] = O(\delta^b)$ when $c$ is chosen sufficiently large.
  
 We distinguish several cases.
 
 \textbf{Case 1:} The segment of $\eta_{u,v}$ from $v$ to $x$ has diameter smaller than $\diam(U_{x,y})$. In the event $E^1_\delta$ we assumed that the segment of $\eta_u$ from $x$ to $\wh{x}$ has diameter at least $\delta^{1+c_1}$. On the event $G_1$ there is a bubble $\wh{V}$ between the segments of $\eta_{u,v}$ and $\eta_u$ from $\wh{x}$ to $x$ whose endpoints are distance at least $\delta^{c(1+c_1)}$ apart. Let $\whwh{x} \in \partial\wh{V}$ be its endpoint last visited by $\eta_{u,v}$, and let $v_1 \in \partial\wh{V}$ be the last point before $\whwh{x}$ where $\eta_{u,v}$ enters $B(\whwh{x},\delta^{c(1+c_1)}/2)$. Then, on the event $G_2$, there is a ball of radius $\delta^{2c(1+c_1)}$ so that if $\wt{v}$ is in the ball, the flow line $\eta_{u,\wt{v}}$ merges into the left side of $\eta_{u,v}$ within distance $\delta^{c(1+c_1)}/4$ of $v_1$.
 
 Suppose now that we stopped the flow line $\eta_u$ within $B(\whwh{x},\delta^{c^2(1+c_1)})$. On the event $G_3$ we see that $\eta_{u,\wt{v}}$ cannot make another crossing of $A(\whwh{x},\delta^{c^2(1+c_1)},\delta^{c(1+c_1)}/4)$ on the right side of its initial segment, otherwise $\eta_{u,v} \cap \partial\wh{V}$ would have also made $4$ crossings (see Figure~\ref{fi:big_bubble_conclude}). We are therefore in the situation of Lemma~\ref{le:strange_big_bubbles}. The result now follows from Lemmas~\ref{le:boundary_normal_big_bubble} and~\ref{le:strange_big_bubbles} by considering a union bound over $u,\wt{v} \in \delta^{2c(1+c_1)}\Z^2 \cap \D$, and over a finite sequence of stopping times for $\eta_u$ stopping each time it travels distance $\delta^{c^2(1+c_1)}$ (on $G_2$ we need to consider at most $O(\delta^{-4c^2(1+c_1)})$ times).
 
 \textbf{Case 2:} The segment of $\eta_{u,v}$ from $v$ to $x$ has diameter at least $\diam(U_{x,y})$. Let $v_1$ be the point on $\eta_{u,v}$ such that the diameter of the segment from $v_1$ to $x$ is equal to $\diam(U_{x,y})$. Let $\whwh{x}$ be the next point where $\eta_{u,v_1}$ hits $\eta_u$. We distinguish two more cases.
 
 \textbf{Case 2a:} The diameter of the segment of $\eta_u$ from $x$ to $\whwh{x}$ has diameter at least $\delta/8$. Then we are in the same situation as Case~1.
 
 \textbf{Case 2b:} The diameter of the segment of $\eta_u$ from $x$ to $\whwh{x}$ has diameter at most $\delta/8$. Then $\abs{\whwh{x}-v_1} \ge \delta/8$. Moreover, the choice of $v_1$ implies that $\eta_{u,v_1}$ cannot reach $B(\whwh{x},\delta/8)$ on the right side of $v_1$ again. The rest of the proof then follows by the same argument as in the previous case.
\end{proof}

\begin{proof}[Proof of Proposition~\ref{prop:tightness_in_nice_bubble}]
 Fix a small constant $a>0$. We work on the event $F_\delta$ that the space-filling SLE associated to $\eta'$ fills a ball of radius $(\delta')^{1+a}$ whenever it travels distance $\delta'$ within $B(0,r)$ for any $\delta' < \delta$. By Lemma~\ref{le:fill_ball} we have $\p[F_\delta^c] = o^\infty(\delta)$.
 
 Let $G_\delta$ be the event that neither $E'_{2^{-k}}$, $E''_{2^{-k}}$ from the Lemmas~\ref{le:tightness_long_fl} and~\ref{le:tightness_bdry_big_bubbles} occur for any $k \ge \log_2(\delta^{-1})$. Then $\p[G_\delta^c] = o^\infty(\delta)$. We claim that $G_\delta \cap F_\delta \subseteq (E^1_\delta)^c$. Let $U = U_{x,y}$ be a region as given in the statement of Proposition~\ref{prop:tightness_in_nice_bubble}. We explain how we can divide $U$ into pieces that are covered by the events $E'_{2^{-k}}$, $E''_{2^{-k}}$.
 
 Let $\log_2(\delta^{-1}) \le k \le \log_2(\epsilon^{-1})$. Let $x',x'' \in \eta_{u,v} \cap \eta_u$, and $U'$ the region bounded between the segments of $\eta_{u,v}$ and $\eta_u$ from $x'$ to $x''$. Suppose that $U' \cap A(x,2^{-k-1},2^{-k}) \neq \varnothing$ and $\diam(U') \le 2^{-k(1+a)}$. On the event $(E'_{2^{-k}})^c$ we have $D_\epsilon^{\partial U'} \le 2^{-k\zeta}+\epsilon^{\epsexp}$ for every such $U'$.
 
 The exact same is true when we replace $x$ by $y$ in the paragraph above.
 
 Finally, when $k > \log_2(\epsilon^{-1})$ and $\diam(U') \le \epsilon$, then $D_\epsilon^{\partial U'} \le \median{\epsilon}^{-1}\ac{\epsilon} \le \epsilon^{2\epsexp}$ by the definition~\eqref{eq:shortcutted_metric} and~\eqref{eq:epsexp2}.
 
 To conclude, we decompose the parts of $U$ that intersect $A(x,2^{-k-1},2^{-k})$ (resp.\ $A(y,2^{-k-1},2^{-k})$) into the components of diameter at least $2^{-k(1+a)}$ and the remaining parts. On the event $F_\delta$ there are $O(2^{4ak})$ components of diameter at least $2^{-k(1+a)}$, and the remaining parts can further be decomposed into $O(2^{4ak})$ regions $U'$ of the type described above. Combining the bounds and applying the monotonicity of the internal metrics, we see that
 \[ D_\epsilon^{\partial U} \lesssim \sum_{\log_2(\delta^{-1}) \le k \le \log_2(\epsilon^{-1})} 2^{4ak}(2^{-k\zeta}+\epsilon^{\epsexp}) \lesssim \delta^{\zeta-4a}+\epsilon^{\epsexp-4a} \]
 on the event $G_\delta \cap F_\delta$. (In the case $\epsilon=0$, we used the continuity of the metric).
\end{proof}

\subsection{Tightness in the interior}
\label{subsec:interior_tightness}

We now turn towards proving the tightness result for the CLE metric in the entire gasket. For this, we will apply our tightness results from the previous subsection to bound the $\Fd_\epsilon$-diameters of the boundaries of regions disconnected by finite chains of loops. This will be done via a resampling argument that transforms those regions to the type considered in the previous subsection. The regions considered in this subsection will be sufficient to cover the CLE gasket, and bound the $\Fd_\epsilon$-distances between any pair of points in the gasket via a chaining argument.

We consider the same setup as in Section~\ref{subsec:tightness_interior_flow_lines}. We start by observing that if a region is disconnected by a single loop, then its boundary is described by an intersecting pair of flow lines as considered in Section~\ref{subsec:tightness_interior_flow_lines}, and therefore the bounds we have proved apply to these regions.

\begin{lemma}\label{le:tightness_bdry_single_loop}
There exists $\zeta > 0$ such that the following is true. Fix $0 < r < 1$ and $c_1 > 0$. For $\delta > 0$, let $E_\delta$ be the event that there exists a loop $\CL \in \Gamma$ and two non-overlapping segments $\ell_1,\ell_2 \subseteq \CL$ such that the left boundaries of $\ell_1,\ell_2$ intersect and there are four points $\wh{x},x,y,\wh{y}$ on their intersection in exactly this order such that
\begin{itemize}
\item $\abs{x-\wh{x}} \ge \delta^{c_1}$, $\abs{y-\wh{y}} \ge \delta^{c_1}$,
\item if $U_{x,y}$ is the region bounded between the left boundaries of $\ell_1,\ell_2$ from $x$ to $y$, then $U_{x,y} \subseteq B(0,r)$, $\diam(U_{x,y}) \le \delta$, and $D_\epsilon^{\partial U_{x,y}} \ge \delta^\zeta+\epsilon^{\epsexp}$.
\end{itemize}
Then $\p[E_\delta] = o^\infty(\delta)$.
\end{lemma}

\begin{proof}
Let $\Gamma$ is coupled with a GFF $h$ on $\D$ as described in Section~\ref{subsec:tightness_interior_flow_lines}. We claim that the region $U_{x,y}$ is described by pairs of intersecting flow lines as in Proposition~\ref{prop:tightness_in_nice_bubble}. Indeed, let $\eta'$ be the branch of the exploration tree that discovers the loop $\CL$, and suppose that it traces $\ell_1$ before $\ell_2$ . Let $u$ be any point in a complementary component that $\eta'$ makes between tracing $\ell_1$ and $\ell_2$. Let $v$ be any point in a complementary component that $\eta'$ makes after tracing $\ell_2$. Then $U_{x,y}$ satisfies the assumption in Proposition~\ref{prop:tightness_in_nice_bubble}, so the result follows.
\end{proof}

Now we extend the bound to regions that are bounded by finite chains of loops.

\begin{lemma}\label{le:tightness_loop_chain}
There exists $\zeta > 0$ such that the following is true. For each $\innexp > 0$ and $b>0$, there exists $N \in \N$ such that the following holds. Fix $0<r<1$. For $z \in B(0,r)$ and $\delta > 0$, let $G_{z,\delta}$ be the event that there is a set of at most $N$ points $x_1,x_2,\ldots \in A(z,\delta,\delta^{1-\innexp})$ with $\abs{x_{i}-x_{i'}} \ge \delta$ for each $i \neq i'$ and such that
\begin{itemize}
 \item The set $\{x_1,x_2,\ldots\}$ separates $\Upsilon_\Gamma \cap B(z,\delta)$ from $\partial B(z,\delta^{1-\innexp})$.
 \item Let $K'$ be any connected component of $\Upsilon_\Gamma \setminus \{x_1,x_2,\ldots\}$ containing some point in $B(z,\delta)$, and let $i,i'$ so that $x_{i},x_{i'} \in \partial K'$. Let $V \in \metregions$ be such that $\ol{V}$ contains all \emph{simple} admissible paths within $K'$ from $x_{i}$ to $x_{i'}$. Then
 \[ \median{\epsilon}^{-1}\metapproxacres{\epsilon}{V}{x_{i}}{x_{i'}}{\Gamma} \le \delta^\zeta+\epsilon^{\epsexp} . \]
\end{itemize}
Then $\p[G_{z,\delta}^c] = O(\delta^b)$.
\end{lemma}

\begin{proof}
We are going to reduce this to the statement of Lemma~\ref{le:tightness_bdry_single_loop} via the resampling procedure from Lemma~\ref{lem:disconnect_interior}.

Let $s>0$, $p>0$. For $j = \lceil\log_2(\delta^{-1+\innexp})\rceil,\ldots,\lfloor\log_2(\delta^{-1})\rfloor$ let $G^\resampled_{z,j}$ be the event that there are points $y_1,y_2,\ldots \in A(z,2^{-j-1},2^{-j})$ such that
\begin{itemize}
\item $\{y_1,y_2,\ldots\}$ separates $B(z,2^{-j-1}) \cap \Upsilon_\Gamma$ from $\partial B(z,2^{-j})$,
\item $\abs{y_i-y_{i'}} \ge 9s2^{-j}$ for each $i \neq i'$,
\item for each $i$ there is a point $x_i$ with $\abs{x_i-y_i} \ge 4s2^{-j}$ on the intersection of the same two loops that intersect at $y_i$, and $\{x_1,x_2,\ldots\}$ separates $\{y_1,y_2,\ldots\}$ from $B(z,2^{-j-1})$,
\item if we apply the resampling procedure in the annulus $A(z,2^{-j-1},2^{-j})$, then for any subset $\{y_{i_k}\}_k \subseteq \{y_1,y_2,\ldots\}$ the conditional probability given $\Gamma$ that the resampling is successful for the target pivotals $\{y_{i_k}\}_k$ is at least $p$.
\end{itemize}
Let $\wt{G}$ be the event that $G^\resampled_{z,j}$ occurs for at least $9/10$ fraction of $j \in \{ \lceil\log_2(\delta^{-1+\innexp})\rceil,\ldots,\lfloor\log_2(\delta^{-1})\rfloor \}$. Using Proposition~\ref{prop:resampling} one can show, following the same proof for Lemma~\ref{lem:disconnect_interior} in \cite{amy-cle-resampling} that for any given $\innexp > 0$, $b>0$ there are $s>0$, $p>0$ such that $\p[\wt{G}^c] = O(\delta^b)$.

When $x_{i},x_{i'}$ are on the boundary of the same connected component $K'$ of $\Upsilon_\Gamma \setminus \{x_1,x_2,\ldots\}$, then on the event $G^\resampled_{z,j}$ there is a subset $\{y_{i_k}\}_k \subseteq \{y_1,y_2,\ldots\}$ such that if the resampling is successful, then $K'$ is disconnected from $\partial\D$ by a single loop $\CL^\resampled \in \Gamma^\resampled_{z,j}$. If we let $\ol{V'}$ be the minimal simply connected set containing all simple admissible paths within $K'$ from $x_i$ to $x_{i'}$, then $V'$ is bounded between the left boundaries of two strands of $\CL^\resampled$. In that case we couple the internal metrics so that $\metapproxacres{\epsilon}{V'}{\cdot}{\cdot}{\Gamma} = \metapproxacres{\epsilon}{V'}{\cdot}{\cdot}{\Gamma^\resampled_{z,j}}$.

Note also that when $s2^{-j} \ge \cserial\epsilon$ where $\cserial$ is the constant in the assumption~\eqref{eq:approx_monotonicity_axiom}, then the monotonicity~\eqref{eq:approx_monotonicity_axiom} applies, so that $\metapproxacres{\epsilon}{V}{\cdot}{\cdot}{\Gamma} \le \metapproxacres{\epsilon}{V'}{\cdot}{\cdot}{\Gamma}+2\ac{\epsilon}$ for every $V \supseteq V'$ (where we also used the definition of $\metapproxac{\epsilon}{\cdot}{\cdot}{\Gamma}$ in~\eqref{eq:shortcutted_metric}).

To conclude, we choose $J \in \{ \lceil\log_2(\delta^{-1+\innexp})\rceil,\ldots,\lfloor\log_2(\delta^{-1})\rfloor \}$ uniformly at random, and apply the resampling procedure in the annulus $A(z,2^{-J-1},2^{-J})$. If we let $E^\resampled_\delta$ denote the event from Lemma~\ref{le:tightness_bdry_single_loop} occurring for $\Gamma^\resampled_{z,J}$, we have $\p[ E^\resampled_\delta \mid G_{z,\delta}^c \cap \wt{G} ] \ge p$ and hence
\[ \p[ G_{z,\delta}^c \cap \wt{G} ] \le \frac{1}{p}\p[ E^\resampled_\delta ] = o^\infty(\delta) . \]
\end{proof}

We now explain that the bounds in Lemma~\ref{le:tightness_loop_chain} are sufficient to bound the distances between any pair of points in the gasket. This is because any admissible path in the gasket needs to pass through intersection points of loops.

\begin{lemma}\label{le:gasket_interior_bound}
 Fix $\innexp > 0$, $N \in \N$, and let $G_{z,\delta}$ be the event in Lemma~\ref{le:tightness_loop_chain}. There is a constant $c$ depending on $N$ such that the following is true. Suppose that $\gamma\colon [0,1] \to \Upsilon_\Gamma$ is an admissible path and there are times $0 < s_1 < \cdots < s_L < 1$ and $z_l,\delta_l$ such that the following hold for each $l$.
 \begin{itemize}
  \item $\gamma[s_l,s_{l+1}] \subseteq B(z_l,\delta_l)$,
  \item $\gamma[0,s_l], \gamma[s_{l+1},1] \nsubseteq B(z_l,\delta_l^{1-\innexp})$,
  \item $G_{z_l,\delta_l}$ occurs.
 \end{itemize}
 Then there exist $s' \in [0,s_1]$, $t' \in [s_L,1]$ such that
 \[ \median{\epsilon}^{-1}\metapproxacres{\epsilon}{V}{\gamma(s')}{\gamma(t')}{\Gamma} \le c\sum_l (\delta_l^\zeta + \epsilon^{\epsexp}) \]
 for each $V \in \metregions$ such that $\ol{V}$ contains every \emph{simple} admissible path $\wt{\gamma}$ between two points on $\gamma$ with $\diamE(\wt{\gamma}) \le 2\max_l \delta_l^{1-\innexp}$.
\end{lemma}

\begin{proof}
 Let $l \in \{1,\ldots,L-1\}$. Let $x_1,x_2,\ldots \in A(z_l,\delta,\delta^{1-\innexp})$ be the points in the event $G_{z_l,\delta_l}$. Consider the connected component $K'$ of $\Upsilon_\Gamma \setminus \{x_1,x_2,\ldots\}$ containing $\gamma[s_l,s_{l+1}]$, and let $s'_l$ (resp.\ $t'_l$) be the last (resp.\ next) time where $\gamma$ passes $\{x_1,x_2,\ldots\}$. These times exist since we assumed $\gamma[0,s_l], \gamma[s_{l+1},1] \nsubseteq B(z_l,\delta_l^{1-\innexp})$. The condition on $V$ implies that it satisfies the condition in Lemma~\ref{le:tightness_loop_chain}, so that
 \[ \median{\epsilon}^{-1}\metapproxacres{\epsilon}{V}{\gamma(s')}{\gamma(t')}{\Gamma} \le \delta_l^\zeta + \epsilon^{\epsexp} \]
 since we are on the event $G_{z_l,\delta_l}$.
 
 Let $s'_{l+1},t'_{l+1}$ be the analogous points for $G_{z_{l+1},\delta_{l+1}}$, and suppose that $s'_l < s'_{l+1} < t'_l < t'_{l+1}$ (in all other cases we can just ignore $l$ or $l+1$). Then $\{x_1,x_2,\ldots\}$ separates $\eta(s'_{l+1})$ from $\eta(t'_{l+1})$. By the generalized parallel law (which also applies to the metric in~\eqref{eq:shortcutted_metric}), there is some $x_i$ such that
 \[ \median{\epsilon}^{-1}\metapproxacres{\epsilon}{V}{x_i}{\gamma(t'_{l+1})}{\Gamma} \le \cparallel(N)\median{\epsilon}^{-1}\metapproxacres{\epsilon}{V}{\gamma(s'_{l+1})}{\gamma(t'_{l+1})}{\Gamma} \le \cparallel(N)(\delta_{l+1}^\zeta + \epsilon^{\epsexp}) . \]
 Hence, again using the condition on $V$, we have
\[ \begin{split}
\median{\epsilon}^{-1}\metapproxacres{\epsilon}{V}{\gamma(s'_l)}{\gamma(t'_{l+1})}{\Gamma} 
&\le \median{\epsilon}^{-1}\metapproxacres{\epsilon}{V}{\gamma(s'_l)}{x_i}{\Gamma} + \median{\epsilon}^{-1}\metapproxacres{\epsilon}{V}{x_i}{\gamma(t'_{l+1})}{\Gamma} \\
&\le (\delta_l^\zeta + \epsilon^{\epsexp}) + \cparallel(N)(\delta_{l+1}^\zeta + \epsilon^{\epsexp}) . 
\end{split} \]
By repeating this argument, we conclude.
\end{proof}

\subsection{Proof of the main theorems}
\label{subsec:thm1_proof}

In this subsection we complete the proofs of the tightness statement in Theorem~\ref{th:tightness_metrics}, the H\"older continuity in Theorem~\ref{th:cle_metric_hoelder}, and the tightness across scales in Theorem~\ref{th:metric_scaling}. Recall the setup from Sections~\ref{se:assumptions}--\ref{subsec:main_statement} and the definition of the metric $\dpath$~\eqref{eq:dpath}. The main result will be the proof of tightness of the approximations, which we establish in Proposition~\ref{prop:interior_tightness} where we show uniform H\"older continuity of the approximations with respect to the metric $\dpath$. As a consequence of Lemma~\ref{le:thin_gasket_hoelder} we obtain in Corollary~\ref{co:tightness_thin_gasket} the uniform H\"older continuity with respect to the Euclidean metric when we restrict to the thin gasket. In Lemma~\ref{le:median_scaling_lb} and~\ref{le:quantiles_comparable}, we complete the proofs of the polynomial scaling and the comparability of the quantiles of the metric.

Recall the definition of $\metregions$ in Section~\ref{se:assumptions}. For each $r>0$, let $\metregions[(r)] \subseteq \metregions$ be the collection of regions such that $V \subseteq B(0,1-r)$ and $V$ is a union of connected components of
$D \setminus (\CL_1 \cup \cdots \cup \CL_n)$ where $\diamE(\CL_i) \ge r$ for each $i$.

\begin{proposition}\label{prop:interior_tightness}
 There exists $\zeta > 0$ such that the following is true. Consider the setup in Section~\ref{subsec:tightness_interior_flow_lines}. For each $r>0$, there is a random variable $X_{\epsilon,r,\zeta}$ with
 \[
  \p[ X_{\epsilon,r,\zeta} > M ] = o^\infty(M^{-1}) 
  \quad\text{ (uniformly in $\epsilon$) as } M \to \infty 
 \]
 such that
 \[ 
  \sup_{V \in \metregions[(r)]} \sup_{x,y \in \ol{V} \cap \Upsilon_\Gamma} \frac{\median{\epsilon}^{-1}\metapproxacres{\epsilon}{V}{x}{y}{\Gamma}}{\dpath[\ol{V}](x,y)^\zeta+\epsilon^{\epsexp}} \le X_{\epsilon,r,\zeta} .
 \]
\end{proposition}

\begin{proof}
 Let $b>0$ be given. For $0 < \delta < r^{1+2\innexp}$, let
 \[ \ol{G}_\delta = \bigcap_{\log_2(\delta^{-1}) \le k \le \log_2(\epsilon^{-1})} \bigcap_{z \in 2^{-2k}\Z^2 \cap B(0,1-r)} G_{z,2^{-k}} \]
 where $G_{z,\delta}$ is the event from Lemma~\ref{le:tightness_loop_chain}. Then $\p[\ol{G}_\delta^c] = O(\delta^b)$. (Note that when $\delta < \epsilon$, we automatically have $\median{\epsilon}^{-1}\metapproxacres{\epsilon}{V}{u}{v}{\Gamma} \le \median{\epsilon}^{-1}\ac{\epsilon} \le \epsilon^{2\epsexp}$ whenever $\dpath[\ol{V}](u,v) \le \delta$ by~\eqref{eq:shortcutted_metric} and~\eqref{eq:epsexp2}.)
 
 Fix $a>0$ small, and let $F_\delta$ be the event that the space-filling SLE associated to $\eta'$ fills a ball of radius $(\delta')^{1+a}$ whenever it travels distance $\delta' \le \delta$ within $B(0,r)$. By Lemma~\ref{le:fill_ball}, we have $\p[F_\delta^c] = o^\infty(\delta)$.
 
 Now suppose we are on the event $\ol{G}_\delta \cap F_\delta$. We explain that $\median{\epsilon}^{-1}\metapproxacres{\epsilon}{V}{x}{y}{\Gamma}$ can be bounded uniformly for all $V \in \metregions[(r)]$, $x,y \in \ol{V} \cap \Upsilon_\Gamma$, and that $X_{\epsilon,r,\zeta} = O(\delta^{-3})$ on the event $\ol{G}_\delta \cap F_\delta$. Let $\gamma$ be an admissible path in $\ol{V}$ from $x$ to $y$ such that $\diamE(\gamma) < 2\dpath[\ol{V}](x,y)$. We distinguish two cases.
 
 In case $\diamE(\gamma) \le \delta$, we cover $\gamma$ by balls of the type $B(z,2^{-k})$ where $z \in 2^{-2k}\Z^2 \cap B(0,1-r)$ and $k \ge \log_2(\diamE(\gamma)^{-1})$. For each such ball, on the event $F_\delta$ we can upper bound the number of crossings. Indeed, if two crossings of $A(z,2^{-j-1},2^{-j})$ by $\gamma$ are not separated in $A(z,2^{-j-1},2^{-j})$ by a loop $\CL \in \Gamma$, then we can find an admissible path within $A(z,2^{-j-1},2^{-j})$ connecting the two crossings. On the event $F_\delta$, there can be at most $O(2^{2aj})$ crossings by loops $\CL \in \Gamma$. Therefore we can modify $\gamma$ so that it crosses each $A(z,2^{-j-1},2^{-j})$ at most $O(2^{2aj})$ times. Moreover, by the definition of $\metregions[(r)]$ and since $2\delta^{1-\innexp} < r$, the region $V$ satisfies the condition in Lemma~\ref{le:gasket_interior_bound}, therefore
 \[
  \median{\epsilon}^{-1}\metapproxacres{\epsilon}{V}{x}{y}{\Gamma} 
  \lesssim \sum_{\log_2(\dpath[\ol{V}](x,y)^{-1}) \le k \le \log_2(\epsilon^{-1})} 2^{2ak}(2^{-k\zeta}+2^{-k\epsexp}\epsilon^{\epsexp}) 
  \lesssim \dpath[\ol{V}](x,y)^\zeta+\epsilon^{\epsexp} .
 \]
 
 In case $\diamE(\gamma) \ge \delta$, we divide $\gamma$ into segments of diameter at most $\delta$. On the event $F_\delta$, by the same argument as before, within each ball of radius $\delta$ we need at most $O(\delta^{-2a})$ segments. Therefore
 \[
  \median{\epsilon}^{-1}\metapproxacres{\epsilon}{V}{x}{y}{\Gamma} \lesssim \left(\frac{\dpath[\ol{V}](x,y)}{\delta}\right)^2 \delta^{-2a} (\delta^\zeta+\epsilon^{\epsexp}) .
 \]
 This concludes the proof.
\end{proof}

Recall that the thin gasket $\CT_\Gamma$ of $\Gamma_D$ is defined to be the set of points in $D$ which are not disconnected from $\partial D$ by any loop in $\Gamma_D$. That is, $\CT_\Gamma$ is equal to the closure of the union of the outer boundaries of the loops of $\Gamma_D$. Recall from Lemma~\ref{le:thin_gasket_hoelder} that when restricted to the thin gasket, the metric $\dpath$ is H\"older continuous with respect to Euclidean distance. Now that we have established H\"older continuity of $\median{\epsilon}^{-1} \metapproxac{\epsilon}{\cdot}{\cdot}{\Gamma}$ respect to $\dpath$, this implies that its restriction to $\CT_\Gamma \times \CT_\Gamma$ is H\"older continuous with respect to the Euclidean metric.

\begin{corollary}\label{co:tightness_thin_gasket}
 For any $b>0$ there exists $\zeta > 0$ such that the following is true. Consider the setup in Section~\ref{subsec:tightness_interior_flow_lines}. Then, for each $r>0$,
 \[ 
  \p \left[ \sup_{x,y \in \CT_\Gamma \cap B(0,1-r)} \frac{\median{\epsilon}^{-1}\metapproxacres{\epsilon}{D}{x}{y}{\Gamma}}{\abs{x-y}^\zeta+\epsilon^{\epsexp}} \geq M \right] = O(M^{-b}) 
  \quad\text{ (uniformly in $\epsilon$) as } M \to \infty .
 \]
\end{corollary}

Our estimates so far concern the metric $\metapproxacres{\epsilon}{V}{\cdot}{\cdot}{\Gamma}$ defined in~\eqref{eq:shortcutted_metric}. With the assumption~\eqref{eq:approx_error_asymp}, we see that the tightness result applies also to the original metric $\metapproxres{\epsilon}{V}{\cdot}{\cdot}{\Gamma}$. For $r>0$, let $A_{\epsilon,r}$ denote the event that $\metapproxres{\epsilon}{V}{x}{y}{\Gamma} \le \ac{\epsilon}$ for every $V \in \metregions$ and $x,y$ with $\dpath[\ol{V}](x,y) < \epsilon$ and $\dpath(x, \Upsilon_\Gamma \setminus \ol{V}) \ge r$.

\begin{corollary}\label{co:tightness_original_metric}
 There exists $\zeta > 0$ such that the following is true. Consider the setup in Section~\ref{subsec:tightness_interior_flow_lines}. For each $r>0$, there is a random variable $Y_{\epsilon,r,\zeta}$ with
 \[
  \p[ Y_{\epsilon,r,\zeta} > M ] = o^\infty(M^{-1}) 
  \quad\text{ (uniformly in $\epsilon$) as } M \to \infty 
 \]
 such that for every $V \in \metregions$ and $U \subseteq \ol{V}$ with $\dpath(U, \Upsilon_\Gamma \setminus \ol{V}) \ge r$, on the event $A_{\epsilon,r}$ we have
 \[ 
  \sup_{x,y \in U \cap \Upsilon_\Gamma} \frac{\median{\epsilon}^{-1}\metapproxres{\epsilon}{V}{x}{y}{\Gamma}}{\dpath[U](x,y)^\zeta+\epsilon^{\epsexp}} \le Y_{\epsilon,r,\zeta} .
 \]
\end{corollary}

\begin{proof}
 The proof of Proposition~\ref{prop:interior_tightness} shows that off an event with probability $o^\infty(\delta)$ we have
 \[ 
  \sup_{x,y \in U \cap \Upsilon_\Gamma} \frac{\median{\epsilon}^{-1}\metapproxacres{\epsilon}{V}{x}{y}{\Gamma}}{\dpath[U](x,y)^\zeta+\epsilon^{\epsexp}} \le \delta^{-1} ,
 \]
 for every $U \subseteq \ol{V}$ as above, where in the metric $\metapproxacres{\epsilon}{V}{\cdot}{\cdot}{\Gamma}$ we can choose the finite sequence of points $(u_i)$ so that $\dpath(u_i,U) \le \delta^{1-\innexp}$ for a small constant $\innexp>0$. In particular, when $\delta^{1-\innexp} < r/2$, on the event $A_{\epsilon,r}$ the same bound applies to $\metapproxres{\epsilon}{V}{x}{y}{\Gamma}$.
\end{proof}

We conclude this section showing a lower bound for the polynomial scaling behavior of CLE metrics which is the counterpart to Lemma~\ref{le:median_scaling}, and the comparability of the quantiles defined in Section~\ref{se:intersections_setup}. We need to assume the analogue of~\eqref{eq:approx_error_asymp} in the setup of Section~\ref{se:intersections_setup}, namely that, if $U$ denotes the region bounded between $\eta_1^1,\eta_2^1$, then
\begin{equation}\label{eq:approx_error_bubbles}
 \lim_{\epsilon \searrow 0} \p\left[ \sup_{\substack{u,v \in B(0,3/4) \\ \dpath(u,v) < \epsilon}} \metapproxres{\epsilon}{U}{u}{v}{\Gamma_1} \le \ac{\epsilon} \right] = 1 .
\end{equation}

\begin{lemma}\label{le:median_scaling_lb}
 Consider the setup in Section~\ref{se:intersections_setup}, and assume~\eqref{eq:approx_error_bubbles}. Fix $q \in (0,1)$, $\epsexp > 0$. For sufficiently small $\epsilon$ we have
 \[ 
  \quant[1]{q}{\epsilon} \le \delta^{-d_\SLE+o(1)}(\median[\delta]{\epsilon} + \epsilon^{-\epsexp}\ac{\epsilon})
  \quad\text{as } \delta \searrow 0 .
 \]
\end{lemma}

\begin{lemma}\label{le:quantiles_comparable}
 Consider the setup in Section~\ref{se:intersections_setup}, and assume~\eqref{eq:approx_error_bubbles}. For any $q,q' \in (0,1)$, $\epsexp > 0$, there exists $c>0$ such that, for sufficiently small $\epsilon$,
 \[ \quant{q}{\epsilon} \le c(\quant{q'}{\epsilon} + \epsilon^{-\epsexp}\ac{\epsilon}) . \]
\end{lemma}

\begin{corollary}\label{co:quantiles_metric}
 Suppose $\met{\cdot}{\cdot}{\Gamma}$ is a \clekp{} (pseudo-)metric in the sense of Definition~\ref{def:cle_metric}. Let $\median[\delta]{}$ and $\quant[\delta]{q}{}$ be as defined in Section~\ref{se:intersections_setup}. Then
 \[ \delta^{d_\SLE+o(1)}\median[r]{} \le \median[r\delta]{} \le \delta^{\ddouble+o(1)}\median[r]{} \]
 for any $r,\delta \in (0,1)$. Further,
 \[ \sup_{r\in (0,1]} \frac{\quant[r]{q}{}}{\quant[r]{q'}{}} < \infty \]
 for any $q,q' \in (0,1)$.
\end{corollary}

\begin{proof}[Proof of Lemma~\ref{le:median_scaling_lb}]
 Assume $\epsilon$ is small enough to that the probability in~\eqref{eq:approx_error_bubbles} is at least $1-(1-q)\pmed/2$.
 
 We can assume that the boundary values of $h$ are as in Lemma~\ref{le:reflected_fl_law}. We can transfer to the boundary values as in the setup of Section~\ref{se:intersections_setup} via a local absolute continuity argument (as done in the proof of Lemma~\ref{le:median_scaling}).
 
 In particular, we assume that $\eta_1,\eta_2$ are coupled with $h$ so that $\eta_1$ is the flow line from $-i$ to $i$, and the reversal of $\eta_2$ is the flow line from $i$ to $-i$, reflected off $\eta_1$. Let $\Gamma$ be the collection of \clekp{} in the components bounded between $\eta_1,\eta_2$, coupled with $h$.
 
 Fix $a>0$ small. We apply Proposition~\ref{prop:interior_tightness} to the metric $\mettapprox{\delta^{-1}\epsilon}{\cdot}{\cdot}{\Gamma} = \metapprox{\epsilon}{\delta\cdot}{\delta\cdot}{\delta\Gamma}$ so that $\mediant{\delta^{-1}\epsilon} = \median[\delta]{\epsilon}$ (cf.\ Lemma~\ref{le:scaled_metric}). By the translation invariance and the monotonicity of the metrics (applied to suitable regions as in the proof of Proposition~\ref{prop:interior_tightness}), it follows that within every ball $B(z,\delta) \subseteq B(0,1/2)$ we have
 \[
  \metapproxac{\epsilon}{u}{v}{\Gamma_1} \le \delta^{-a}((\delta^{-1}\dpath(u,v))^\zeta \median[\delta]{\epsilon} + (\delta^{-1}\epsilon)^{-\epsexp}\ac{\epsilon})
  \quad\text{for } u,v \in B(z,\delta) \cap \Upsilon_{\Gamma}
 \]
 off an event with probability $o^\infty(\delta)$. Taking a union bound, this holds simultaneously for each ball $B(z,\delta) \subseteq B(0,1/2)$.
 
 By Lemma~\ref{le:variation_one_scale} and local absolute continuity, we can find an event with probability as close to $1$ as we like such that for each $\delta > 0$ and $\eta_1(s),\eta_1(t) \in B(0,1/2)$ there is a sequence of $O(\delta^{-d_\SLE-a})$ time points $(t_l)$ such that $t_1 = s$, $t_L = t$, and $\diam(\eta_1[t_l,t_{l+1}]) \le \delta$ for each $l$. In particular, $\dpath(\eta_1(t_l),\eta_1(t_{l+1})) \le \delta$ for each $l$. We conclude that in the setup of Section~\ref{se:intersections_setup} we have
 \[
  \p\left[ \sup_{(x',x,y,y')\in\intptsapprox{1}{\epsilon}}\metapproxacres{\epsilon}{U_{x',y'}}{x}{y}{\Gamma_1} > \delta^{-d_\SLE-2a}(\median[\delta]{\epsilon} + \epsilon^{-\epsexp}\ac{\epsilon}) \right] = o(1) 
  \quad\text{as } \delta \searrow 0 .
 \]
 In particular, for $\delta$ small enough the probability is smaller than $(1-q)\pmed/2$. If we are on the event in~\eqref{eq:approx_error_bubbles}, then $\metapproxres{\epsilon}{U_{x',y'}}{x}{y}{\Gamma_1} = \metapproxacres{\epsilon}{U_{x',y'}}{x}{y}{\Gamma_1}$. We conclude that $\quant[1]{q}{\epsilon} \le \delta^{-d_\SLE-2a}(\median[\delta]{\epsilon} + \epsilon^{-\epsexp}\ac{\epsilon})$.
\end{proof}

\begin{proof}[Proof of Lemma~\ref{le:quantiles_comparable}]
 By Lemma~\ref{le:median_scaling}, we have $\median[\delta]{\epsilon} \le \delta^{\ddouble+o(1)}\quant[1]{q'}{\epsilon}$ for $\epsilon < \delta$. In particular, we can find $\delta > 0$ such that $\median[\delta]{\epsilon} \le \quant[1]{q'}{\epsilon}$. By Lemma~\ref{le:median_scaling_lb}, we then have $\quant[1]{q}{\epsilon} \le \delta^{-d_\SLE+o(1)}(\median[\delta]{\epsilon} + \epsilon^{-\epsexp}\ac{\epsilon}) \le \delta^{-d_\SLE+o(1)}(\quant[1]{q'}{\epsilon} + \epsilon^{-\epsexp}\ac{\epsilon})$. 
\end{proof}

\section{Construction of the metric}
\label{se:construction_metric}

\newcommand*{\FS}{\mathfrak{S}}
\newcommand*{\FQ}{\mathfrak{Q}}

We have shown in Section~\ref{se:tightness_proof} that the approximate CLE metrics are tight. This yields weak limits along subsequences. In this section, we construct a CLE metric from the subsequential limits. We show in Section~\ref{se:proof_axioms} that it satisfies the axioms of a CLE metric in Section~\ref{se:assumptions}, and in Section~\ref{se:nondegeneracy} that it either is a true (positive) metric or it is identically zero. Finally, we show in Section~\ref{se:geodesic_metric} that if the approximation scheme is geodesic, the limiting metric is geodesic.

We consider the setup from Sections~\ref{se:assumptions}--\ref{subsec:main_statement}. In particular, let $\Gamma_\D$ be a nested $\CLE_{\kappa'}$ in $\D$, let $\CL$ be the outermost loop of $\Gamma_\D$ that surrounds~$0$, and let $D$ be the regions surrounded by $\CL$. Let $\Gamma_D$ be the loops of $\Gamma_\D$ contained in $\ol{D}$, and let $\Gamma = \{\CL\} \cup \Gamma_D$. Let $\Upsilon_\Gamma$ be the gasket of $\Gamma_D$. Recall the definition of the set of regions $\metregions$ on which the internal metrics are defined. We will assume from now on that $(\metapprox{\epsilon}{\cdot}{\cdot}{\Gamma})_{\epsilon \in (0,\epsilon_0]}$ is a good approximation scheme in the sense of Definition~\ref{def:good_approximation} (this includes the assumption that $\median{\epsilon} > 0$).

Throughout this section, we let
\begin{align*}
 \CD_k &= 2^{-k}\Z^2 ,\quad \CD = \bigcup_{k \in \N} \CD_k ,\\
 \FS_k &= \left\{ [i2^{-k},(i+1)2^{-k}] \times [j2^{-k},(j+1)2^{-k}] \subseteq \D \ :\  i,j\in\Z \right\} ,\\
 \FQ_k &= \{ \operatorname{int}(S_1 \cup S_2 \cup \cdots) \text{ simply connected} : S_1,S_2,\ldots \in \FS_k \} ,\\
 \FQ &= \bigcup_{k \in \N} \FQ_k ,
\end{align*}
where $\operatorname{int}(S)$ denotes the topological interior of a set $S$.

For each open, simply connected $Q \subseteq \D$, we define a metric $D_\epsilon^{Q}$ as follows. Let $\cserial$ be the constant in the assumption~\eqref{it:mon_large_loops} of the monotonicity~\eqref{eq:approx_monotonicity_axiom}. By \cite[Lemma~5.14]{amy-cle-resampling}, for each $Q_1 \Subset Q$ there exists a finite set of points $z_1,z_2,\ldots \in (Q \setminus Q_1) \cap \Upsilon_\Gamma$ that separate $Q_1 \cap \Upsilon_\Gamma$ from $\partial Q$. For each $\epsilon > 0$, we select a set of points $\{z_1,z_2,...\} \subseteq Q \cap \Upsilon_\Gamma$ with $\abs{z_i-z_{i'}} \ge \cserial\epsilon$ and $\distE(z_i,\partial Q) \ge \cserial\epsilon$ for each $i\neq i'$, and let $K \subseteq Q \cap \Upsilon_\Gamma$ be the set of points separated from $\partial Q$ by $\{z_1,z_2,...\}$. We select $\{z_1,z_2,...\}$ so that $\max_{x \in Q \cap \Upsilon_\Gamma \setminus K} \distE(x, \partial Q)$ is minimized among all such sets of points (breaking ties e.g.\ with the lexicographical ordering of $\{\distE(z_i,\partial Q)\}_i$). Let $\wt{Q}_\epsilon$ be the minimal simply connected set with $K \subseteq \ol{\wt{Q}_\epsilon}$ (i.e.\ $\ol{\wt{Q}_\epsilon}$ is the union of $\ol{K}$ with its bounded complementary components). Let $D_\epsilon^{Q} = \median{\epsilon}^{-1}\metapproxres{\epsilon}{\wt{Q}_\epsilon}{\cdot}{\cdot}{\Gamma}$.\footnote{The concrete choices we made in the construction of $D_\epsilon^{Q}$ are not relevant as long as the monotonicity~\eqref{eq:approximations_monotone} is satisfied. This is because we will construct the limiting metric via the infimum~\eqref{eq:limiting_metric} below.}

For $Q_1 \Subset Q$, let $D_\epsilon^{Q,Q_1}$ be the restriction of $D_\epsilon^{Q}$ to $Q_1 \times Q_1$. By the monotonicity~\eqref{eq:approx_monotonicity_axiom} of the approximating metrics, we have
\begin{equation}\label{eq:approximations_monotone}
 D_\epsilon^{Q',Q_1} \le D_\epsilon^{Q,Q_1} \quad\text{for each } Q_1 \Subset Q \Subset Q'
\end{equation}
when $\epsilon$ is sufficiently small.

For each $Q_1,Q \in \FQ$ with $Q_1 \Subset Q$, consider the random function $D_\epsilon^{Q,Q_1}$ in the topology defined in Section~\ref{subsec:topology}. By Corollary~\ref{co:tightness_original_metric}, the assumption~\eqref{eq:approx_error_asymp}, Lemma~\ref{le:gasket_compact}, and Proposition~\ref{pr:ghf_tightness_hoelder_approx}, the laws of
\begin{equation}\label{eq:converging_metrics}
 \left( \Upsilon_\Gamma ,\, \left( D_\epsilon^{Q,Q_1} \right)_{Q,Q_1} \right)
\end{equation}
are tight as we let $\epsilon \searrow 0$. Pick a subsequence $\epsilon_n \searrow 0$ so that the joint vector for all $Q_1,Q \in \FQ$, $Q_1 \Subset Q$ converges in law. Write
\[ \left( \Upsilon_\Gamma ,\, \left( D^{Q,Q_1} \right)_{Q,Q_1} \right) \]
for the limit.

\begin{lemma}\label{le:limiting_rv}
 The following hold almost surely.
 \begin{enumerate}[(i)]
  \item\label{it:extension} If $Q_1 \subseteq Q_1' \Subset Q$, then $D^{Q,Q_1} = D^{Q,Q_1'}$ on $(Q_1 \cap \Upsilon_\Gamma) \times (Q_1 \cap \Upsilon_\Gamma)$.
  \item\label{it:monotonicity} If $Q_1 \Subset Q \Subset Q'$, then $D^{Q',Q_1} \le D^{Q,Q_1}$.
 \end{enumerate}
\end{lemma}

\begin{proof}
 These properties are satisfied by the approximations $(D_\epsilon^{Q,Q_1})$, in particular~\eqref{eq:approximations_monotone}. They are preserved by the GHf convergence, cf.\ Lemmas~\ref{le:ghf_subspace},~\ref{le:ghf_subspace_proj}, and~\ref{le:ghf_monotonicity}.
\end{proof}

Lemma~\ref{le:limiting_rv}\eqref{it:extension} allows us to define the random function $D^Q$ on $(Q \cap \Upsilon_\Gamma) \times (Q \cap \Upsilon_\Gamma)$ as the common extension of $D^{Q,Q_1}$, $Q_1 \Subset Q$. For $V \in \metregions$ define
\begin{equation}\label{eq:limiting_metric}
 \metres{V}{x}{y}{\Gamma} = \sup_{Q \in \FQ, Q \Supset V} D^Q(x,y) ,\quad x,y \in \ol{V} \cap \Upsilon_\Gamma .
\end{equation}

It follows immediately from~\eqref{eq:limiting_metric} that the family of metrics just constructed satisfy the monotonicity
\[ \metres{V'}{x}{y}{\Gamma} \le \metres{V}{x}{y}{\Gamma} ,\quad x,y \in \ol{V} \cap \Upsilon_\Gamma , \]
for $V \subseteq V'$, and the separability condition
\[ \metres{V}{x}{y}{\Gamma} = \lim_{V' \searrow V} \metres{V'}{x}{y}{\Gamma} ,\quad x,y \in \ol{V} \cap \Upsilon_\Gamma . \]

\begin{theorem}\label{th:metric}
 Let $(\metapprox{\epsilon}{\cdot}{\cdot}{\Gamma})_{\epsilon \in (0,\epsilon_0]}$ be a good approximation scheme. Let $(\epsilon_n)$ be any subsequence as above. Then the random family of functions defined in~\eqref{eq:limiting_metric} is almost surely a \clekp{} metric in the sense of Definition~\ref{def:cle_metric} (with $\epsilon = 0$).
\end{theorem}

We complete the proof of Theorem~\ref{th:metric} in Section~\ref{se:proof_axioms}. At this point, it is not guaranteed that the pseudo-metrics $(\metres{V}{\cdot}{\cdot}{\Gamma})$ are true metrics. We will show in Section~\ref{se:nondegeneracy} that either $(\metres{V}{\cdot}{\cdot}{\Gamma})$ are true metrics or they are identically zero. For approximations of geodesic metrics we will show in Section~\ref{se:geodesic_metric} that the former is the case.

\begin{remark}\label{rm:compatibility_renormalisation}
The choice of our renormalization factors $\median{\epsilon}$ in Section~\ref{se:intersections_setup} suggest that the subsequential limits should not be identically zero. However, without further assumptions it is not clear that the random variables $\median{\epsilon}^{-1}\metapproxres{\epsilon}{U_{x',y'}}{x}{y}{\Gamma_1}$ in Section~\ref{se:intersections_setup} converge in law to the random variable $\metres{U_{x',y'}}{x}{y}{\Gamma_1}$ constructed in~\eqref{eq:limiting_metric}, and hence the non-triviality of the latter is not clear. This is due to the fact that the setup of Section~\ref{se:intersections_setup} corresponds to the region bounded between the \emph{outer boundaries} of two loop segments. If we let $U'$ contain a neighborhood of the region, i.e.\ we attach to $U_{x',y'}$ the ``dead ends'' disconnected by either $\eta'_{1,\delta}$ or $\eta'_{2,\delta}$, then for $\epsilon > 0$ the internal metric $\metapproxres{\epsilon}{U'}{x}{y}{\Gamma_1}$ might be different from $\metapproxres{\epsilon}{U_{x',y'}}{x}{y}{\Gamma_1}$ because we have imposed an additional condition in the compatibility assumption in Section~\ref{se:assumptions} which does not hold for $U_{x',y'}$ and $U'$. (In the case of geodesic approximation schemes, the compatibility property holds without the extra restriction, hence we do have $\median{\epsilon}^{-1}\metapproxres{\epsilon}{U_{x',y'}}{x}{y}{\Gamma_1} \to \metres{U_{x',y'}}{x}{y}{\Gamma_1}$ in that case, see Lemma~\ref{le:geodesic_metric_positive}.)
\end{remark}

\subsection{Proof that the limits are CLE metrics}
\label{se:proof_axioms}

\begin{proposition}[compatibility]\label{pr:compatibility}
 Almost surely the following holds. Let $V,V' \in \metregions$, $V \subseteq V'$, and $x,y \in \ol{V} \cap \Upsilon_\Gamma$ such that for every $u \in V' \setminus V$ there is a point $z \in \ol{V}$ that separates $u$ from $x,y$ in $\ol{V'} \cap \Upsilon_\Gamma$. Then $\metres{V}{x}{y}{\Gamma} = \metres{V'}{x}{y}{\Gamma}$.
\end{proposition}

\begin{proof}
 By the definition of $\metregions$, the boundary of $V'$ is part of finitely many loops of $\Gamma$. Therefore, if $Q' \supseteq \ol{V'}$ is contained in a sufficiently small neighborhood of $\ol{V'}$, then every $u \in Q' \setminus \ol{V'}$ is separated from $V'$ in $Q' \cap \Upsilon_\Gamma$ by a single point $z' \in \partial V'$. Let $Q,Q' \in \FQ$, $V \Subset Q \Subset Q'$. We argue that $D^Q(x,y) = D^{Q'}(x,y)$ for every pair $x,y \in \ol{V} \cap \Upsilon_\Gamma$ as above. This will imply the result due to~\eqref{eq:limiting_metric} and Lemma~\ref{le:limiting_rv}.
  
 Let $0 < r < \distE(V,\partial Q)$, and let $V_r$ be the $r$-neighborhood of $V$.
 
 Suppose $\Gamma$ and $(\Gamma_n)$ together with their internal metrics are coupled so that the random vector~\eqref{eq:converging_metrics} converges almost surely. Recall that we have chosen the topology in Section~\ref{subsec:topology} to keep track of the separation points. Therefore, almost surely, for each pair $x,y \in \ol{V} \cap \Upsilon_\Gamma$ as above there are sequences of $x_n,y_n \in Q \cap \Upsilon_{\Gamma_n}$, $x_n \to x$, $y_n \to y$ such that for every $u \in (Q' \setminus V_r) \cap \Upsilon_{\Gamma_n}$ there is a point $z \in V_r \cap \Upsilon_{\Gamma_n}$ that separates $u$ from $x_n,y_n$ in $Q' \cap \Upsilon_{\Gamma_n}$. Therefore we can apply the compatibility assumption for the approximating metrics which implies that $D_{\epsilon_n}^{Q}(x_n,y_n) = D_{\epsilon_n}^{Q'}(x_n,y_n)$ for large $n$, and hence $D^Q(x,y) = D^{Q'}(x,y)$.
 \end{proof}

\begin{lemma}
 Almost surely, for each $V \in \metregions$, the function $\metres{V}{\cdot}{\cdot}{\Gamma}$ defines a pseudo-metric on $\ol{V} \cap \Upsilon_\Gamma$, and is continuous with respect to $\dpath[\ol{V}]$.
\end{lemma}

\begin{proof}
 Since $D_\epsilon^{Q}$ is symmetric and satisfies the triangular inequality for each $\epsilon$, these properties are preserved for $D^Q$ by the topology in which we take limits. Thanks to Lemma~\ref{le:limiting_rv}, they are also preserved under taking the supremum in~\eqref{eq:limiting_metric}. The fact that it is zero on the diagonal also follows from Propositions~\ref{prop:interior_tightness} and~\ref{pr:ghf_tightness_hoelder_approx}. Together, this implies also the continuity with respect to $\dpath$, at least in the interior of $V$. To see that the continuity also holds up to the boundary, note that by the definition of $\metregions$, the boundary of $V$ is part of finitely many loops of $\Gamma$. Therefore, if $V' \supseteq \ol{V}$ is contained in a sufficiently small neighborhood of $\ol{V}$, then every $u \in \ol{V'} \setminus \ol{V}$ is separated from $V$ in $\ol{V'} \cap \Upsilon_\Gamma$ by a single point $z \in \partial V$. By the compatibility Proposition~\ref{pr:compatibility} we have $\metres{V}{\cdot}{\cdot}{\Gamma} = \metres{V'}{\cdot}{\cdot}{\Gamma}$ on $\ol{V}$, in particular $\metres{V}{\cdot}{\cdot}{\Gamma}$ is continuous on $\ol{V} \cap \Upsilon_\Gamma$.
\end{proof}

\subsubsection{Proof of the Markovian property}

\begin{proposition}\label{pr:markov_property}
 Let $U \subseteq \D$ be a simply connected domain. The conditional law of the collection $(\metres{V}{\cdot}{\cdot}{\Gamma})_{V \in \metregions[U]}$ given $\Gamma\setminus\Gamma_{U^*}$ and $(\metres{V'}{\cdot}{\cdot}{\Gamma})_{V' \in \metregions[\D\setminus\ol{U}]}$ is almost surely measurable with respect to $U^*$.
\end{proposition}

We first prove a weaker variant of Proposition~\ref{pr:markov_property} where we do not condition on the internal metrics outside of $\ol{U}$.

\begin{lemma}\label{le:weak_markov}
 Let $U_1,\ldots,U_m \subseteq \D$ be disjoint open, simply connected subsets. The conditional law of the collection $(\metres{V}{\cdot}{\cdot}{\Gamma})_{V \in \metregions[U_i] ,\, i=1,\ldots,m}$ given $\Gamma\setminus\bigcup_{i=1,\ldots,m}\Gamma_{U_i^*}$ is almost surely measurable with respect to $(U_1^*,\ldots,U_m^*)$.
\end{lemma}

\begin{proof}
 This property holds for the approximating metrics which we can see by repeatedly applying their Markovian property. To see that this continues to hold for their weak limits, we can apply Lemma~\ref{le:cond_independence_limit_fixed_kernel} with $\wt{\mu}$ being the conditional law of $\Gamma\setminus\bigcup_{i=1,\ldots,m}\Gamma_{U_i^*}$ given $(U_1^*,\ldots,U_m^*)$.
\end{proof}

The main step for proving Proposition~\ref{pr:markov_property} is to show the following.

\begin{lemma}\label{le:cond_indep_boxes}
 Let $U \subseteq \D$ be a simply connected domain. Let $Q_1,\ldots,Q_n,Q'_1,\ldots,Q'_{n'} \in \FQ$ such that $Q_i \Subset U$ and $\dist(Q'_{i'},U) > 0$ for every $i,i'$. Then the conditional law of $(D^{Q_i})_i$ given $\Gamma\setminus\Gamma_{U^*}$ and $(D^{Q_{i'}})_{i'}$ is measurable with respect to $U^*$.
\end{lemma}

The proof relies on the continuity result Lemma~\ref{le:partial_cle_tv}. Recall the partial exploration $\Gamma_\outside^{*,V,U}$ of $\Gamma$ introduced in Section~\ref{se:cle_resampling}. Let $\alpha^*$ denote the interior link pattern induced by the remainders of the unfinished loops in $\Gamma_\outside^{*,V,U}$. By Lemma~\ref{le:partial_cle_tv}, for each $U_1 \Subset U$, the conditional law of the pair $(U_1^*,\alpha^*)$ is a continuous in total variation function of $\Gamma_\outside^{*,V,U}$, and by Theorem~\ref{thm:cle_partially_explored} the conditional probability of $\alpha^* = \alpha_0$ is positive for each link pattern $\alpha_0$. It follows that the conditional law of $U_1^*$ given $(\Gamma_\outside^{*,V,U}, \alpha^*)$ is continuous in total variation.

\begin{lemma}\label{le:cond_indep_partial}
 Let $U,V$ be open, simply connected such that $U \Subset V$. Let $Q_1,\ldots,Q_n,Q'_1,\ldots,Q'_{n'} \in \FQ$ such that $Q_i \Subset U$ and $Q'_{i'} \cap V = \varnothing$ for every $i,i'$. Then the law of $(D^{Q_i})_i$ is conditionally independent of $(D^{Q_{i'}})_{i'}$ given $(\Gamma_{\mathrm{out}}^{*,V,U}, \alpha^*)$.
\end{lemma}

\begin{proof}
 Pick $U_1 \Subset U$ simply connected such that $Q_i \Subset U_1$ for every $i$. Note that the gasket outside $V$ is determined by $(\Gamma_{\mathrm{out}}^{*,V,U}, \alpha^*)$. Therefore, by the Markovian property of the approximating metrics, the law of $(D_{\epsilon_n}^{Q_i})_i$ is conditionally independent of $(D_{\epsilon_n}^{Q'_{i'}})_{i'}$ given $(\Gamma_{\mathrm{out}}^{*,V,U}, \alpha^*)$. Further, the conditional law is obtained by first sampling $U_1^*$ and then sampling the internal metrics in $U_1$ according their conditional laws given $U_1^*$. As discussed in the paragraph above, the conditional law of $U_1^*$ given $(\Gamma_{\mathrm{out}}^{*,V,U}, \alpha^*)$ is a continuous function of $(\Gamma_{\mathrm{out}}^{*,V,U}, \alpha^*)$ in the total variation sense. This implies that also the conditional law of the internal metrics is a continuous function of $(\Gamma_{\mathrm{out}}^{*,V,U}, \alpha^*)$ in the total variation sense, uniformly in $\epsilon_n$. Therefore, by Propositions~\ref{prop:interior_tightness} and~\ref{pr:ghf_tightness_hoelder_approx} combined with Corollary~\ref{co:subsequential_limit_cond_laws}, we can find a further subsequence of $(\epsilon_n)$ so that the conditional laws of the joint vector converge almost surely. The result follows by applying Lemma~\ref{le:cond_independence_limit} to this subsequence.
\end{proof}

\begin{proof}[Proof of Lemma~\ref{le:cond_indep_boxes}]
 Pick $U',V$ open, simply connected such that $U \Subset U' \Subset V$ and $Q'_{i'} \cap V = \varnothing$ for every $i'$. By Lemma~\ref{le:cond_indep_partial}, $(D^{Q_i})_i$ and $(D^{Q_{i'}})_{i'}$ are conditionally independent given $(\Gamma_{\mathrm{out}}^{*,V,U'}, \alpha^*, U^*)$. Lemma~\ref{le:weak_markov} then implies that $(D^{Q_i})_i$, $(D^{Q_{i'}})_{i'}$, and the remainder of $\Gamma\setminus\Gamma_{U^*}$ are jointly conditionally independent given $(\Gamma_{\mathrm{out}}^{*,V,U'}, \alpha^*, U^*)$. Applying Lemma~\ref{le:weak_markov} one more time with just $U$ shows the result.
\end{proof}

\begin{proof}[Proof of Proposition~\ref{pr:markov_property}]
 By~\eqref{eq:limiting_metric} and Lemma~\ref{le:limiting_rv}, the collection of $(\metres{V}{\cdot}{\cdot}{\Gamma})_{V \in \metregions[U]}$ is determined by the countable collection of $(D^{Q_1})_{Q_1 \in \FQ, Q_1 \Subset U}$, and similarly the collection of $(\metres{V'}{\cdot}{\cdot}{\Gamma})_{V' \in \metregions[\D\setminus\ol{U}]}$ is determined by the countable collection of $(D^{Q'})_{Q' \in \FQ, Q' \Subset \D\setminus\ol{U}}$. These are generated by the finite collections of $D^{Q_1}$ (resp.\ $D^{Q'}$). By Lemma~\ref{le:cond_indep_boxes}, the conditional independence holds for the intersection stable generators. This implies the result.
\end{proof}

\subsubsection{Proof of translation invariance}

For $U \subseteq \D$ open, simply connected, we let $\p_U$ denote the law of $U^*$. For $z \in \C$, let $T_z$ be the translation operator $T_z\Fd(x,y) = \Fd(x-z,y-z)$.

The following lemma shows that the conditional law of $(D^Q)_{Q \Subset U}$ given $U^*$ is a translation invariant function of $U^*$.

\begin{lemma}\label{le:translation_invariant_limit}
 Let $(\epsilon_n)$ be chosen as in the beginning of Section~\ref{se:construction_metric}. Let $U \subseteq \D$ be open, simply connected. There exists a probability kernel $\mu^{U^*}$ from the space of $U^*$ to the space of $(D^Q)_{Q \in \FQ, Q \Subset U}$ such that the following is true. For every $z \in \C$ the joint laws of $((U+z)^*, (D_{\epsilon_n}^{Q+z,Q_1+z})_{Q_1 \Subset Q \Subset U})$ converge to the annealed law $\p_{U+z} \otimes (T_z)_* \mu^{U^*}(\cdot -z)$.
\end{lemma}

\begin{proof}
 Let $\mu^{U^*}_n$ be the conditional law of $(D_{\epsilon_n}^{Q,Q_1})_{Q_1 \Subset Q \Subset U}$ given $U^*$. By the translation invariance of the approximating metrics, there is a version of $\mu^{U^*}_n$ such that for each $z$, we have $\mu^{(U+z)^*}_n = (T_z)_* \mu^{U^*}_n(\cdot -z)$ almost surely. We have picked the sequence $(\epsilon_n)$ so that $\p_{U+z} \otimes \mu^{(U+z)^*}_n$ converges weakly for every $z\in\CD$. By applying Lemma~\ref{le:same_cond_law_weak_limit} repeatedly, we find a probability kernel $\mu^{U^*}$ so that
 \begin{equation}\label{eq:weak_convergence_translation}
  \p_{U+z} \otimes \mu^{(U+z)^*}_n = \p_{U+z} \otimes (T_z)_* \mu^{U^*}_n(\cdot -z) \to \p_{U+z} \otimes (T_z)_* \mu^{U^*}(\cdot -z)
 \end{equation}
 weakly for every $z\in\CD$.
  
 For general $z$, by the total variation convergence $\lim_{w\to z} (T_{z-w})_* \p_{Q+w} = \p_{Q+z}$ from Lemma~\ref{le:translation_tv}, the law $\p_{Q+z}$ is absolutely continuous with respect to the collection $(\p_{Q+w})_{w\in\CD}$. Therefore, by Lemma~\ref{le:weak_limit_abs_cont}, the convergence~\eqref{eq:weak_convergence_translation} holds for each $z\in\C$.
\end{proof}

\begin{proposition}\label{pr:translation_invariance}
 Let $U \subseteq \D$ be a simply connected domain. There exists a probability kernel $\mu^{U^*}$ such that for each $z$, the conditional law of $(\metres{V}{\cdot}{\cdot}{\Gamma})_{V \in \metregions[U+z]}$ given $(U+z)^*$ is $(T_z)_* \mu^{U^*}(\cdot -z)$.
\end{proposition}

\begin{proof}
 For $z \in \CD$, this is an immediate consequence of Lemma~\ref{le:translation_invariant_limit} and~\eqref{eq:limiting_metric}. It remains to show the case for general $z$.
 
 For every $V \in \metregions[U]$, due to Lemma~\ref{le:limiting_rv}, we can pick a decreasing sequence of $Q_m \in \FQ$, $Q_m \searrow \ol{V}$, such that $\metres{V}{\cdot}{\cdot}{\Gamma} = \lim_m D^{Q_m}$. Similarly, for every $\wt{V} \in \metregions[U+z]$ we can pick a decreasing sequence of $\wt{Q}_m \in \FQ$, $\wt{Q}_m \searrow \ol{\wt{V}}$, such that $\metres{\wt{V}}{\cdot}{\cdot}{\Gamma} = \lim_m D^{\wt{Q}_m}$. Pick a further subsequence of $(\epsilon_n)$ such that the law of the joint vector $((U+z)^*, (D_{\epsilon_n}^{Q+z})_{Q \Subset U}, (D_{\epsilon_n}^{\wt{Q}})_{\wt{Q} \Subset U+z})$ converges (when restricted to compact subsets of $Q$) to the law of a random vector $((U+z)^*, (\wt{D}^{Q+z})_{Q \Subset U}, (D^{\wt{Q}})_{\wt{Q} \Subset U+z})$. By Lemma~\ref{le:translation_invariant_limit}, the conditional law of $(\wt{D}^{Q+z})_{Q \Subset U}$ given $(U+z)^*$ agrees with the pushforward of the conditional law of $(D^{Q})_{Q \Subset U}$ under translation. Hence, the conditional law of $(\lim_m \wt{D}^{Q_m+z})_{\wt{V} \in \metregions[U+z]}$ agrees with the conditional law of $(\metres{V}{\cdot}{\cdot}{\Gamma})_{V \in \metregions[U]}$ under translation. On the other hand, the proof of Lemma~\ref{le:limiting_rv}\eqref{it:monotonicity} shows that the same monotonicity property holds for the collection of metrics $((\wt{D}^{Q+z})_{Q \Subset U}, (D^{\wt{Q}})_{\wt{Q} \Subset U+z})$. Therefore we have $\lim_m D^{\wt{Q}_m} = \lim_m \wt{D}^{Q_m+z}$ almost surely for each $\wt{V} \in \metregions[U+z]$ which concludes the proof.
\end{proof}

\subsubsection{Proof of the series law and the generalized parallel law}

\begin{proposition}
 Almost surely the following holds. Let $V \in \metregions$ and $x,y,z \in \ol{V} \cap \Upsilon_\Gamma$ such that $z$ disconnects $x$ from $y$ in $\ol{V} \cap \Upsilon_\Gamma$. Then
 \[ \metres{V}{x}{y}{\Gamma} = \metres{V}{x}{z}{\Gamma}+\metres{V}{z}{y}{\Gamma} . \]
\end{proposition}

\begin{proof}
 By the definition of $\metregions$, the boundary of $V$ is part of finitely many loops of $\Gamma$. Therefore, if $Q \supseteq \ol{V}$ is contained in a sufficiently small neighborhood of $\ol{V}$, then every $u \in Q \setminus \ol{V}$ is separated from $V$ in $Q \cap \Upsilon_\Gamma$ by a single point $z' \in \partial V$. Therefore, due to~\eqref{eq:limiting_metric} and Lemma~\ref{le:limiting_rv}, it suffices to show that
 \[ D^Q(x,y) \ge D^Q(x,z)+D^Q(z,y) \]
 for each $Q \in \FQ$ and each $x,y,z \in Q \cap \Upsilon_\Gamma$ such that $z$ disconnects $x$ from $y$ in $Q \cap \Upsilon_\Gamma$.
 
 Suppose $\Gamma$ and $(\Gamma_n)$ together with their internal metrics are coupled so that the random vector~\eqref{eq:converging_metrics} converges almost surely. Since the set of double points of non-simple SLE does not contain isolated points, for each $x,y,z$ as above there is a non-constant sequence of $z^m \to z$ such that each $z^m$ disconnects $x$ from $y$ in $Q \cap \Upsilon_\Gamma$. Upon swapping the roles of $x,y$, we can assume that $z^m$ disconnects $x$ from $z$. By continuity, it suffices to show that
 \[ D^Q(x,y) \ge D^Q(x,z^m)+D^Q(z,y) \]
 for each $m$.
 
 Recall that we have chosen the topology in Section~\ref{subsec:topology} to keep track of the separation points. Hence, there are sequences of $x_n,y_n,z_n,z^m_n \in Q \cap \Upsilon_{\Gamma_n}$ converging to $x,y,z,z^m$, respectively, such that $z^m_n$ disconnects $x_n$ from $z_n,y_n$, and $z_n$ disconnects $y_n$ from $x_n,z^m_n$ in $Q \cap \Upsilon_{\Gamma_n}$. Although the connected components of $Q \cap \Upsilon_{\Gamma_n} \setminus \{z_n,z^m_n\}$ containing $x_n$ and $y_n$ may not have positive Euclidean distance from each other due to double points of CLE loops, we can use the compatibility Proposition~\ref{pr:compatibility} to restrict to sub-components that do have uniformly positive Euclidean distance. For $\epsilon_n$ small enough, the series law then implies
 \[ D_{\epsilon_n}^{Q}(x_n,y_n) \ge D_{\epsilon_n}^{Q}(x_n,z^m_n)+D_{\epsilon_n}^{Q}(z_n,y_n) . \]
 This concludes the proof.
\end{proof}

 \begin{proposition}
  Let $N \in \N$ and $\cparallel(N) > 0$ be the constant in the generalized parallel law for $\metapprox{\epsilon}{\cdot}{\cdot}{\Gamma}$. Almost surely the following holds. Let $V \in \metregions$ and let $x,y,z_1,\ldots,z_N \in \ol{V} \cap \Upsilon_\Gamma$ such that $x,y$ are separated in $\ol{V} \cap \Upsilon_\Gamma \setminus\{z_1,\ldots,z_N\}$. Then
 \[ \cparallel(N)\metres{V}{x}{y}{\Gamma} \ge \min_{i}\metres{V_x}{x}{z_i}{\Gamma} \]
 where $V_x \subseteq V$ is any region such that $\ol{V_x}$ contains the connected component of $\ol{V} \cap \Upsilon_\Gamma \setminus\{z_1,\ldots,z_N\}$ containing $x$.
 \end{proposition}
 
 \begin{proof}
  Let $V \in \metregions$, $x,y,z_1,\ldots,z_N \in \ol{V} \cap \Upsilon_\Gamma$, and $V_x$ be as in the proposition statement. We argue that we can find sequences $Q_m \searrow \ol{V}$, $(Q_x)_m \searrow \ol{V_x}$ such that
  \[ \cparallel(N)D^{Q_m}(x,y) \ge \min_{i}D^{(Q_x)_m}(x,z_i) \quad\text{for each } m . \]
  This will imply the result due to~\eqref{eq:limiting_metric} and Lemma~\ref{le:limiting_rv}.
  
  Suppose $\Gamma$ and $(\Gamma_n)$ together with their internal metrics are coupled so that the random vector~\eqref{eq:converging_metrics} converges almost surely. By the definition of $\metregions$, the boundary of $V$ is part of finitely many loops of $\Gamma$. Therefore, if $Q \supseteq \ol{V}$ is contained in a sufficiently small neighborhood of $\ol{V}$, then every $u \in Q \setminus \ol{V}$ is separated from $V$ in $Q \cap \Upsilon_\Gamma$ by a single point $z' \in \partial V$. Recall that we have chosen the topology in Section~\ref{subsec:topology} to keep track of the separation points. Therefore there are sequences of points with $x_n \to x$, $y_n \to y$, $(z_i)_n \to z_i$ for each $i$ such that $x_n,y_n$ are separated in $Q \cap \Upsilon_{\Gamma_n} \setminus\{(z_1)_n,\ldots,(z_N)_n\}$.
    
  Let $Q \in \FQ$, $V \Subset Q$. Let $\wt{Q}_\epsilon$ be as in the definition of $D_\epsilon^{Q}$. We apply the generalized parallel law to $\metapproxres{\epsilon_n}{\wt{Q}_{\epsilon_n}}{x_n}{y_n}{\Gamma_n}$. Let $\wt{V}_x \in \metregions$ contain a small neighborhood of the connected component of $\wt{Q}_{\epsilon_n} \cap \Upsilon_{\Gamma_n} \setminus\{(z_1)_n,\ldots,(z_N)_n\}$. We can choose $Q_x \in \FQ$ such that $\wt{V}_x \Subset Q_x$ and $Q_x$ is contained in a small neighborhood of $\ol{V_x}$. By the generalized parallel law of the approximating metrics and the monotonicity~\eqref{eq:approximations_monotone} we have
  \[ \cparallel(N)D_{\epsilon_n}^{Q}(x_n,y_n) \ge \min_{i}\median{\epsilon}^{-1}\metapproxres{\epsilon_n}{\wt{V}_x}{x_n}{(z_i)_n}{\Gamma_n} \ge \min_{i}D_{\epsilon_n}^{Q_x}(x_n,(z_i)_n) . \]
  The claim then follows since monotonicity is preserved under GHf convergence (Lemma~\ref{le:ghf_monotonicity}).
 \end{proof}

\subsection{Proof of non-degeneracy}
\label{se:nondegeneracy}

\ExplSyntaxOn

\NewDocumentCommand \footnotenumber { O{} }
  {
    \footnotemark
      {
        \tl_if_blank:nF { #1 }
          {
            \addtocounter { footnote } { -1 }
            \refstepcounter { footnote }
            \label { #1 }
          }
      }
  }

\NewDocumentCommand \footnotemultiple { O{} m }
  {
    \group_begin:
    \tl_if_blank:nF { #1 }
      {
        \clist_clear:N \l_tmpa_clist
        \clist_map_inline:nn { #1 }
          {
            \clist_put_right:Nn \l_tmpa_clist { \ref{ ##1 } }
          }
        \def \thefootnote { \clist_use:Nn \l_tmpa_clist { , } }
      }
    \footnotetext { #2 }
    \group_end:
  }

\ExplSyntaxOff

\begin{proposition}\label{pr:metric_positive}
 Suppose $\met{\cdot}{\cdot}{\Gamma}$ is a \clekp{} metric (with $\epsilon = 0$). Then either a.s.\ $\metres{V}{\cdot}{\cdot}{\Gamma} = 0$ for each $V \in \metregions$ or a.s.\ $\metres{V}{\cdot}{\cdot}{\Gamma}$ is a true metric for each $V \in \metregions$ (i.e.\ $\metres{V}{x}{y}{\Gamma} > 0$ whenever $x \neq y$).
\end{proposition}

\begin{proof}
 We argue that if with positive probability $\metres{V}{x}{y}{\Gamma} = 0$ for some $V$ and $x \neq y$, then there exists $q>0$ such that $\quant{q}{} = 0$ where $\quant{q}{}$ is the quantile defined in Section~\ref{se:intersections_setup}. By Corollary~\ref{co:quantiles_metric}, this would imply $\median{} = 0$, and then Proposition~\ref{prop:interior_tightness} would imply that the metric is constant zero.
 
 Suppose that with positive probability $\metres{V}{x}{y}{\Gamma} = 0$ for some $x \neq y$. Almost surely, each pair of points is separated by a finite chain of loops in the sense that there exist $\CL_1,\ldots,\CL_n \in \Gamma$ (some of them may belong to the boundary of $V$) and points $z_i \in \CL_i \cap \CL_{i+1}$, $i=1,\ldots,n$ (where $\CL_{n+1} = \CL_1$), that separate $x$ from $y$ in $\ol{V} \cap \Upsilon_\Gamma$ (see e.g.\ \cite[Lemma~2.1]{amy-cle-resampling}). For each such separating set of points, by the generalized parallel law, we must have $\metres{V_x}{x}{z_i}{\Gamma} = 0$ for some $i$. Since each pair of intersecting loops intersect at infinitely many points, there must be two distinct points $z_i,z'_i \in \CL_i \cap \CL_{i+1}$ for some $i$ such that $\metres{V_x}{x}{z_i}{\Gamma} = \metres{V_x}{x}{z'_i}{\Gamma} = 0$ (where we assumed $z'_i$ separates $z_i$ from $x$ in $\ol{V_x}$). In particular, $\metres{V_x}{z_i}{z'_i}{\Gamma} = 0$. By the compatibility property, we then have $\metres{V'}{z_i}{z'_i}{\Gamma} = 0$ where $V'$ is the region bounded between the segments of $\CL_i$ and $\CL_{i+1}$ from $z_i$ to $z'_i$.
 
 Let $\Gamma$ be coupled with $h$ as in Section~\ref{subsec:tightness_interior_flow_lines}. Sample $u,v \in B(0,1)$ uniformly at random. Let $\eta_u$ be the flow lines of $h$ starting from $u$. Given $\eta_u$, let $\eta_{u,v}$ be the flow line of $h$ starting from $v$ and reflected off $\eta_u$ in the opposite direction. Then the following event has positive probability. There exist $z,z' \in \eta_{u,v} \cap \eta_u$ such that the segment of $\eta_{u,v}$ from $z$ to $z'$ intersects $\eta_u$ only on the right side of $\eta_u$ with angle difference $0$ (before reflecting off), and $\metres{U_{z,z'}}{z}{z'}{\Gamma} = 0$ where $U_{z,z'}$ is the region bounded between the segments of $\eta_{u,v}$ and $\eta_u$ from $z$ to $z'$. Using the Markovian property and translation invariance, together with a local absolute continuity argument (similarly as in the first part of the proof of Lemma~\ref{le:a_priori_loop_intersecting_one_side})\footnotenumber[fn:transfer1], this implies that the following event has positive probability. Let $\wt{h}$ be a GFF on $\D$ with boundary values as in Lemma~\ref{le:reflected_fl_law}. Let $\wt{\eta}_1$ be the flow line of $\wt{h}$ from $-i$ to $i$, and let $\wt{\eta}_2$ be the flow line of $\wt{h}$ in the components of $\D\setminus\wt{\eta}_1$ to the right of $\wt{\eta}_1$ from $i$ to $-i$, reflected off $\wt{\eta}_1$. Then there exist $\wt{z},\wt{z}' \in \wt{\eta}_1 \cap \wt{\eta}_2 \cap B(0,1/2)$ such that
 \[ \p[ \met{\wt{z}}{\wt{z}'}{\wt{\Gamma}} = 0 ] > 0 . \]
 Finally, using the reversal symmetry (Lemma~\ref{le:reflected_fl_law}) and another local absolute continuity argument (similarly as in the proof of Lemma~\ref{le:a_priori_disc} or~\ref{le:tightness_bdry_bubble_fl})\footnotenumber[fn:transfer2]\footnotemultiple[fn:transfer1,fn:transfer2]{In fact, the argument simplifies in this case as we do not need quantitative bounds on the probabilities.}, we conclude that the following event has positive probability. Consider the setup in Section~\ref{se:intersection_exponent}. Then
 \[ \p\left[ \sup_{(x,y) \in \intpts{1}} \metres{U_{x,y}}{x}{y}{\Gamma} = 0 \right] > 0 . \]
 This means that $\quant{q}{} = 0$ for some $q>0$ which is what we wanted to show.
\end{proof}

\subsection{Geodesic metric}
\label{se:geodesic_metric}

In this subsection we define geodesic approximation schemes and prove that their subsequential limits yield geodesic metrics.

Let $\spaths{x}{y}{U}{\Gamma} \subseteq \paths{x}{y}{U}{\Gamma}$ denote the set of \emph{simple} admissible paths from $x$ to $y$ within $U$. We identify paths with different parameterizations and orientations. Let $\len{\Fd}{\gamma}$ denote the length of a path $\gamma$ under a metric $\Fd$.

Suppose that for each $\epsilon > 0$ we have a deterministic measurable function $\lebneb{\epsilon}(\gamma) \ge 0$ defined for each trace $\gamma$ of a simple path (i.e.\ not distinguishing orientation). Suppose that $\lebneb{\epsilon}$ satisfies the following properties.
\begin{itemize}
 \item Positivity: $\lebneb{\epsilon}(\gamma) > 0$ for every $\gamma$ consisting of more than a point.
 \item Translation invariance: $\lebneb{\epsilon}(\gamma) = \lebneb{\epsilon}(\gamma+z)$ for every $z\in\C$.
 \item Monotonicity: If $\gamma \subseteq \gamma'$, then $\lebneb{\epsilon}(\gamma) \le \lebneb{\epsilon}(\gamma')$.
 \item Subadditivity: $\lebneb{\epsilon}(\gamma_1 \oplus \gamma_2) \le \lebneb{\epsilon}(\gamma_1)+\lebneb{\epsilon}(\gamma_2)$ if the concatenation $\gamma_1 \oplus \gamma_2$ is a simple path.
 \item Approximate additivity: If $\gamma_1,\gamma_2 \subseteq \gamma$ and $\distE(\gamma_1,\gamma_2) \ge \cserial\epsilon$, then $\lebneb{\epsilon}(\gamma) \ge \lebneb{\epsilon}(\gamma_1)+\lebneb{\epsilon}(\gamma_2)$.
 \item Approximation scale: If $\diamE(\gamma) \le \epsilon$, then $\lebneb{\epsilon}(\gamma) \le \ac{\epsilon}$ where $\ac{\epsilon}$ is a fixed constant.
\end{itemize}
We consider the approximate CLE metrics defined by
\begin{equation}\label{eq:geodesic_approx_metric}
\metapproxres{\epsilon}{V}{x}{y}{\Gamma} = \inf_{\gamma \in \spaths{x}{y}{\ol{V}}{\Gamma}} \lebneb{\epsilon}(\gamma) .
\end{equation}
We further say that $(\metapprox{\epsilon}{\cdot}{\cdot}{\Gamma})_{\epsilon > 0}$ is a \emph{good geodesic approximation scheme} if there is a constant $\epsexp > 0$ such that for every $r>0$ there is $\epsilon_0 > 0$ such that
\begin{equation}\label{eq:eps_bound_geodesic}
 \inf_{\diamE(\gamma) \ge r} \lebneb{\epsilon}(\gamma) \ge \epsilon^{-\epsexp}\ac{\epsilon}
 \quad\text{for } \epsilon < \epsilon_0 .
\end{equation}
Note that this condition implies in particular~\eqref{eq:eps_bound_ass}. The following is immediate from definition.
\begin{proposition}
 Suppose we have the setup above. Then $\metapprox{\epsilon}{\cdot}{\cdot}{\Gamma}$ is an approximate \clekp{} metric in the sense of Definition~\ref{def:cle_metric} with the constant $\cserial$ as above, and $\cparallel(N) = 1$. If further~\eqref{eq:eps_bound_geodesic} holds, then $(\metapprox{\epsilon}{\cdot}{\cdot}{\Gamma})_{\epsilon > 0}$ is a good approximation scheme in the sense of Definition~\ref{def:good_approximation}.
\end{proposition}

\begin{example}\label{ex:geodesic_approx}
We give two examples of good geodesic approximation schemes.
\begin{enumerate}
 \item Let $\lebneb{\epsilon}(\gamma)$ be the Lebesgue measure of the $\epsilon$-neighborhood of $\gamma$. Here we can set $\ac{\epsilon} = 4\pi\epsilon^2$, and we have $\inf_{\diamE(\gamma) \ge r} \lebneb{\epsilon}(\gamma) \ge 2r\epsilon$.
 \item Let $\lebneb{\epsilon}(\gamma)$ be the largest number $N$ such that there exist $t_1 < t_2 < \cdots < t_N$ with $\abs{\gamma(t_j)-\gamma(t_{j-1})} \ge \epsilon$ for each $j$. Here we can set $\ac{\epsilon} = 1$, and we have $\inf_{\diamE(\gamma) \ge r} \lebneb{\epsilon}(\gamma) \ge r\epsilon^{-1}$.
\end{enumerate}
\end{example}

\begin{lemma}\label{le:approximate_midpoint}
 Let $\lebneb{\epsilon}$ satisfy the properties stated above. Then for every simple path $\gamma\colon [0,1] \to \C$ there exists $s \in [0,1]$ such that
 \[ \abs*{ \lebneb{\epsilon}(\gamma[0,s])-\frac{1}{2}\lebneb{\epsilon}(\gamma) } \le \ac{\epsilon} . \]
\end{lemma}

\begin{proof}
 In case $\lebneb{\epsilon}(\gamma) \le 2\ac{\epsilon}$, this is trivial. Otherwise let
 \[ s = \inf\left\{ t \ :\  \lebneb{\epsilon}(\gamma[0,t]) \ge \frac{1}{2}\lebneb{\epsilon}(\gamma) \right\} . \]
 Then, for any $s' \in [0,s]$ with $\diamE(\gamma[s',s]) \le \epsilon$, by the approximation scale property $\lebneb{\epsilon}(\gamma[s',s]) \le \ac{\epsilon}$. Therefore
 \[ \lebneb{\epsilon}(\gamma[0,s]) \le \lebneb{\epsilon}(\gamma[0,s'])+\lebneb{\epsilon}(\gamma[s',s]) \le \frac{1}{2}\lebneb{\epsilon}(\gamma)+\ac{\epsilon} . \]
\end{proof}

In the remainder of this subsection, we let $(\metapprox{\epsilon}{\cdot}{\cdot}{\Gamma})_{\epsilon > 0}$ be a good geodesic approximation scheme as defined above. Let $(\epsilon_n)$ be chosen as in Section~\ref{se:construction_metric}, and consider the metrics defined in~\eqref{eq:limiting_metric}.

\begin{proposition}\label{pr:internal_length_metric}
 Consider a good geodesic approximation scheme. Then $\metres{V}{\cdot}{\cdot}{\Gamma}$ defined in~\eqref{eq:limiting_metric} is a non-degenerate geodesic metric for each $V \in \metregions$. Moreover, for each $V \subseteq V'$ we have
 \[ \metres{V}{x}{y}{\Gamma} = \min_{\gamma \in \paths{x}{y}{\ol{V}}{\Gamma}} \lenmetres{V'}{\gamma} ,\quad x,y \in \ol{V} \cap \Upsilon_\Gamma . \]
\end{proposition}

\begin{lemma}\label{le:geodesic_metric_dyadic}
 Consider a good geodesic approximation scheme. Then for each $Q,Q' \in \FQ$ with $Q \Subset Q'$ we have
 \[ 
   \inf_{\gamma \in \paths{x}{y}{Q'}{\Gamma}} \len{D^{Q'}}{\gamma} \le D^Q(x,y) ,\quad x,y \in Q \cap \Upsilon_\Gamma .
 \]
\end{lemma}

\begin{proof}
 We show that for each pair $x,y \in Q \cap \Upsilon_\Gamma$ there is $z \in \ol{Q} \cap \Upsilon_\Gamma$ with
 \begin{equation}\label{eq:approx_midpoint}
  D^{Q'}(x,z) \le D^Q(x,y)/2 ,\qquad D^{Q'}(y,z) \le D^Q(x,y)/2 .
 \end{equation}
 By applying~\eqref{eq:approx_midpoint} iteratively with a sequence $(Q'_m)$ with $Q \Subset Q'_1 \Subset Q'_2 \Subset \cdots \Subset Q'$ and using the completeness of $\ol{Q'} \cap \Upsilon_\Gamma$ we can construct a path $\gamma \in \paths{x}{y}{Q'}{\Gamma}$ as desired (cf.\ \cite[Proof of Theorem~2.4.16]{bbi-metric-geometry}).
 
 Suppose $\Gamma$ and $(\Gamma_n)$ are coupled so that the random vector~\eqref{eq:converging_metrics} converges almost surely. To control the interference between path segments with small Euclidean distance (resulting from the approximation), we introduce the following event. Let $F_\delta$ be the event that there is no loop of $\Gamma$ whose exterior boundary makes $4$ crossings across some annulus $A(w,\delta,\delta^\zeta) \subseteq \D$. By Lemma~\ref{lem:thin_cle_lwb} we can choose $\zeta>0$ so that $\lim_{\delta\searrow 0} \p[F_\delta^c] = 0$. On $F_\delta$, for any \emph{simple} admissible path $\gamma$ from $x$ to $y$ and $u,v \in \gamma$, $\abs{u-v} < \delta$, $\distE(\{u,v\},\{x,y\}) \ge \delta^\zeta$ we have $\dpath(u,v) \le \abs{u-v}^\zeta$.
 
 Suppose $Q \Subset Q'$. Fix $\delta_0 > 0$. By the continuity of $D^{Q'}$, there exists $\delta > 0$ such that $D^{Q'}(u,v) \le \delta_0$ for every $u,v \in Q$ with $\dpath(u,v) \le \delta^\zeta$. In the remainder of the proof, we suppose that we are on the event $F_\delta$.
 
 Let $x_n,y_n \in Q \cap \Upsilon_{\Gamma_n}$ such that $x_n \to x$, $y_n \to y$, and $D_{\epsilon_n}^{Q}(x_n,y_n) \to D^Q(x,y)$. Let $\gamma_n \in \spaths{x_n}{y_n}{Q}{\Gamma_n}$ be an admissible path with
 \[ \lebneb{\epsilon_n}(\gamma_n) \le \median{\epsilon_n}D_{\epsilon_n}^{Q}(x_n,y_n)+o(\median{\epsilon_n}) . \]
 Let $x'_n$ (resp.\ $y'_n$) be the first (resp.\ last) point on $\gamma_n$ such that $\dpath(x_n,x'_n) = \delta^\zeta$ (resp.\ $\dpath(y_n,y'_n) = \delta^\zeta$). Let $\gamma'_n \subseteq \gamma_n$ be the subsegment from $x'_n$ to $y'_n$. By Lemma~\ref{le:approximate_midpoint} and the assumption~\eqref{eq:eps_bound_ass}, there is $t_n \in [0,1]$ such that
 \begin{equation}\label{eq:half_length}
  \abs*{ \lebneb{\epsilon_n}(\gamma'_n[0,t_n])-\frac{1}{2}\lebneb{\epsilon_n}(\gamma'_n) } \le \ac{\epsilon} = o(\median{\epsilon_n}) .
 \end{equation}
 Let $u_n \in [t_n,1]$ be the last time that $\distE(\gamma'_n(u_n),\gamma'_n[0,t_n]) \le \cserial\epsilon_n$, and let $s_n \in [0,t_n]$ be the point with $\abs{\gamma'_n(u_n)-\gamma'_n(s_n)} = \cserial\epsilon_n$. By compactness we can assume that $x'_n$, $y'_n$, $\gamma'_n(s_n)$, $\gamma'_n(u_n)$ converge along a further subsequence to $x',y',z,w \in \ol{Q} \cap \Upsilon_\Gamma$. Since $\abs{z-w} = 0$, on $F_\delta$ we must have $\dpath(z,w) = 0$, i.e.\ $z=w$. We now argue that $z$ is an approximate midpoint of $x,y$.
 
 By the approximate additivity we have $\lebneb{\epsilon_n}(\gamma'_n[0,t_n])+\lebneb{\epsilon_n}(\gamma'_n[u_n,1]) \le \lebneb{\epsilon_n}(\gamma'_n)$. In particular, by~\eqref{eq:half_length},
 \[
  \lebneb{\epsilon_n}(\gamma'_n[u_n,1]) \le \frac{1}{2}\lebneb{\epsilon_n}(\gamma'_n)+o(\median{\epsilon_n}) \le \frac{1}{2}\median{\epsilon_n}D_{\epsilon_n}^{Q}(x_n,y_n)+o(\median{\epsilon_n}) .
 \]
 Since $\gamma_n(s_n),\gamma_n(u_n) \to z$, we have $D_{\epsilon_n}^{Q'}(x_n,\gamma_n(s_n)) \to D^{Q'}(x,z)$ and $D_{\epsilon_n}^{Q'}(y_n,\gamma_n(u_n)) \to D^{Q'}(y,z)$ for $Q' \Supset Q$. Finally, since $\dpath(x,x') = \dpath(y,y') = \delta^\zeta$, we have $D^{Q'}(x,x') \le \delta_0$, $D^{Q'}(y,y') \le \delta_0$. Therefore we have shown
 \begin{align*}
  D^{Q'}(x,z) &\le D^{Q'}(x,x')+D^{Q'}(x',z) \le \frac{1}{2}D^Q(x,y)+\delta_0 ,\\
  D^{Q'}(y,z) &\le D^{Q'}(y,y')+D^{Q'}(y',z) \le \frac{1}{2}D^Q(x,y)+\delta_0 .
 \end{align*}
 Since $\delta_0$ and $\delta$ can be chosen arbitrarily small, \eqref{eq:approx_midpoint} follows by the compactness of $\ol{Q} \cap \Upsilon_\Gamma$ and the continuity of $D^{Q'}$.
\end{proof}

\begin{lemma}\label{le:length_metric}
 Consider a good geodesic approximation scheme. Then for each $V \in \metregions$ we have
 \[ \metres{V}{x}{y}{\Gamma} = \min_{\gamma \in \paths{x}{y}{\ol{V}}{\Gamma}} \lenmetres{V}{\gamma} ,\quad x,y \in \ol{V} \cap \Upsilon_\Gamma . \]
\end{lemma}

\begin{proof}
By the definition of $\metregions$, the boundary of $V$ is part of finitely many loops of $\Gamma$. Therefore, if $Q \supseteq \ol{V}$ is contained in a sufficiently small neighborhood of $\ol{V}$, then every $u \in Q \setminus \ol{V}$ is separated from $V$ in $Q \cap \Upsilon_\Gamma$ by a single point $z' \in \partial V$. In that case, by the compatibility Proposition~\ref{pr:compatibility} we have $D^Q(x,y) = \metres{V}{x}{y}{\Gamma}$. Moreover, for every admissible path in $Q$ between $x,y \in \ol{V}$ we can find a subpath within $\ol{V}$. The result then follows from~\eqref{eq:limiting_metric} and Lemma~\ref{le:geodesic_metric_dyadic}.
\end{proof}

\begin{lemma}\label{le:geodesic_metric_positive}
 Consider a good geodesic approximation scheme. Then, for each $V \in \metregions$, the metric $\metres{V}{\cdot}{\cdot}{\Gamma}$ is a true metric.
\end{lemma}

\begin{proof}
By Proposition~\ref{pr:metric_positive}, either $\metres{V}{\cdot}{\cdot}{\Gamma}$ is a true metric for each $V$ or it is identically zero for each $V$. To see that the latter cannot be the case, note that the compatibility property in Section~\ref{se:assumptions} holds for the geodesic approximation schemes $\metapprox{\epsilon}{\cdot}{\cdot}{\Gamma}$ even without the restriction $\dpath[\ol{V'}](z, V' \setminus V) \ge \cserial\epsilon$. In particular, in the setup of Section~\ref{se:intersections_setup} we have $\metapproxres{\epsilon}{U_{x,y}}{\cdot}{\cdot}{\Gamma_\delta} = \metapproxres{\epsilon}{U'}{\cdot}{\cdot}{\Gamma'_\delta}$ for any $U' \supseteq U_{x,y}$. Therefore (by absolute continuity, cf.\ Lemma~\ref{le:weak_limit_abs_cont}) the law of the random variables $\median{\epsilon}^{-1}\metapproxres{\epsilon}{U_{x,y}}{x}{y}{\Gamma_\delta}$ converge to the corresponding $\metres{U_{x,y}}{x}{y}{\Gamma_\delta}$ for the metric in~\eqref{eq:limiting_metric}, and $\median{\epsilon}^{-1}\metapproxres{\epsilon}{U_{x,y}}{x}{y}{\Gamma_1}$ is normalized to have median $1$.
\end{proof}

\begin{lemma}\label{le:lengths_compatible}
 Consider a good geodesic approximation scheme. Then for each $V \subseteq V'$ and each admissible path $\gamma$ within $\ol{V}$ we have $\lenmetres{V}{\gamma} = \lenmetres{V'}{\gamma}$.
\end{lemma}

\begin{proof}
 We argue that for any $Q_1,Q \in \FQ$, $Q_1 \Subset Q$ there exists a (random) $\delta > 0$ such that if $x,y \in Q_1$ and $\dpath[Q_1](x,y) < \delta$, then $D^Q(x,y) = D^{Q'}(x,y)$ for any $Q' \in \FQ$, $Q' \supseteq Q$. This will imply the result since we have seen in Lemma~\ref{le:length_metric} that $D^Q(x,y) = \metres{V}{x}{y}{\Gamma}$ when $Q$ is contained in a sufficiently small neighborhood of $\ol{V}$.
 
 Suppose $\Gamma$ and $(\Gamma_n)$ are coupled so that the random vector~\eqref{eq:converging_metrics} converges almost surely.
 
 Pick any $Q_2$ such that $Q_1 \Subset Q_2 \Subset Q$. Let $r = \dist_{D^Q}(Q_1, \partial Q_2)$. By the non-degeneracy (Lemma~\ref{le:geodesic_metric_positive}), we have $r>0$ almost surely. Let $X_\zeta$ be as in Proposition~\ref{prop:interior_tightness}, so that $D^{Q,Q_1}(x,y) \le X_\zeta \dpath[Q_1](x,y)^\zeta$. Suppose $\delta$ is small enough so that $X_\zeta \delta^\zeta < r/4$, and consider $x,y$ such that $\dpath[Q_1](x,y) < \delta$. Let $x_n,y_n \in \Upsilon_{\Gamma_n}$ such that $x_n \to x$, $y_n \to y$, and $D_{\epsilon_n}^{Q}(x_n,y_n) \to D^Q(x,y)$. Then
 \[ D_{\epsilon_n}^{Q}(x_n,y_n) \le X_\zeta \delta^\zeta+o(1) < r/2 . \]
 On the other hand, we have
 \[ \inf_{z \in \partial Q_2} D_{\epsilon_n}^{Q}(x_n,z) \ge r-o(1) . \]
 In particular, any admissible path from $x$ to $y$ that exits $Q$ must satisfy $\median{\epsilon_n}^{-1}\lebneb{\epsilon_n}(\gamma) > r/2$. Therefore $D_{\epsilon_n}^{Q}(x_n,y_n) = D_{\epsilon_n}^{Q'}(x_n,y_n)$ and hence $D^Q(x,y) = D^{Q'}(x,y)$ for any $Q' \supseteq Q$.
\end{proof}

\begin{proof}[Proof of Proposition~\ref{pr:internal_length_metric}]
 This follows from Lemmas~\ref{le:length_metric} and~\ref{le:lengths_compatible}.
\end{proof}

\appendix

\section{Gromov-Hausdorff-function topology}
\label{app:ghf}

We first recall the definition of the Gromov-Hausdorff topology.  Suppose that $(X_i,d_i)$ for $i=1,2$ are non-empty compact metric spaces.  Then the \emph{Gromov-Hausdorff distance} between $(X_1,d_1)$ and $(X_2,d_2)$ is defined as
\[ \dGH( (X_1,d_1), (X_2,d_2) ) = \inf\{ \dH( \psi_1(X_1), \psi_2(X_2)) \}\]
where the infimum is over all metric spaces $(W,d_W)$ and isometric embeddings $\psi_i \colon X_i \to W$ for $i=1,2$. It is a standard fact \cite[Theorem~7.4.15]{bbi-metric-geometry} that a collection $\CK$ of compact metric spaces is compact with respect to $\dGH$ if and only if $\CK$ is uniformly bounded and uniformly totally bounded.

Let $(Y,d_Y)$ be a proper metric space (in the sense that every closed ball is compact). Suppose that $(W,d_W)$ is a metric space and for $i=1,2$ we have that $K_i \subseteq W$ is a non-empty compact subset and $f_i \colon K_i \to Y$ a bounded function. Then we define
\[ 
d_\infty(f_1,f_2) = \inf\left\{ \delta > 0 \ :\  \parbox{.5\linewidth}{For each $i=1,2$ and $x \in K_i$ there is $x' \in K_{3-i}$ with $d_W(x,x') \le \delta$ and $d_Y(f_i(x),f_{3-i}(x')) \le \delta$} \right\} .
\]
Note that when $f_1$ or $f_2$ is continuous, then $d_\infty(f_1,f_2) = 0$ if and only if $f_1=f_2$. For general bounded functions, we can have $d_\infty(f_1,f_2) = 0$ even when $f_1$, $f_2$ do not agree on every point.

Suppose now that $(X_i,d_i,f_i)$ for $i=1,2$ where $(X_i,d_i)$ are non-empty compact metric spaces and $f_i \colon K_i \to Y$ are bounded functions defined on some compact subsets $K_i \subseteq X_i$.  The \emph{Gromov-Hausdorff-function distance} between $(X_1,d_1,f_1)$ and $(X_2,d_2,f_2)$ is defined by
\[ \dGHf( (X_1,d_1,f_1), (X_2,d_2,f_2) ) = \inf\{ \dH( \psi_1(X_1), \psi_2(X_2)) + d_\infty( f_1 \circ \psi_1^{-1}, f_2 \circ \psi_2^{-1}) \} \]
where the infimum is over all metric spaces $(W,d_W)$ and isometric embeddings $\psi_i \colon X_i \to W$ for $i=1,2$.

Note that $\dGHf$ defines a pseudo-metric on the space of non-empty compact metric spaces equipped with bounded functions. To see the triangular inequality, suppose that we have isometric embeddings $\psi_1\colon X_1 \to W_{12}$, $\psi_2\colon X_2 \to W_{12}$, and $\psi'_2\colon X_2 \to W_{23}$, $\psi'_3\colon X_3 \to W_{23}$. Then we can define a metric space $W_{13}$ as the quotient $(W_{12} \sqcup W_{23})/{\sim}$ identifying $\psi_2(x) \sim \psi'_2(x)$ for every $x \in X_2$.

\begin{proposition}\label{pr:ghf_positive}
 Let $(X,d,f)$, $(X',d',f')$ be two non-empty compact metric spaces each equipped with a continuous function. Then $\dGHf( (X,d,f), (X',d',f') ) = 0$ if and only if there is an isometry $\psi$ between $(X,d)$ and $(X',d')$ such that $d_\infty(f, f' \circ \psi) = 0$.
\end{proposition}

\begin{lemma}\label{le:ghf_common_embedding}
 Suppose that $(X_n,d_n,f_n) \to (X,d,f)$ in the GHf topology. Then there exists a compact metric space $(W,d_W)$ and isometric embeddings $\psi_n\colon X_n \to W$ (resp.\ $\psi\colon X \to W$) such that $\dH( \psi(X), \psi_n(X_n)) + d_\infty(f \circ \psi^{-1}, f_n \circ \psi_n^{-1}) \to 0$.
\end{lemma}

\begin{proof}
 The proof is identical to the analogous statement for GH convergence, see e.g.\ \cite[Lemma~A.1]{gpw-gromov-prokhorov} or \cite[Proposition~1.5]{gm-uihpq}.
\end{proof}

\begin{proof}[Proof of Proposition~\ref{pr:ghf_positive}]
 Suppose that $\dGHf( (X,d,f), (X',d',f') ) = 0$. By Lemma~\ref{le:ghf_common_embedding}, we can find a compact space $(W,d_W)$ containing $X$ as a subspace, and a sequence of isometric embeddings $\psi_n\colon X' \to W$ such that $\dH(X, \psi_n(X')) + d_\infty(f,f'\circ\psi_n^{-1}) \to 0$. The sequence $(\psi_n)$ is equicontinuous, hence converges along a subsequence uniformly to a function $\psi\colon X' \to W$.
 
 We show that $\psi$ is an isometry between $X'$ and $X$. By the uniform convergence $\psi_n \to \psi$, we see immediately that $\psi(X') \subseteq X$ and $d_W(\psi(x'),\psi(y')) = d'(x',y')$. It remains to show that $\psi(X') = X$. Let $x \in X$. Since $\dH(X, \psi_n(X')) \to 0$, there exists a sequence of $x'_n \in X'$ such that $d_W(x,\psi_n(x'_n)) \to 0$. Pick a further subsequence such that $x'_n$ converges to some $x'$ in $X'$. Since $\psi_n \to \psi$ uniformly, we get $\psi_n(x'_n) \to \psi(x')$, hence $\psi(x') = x$. 
 
 Finally, we show that $d_\infty(f, f' \circ \psi^{-1}) = 0$. Let $K$ (resp.\ $K'$) be the domain of $f$ (resp.\ $f'$). Let $x \in K$. Since $d_\infty(f,f'\circ\psi_n^{-1}) \to 0$, there exists a sequence of $x'_n \in K'$ such that $d_W(x,\psi_n(x'_n)) \to 0$ and $f'(x'_n) \to f(x)$. By the same argument as above, we have $x'_n \to \psi^{-1}(x) \in K'$ along a subsequence. Then also $\psi(x'_n) \to x$ due to the uniform convergence $\psi_n \to \psi$. Conversely, let $x' \in K'$. Again, since $d_\infty(f,f'\circ\psi_n^{-1}) \to 0$, there is a sequence of $x_n \in K$ with $d_W(x_n,\psi_n(x')) \to 0$ and $f(x_n) \to f'(x')$. Since also $\psi_n(x') \to \psi(x')$, we have $x_n \to \psi(x')$.
\end{proof}

We now prove compactness of equicontinuous families in the $d_\infty$ and GHf topologies.

\begin{proposition}\label{pr:hf_compactness}
 Let $(W,d_W)$ be a compact metric space. Suppose that $\CK^W$ is a collection of non-empty compact subsets $K \subseteq W$ each equipped with a continuous function $f\colon K \to Y$. Suppose that
 \begin{enumerate}[(i)]
  \item\label{it:uniformly_bounded} There exist $0_Y \in Y$ and $M \in (0,\infty)$ such that $\sup_{x \in K} d_Y(f(x),0_Y) \leq M$ for all $(K,f) \in \CK^W$.
  \item\label{it:equicontinuous} For every $\varepsilon > 0$ there exists $\delta > 0$ so that $d(x,y) < \delta$ implies $d_Y(f(x),f(y)) < \varepsilon$ for all $(K,f) \in \CK^W$.
 \end{enumerate}
 Then $\CK^W$ is relatively compact with respect to the metric $d_\infty$.
\end{proposition}

\begin{proof}
 Let $(K_n,f_n)$ be a sequence in $\CK^W$. Recall that the space of compact subsets of $W$ is compact in the Hausdorff topology \cite[Theorem~7.3.8]{bbi-metric-geometry}. Therefore we can find a converging subsequence (again indexed by $n$) such that $(K_n)$ converges to a non-empty compact subset $K \subseteq W$. We need to construct a function $f\colon K \to Y$ such that $d_\infty(f,f_n) \to 0$ along a further subsequence.
 
 Since $\dH(K,K_n) \to 0$, for every $x \in K$ there is a sequence of $x_n \in K_n$ such that $d_W(x,x_n) \to 0$. By~\eqref{it:uniformly_bounded}, we can pick a subsequence (again indexed by $n$) such that $f_n(x_n)$ converges to a limit which we define as $f(x)$. By~\eqref{it:equicontinuous}, this definition of $f(x)$ does not depend on the choice of $x_n$. Repeating this procedure for a countable dense set of points $x \in N \subseteq K$, this defines a function $f$ on $N$. Again by~\eqref{it:equicontinuous}, we see that $f$ is uniformly continuous on $N$, so that $f$ extends to $K$, and $d_\infty(f,f_n) \to 0$ along the chosen subsequence.
\end{proof}

\begin{proposition}\label{pr:ghf_compactness}
 Suppose $\CK$ is a collection of non-empty compact metric spaces $(X,d)$ each equipped with a continuous function $f\colon K \to Y$. Suppose that
 \begin{enumerate}[(i)]
  \item $\{ (X,d) : (X,d,f) \in \CK\}$ is relatively compact with respect to the GH topology.
  \item There exist $0_Y \in Y$ and $M \in (0,\infty)$ such that $\sup_{x \in K} d_Y(f(x),0_Y) \leq M$ for all $(X,d,f) \in \CK$.
  \item For every $\varepsilon > 0$ there exists $\delta > 0$ so that $d(x,y) < \delta$ implies $d_Y(f(x),f(y)) < \varepsilon$ for all $(X,d,f) \in \CK$.
 \end{enumerate}
 Then $\CK$ is relatively compact with respect to the GHf topology.
\end{proposition}

\begin{proof}
 Pick any sequence of $(X_n,d_n,f_n) \in \CK$ such that $(X_n,d_n)$ converges to some $(X,d)$ in the GH topology. We can isometrically embed them in a common compact space $(W,d_W)$ so that $\dH(\psi_n(X_n),\psi(X)) \to 0$ \cite[Lemma~A.1]{gpw-gromov-prokhorov}. By Proposition~\ref{pr:hf_compactness}, we find a compact subset $K \subseteq W$, a function $\wt{f}\colon K \to Y$, and a subsequence (again indexed by $n$) along which $d_\infty(\wt{f},f_n\circ\psi_n^{-1}) \to 0$. In particular, $K \subseteq \psi(X)$ and $\dGHf((X,d,\wt{f}\circ\psi),(X_n,d_n,f_n)) \to 0$.
\end{proof}

We need variants of Proposition~\ref{pr:hf_compactness} and~\ref{pr:ghf_compactness} when the equicontinuity holds only up to an approximate scale $r_n$ where $r_n \searrow 0$. Their proofs are completely analogous, and therefore omitted.

\begin{lemma}\label{le:hf_approx_equicont}
 Let $(W,d_W)$ be a compact metric space. Suppose we have a sequence of non-empty compact subsets $K_n \subseteq W$ each equipped with a bounded function $f_n\colon K_n \to Y$. Suppose that there is a sequence $r_n \searrow 0$ such that
 \begin{enumerate}[(i)]
  \item There exist $0_Y \in Y$ and $M \in (0,\infty)$ such that $\sup_{x \in K} d_Y(f_n(x),0_Y) \leq M$ for every $n$.
  \item For every $\varepsilon > 0$ there exists $\delta > 0$ so that $d_n(x,y) < \delta$ implies $d_Y(f_n(x),f_n(y)) < \varepsilon \vee r_n$ for every $n$.
 \end{enumerate}
 Then there is a subsequence $(K_{n_k},f_{n_k})$ that converges in the $d_\infty$ metric.
\end{lemma}

\begin{lemma}\label{le:ghf_approx_equicont}
 Suppose we have a sequence $(X_n,d_n,f_n)$ of non-empty compact metric spaces each equipped with a bounded function $f_n\colon K_n \to Y$. Suppose that there is a sequence $r_n \searrow 0$ such that
 \begin{enumerate}[(i)]
  \item $(X_n,d_n) \to (X,d)$ in the GH topology.
  \item There exist $0_Y \in Y$ and $M \in (0,\infty)$ such that $\sup_{x \in K_n} d_Y(f_n(x),0_Y) \leq M$ for every $n$.
  \item For every $\varepsilon > 0$ there exists $\delta > 0$ so that $d_n(x,y) < \delta$ implies $d_Y(f_n(x),f_n(y)) < \varepsilon \vee r_n$ for every $n$.
 \end{enumerate}
 Then there is a subsequence $(X_{n_k},d_{n_k},f_{n_k})$ that converges in the GHf topology.
\end{lemma}

Suppose $\CK^0$ is a compact set in the GH topology, and $0_Y \in Y$. For $\alpha, C > 0$ and $r \ge 0$ we let $\CK_{\alpha,C;r}$ be the set of $(X,d,f)$ such that
\begin{itemize}
 \item $(X,d) \in \CK^0$,
 \item $\sup_{x \in K} d_Y(f(x),0_Y) \leq C$,
 \item $d_Y(f(x),f(y)) \le C(d(x,y) \vee r)^\alpha$ for every $x,y \in K$.
\end{itemize}
Let $\CK_{\alpha,C} = \CK_{\alpha,C;0}$.

It follows from Proposition~\ref{pr:ghf_compactness} that $\CK_{\alpha,C}$ is a compact set in the GHf topology. For fixed $r>0$, the set $\CK_{\alpha,C;r}$ is not compact. However, given a sequence of $(X_n,d_n,f_n) \in \CK_{\alpha,C;r_n}$ with $r_n \to 0$, this sequence satisfies the assumptions of Lemma~\ref{le:ghf_approx_equicont}, and hence has subsequential limits in the GHf topology which lie in $\CK_{\alpha,C}$. The following lemma is a similar statement regarding the tightness of a sequence of probability measures with uniform bounds for $\p_n[\CK_{\alpha,C;r_n}^c]$.

\begin{remark}
 It is easy to show that the spaces $\CK_{\alpha,C;r}$ are separable. However, they are not complete when $r>0$. Their completions contain multivalued functions. In order to not distract ourselves with discussions on the reference spaces, we have avoided using the term ``tightness'' in the statement of the proposition below.
\end{remark}

\begin{proposition}\label{pr:ghf_tightness_hoelder_approx}
 Let $\CK_{\alpha,C;r}$ be given as above. Let $(\p_n)$ be a sequence of probability measures on the space of non-empty compact metric spaces equipped with bounded functions, equipped with the GHf topology. Fix $\alpha > 0$. Suppose that for every $\varepsilon > 0$ there is $C>0$ such that for any $r > 0$ there is $N \in \N$ such that
 \begin{equation}\label{eq:tightness_approx_hoelder}
  \sup_{n \ge N} \p_n[\CK_{\alpha,C;r}^c] < \varepsilon .
 \end{equation}
 Then there is a subsequence $(\p_{n_k})$ that converges in the Prokhorov metric to a probability measure supported on $\bigcup_{C>0} \CK_{\alpha,C}$.
\end{proposition}

\begin{proof}
 Let $\varepsilon > 0$, and pick the corresponding $C>0$. By Proposition~\ref{pr:ghf_compactness}, $\CK_{\alpha,C}$ is compact. Moreover, by the observation in the paragraph above,
 \begin{equation}\label{eq:ghf_dist_approx_hoelder}
  \lim_{r \searrow 0} \sup_{(X,d,f) \in \CK_{\alpha,C;r}}\dist_{\mathrm{GHf}}((X,d,f), \CK_{\alpha,C}) = 0 .
 \end{equation}
 Let $P_r\colon \CK_{\alpha,C;r} \to \CK_{\alpha,C}$ be the metric projection. Then~\eqref{eq:ghf_dist_approx_hoelder} can be equivalently written as
 \begin{equation}\label{eq:ghf_approx_hoelder_projection}
  \lim_{r \searrow 0} \sup_{(X,d,f) \in \CK_{\alpha,C;r}}\dGHf((X,d,f), P_r(X,d,f)) = 0 .
 \end{equation}
 Fix $r>0$ and consider the probability measures $(P_r)_*\p_n$ which are tight since $\CK_{\alpha,C}$ is compact. Therefore $((P_r)_*\p_n)$ converges along a subsequence. By~\eqref{eq:tightness_approx_hoelder} and~\eqref{eq:ghf_approx_hoelder_projection} we see that
 \[ \limsup_{n\to\infty} \dP(\p_n, (P_r)_*\p_n) \le \varepsilon . \]
 Since we can repeat this for every $\varepsilon > 0$, we conclude that there is a Cauchy and hence convergent subsequence of $(\p_n)$.
\end{proof}

\begin{lemma}\label{le:ghf_subspace}
 Let $(X'_n,d'_n)$ be a non-empty compact subspace of a compact metric space $(X_n,d_n)$ for each $n$. Let $f_n$ be a bounded function defined on $X_n$, and $f'_n = f_n\big|_{X'_n}$. Suppose that $(X_n,d_n,f_n) \to (X,d,f)$ and $(X'_n,d'_n,f'_n) \to (X',d',f')$ in the GHf topology, and that $f$ is continuous. Then there is an isometric embedding $\psi$ of $(X',d')$ into $(X,d)$ such that $f' = f\circ\psi$.
\end{lemma}

\begin{proof}
 By Lemma~\ref{le:ghf_common_embedding}, we can embed $(X_n,d_n)$ and $(X,d)$ into a common space $(W,d_W)$ so that $d_\infty(f,f_n) \to 0$ in $W$. Let us identify $X'_n \subseteq X_n$ with their embeddings into $W$. By the compactness in the Hausdorff topology, we can find a subsequence along which $(X'_n)$ converges to some $\wt{X}' \subseteq X$. It follows from the uniform continuity of $f$ that also $d_\infty(f\big|_{\wt{X}'},f'_n) \to 0$ in $W$. Then Proposition~\ref{pr:ghf_positive} implies that $(X',d',f')$ is isometric to $(\wt{X}',d,f\big|_{\wt{X}'})$.
\end{proof}

\begin{lemma}\label{le:ghf_monotonicity}
 Let $(X_n,d_n)$ be a non-empty compact metric space and $f_n$, $f'_n$ be bounded real-valued functions defined on $X_n$ with $f_n \le f'_n$ pointwise for each $n$. Suppose that $(X_n,d_n,f_n) \to (X,d,f)$ and $(X_n,d_n,f'_n) \to (X',d',f')$ in the GHf topology, and that $f$ is continuous. Then there is an isometry $\psi$ between $(X,d)$ and $(X',d')$ such that $f \le f'\circ\psi$ pointwise.
\end{lemma}

\begin{proof}
 By Lemma~\ref{le:ghf_common_embedding}, we can embed $(X_n,d_n)$ and $(X',d')$ into a common space $(W,d_W)$ so that $d_\infty(f',f'_n) \to 0$ in $W$. Using the convergence $(X_n,d_n,f_n) \to (X,d,f)$ and the uniform continuity of $f$, it follows from Lemma~\ref{le:hf_approx_equicont} that there is a continuous function $\wt{f}\colon X' \to \R$ and a subsequence along which also $d_\infty(\wt{f},f_n) \to 0$ in $W$. It follows from the continuity of $\wt{f}$ that $\wt{f} \le f'$. Finally, Proposition~\ref{pr:ghf_positive} implies that $(X',d',\wt{f})$ is isometric to $(X,d,f)$.
\end{proof}

\section{Tightness and convergence of conditional laws}
\label{app:conditional_laws}

In general, conditional laws and conditional independence may not be preserved under weak limits, e.g.\ let $X \sim \mathrm{Leb}\big|_{[0,1]}$ and $Y_n = Z_n = \sin(nX)$. In this section, we prove some sufficient conditions that guarantee conditional independence in the weak limit. We formulate the results for Polish spaces, without making an effort in finding the most general assumptions for the spaces.

We will use the following notation for composition-products of kernels. If $\nu$ is a probability measure on $X$, and $\mu = \mu[dy \mid x]$ (resp.\ $\wt{\mu} = \wt{\mu}[dz \mid x]$) is a probability kernel from $X$ to $Y$ (resp.\ $Z$), then we write
\begin{align*}
 \nu\otimes\mu[dx\,dy] &= \nu[dx]\,\mu[dy \mid x] ,\\
 \mu\otimes\wt{\mu}[dy\,dz \mid x] &= \mu[dy \mid x]\,\wt{\mu}[dz \mid x] .
\end{align*}

We begin with a useful lemma.

\begin{lemma}\label{le:weak_limit_abs_cont}
 Let $X,Y$ be Polish spaces. Let $\nu$ be a probability measure on $X$, and $d\wt{\nu} = fd\nu$ where $f \in L^1(\nu)$. Let $\mu_n[dy \mid x]$ be a sequence of probability kernels from $X$ to $Y$. Suppose that $\nu\otimes\mu_n \to \nu\otimes\mu$ weakly. Then also $\wt{\nu}\otimes\mu_n \to \wt{\nu}\otimes\mu$ weakly.
\end{lemma}

\begin{proof}
 Since $X$ is Polish, the set of bounded continuous functions is dense in $L^1(\nu)$. Therefore we can find a sequence $f_m \in C_b(X)$ with $f_m \to f$ in $L^1(\nu)$. By the weak convergence $\nu\otimes\mu_n \to \nu\otimes\mu$ we have
 \[ \iint g(x,y) f_m(x) \mu_n[dy \mid x] \nu[dx] \xrightarrow{n\to\infty} \iint g(x,y) f_m(x) \mu[dy \mid x] \nu[dx] . \]
 for every bounded continuous $g$. On the other hand, we have
 \[ \sup_n \iint \abs{g(x,y)} \abs{f(x)-f_m(x)} \mu_n[dy \mid x] \nu[dx] \le \norm{g}_\infty \norm{f-f_m}_{L^1(\nu)} \xrightarrow{m\to\infty} 0 . \]
 Therefore also
 \[ \iint g(x,y) f(x) \mu_n[dy \mid x] \nu[dx] \xrightarrow{n\to\infty} \iint g(x,y) f(x) \mu[dy \mid x] \nu[dx] \]
 which shows that $\wt{\nu}\otimes\mu_n \to \wt{\nu}\otimes\mu$.
\end{proof}

\begin{lemma}\label{le:same_cond_law_weak_limit}
 Let $X,Y$ be Polish spaces. Let $\nu,\wt{\nu}$ be two probability measures on $X$, and $\mu_n[dy \mid x]$ a sequence of probability kernels from $X$ to $Y$. Suppose that $\nu\otimes\mu_n \to \nu\otimes\mu$ and $\wt{\nu}\otimes\mu_n \to \wt{\nu}\otimes\wt{\mu}$ weakly. Then there are versions of $\mu$ (with respect to $\nu$) and $\wt{\mu}$ (with respect to $\wt{\nu}$) such that $\mu[\cdot \mid x] = \wt{\mu}[\cdot \mid x]$ for every $x \in X$.
\end{lemma}

\begin{proof}
 Decompose $\nu = \nu_c+\nu_s$ where $\nu_c$ (resp.\ $\nu_s$) is absolutely continuous (resp.\ singular) with respect to $\wt{\nu}$. By Lemma~\ref{le:weak_limit_abs_cont}, the convergence $\wt{\nu}\otimes\mu_n \to \wt{\nu}\otimes\wt{\mu}$ implies also $\nu_c\otimes\mu_n \to \nu_c\otimes\wt{\mu}$.
 
 Let $A$ be a measurable set so that $\nu_s(A) = 0$, $\nu_c(A^c) = 0$. Applying Lemma~\ref{le:weak_limit_abs_cont} to $f=1_A$, we see that the convergence $\nu\otimes\mu_n \to \nu\otimes\mu$ implies $\nu_c\otimes\mu_n \to \nu_c\otimes\mu$. This shows that $\mu = \wt{\mu}$ holds $\nu_c$-almost everywhere. Since $\nu_s$ is singular with respect to $\wt{\nu}$, there is a version of $\wt{\mu}$ so that $\mu = \wt{\mu}$ also holds $\nu_s$-almost everywhere. We conclude that there is a version of $\mu$ so that $\mu=\wt{\mu}$.
\end{proof}

The following lemma guarantees conditional independence in the weak limit when the conditional laws of one of the random variables remain fixed.

\begin{lemma}\label{le:cond_independence_limit_fixed_kernel}
 Let $X,Y,Z$ be Polish spaces. Let $\nu$ be a probability measure on $X$, and $\mu_n[dy \mid x]$ a sequence of probability kernels from $X$ to $Y$, and suppose that $\nu\otimes\mu_n \to \nu\otimes\mu$ weakly. Then, if $\wt{\mu}$ is another probability kernel from $X$ to $Z$, then $\nu\otimes(\mu_n\otimes\wt{\mu}) \to \nu\otimes(\mu\otimes\wt{\mu})$ weakly.
\end{lemma}

\begin{proof}
 It is straightforward to verify that the sequence $(\nu\otimes(\mu_n\otimes\wt{\mu}))$ is tight. Therefore we only need to identify the limit. Let $g \in C_b(X \times Y)$, $h \in C_b(Z)$. Then Lemma~\ref{le:weak_limit_abs_cont} applied to $f(x) = \int h(z)\wt{\mu}[dz \mid x]$ implies
 \[ \iiint g(x,y)h(z) \wt{\mu}[dz \mid x] \mu_n[dy \mid x] \nu[dx] \to \iiint g(x,y)h(z) \wt{\mu}[dz \mid x] \mu[dy \mid x] \nu[dx] . \]
\end{proof}

For more general sequences of random variables we can make use of the convergence of their conditional laws. The following lemmas aim to find conditions under which we can extract subsequences for which the conditional laws converge.

\begin{lemma}\label{le:tightness_cond_law}
 Let $X,Y$ be Polish spaces. Let $\nu$ be a probability measure on $X$, and $\mu_n[dy \mid x]$ a sequence of probability kernels from $X$ to $Y$. Suppose that the sequence of measures $(\nu\otimes\mu_n)$ on $X \times Y$ is tight. Then there exists a sequence $(K^m)$ of compact sets in $Y$ such that
 \[ \lim_{m \to \infty} \liminf_{n \to \infty} \mu_n[(K^m)^c \mid x] = 0 . \]
 for $\nu$-almost every $x \in X$.
\end{lemma}

\begin{proof}
 By the assumption, there exists a sequence of compact sets $\wt{K}^m \subseteq X \times Y$ such that
 \begin{equation}\label{eq:product_tight}
  \sup_n (\nu\otimes\mu_n)[(K^m)^c] < 2^{-2m} .
 \end{equation}
 Let $K^m = \{ y : (x,y) \in \wt{K}^m \}$, and $A^m_n = \{ x \in X : \mu_n[(K^m)^c \mid x] > 2^{-m} \}$. Then it follows from~\eqref{eq:product_tight} that $\sup_n \nu[A^m_n] \le 2^{-m}$. By Fatou's lemma,
 \[ \nu[\liminf_n A^m_n] \le \liminf_n \nu[A^m_n] \le 2^{-m} . \]
 In particular, for $\nu$-almost every $x$ there exists $m(x) \in \N$ such that $x \in (\liminf_n A^m_n)^c$ for every $m \ge m(x)$. In other words, there is a subsequence $(n_{k(x)})$ such that
 \[ \limsup_k \mu_{n_{k(x)}}[(K^m)^c \mid x] \le 2^{-m} \]
 for every $m \ge m(x)$.
\end{proof}

\begin{corollary}\label{co:subsequential_limit_cond_laws}
 Let $X,Y$ be Polish spaces. Let $\nu$ be a probability measure on $X$, and $\mu_n[dy \mid x]$ a sequence of probability kernels from $X$ to $Y$. Suppose that the sequence of measures $(\nu\otimes\mu_n)$ on $X \times Y$ is tight, and that the sequence $(\mu_n)$ is equicontinuous as functions from $X$ to the space of probability measures on $Y$ equipped with the Prokhorov metric. Then there exists a subsequence $(n_k)$ such that $\mu_{n_k}[\cdot \mid x]$ converges weakly for $\nu$-almost every every $x \in X$.
\end{corollary}

\begin{proof}
 By Lemma~\ref{le:tightness_cond_law}, there is a set $\wt{X} \subset X$ of full measure such that for every $x \in \wt{X}$ the sequence $(\mu_n[dy \mid x])$ contains a tight subsequence. In particular, we can find a subsequence along which it converges for a countable dense set of $x \in \wt{X}$. By the equicontinuity, the convergence will then also hold for every $x \in \wt{X}$.
\end{proof}

\begin{lemma}\label{le:limit_cond_laws}
 Let $X,Y$ be Polish spaces. Let $\nu$ be a probability measure on $X$, and $\mu_n[dy \mid x]$ a sequence of probability kernels from $X$ to $Y$. Suppose that $\mu_n[dy \mid x] \to \mu[dy \mid x]$ weakly for $\nu$-almost every $x$. Then $\nu\otimes\mu_n \to \nu\otimes\mu$ weakly.
\end{lemma}

\begin{proof}
 Let $f = f(x,y)$ be a bounded continuous function. Then we have $\int_Y f(x,y) \mu_n[dy \mid x] \to \int_Y f(x,y) \mu[dy \mid x]$ for almost every $x$. Since $f$ is bounded, this implies also
 \[ \iint f(x,y) \mu_n[dy \mid x] \nu[dx] \to \iint f(x,y) \mu[dy \mid x] \nu[dx] . \]
\end{proof}

The following lemma generalizes the Lemmas~\ref{le:cond_independence_limit_fixed_kernel} and~\ref{le:limit_cond_laws}.

\begin{lemma}\label{le:cond_independence_limit}
 Let $X,Y,Z$ be Polish spaces. Let $\nu$ be a probability measure on $X$, and $\mu_n[dy \mid x]$ (resp.\ $\wt{\mu}_n[dz \mid x]$) a sequence of probability kernels from $X$ to $Y$ (resp.\ $Z$). Suppose that $\nu\otimes\wt{\mu}_n \to \nu\otimes\wt{\mu}$ weakly, and that $\mu_n[dy \mid x] \to \mu[dy \mid x]$ weakly for $\nu$-almost every $x$. Then $\nu\otimes(\mu_n\otimes\wt{\mu}_n) \to \nu\otimes(\mu\otimes\wt{\mu})$ weakly.
\end{lemma}

\begin{proof}
 By Lemma~\ref{le:limit_cond_laws}, the sequence $(\nu\otimes\mu_n)$ also converges weakly. It is then straightforward to verify that the sequence $(\nu\otimes(\mu_n\otimes\wt{\mu}_n))$ is tight. Therefore we only need to identify the limit. Let $g \in C_b(X \times Z)$, $h \in C_b(Y)$. Then
 \begin{multline*}
  \abs*{ \iiint g(x,z)h(y) \mu_n[dy \mid x] \wt{\mu}_n[dz \mid x] \nu[dx] - \iiint g(x,z)h(y) \mu[dy \mid x] \wt{\mu}_n[dz \mid x] \nu[dx] } \\
  \le \norm{g}_\infty \int \abs*{ \int h(y) \mu_n[dy \mid x] - \int h(y) \mu[dy \mid x] } \nu[dx] 
  \to 0
 \end{multline*}
 due to the assumption $\mu_n[dy \mid x] \to \mu[dy \mid x]$ for almost every $x$. Moreover, by Lemma~\ref{le:cond_independence_limit_fixed_kernel},
 \[
  \iiint g(x,z)h(y) \mu[dy \mid x] \wt{\mu}_n[dz \mid x] \nu[dx] \to \iiint g(x,z)h(y) \mu[dy \mid x] \wt{\mu}[dz \mid x] \nu[dx] .
 \]
 Combining the above, we conclude.
\end{proof}

\section{Arm exponents}
\label{app:arms}

In this section we collect a few results on the double point and intersection point probabilities for various types of \slek{} curves. In Section~\ref{se:regularity} we will use them to prove certain regularity statements for domains arising from explorations considered in this paper. In Section~\ref{se:cle_outer_boundary_exponent} we apply them to bound the probability of multiple crossings of the exterior boundaries of \clekp{} loops.

Let
\begin{equation}
\label{eqn:double_exponent_simple}
\alpha_{4,\kappa} = \frac{(12-\kappa)(4+\kappa)}{8\kappa} > 2
\end{equation}
be the double point exponent for $\SLE_\kappa$.

The following result is essentially proved in \cite{Wu2016}, but there some extra constraints are assumed. We will explain below how to extend the result to the general setup.

\begin{proposition}\label{pr:4arm_simple}
 For any $M>0$, $a>0$ there exists $c>0$ such that the following holds. Let $D \subsetneq \C$ be a simply connected domain, $x,y \in \partial D$ distinct, and $\kappa \in (0,4)$. Let $\eta$ be a chordal \slekr{\ul{\rho}} in $D$ from $x$ to $y$ where all partial sums of the weights are bounded by $M$. Let $z\in D$, $r>0$ such that $B(z,r) \subseteq D$. For $\epsilon\in(0,1)$, let $E$ be the event that $\eta$ makes $4$ crossings of $A(z,\epsilon r,r)$. Then
 \[
  \p[E] \le c\epsilon^{\alpha_{4,\kappa}-a} .
 \]
\end{proposition}

This result can be stated more generally in terms of crossings of flow lines.

\begin{proposition}\label{pr:4arm_fl}
 For any $M>0$, $a>0$ there exists $c>0$ such that the following holds. Let $h$ be a GFF in $D \subseteq \C$ with boundary values bounded by $M$, or a whole-plane GFF modulo additive constant $2\pi\chi$. Let $z\in D$, $r>0$ such that $B(z,r) \subseteq D$. For $\epsilon\in(0,1)$, let $E$ be the event that there are two disjoint flow lines $\eta_{w_1},\eta_{w_2}$ crossing in and out of the annulus $A(z,\epsilon r,r)$ and their heights are so that a flow line from $\partial D$ can merge into $\eta_{w_1}$ and then into $\eta_{w_2}$ without further crossings of the annulus. Then
 \[
  \p[E] \le c\epsilon^{\alpha_{4,\kappa}-a} .
 \]
\end{proposition}

We first recall the setup of \cite[Proposition~4.1]{Wu2016}. Let $\kappa\in(0,4)$ and let $\eta$ be an $\SLE_\kappa$ in $\D$ from $-i$ to $i$. Let $\partial_- \D$ be the clockwise arc of $\partial \D$ from $-i$ to $i$, and suppose $y \in \partial_- \D$. Fix $0<s<1/8$ such that $\pm i \notin B(y,s)$ and let $y_-$ be the clockwise most point on the arc $\partial\D \cap B(y,s)$. Let $\tau_1$ be the first time $\eta$ hits $B(0,\epsilon)$. Let
\[ \CE_2 = \{\tau_1<\infty\} .\]
Let $\sigma_1$ be the first time after $\tau_1$ that $\eta$ hits the connected component of $\partial B(y,s)\setminus \eta[0,\tau_1]$ containing $y_-$. Let $\CE^g$ be the event that $0$ is connected to $i$ in $\D\setminus(\eta[0,\sigma_1]\cup B(y,s))$. Given $\eta[0,\sigma_1]$, let $C$ be the connected component of $B(0,\epsilon)\setminus\eta[0,\sigma_1]$ that contains $0$. Let $C^b$ be the connected component of $\partial C \cap \partial B(0,\epsilon)$ that is connected to $i$ in $\D\setminus(\eta[0,\sigma_1]\cup B(0,\epsilon))$. Let $\tau_2$ be the first time after $\sigma_1$ that $\eta$ hits $C^b$, and let
\[ \CE_{4}=\CE^g\cap\{\tau_2<\infty\} .\]

\begin{proposition}[{\cite[Proposition~4.1]{Wu2016}}]\label{pr:alternating_crossing_simple}
Fix $\kappa\in(0,4)$. There exist $y \in \partial_- \D$, $0<s<1/8$ with $\pm i \notin B(y,s)$ such that
\[\p[\CE_{4}\cap \{\eta[0,\tau_1] \cap B(i,1/8) = \varnothing \}]=O(\epsilon^{\alpha_{4,\kappa}+o(1)}).\]
\end{proposition}

In order to deduce Proposition~\ref{pr:4arm_fl} from Proposition~\ref{pr:alternating_crossing_simple}, we use a resampling argument to resample the connections between the crossings of the flow lines. Using the good scales for the GFF (see Section~\ref{se:gff}), we see that it is very likely that at a random scale the probability of successfully resampling the connections is bounded from below.

\begin{lemma}\label{le:4arm_good_fl}
Fix $p \in (0,1)$. Let~$h$ be a GFF on~$\D$ with boundary values so that its flow line $\eta$ is an \slek{} from $-i$ to $i$. Let $X^0_{0,1},X^0_{0,\epsilon}$ be as defined in~\eqref{eq:fl_annulus}. Let $I_\epsilon$ be the event that there exist two strands $\eta_{w_1;\mathrm{in}}$, $\eta_{w_3;\mathrm{in}}$ of $X^0_{0,1}$ and a strand $\eta_{w_2;\mathrm{out}}$ of $X^0_{0,\epsilon}$ so that the following hold.
\begin{itemize}
 \item The continuations of $\eta_{w_i;\mathrm{in}}$, $i=1,3$, exit $A(0,\epsilon/4,3/4)$ on the inner boundary and without merging, and the continuation of $\eta_{w_2;\mathrm{out}}$ exits $A(0,\epsilon/4,3/4)$ on the outer boundary.
 \item Let $\wt{I}_\epsilon$ be the event that
 \begin{itemize}
  \item $\eta$ merges consecutively into $\eta_{w_1;\mathrm{in}}$, $\eta_{w_2;\mathrm{out}}$, $\eta_{w_3;\mathrm{in}}$ before hitting $B(i,1/8)$ and without crossing $A(0,1/4,3/4)$ (resp.\ $A(0,\epsilon/4,3\epsilon/4)$) inbetween,
  \item the event $\CE_4$ in Proposition~\ref{pr:alternating_crossing_simple} occurs for $\eta$.
 \end{itemize}
 The conditional probability of $\wt{I}_\epsilon$ given $\eta_{w_1;\mathrm{in}}$, $\eta_{w_2;\mathrm{out}}$, $\eta_{w_3;\mathrm{in}}$, $X^0_{0,1},X^0_{0,\epsilon}$, and the values of $h$ on $X^0_{0,1},X^0_{0,\epsilon}$ is at least $p$.
\end{itemize}
Then
\[\p[I_\epsilon]=O(\epsilon^{\alpha_{4,\kappa}+o(1)}).\]
\end{lemma}

\begin{proof}
By the definition of $I_\epsilon$, we have that
\[ \p[\wt{I}_\epsilon \giv I_\epsilon] \geq p.\]
By Proposition~\ref{pr:alternating_crossing_simple} we have $\p[\wt{I}_\epsilon]=O(\epsilon^{\alpha_{4,\kappa}+o(1)})$. 
This, in turn, implies that
\[ \p[ I_\epsilon] \leq \frac{1}{p} \p[\wt{I}_\epsilon] =  O(\epsilon^{\alpha_{4,\kappa}+o(1)}) . \]
\end{proof}

Note that the event $I_\epsilon$ in Lemma~\ref{le:4arm_good_fl} is measurable with respect to the values of a GFF in $A(0,\epsilon/4,3/4)$.

\begin{proof}[Proof of Proposition~\ref{pr:4arm_fl}]
By translation and scaling, it suffices to assume $z=0$, $r=1$. By symmetry, it suffices to consider the case where the second time $\eta$ crosses into the annulus $A(0,\epsilon,1)$, it does so on the right of the first crossing.

Let the fields $\wt{h}_{0,r}$ be as defined in Section~\ref{se:gff}, with boundary values compatible with an \slek{} in $B(0,r)$ from $-ir$ to $ir$. Let $y,s$ be as in Proposition~\ref{pr:alternating_crossing_simple}.

We define variants of the event in Lemma~\ref{le:good_scales_merging_multiple} where in the event $G$ we consider for each choice of strands $\eta_{w_1;\mathrm{in}},\eta_{w_2;\mathrm{out}},\eta_{w_3;\mathrm{in}}$ of $X^0_{0,1}$ ending on $\partial B(0,1/4),\partial B(0,3/4),\partial B(0,1/4)$, respectively, the conditional probability that the flow line from $-i$ merges into $\eta_{w_1;\mathrm{in}}$ before hitting $B(i,1/8)$ or entering $B(0,1/4)$, and the continuation of $\eta_{w_2;\mathrm{out}}$ first hits $B(y,s)$ and then merges into $\eta_{w_3;\mathrm{in}}$ before entering $B(0,1/4)$. Let $G^{(1)}$ denote this (stronger) event that the conditional probability is either $0$ or at least $p$.

We also define a similar event $G^{(2)}$ where we consider for each pair of strands $\eta_{w_1;\mathrm{in}},\eta_{w_2;\mathrm{out}}$ of $X^0_{0,1}$ ending on $\partial B(0,1/4)$ (resp.\ $\partial B(0,3/4)$) the conditional probability that the continuation of $\eta_{w_1;\mathrm{in}}$ merges into $\eta_{w_2;\mathrm{out}}$ before exiting $B(0,3/4)$ and passes to the left of $0$. 

Fix $a'>0$ sufficiently small. Let $G^{(1)}_{0,r},G^{(2)}_{0,r}$ be defined analogously as in Lemma~\ref{le:good_scales_merging_multiple}, and let $F^{(1)}$ (resp.\ $F^{(2)}$) be the event that the fraction of $1 \leq j_1 \leq a'\log_2(\epsilon^{-1})$ (resp.\ $(1-a')\log_2(\epsilon^{-1}) \leq j_2 \leq \log_2(\epsilon^{-1})$) so that $G^{(1)}_{0,2^{-j_1}}$ (resp.\ $G^{(2)}_{0,2^{-j_2}}$) occurs is at least $1/2$. By the exact same proof we can choose $M,p>0$ so that $\p[(F^{(1)})^c], \p[(F^{(2)})^c] = O(\epsilon^{\alpha_{4,\kappa}})$.

Now let $J_1 \in \{1,\ldots,\lfloor a'\log_2(\epsilon^{-1})\rfloor\}$, $J_2 \in \{\lceil(1-a')\log_2(\epsilon^{-1})\rceil,\ldots,\lfloor\log_2(\epsilon^{-1})\rfloor\}$ be sampled uniformly at random, independently of $h$. On the event $F^{(1)} \cap F^{(2)}$ the conditional probability is at least $1/4$ that we get scales where $G^{(1)}_{0,2^{-J_1}} \cap G^{(2)}_{0,2^{-J_2}}$ occurs. Note that on the event $E \cap G^{(1)}_{0,2^{-J_1}} \cap G^{(2)}_{0,2^{-J_2}}$, since there are flow lines merging into the corresponding strands of $\eta_{w_1},\eta_{w_2}$, the merging probabilities as described in the events $G^{(1)}_{0,2^{-J_1}}, G^{(2)}_{0,2^{-J_2}}$ are a.s.\ at least $p$. Therefore, if $\wt{E}_{J_1,J_2}$ denotes the event that the event in Lemma~\ref{le:4arm_good_fl} occurs for the field $\wt{h}_{0,2^{-J_1}}$ and $\wt{\epsilon} = 2^{-J_2+J_1}$, then
\[ \p[ \wt{E}_{J_1,J_2} \cap G^{(1)}_{0,2^{-J_1}} \cap G^{(2)}_{0,2^{-J_2}} ] \ge (1/4) \p[ E \cap F^{(1)} \cap F^{(2)} ] . \]
On the other hand, for each $j_1,j_2$, by Lemma~\ref{le:4arm_good_fl} and absolute continuity we see that
\[ \p[ \wt{E}_{j_1,j_2} \cap G^{(1)}_{0,2^{-j_1}} \cap G^{(2)}_{0,2^{-j_2}} ] \lesssim \epsilon^{(1-2a')\alpha_{4,\kappa}} . \]
Since we can equally first sample $J_1,J_2$ and then independently $h$, we conclude that
\[
 \p[E] = \p[ E \cap F^{(1)} \cap F^{(2)} ] + \p[(F^{(1)})^c] + \p[(F^{(2)})^c] \lesssim \epsilon^{(1-2a')\alpha_{4,\kappa}} .
\]
\end{proof}

Next, we recall that the double point exponent for \slekp{} is $2-\ddouble$ where $\ddouble$ is given by~\eqref{eq:dim_double}.

\begin{proposition}\label{pr:dbl_exponent}
 Let $\angledouble$ be given by~\eqref{eq:angle_double}. For any $M>0$, $a>0$ there exists $c>0$ such that the following holds. Let $h$ be a GFF in $D \subseteq \C$ with boundary values bounded by $M$, or a whole-plane GFF modulo additive constant $2\pi\chi$. Let $z\in D$, $r>0$ such that $B(z,r) \subseteq D$. For $\epsilon\in(0,1)$, let $E$ be the event that there exist flow lines $\eta^0_{w_1},\eta^{\angledouble}_{w_2}$ with angles $0$ (resp.\ $\angledouble$) starting from some $w_1,w_2 \notin B(z,r)$, both crossing the annulus $A(z,\epsilon r,r)$, and their height difference upon crossing is $\angledouble$. Then
 \[
  \p[E] \le c\epsilon^{2-\ddouble-a} .
 \]
\end{proposition}

\begin{proof}
This is proved in \cite[Lemma~4.3(i)]{mw2017intersections} under some extra constraints. The extra constraints can be relaxed using the same argument used in the proof of Proposition~\ref{pr:4arm_fl}.
\end{proof}

\subsection{Bottleneck bounds}
\label{se:regularity}

In this section we will establish a regularity statement for domains that arise from exploring certain regions on one side of an \slekr{\rho} curve. In particular, we are interested in the case where we explore some loops of a \clekp{} that lie on one side of the curve.  This regularity statement quantifies the number of nested bottlenecks that can occur in such domains. One novelty of this result is that the assertion holds \emph{simultaneously} for an entire class of such explorations. The main result is Proposition~\ref{pr:regularity}.

\begin{definition}
\label{def:regularity}
Suppose that $D \subsetneq \C$ is a simply connected domain and $x,y \in \partial D$ are distinct. For $U \subseteq \C$ and constants $M , a > 0$ we say that \emph{$(D,x,y)$ is $(M,a)$-good within $U$} if there is a conformal transformation $\varphi \colon \h \to D$ with $\varphi(0) = x$ and $\varphi(\infty) = y$ such that the following is true. Let
\begin{equation}\label{eq:bottleneck_def}
\begin{split}
 \CZ^U_k &= \{j \in \Z : \varphi(i 2^{j}) \in U,\ 2^{-k-1} < \dist(\varphi(i 2^{j}), \partial D) \leq 2^{-k}\} ,\quad k \in \N ,\\
 \CZ_k &= \CZ^{\C}_k .
\end{split}
\end{equation}
Then $|\CZ^U_k| \leq M2^{ak}$ for every $k \in \N$.

In the case $U = \C$, we say that $(D,x,y)$ is $(M,a)$-good.
\end{definition}

The following scaling property is immediate from the definition.
\begin{lemma}\label{le:good_domain_scaling}
 Suppose $\operatorname{inrad}(D) \le 1$, and let $\lambda > 0$. If $(D,x,y)$ is $(M,a)$-good, then $(\lambda D, \lambda x, \lambda y)$ is $((2\lambda)^a M,a)$-good.
\end{lemma}

We state the main result of this section. Let $\eta$ be an \slekr{\rho} in $(\D,-i,i)$ where $\rho > -2$ and the force point is on the boundary. In fact, we can allow an arbitrary number of force points on the boundary. Let $D^L$ (resp.\ $D^R$) be the union of connected components of $\D \setminus \eta$ that lie to the left (resp.\ right) of $\eta$. For $0 \le s < t \le \infty$ and $A \subseteq D^L$, let $D_{s,t,A}$ be the connected component of $\D \setminus (\eta[0,s] \cup \eta[t,\infty] \cup \ol{A})$ with $\eta(s),\eta(t)$ on its boundary (supposing that the two points are on the boundary of the same connected component).

\begin{proposition}\label{pr:regularity}
 Fix $a>0$. Let $G_{M,a}$ be the event that $(D_{s,t,A}, \eta(s), \eta(t))$ is $(M,a)$-good within $B(0,3/4)$ for every $0 \le s < t \le \infty$ and $A \subseteq D^L$ as above. Then
 \[ \p[G_{M,a}^c] = o^\infty(M^{-1}) \quad\text{as } M \to \infty . \]
\end{proposition}

We explain that the event $G_{M,a}$ defined in Proposition~\ref{pr:regularity} is indeed measurable with respect to $\eta$. Suppose first that $s,t$ and $A$ are fixed. Then the event that $(D_{s,t,A}, \eta(s), \eta(t))$ is $(M,a)$-good is measurable with respect to $\eta$ because the conformal map $\varphi\colon \h \to D_{s,t,A}$ (normalized in an arbitrary way) restricted to any compact subset depends continuously on $\eta$. By considering $s,t \in \Q$ and sets $A$ consisting of dyadic squares, and using the locally continuous dependence of $\varphi$ on $s,t,A$, we conclude that $G_{M,a}$ is measurable with respect to $\eta$.

\begin{remark}
 The analogous result without restricting to $B(0,3/4)$ holds in case the (potential) force point to the right of $\eta$ has weight strictly larger than $\kappa/2-2$. In that regime, the boundary intersection exponent is strictly larger than $1$. From this, that variant of the result can be proved by a generalization of our strategy.
\end{remark}

By scaling and Lemma~\ref{le:good_domain_scaling}, we deduce immediately the following.
\begin{corollary}\label{co:regularity_scaled}
 Let $\delta \in (0,1]$. Let $\eta$ be an \slekr{\rho} in $(\delta\D,-i\delta,i\delta)$ where $\rho > -2$ and the force point is on the boundary. Let $D^L$, $D^R$, and $D_{s,t,A}$ be defined as above. Let $G_{M,a}$ be the event that $(D_{s,t,A}, \eta(s), \eta(t))$ is $(M,a)$-good within $B(0,3\delta/4)$ for every $0 \le s < t \le \infty$ and $A \subseteq D^L$ as above. Then
 \[ \p[G_{M,a}^c] = o^\infty(M^{-1}\delta) \quad\text{as } M^{-1}\delta \searrow 0 . \]
\end{corollary}

\begin{example}
We are in particular interested in the following examples. Suppose that we have the setup from Section~\ref{se:intersections_setup}. We will apply Corollary~\ref{co:regularity_scaled} to the following domains $D$ (obtained from $\eta=\eta_2^\delta$).
\begin{enumerate}[(i)]
 \item\label{it:expl_regularity} Consider a subset $\CC \subseteq \Gamma_\delta$ and let $A$ be the the union of $\eta_1^\delta$ and the loops in $\CC$. Let $D$ be the component of $\delta\D \setminus (\eta_2^\delta([0,s]\cup[t,\infty]) \cup \ol{A})$ with $\eta_2^\delta(s),\eta_2^\delta(t)$ on its boundary. 
 \item\label{it:intersections_regularity} Let $0 \le s_1 < t_1 \le \infty$ and $0 \le s_2 < t_2 \le \infty$ such that $\eta_1^\delta(s_1) = \eta_2^\delta(s_2)$ and $\eta_1^\delta(t_1) = \eta_2^\delta(t_2)$. Consider the component of $\delta\D \setminus (\eta_1^\delta([0,s_1]\cup[t_1,\infty]) \cup \eta_2^\delta([0,s_2]\cup[t_2,\infty]))$ with $\eta_2^\delta(s_2),\eta_2^\delta(t_2)$ on its boundary.
\end{enumerate}
\end{example}

We also record a variant of Corollary~\ref{co:regularity_scaled} for two-sided whole-plane \slek{}.

\begin{corollary}\label{co:regularity_wpsle}
 Let $\wt{\eta}$ be a two-sided whole-plane \slek{} from $\infty$ to $\infty$ through $0$. For $\delta > 0$ and $w \in \C$, let $G_{M,a}$ be the event that $(\C \setminus \wt{\eta}([-\infty,s] \cup [t,\infty]), \wt{\eta}(s), \wt{\eta}(t))$ is $(M,a)$-good within $B(w,\delta)$ for every $-\infty \le s < t \le \infty$. Then
 \[ \p[G_{M,a}^c] = o^\infty(M^{-1}\delta) \quad\text{as } M^{-1}\delta \searrow 0 . \]
\end{corollary}

\begin{proof}
 We only need to consider the case when $\wt{\eta}$ intersects $B(w,2\delta)$. Let $\sigma = \inf\{t\in\R : \wt{\eta}(t) \in \partial B(w,2\delta)\}$ and $\tau = \sup\{t\in\R : \wt{\eta}(t) \in \partial B(w,2\delta)\}$.
 
 Suppose that we are on the event that either $0 \le \sigma < \tau$ or $\sigma < \tau \le 0$. Then the result follows from Corollary~\ref{co:regularity_scaled} by conformally mapping $\C \setminus \wt{\eta}([-\infty,\sigma] \cup [\tau,\infty])$ to $\delta\D$.
 
 To handle the case $\sigma < 0 < \tau$, we can use the stationarity of the two-sided whole-plane \slek{} along with a union bound.
\end{proof}

In the remainder of this section, we prove Proposition~\ref{pr:regularity}. Recall the double point exponent $\alpha_{4,\kappa}$ for \slek{} given by~\eqref{eqn:double_exponent_simple}. For $k \in \N$, we let
\[ \CD_k = 2^{-k}\Z^2\cap B(0,3/4) . \]

The main idea of the proof of Proposition~\ref{pr:regularity} is that whenever $D_{s,t,A}$ has too many small bottlenecks, $\eta$ must come very close to itself. See Figure~\ref{fig:bad-domains-examples} for illustrations. Since $\alpha_{4,\kappa} > 2$, it is unlikely for $\eta$ to make ``approximate double points'', and it is extremely unlikely to make many of those.

To make this precise, we introduce the definition of $(r,s)$-double points. We then state a tail bound for the number of these points in Lemma~\ref{le:jk-double-points}. 

\begin{definition}
\label{def:double-points}  
Let $\eta$ be a curve contained in $\D$. Let $z\in\CD_k$ and $r>s>0$. We say $z$ is a \emph{$(r,s)$-double point for $\eta$} if $\eta$ crosses the annulus $A(z,s,r)$ at least $4$ times.
\end{definition}

We note that different double points in $\CD_k$ can be placed within distance $2^{-k}$ of each other, even in the case where $\eta$ makes only $4$ crossings of $A(z,s,r)$. This is because the crossings corresponding to different double points are not required to be disjoint from each other. See Figure~\ref{fig:pinch-points} for an illustration.

\begin{figure}[ht]
\centering
\includegraphics[width=0.5\textwidth]{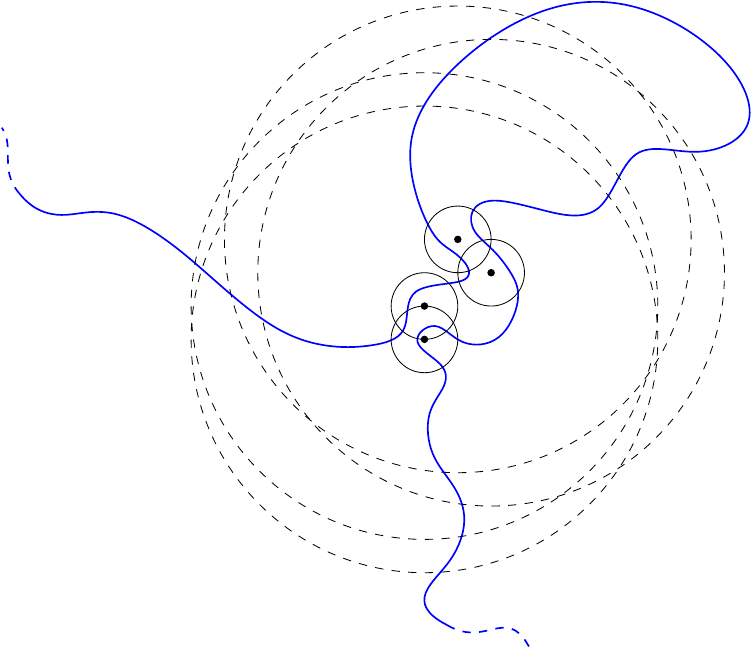}
\caption{Illustration of Definition~\ref{def:double-points}. The black dots depict $(r,s)$-double points for $\eta$. The smaller circles with solid boundary have radius $s$, the dashed circles have radius $r$.}
\label{fig:pinch-points}
\end{figure}

\begin{lemma}
\label{le:jk-double-points}
Let $\eta$ be an \slekr{\rho} in $(\D,-i,i)$ where $\rho > -2$ and the force point is on the boundary. Fix $a >0$. For $w \in B(0,3/4)$ and $r>s>0$, let $N^{w}_{r,s}$ be the number of $(r,s)$-double points for $\eta$ in $B(w,r)\cap\CD_k$. Then
\[ \p\left[ \bigcup_{w \in B(0,3/4)} \bigcup_{r > 2^{-k}} \{N^{w}_{r,2^{-k}} \geq 2^{a k}\} \right] = o^\infty(2^{-k}) \quad\text{as}\quad k \to \infty . \]
\end{lemma}

The intuition behind the proof is that, as mentioned above, the exponent $\alpha_{4,\kappa}$ is strictly larger than $2$ for $\kappa \in (0,4)$, hence the probability of any double point occurring is small. We will further argue that there is sufficient independence so that the probability of having as many as $2^{a k}$ double points is extremely small. To this end, we localize the event of a double point occurring and use independence across disjoint balls containing the double points. An additional argument is needed in order to make this work, as in principle the double points could be placed very close to each other and correspond to the same excursion of $\eta$ across a common large annulus as in Figure~\ref{fig:pinch-points}.

We first prove an upper bound for the probability that a given collection of points in $\CD_k$ are all $(r,2^{-k})$-double points for $\eta$. To this end, fix distinct non-neighboring points $z_1,\ldots,z_m \in \CD_k$ and let $\ul{z} = (z_1,\ldots,z_m)$. For $1 \leq u \leq m$, let $2^{-k} < r_u < \min_{v\neq u}\abs{z_u-z_v}$. Let $F(\ul{z},\ul{r})$ be the event that $z_u$ is a $(r_u,2^{-k})$-double point for $\eta$ for each $u$.

We will bound the probability of $F(\ul{z},\ul{r})$ by comparing to an \slek{} in each $B(z_u,r_u)$, and use the approximate independence for the GFF within disjoint regions. For this, we we recall the notation introduced in Section~\ref{se:gff}. We let the fields $\wt{h}_{z,r}$ have boundary values compatible with a usual \slek{}.

In the remainder of the section, we let $G_{z,r}$ be the events from Lemma~\ref{le:good_scales_merging_multiple} with suitable parameters $M,p > 0$ chosen below. Fix $\innexp > 0$. Let $E_{z,j,k}$ be the event that at least $9/10$ fraction of $j' = j,\ldots,j+\lfloor\innexp k\rfloor$ satisfy $G_{z,2^{-j'}}$. (The dependence on $\innexp$ is implicitly assumed in the notation.) Let
\begin{equation}\label{eq:good_merging_everywhere}
 E_k = \bigcap_{z \in \CD_k} \bigcap_{j \le k} E_{z,j,k} .
\end{equation}
By Lemma~\ref{le:good_scales_merging_multiple}, for any given $\innexp,b > 0$ there are $M,p > 0$ so that $\p[E_k^c] = O(e^{-bk})$.

\begin{lemma}
\label{le:double_bound_given_pts} 
Fix $\innexp,a' > 0$. Then for each $\ul{z} \in \CD_k^m$, $\ul{r} \in \R_+^m$ with $r_u > 2^{-k(1-\innexp)}$ for each $u$ we have
\begin{equation}\label{eq:bound-double-points-u}
\p[F(\ul{z},\ul{r}) \cap E_k] \lesssim \prod_u \left(\frac{ 2^{-k(1-\innexp)}}{r_u}\right)^{\alpha_{4,\kappa}-a'}
\end{equation}
where the implicit constant does not depend on the choice of $\ul{z},\ul{r}$.
\end{lemma}

The main input for proving Lemma~\ref{le:double_bound_given_pts} is the following lemma which is analogous to Lemma~\ref{le:4arm_good_fl}.

\begin{lemma}
\label{lem:double_flow_lines_interior}
Fix $p \in (0,1)$, $a'>0$. Let~$h$ be a GFF on~$\D$ with boundary values so that its flow line $\eta$ is an \slek{} from $-i$ to $i$. Let $X^0_{0,1}$ be as defined in~\eqref{eq:fl_annulus}. Let $I_\epsilon$ be the event that there exist two strands $\eta_{w_1}$, $\eta_{w_2}$ of $X^0_{0,1}$ so that the following hold.
\begin{itemize}
 \item The continuations of $\eta_{w_i}$, $i=1,2$, hit $B(0,\epsilon)$ before exiting $B(0,3/4)$ and without merging.
 \item The conditional probability given $X^0_{0,1}$ and the values of $h$ on $X^0_{0,1}$ that $\eta$ merges into $\eta_{w_1}$ before entering $B(0,1/4)$ and then into $\eta_{w_2}$ before entering $B(0,1/4)$ again is at least $p$.
\end{itemize}
Then there exists $\epsilon_0 > 0$ so that
\[
 \p[I_\epsilon] \leq \epsilon^{\alpha_{4,\kappa}-a'}  \quad\text{for}\quad \epsilon \in (0,\epsilon_0).
\]
\end{lemma}

\begin{proof} 
Let $F$ be the event that $\eta$ hits $B(0,\epsilon)$, then exits $B(1/4)$, and hits $B(0,\epsilon)$ again. By Proposition~\ref{pr:4arm_simple} we have that
\begin{equation}
\label{eqn:p_f_probability}
\p[F] = O(\epsilon^{\alpha_{4,\kappa}-a'}).
\end{equation}
By the definition of $I_\epsilon$, we have that
\[ \p[F \giv I_\epsilon] \geq p.\]
This, in turn, implies that
\[ \p[ I_\epsilon] \leq \frac{1}{p} \p[F] =  O(\epsilon^{\alpha_{4,\kappa}-a'}) . \]
\end{proof}

For the proof of Lemma~\ref{le:double_bound_given_pts}, recall the notation $\CF_{z,r}$, $\wt{h}_{z,r}$ introduced in Section~\ref{se:gff} where let the fields $\wt{h}_{z,r}$ have boundary values compatible with a usual \slek{}.

\begin{proof}[Proof of Lemma~\ref{le:double_bound_given_pts}]
 We claim that
 \begin{equation}\label{eq:double_point_given_ball}
  \p[ F(z_u,r_u) \cap E_{z_u,\log_2(r_u^{-1}),k} \mid \CF_{z_u,r_u} ] \lesssim \left(\frac{ 2^{-k(1-\innexp)}}{r_u}\right)^{\alpha_{4,\kappa}-a'} .
 \end{equation}
 To see this, let $J \in \{ \lceil\log_2(r_u^{-1})\rceil,\ldots,\lfloor\log_2(r_u^{-1})+\innexp k\rfloor \}$ be sampled uniformly at random, independently of $h$. Then with conditional probability at least $9/10$ we find a scale where $G_{z_u,2^{-J}}$ occurs. In particular, if $\wt{I}_{z_u,2^{-J},2^{-k}}$ denotes the event from Lemma~\ref{lem:double_flow_lines_interior} occurring for $\wt{h}_{z,2^{-J}}$ and $\epsilon = 2^{-(k-J)}$, then
 \[ \p[ \wt{I}_{z_u,2^{-J},2^{-k}} \cap G_{z_u,2^{-J}} \mid \CF_{z_u,r_u} ] \ge (9/10) \p[ F(z_u,r_u) \cap E_{z_u,\log_2(r_u^{-1}),k} \mid \CF_{z_u,r_u} ] . \]
 On the other hand, for each $j$, by Lemma~\ref{lem:double_flow_lines_interior} and absolute continuity we see that
 \[ \p[ \wt{I}_{z_u,2^{-j},2^{-k}} \cap G_{z_u,2^{-j}} \mid \CF_{z_u,r_u} ] \lesssim 2^{-(k-j)(\alpha_{4,\kappa}-a')} \le (2^{-k}/(r_u 2^{-\innexp k}))^{\alpha_{4,\kappa}-a'} . \]
 Since we can equally first sample $J$ and then independently $h$, we conclude~\eqref{eq:double_point_given_ball}.
 
 Since we have chosen the balls $B(z_u,r_u)$ mutually disjoint, \eqref{eq:bound-double-points-u} follows from~\eqref{eq:double_point_given_ball}.
\end{proof}

\newcommand*{\stepsize}{a_2}

We turn to the proof of Lemma~\ref{le:jk-double-points}. We will sum the bound~\eqref{eq:bound-double-points-u} over a chosen set of point configurations, and then argue that whenever there are too many $(r,2^{-k})$-double points, we can find a configuration among them that is in our chosen set.

Let $\stepsize > 0$. For $i = 1,\ldots,\lfloor\stepsize^{-1}\rfloor$, let $\CA_{i,m}$ be the set of $\ul{z} = (z_1,\ldots,z_m)$ such that $2^{-i\stepsize k} \le \abs{z_u-z_v} \le 2^{-(i-1)\stepsize k}$ for each pair $u \neq v$.

\begin{lemma}\label{le:double_bound_integrated}
 Let $c_0 > 0$. There exists $\zeta > 0$ (depending on $c_0$) such that the following is true. Suppose $\innexp,\stepsize > 0$ are sufficiently small (depending on $c_0$). Let $E_k$ be the event defined in~\eqref{eq:good_merging_everywhere}. Then
 \[
  \max_{i < (1-c_0)\stepsize^{-1}} \p\left[ \bigcup_{\ul{z} \in \CA_{i,m}} F(\ul{z},2^{-i\stepsize k}) \cap E_k \right] \lesssim 2^{2k-\zeta km} .
 \]
\end{lemma}

\begin{proof}
 For fixed $\ul{z} \in \CA_{i,m}$, by Lemma~\ref{le:double_bound_given_pts}, we have
 \[ \p[ F(\ul{z},2^{-i\stepsize k}) \cap E_k ] \lesssim (2^{-k(1-\innexp)+i\stepsize k})^{(\alpha_{4,\kappa}-a')m} . \]
 We need to count the number of elements in $\CA_{i,m}$. Due to the constraint $\abs{z_u-z_v} \le 2^{-(i-1)\stepsize k}$ for each $u,v$, we see that
 \[ |\CA_{i,m}| \lesssim 2^{2k} 2^{(2k-2(i-1)\stepsize k)(m-1)} . \] 
 Since $\alpha_{4,\kappa} > 2$, we can pick $a' > 0$ small enough so that $\alpha_{4,\kappa}-a'-2 > 0$. For $i\stepsize < 1-c_0$ we get
 \[ \begin{split}
  |\CA_{i,m}| (2^{-k(1-\innexp)+i\stepsize k})^{(\alpha_{4,\kappa}-a')m} 
  &\le 2^{2k} 2^{-[(\alpha_{4,\kappa}-a')(1-\innexp)-2 -i\stepsize(\alpha_{4,\kappa}-a'-2)-2\stepsize]km} \\
  &\le 2^{2k} 2^{-[c_0(\alpha_{4,\kappa}-a'-2) -\innexp(\alpha_{4,\kappa}-a') -2\stepsize]km} .
 \end{split} \] 
 When $\innexp,\stepsize > 0$ are sufficiently small, we have $\innexp(\alpha_{4,\kappa}-a') +2\stepsize < c_0(\alpha_{4,\kappa}-a'-2)/2 \eqdef \zeta$. This concludes the proof.
\end{proof}

\begin{proof}[Proof of Lemma~\ref{le:jk-double-points}]
 Throughout the proof, we fix $\stepsize > 0$ sufficiently small depending on $a$ (how small will become apparent later). We can assume that $r > 2^{-k(1-a/4)}$, otherwise the number of points in $\CD_k \cap B(w,2^{-k(1-a/4)})$ is trivially bounded by $2^{ak/2}$.
 
 Let $b>0$ be given. We want to show that
 \[ \p\left[ \bigcup_{w \in B(0,3/4)} \bigcup_{r > 2^{-k}} \{N^{w}_{r,2^{-k}} \geq 2^{a k}\} \right] = O(2^{-bk}) . \]
 Let $E_k$ be the event defined in~\eqref{eq:good_merging_everywhere}. Pick $M,p>0$ so that $\p[E_k^c] = O(e^{-bk})$. Then it suffices to show $\p[\cdot \cap E_k] = O(2^{-bk})$.
 
 For $i=1,\ldots,\lfloor(1-a/4)\stepsize^{-1}\rfloor$, let
 \[ F_i = \bigcup_{\ul{z} \in \CA_{i,k}} F(\ul{z},2^{-i\stepsize k}) . \]
 Then $\p[F_i \cap E_k] \lesssim 2^{2k-\zeta k^2} \lesssim 2^{-\zeta' k^2}$ by Lemma~\ref{le:double_bound_integrated}.
 
 Suppose that there are $2^{ak}$ many $(r,2^{-k})$-double points in some $B(w,r) \cap \CD_k$. We claim that we can find a scale $i$ so that $F_i$ occurs. This will conclude the proof of the lemma.
 
 For each $i$, we divide the set $B(0,3/4)$ into $2^{2i\stepsize k}$ boxes of radius $2^{-i\stepsize k}$. Let $i_r = \lfloor \log_2(r^{-1})/(\stepsize k) \rfloor$. Suppose that there are $2^{ak}$ many $(r,2^{-k})$-double points in one of the boxes $Q_{i_r}$ of radius $2^{-i_r\stepsize k}$.
 
 We distinguish between two cases. Consider the boxes of radius $2^{-(i_r+1)\stepsize k}$ inside $Q_{i_r}$. In case at least $k$ of them contain a $(r,2^{-k})$-double point, then we see that the event $F_{i_r+1}$ occurs. If not, then there must be a box $Q_{i_r+1} \subseteq Q_{i_r}$ of radius $2^{-(i_r+1)\stepsize k}$ that contains at least $2^{ak}/k$ many $(r,2^{-k})$-double points. We then attempt to find at least $k$ boxes of radius $2^{-(i_r+2)\stepsize k}$ each containing a $(r,2^{-k})$-double point. Inductively, if a box $Q_{i-1} \subseteq B(0,3/4)$ of radius $2^{-(i-1)\stepsize k}$ contains at least $2^{ak}/k^{i-1}$ many $(r,2^{-k})$-double points, then among the $2^{2\stepsize k}$ many sub-boxes of radius $2^{-i\stepsize k}$ within $Q_{i-1}$, either $k$ of them contain a $(r,2^{-k})$-double point, or there is a sub-box $Q_i \subseteq Q_{i-1}$ of radius $2^{-i\stepsize k}$ that contains at least $2^{ak}/k^{i}$ many $(r,2^{-k})$-double points. Finally, if $i\stepsize > 1-a/4$, then a box of radius $2^{-(1-a/4)k}$ can contain at most $2^{(a/2)k}$ (double) points in $\CD_k$. In particular, it cannot contain more than $2^{ak}/k^{i}$ many. We conclude that $F_i$ must have occurred for some $i=1,\ldots,\lfloor(1-a/4)\stepsize^{-1}\rfloor$.
\end{proof}

\begin{figure}[ht]
\centering
\includegraphics[width=0.3\textwidth]{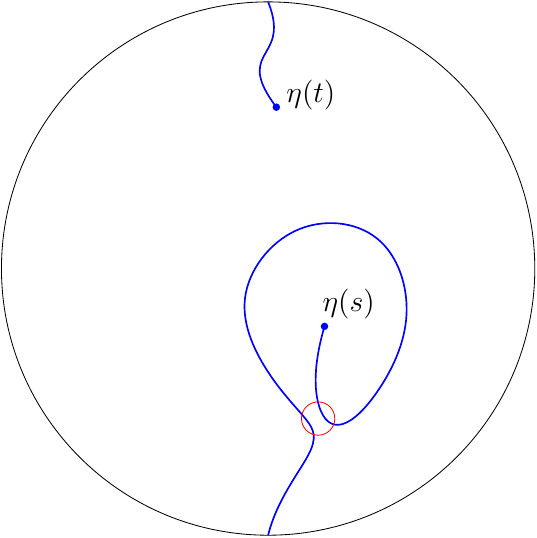}\hspace{0.04\textwidth}\includegraphics[width=0.3\textwidth]{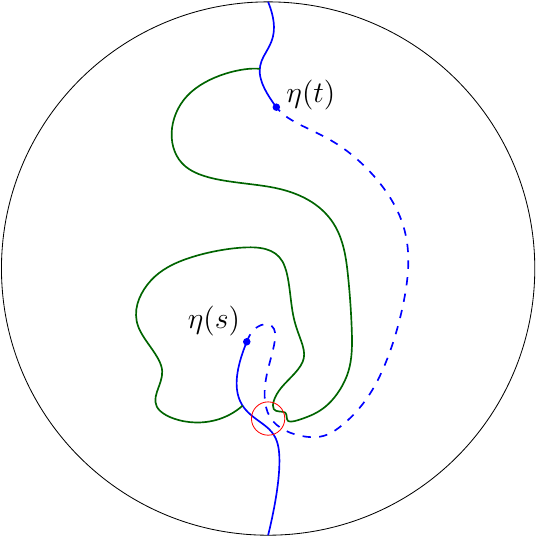}\hspace{0.04\textwidth}\includegraphics[width=0.3\textwidth]{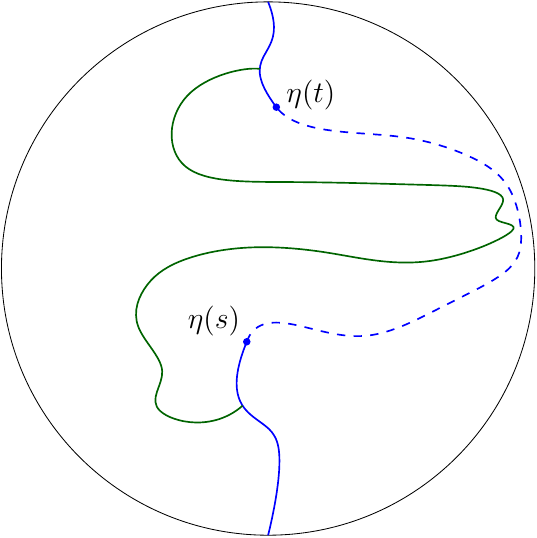}
\caption{We illustrate some of the possible events that produce bottlenecks in the domain $D_{s,t,A}$ in Proposition~\ref{pr:regularity}. These are the places where the left and the right boundary of $D_{s,t,A}$ are very close to another. The boundary of $A$ is depicted in green. In the first two depicted scenarios, $\eta$ makes an approximate double point. In the third depicted scenario (which we do not consider in the statement of Proposition~\ref{pr:regularity}), $\eta$ needs to pass close to the boundary of $\D$.}
\label{fig:bad-domains-examples}
\end{figure}

\begin{proof}[Proof of Proposition~\ref{pr:regularity}] 
 We are going to deduce the result from Lemma~\ref{le:jk-double-points}.
 
 Let $k_0 \in \N$ such that $M \asymp 2^{2k_0}$. Let $E$ be the event that for every $k \ge k_0$ and $n < k$ the number of $(2^{-n},2^{-k})$-double points for $\eta$ in any ball $B(w,2^{-n}) \cap \CD_k$ is at most $2^{ak}$. By Lemma~\ref{le:jk-double-points} we have $\p[E^c] = o^\infty(2^{-k_0}) = o^\infty(M^{-1})$.
 
 We argue that $E \subseteq G_{M,a}$. First note that for $k \le k_0$ the number of points in $\CZ_k$ is trivially bounded by $O(2^{2k_0}) = O(M)$. 
 
 Suppose that $0 \le s < t \le \infty$ and $A \subseteq D^L$ as in the statement of the proposition. Let $z$ be a point on the hyperbolic geodesic in $D_{s,t,A}$ from $\eta(s)$ to $\eta(t)$, and suppose that $\dist(z,\partial D_{s,t,A}) \in [2^{-k-1},2^{-k}]$. By Lemma~\ref{lem:hyperbolic_geodesic_facts}, it has distance at most $8\cdot 2^{-k}$ to both the left and right boundary of $D_{s,t,A}$. In particular, $\eta[s,t] \cap B(z,8\cdot 2^{-k}) \neq \varnothing$. The right boundary of $D_{s,t,A}$ consists of the right arc of $\partial\D$ and the right sides of $\eta[0,s]$, $\eta[t,\infty]$. Since we are only considering points within $B(0,3/4)$, we must have $\dist(z,\eta[0,s]) \le 8\cdot 2^{-k}$ or $\dist(z,\eta[t,\infty]) \le 8\cdot 2^{-k}$. Let us assume the former (the latter case is exactly the same by symmetry).
 
 Let $n \in \N$ such that $\abs{z-\eta(s)} \in [2^{-n-1},2^{-n}]$. Let $\wt{z} \in \CD_{k+2}$ be the point closest to $z$. Then $\wt{z}$ is a $(2^{-n-1},2^{-k+4})$-double point for $\eta$. Moreover, if $z = \varphi(i2^{j})$, $z' = \varphi(i2^{j'})$ where $j \neq j' \in \CZ^{B(0,3/4)}_k$, then $\abs{z-z'} \gtrsim 2^{-k}$ due to Koebe's $1/4$-theorem. In particular, on the event that there are at most $2^{ak}$ many $(2^{-n},2^{-k})$-double points in $B(\eta(s),2^{-n}) \cap \CD_k$, we can have at most $O(2^{ak})$ many points $j \in \CZ^{B(0,3/4)}_k$ with $\varphi(i2^j) \in A(\eta(s),2^{-n-1},2^{-n})$. We conclude by summing over $n=1,\ldots,k$.
\end{proof}

\subsection{Multiple crossings of \clekp{} exterior boundaries}
\label{se:cle_outer_boundary_exponent}

\begin{lemma}
\label{lem:thin_cle_lwb}
Let $\Gamma$ be a nested $\CLE_{\kappa'}$ in $\D$.  For $0 < \epsilon < r < 1$ let $E_{r,\epsilon}^4$ be the event that there exists $z \in B(0,1-r)$ and a loop of~$\Gamma$ whose exterior boundary makes~$4$ crossings of the annulus $A(z,\epsilon,r)$.  Then $\p[E_{r,\epsilon}^{4}] = O(\epsilon^{\alpha_{4,\kappa}-2+o(1)}r^{-\alpha_{4,\kappa}})$ as $\epsilon \to 0$ where $\alpha_{4,\kappa} > 2$ is as in~\eqref{eqn:double_exponent_simple}.
\end{lemma}

The difficulty in proving Lemma~\ref{lem:thin_cle_lwb} is that the exterior boundary of a $\CLE_{\kappa'}$ loop does not correspond to a flow line of the GFF.  On the other hand, it does hold that the boundary of the complementary components cut out by a $\CLE_{\kappa'}$ loop do correspond to GFF flow lines.  The idea of the following proof is to use the inversion symmetry of whole-plane nested $\CLE_{\kappa'}$ proved in \cite{gmq2021sphere} in order to convert the exterior boundary of a loop into an interior boundary and then apply Proposition~\ref{pr:4arm_simple} to the corresponding flow line.

\begin{lemma}
\label{lem:thin_cle_wp}
Let $\Gamma$ be a whole-plane nested $\CLE_{\kappa'}$.  For $0 < \epsilon < 1$ let $E_{\epsilon}^{4,0}$ be the event that~$\Gamma$ contains a loop surrounding $0$ whose exterior boundary makes~$4$ crossings of the annulus $A(0,\epsilon,1)$.  Then $\p[E_{\epsilon}^{4,0}] = O(\epsilon^{\alpha_{4,\kappa}+o(1)})$ as $\epsilon \to 0$ where $\alpha_{4,\kappa}$ is as in~\eqref{eqn:double_exponent_simple}.
\end{lemma}

\begin{proof}
We apply the inversion map $1/z$ and let $\wh{\Gamma}$ be the image of $\Gamma$.  By \cite{gmq2021sphere} we have that $\wh{\Gamma}$ is a whole-plane nested $\CLE_{\kappa'}$.  If $E_{\epsilon}^{4,0}$ occurs for $\Gamma$, then there exists a loop $\wh{\CL}$ of $\wh{\Gamma}$ that surrounds $0$ and such that if $V$ is the component of $\C \setminus \wh{\CL}$ containing $0$, then $\partial V$ makes $4$ crossings across $A(0,1,1/\epsilon)$. Let $\eta'$ be the space-filling $\SLE_{\kappa'}$ in $\C$ from $\infty$ to $\infty$ through $0$ associated with $\wh{\Gamma}$.\footnote{To see that the space-filling \slekp{} associated to a whole-plane \clekp{} is well-defined, we can argue as follows. By repeatedly applying the domain Markov property, we see that a.s.\ for every $R>0$ there is a loop of $\wh{\Gamma}$ contained in $\C\setminus B(0,R)$ that does not surround $0$ but disconnects $0$ from $\infty$. For such a loop, the first point on the interior boundary drawn by the CLE exploration is determined (regardless of which point on the exterior boundary of the loop is visited first by the exploration), and so is the entire exploration inside the interior component.} Let $h$ be the whole-plane GFF modulo additive constant $2\pi\chi$ coupled with $\eta'$. We would like to deduce from Proposition~\ref{pr:4arm_fl} that the probability of this event is $O( \epsilon^{\alpha_{4,\kappa} + o(1)})$.

Strictly speaking, Proposition~\ref{pr:4arm_fl} applies only when the $4$ crossings of the interior boundary occur in the specified order ``in-out-in-out'', but we can use the rerooting invariance of \clekp{} together with a resampling argument similar to the proof of Proposition~\ref{pr:4arm_fl} to overcome this constraint.

Fix $a>0$ small. Suppose that $\wh{\CL} \in \wh{\Gamma}$ and $V$ is the component of $\C \setminus \wh{\CL}$ containing $0$ such that $\partial V$ makes $4$ crossings across $A(0,\epsilon,1)$. Let $x_0 \in \wh{\CL}$ be the first point visited by the exploration path, and let $y_0 \in \wh{\CL}$ be the first point visited on $\partial V$. In case $y_0 \notin B(0,\epsilon^a)$, then we are in the situation of Proposition~\ref{pr:4arm_fl} and conclude immediately. Suppose we are on the event that $x_0,y_0 \in B(0,\epsilon^a)$. The case when $x_0 \notin B(0,\epsilon^a)$ is similar and simpler.

\begin{figure}[ht]
\centering
\includegraphics[width=0.45\textwidth]{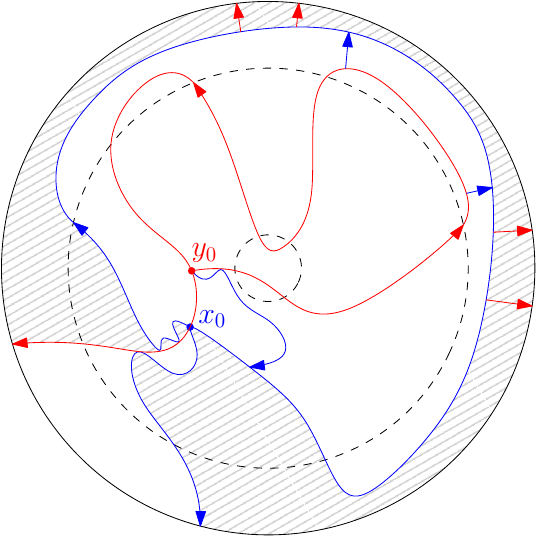}
\caption{Illustration of the resampling event used in the proof of Lemma~\ref{lem:thin_cle_wp}. The red (resp.\ blue) flow lines have angle $0$ (resp.\ $-\pi$). The region bounded by the shaded parts is the resampled version of $U$.}
\label{fi:interior_bdry_rerooting}
\end{figure}

We will now use some notation from Section~\ref{subsec:ig}, we refer to Section~\ref{se:fl_merging} for details and for the statement of Lemma~\ref{le:good_scales_merging_multiple} which we will use shortly. Consider a variant of the event in Lemma~\ref{le:good_scales_merging_multiple} where in the event $G$ we consider for each choice of strands $\eta_{w_1;\mathrm{out}},\eta_{w_2;\mathrm{in}},\eta_{w_3;\mathrm{out}},\eta_{w_4;\mathrm{in}},\eta_{w_5;\mathrm{out}}$ of $X^0_{0,1}$ alternatingly ending on $\partial B(0,3/4),\partial B(0,1/4)$ and each choice of strands $\eta^\pi_{z_1;\mathrm{out}},\eta^\pi_{z_2;\mathrm{in}},\eta^\pi_{z_3;\mathrm{out}}$ of $X^\pi_{0,1}$ alternatingly ending on $\partial B(0,3/4),\partial B(0,1/4)$, the conditional probability that the following occurs is either $0$ or at least $p$:
\begin{itemize}
 \item The continuation of $\eta_{w_1;\mathrm{out}}$ (resp.\ $\eta_{w_3;\mathrm{out}}$) merges into $\eta_{w_2;\mathrm{in}}$ (resp.\ $\eta_{w_4;\mathrm{in}}$) before entering $B(0,1/4)$, and the continuation of $\eta_{w_5;\mathrm{out}}$ hits the clockwise arc of $\partial\D$ between $-i$ and $-1$ without entering $B(0,1/4)$.
 \item The continuation of $\eta^\pi_{z_1;\mathrm{out}}$ merges into $\eta^\pi_{z_2;\mathrm{in}}$ before entering $B(0,1/4)$, and the continuation of $\eta^\pi_{z_3;\mathrm{out}}$ hits the clockwise arc of $\partial\D$ between $-i$ and $-1$ without entering $B(0,1/4)$.
 \item There are points $v_1,v_2,v_3,v_4$ on the continuation of $\eta^\pi_{z_1;\mathrm{out}}$ such that if we let $\eta^L_{v_i}$ be the corresponding angle $0$ flow lines starting from $v_i$ in the component to the left, then they hit the clockwise arcs of $\partial\D$ from, respectively, $-1$ to $i$, $i$ to $1$, $i$ to $1$, and $1$ to $-i$ before entering $B(0,1/4)$, and none of them merge with each other before.
 \item There are points $u_1,u_2$ on the continuation of $\eta_{w_1;\mathrm{out}}$ such that if we let $\eta^R_{u_1}$ (resp.\ $\eta^R_{u_2}$) be the corresponding $-\pi$ flow line starting from $u_1$ (resp.\ $u_2$) in the component to the right, then it hits the segment of the continuation of $\eta^\pi_{z_1;\mathrm{out}}$ between $v_2$ and $v_3$ before entering $B(0,1/4)$, and is contained in the sector of $\D$ with angles between $0$ and $\pi/8$ (resp.\ $3\pi/8$ and $\pi/2$).
\end{itemize}
See Figure~\ref{fi:interior_bdry_rerooting} for an illustration.

As in the proof of Proposition~\ref{pr:4arm_fl}, we select a scale $J \in \{1,\ldots, \lfloor a\log_2(\epsilon^{-1})\rfloor \}$ uniformly at random. On an event $F$ with $\p[F^c] = O(\epsilon^{\alpha_{4,\kappa}})$, the conditional probability is at least $1/2$ that we select a scale where $G_{0,2^{-J}}$ (defined analogously as in Lemma~\ref{le:good_scales_merging_multiple}) occurs.

Suppose that $\wt{h}_{0,2^{-J}}$ (defined in Section~\ref{se:gff}) has boundary values compatible with a \clekp{} exploration rooted at $-i2^{-J}$, and its values in $B(0,(3/4)2^{-J})$ agree with those of $h$ modulo $2\pi\chi$. Suppose we are on the event above where $x_0,y_0 \in B(0,2^{-J}/4)$. Let $U$ be the remaining connected component containing $0$ at the time $\eta'$ first visits $\wh{\CL}$ at $x_0$. On the event $G_{0,2^{-J}}$ we can resample the connections of $\partial U$ outside $B(0,(3/4)2^{-J})$ as illustrated in Figure~\ref{fi:interior_bdry_rerooting}. Similarly, we can resample the connections of $\partial V$ outside $B(0,(3/4)2^{-J})$. Note also that the right boundary of $\eta'$ upon completing the loop $\partial V$ at $y_0$ does not hit the segment of its left boundary between $x_0$ and $y_0$ (otherwise it would branch instead of completing the same loop $\wh{\CL}$), and this property is preserved upon resampling. We conclude that with conditional probability at least $p$ the resampling achieves a configuration as depicted in Figure~\ref{fi:interior_bdry_rerooting}.

On the event depicted in Figure~\ref{fi:interior_bdry_rerooting}, if we first explore the \clekp{} loops intersecting the counterclockwise arc of $\partial\D$ from $1$ to $i$, and then consider the exploration of the remainder (upon mapping back to $\D$) rooted at $e^{i\pi/4}$ with force point at $e^{i\pi/4-}$, then we see that the first point of the interior boundary of (the resampled version $\wt{\CL}$ of) $\wh{\CL}$ visited by the exploration is on the segment between $u_1$ and $u_2$ as in the definition of the event $G$ above. In particular, it is outside $B(0,1/4)$, so that the crossings of the interior boundary are traversed in the order ``in-out-in-out'' as in Proposition~\ref{pr:4arm_fl}. We now can bound the probability of this by Proposition~\ref{pr:4arm_fl}, concluding the proof.
\end{proof}

\begin{proof}[Proof of Lemma~\ref{lem:thin_cle_lwb}]
Let $\wt{\Gamma}$ be a whole-plane nested $\CLE_{\kappa'}$ and let $\wt{\CL}_0$ be the outermost loop that intersects $\D$. Let $\wt{D}$ be the component of $\C \setminus \wt{\CL}_0$ containing $0$, and let $\wt{E}_{r,\epsilon}^4$ be the event that $\partial\wt{D} \subseteq A(0,1/2,1)$ and there exists $z \in \wt{D}$ and a loop $\wt{\CL}$ of $\wt{\Gamma}$ whose exterior boundary makes $4$ crossings across $A(z,\epsilon,r)$. By applying a conformal map $\wt{D} \to \D$ (on the event $\partial\wt{D} \subseteq A(0,1/2,1)$ which has positive probability and is independent of $E_{r,\epsilon}^4$), we see that there exists a constant $c > 1$ so that
\[ \p[ E_{r,\epsilon}^4] \leq c\,\p[ \wt{E}_{r/c, c\epsilon}^4].\]
Therefore it suffices to prove the upper bound for $\p[ \wt{E}_{r,  \epsilon}^4]$ in place of $\p[ E_{r,\epsilon}^4]$.

Let $\wt{\eta}'$ be the space-filling $\SLE_{\kappa'}$ in $\wt{D}$ associated with the \clekp{} within $\wt{D}$. Fix $a > 0$ and let $\wt{G}$ be the event that whenever $\wt{\eta}'$ travels distance $\epsilon$ while being distance $r$ away from $\partial\wt{D}$, it fills a ball of diameter at least $\epsilon^{1+a}$.  By Lemma~\ref{le:fill_ball} we have $\p[\wt{G}^c] = o^\infty(\epsilon)$.  

Suppose that we are on the event $\wt{E}_{r,  \epsilon}^4 \cap \wt{G}$ and let $\wt{\CL}$ be a loop of $\wt{\Gamma}$ whose exterior boundary makes $4$ crossings across $A(z,\epsilon,r)$.  Since we are working on $\wt{G}$, we know that when $\wt{\CL}$ travels across $A(z,\epsilon,2\epsilon)$, it surrounds a ball $B$ of diameter $\epsilon^{1+a}$. Suppose that $w \in \D$ is sampled from Lebesgue measure independently of $\wt{\Gamma}$.  Then on $\wt{E}_{r,  \epsilon}^4 \cap \wt{G}$ the conditional probability that $w$ is sampled in such a ball $B$ is at least $\epsilon^{2+2a}$. Since the translation $\wt{\Gamma}-w$ has the same law as $\wt{\Gamma}$, it follows that the following is true.  Let $\wt{E}_{r,\epsilon}^{4,0}$ be the event that there is a loop of $\wt{\Gamma}$ surrounding $0$ whose exterior boundary makes $4$ crossings across $A(0,\epsilon,r)$, then
\[ \p[\wt{E}_{r,\epsilon}^4] \lesssim \epsilon^{-2-2a} \p[\wt{E}_{r,2\epsilon}^{4,0}] + o^\infty(\epsilon).\]
By Lemma~\ref{lem:thin_cle_wp} we have $\p[\wt{E}_{r,\epsilon}^{4,0}] = O((\epsilon/r)^{\alpha_{4,\kappa}+o(1)})$. Since this holds for any given choice of $a>0$, this concludes the proof.
\end{proof}

\bibliographystyle{alpha}
\providecommand{\noopsort}[1]{}

\end{document}